\documentclass{amsart}

\title[Some vanishing results for the rational completed cohomology]{Some vanishing results for the rational completed cohomology of Shimura varieties}

\author{Kai-Wen Lan and Lue Pan}

\address{University of Minnesota, 127 Vincent Hall, 206 Church Street SE, Minneapolis, MN 55455, USA}
\email{kwlan@umn.edu}

\address{University of Michigan, 2074 East Hall, 530 Church Street, Ann Arbor, MI 48109, USA}
\email{luepan@umich.edu}


\usepackage{amsfonts, latexsym, amssymb, mathrsfs}

\usepackage[driverfallback=hypertex, unicode]{hyperref}
\usepackage{color}

\usepackage[neveradjust]{paralist}

\usepackage{upref}
\usepackage[shortlabels]{enumitem}
\usepackage{verbatim}
\usepackage{xspace}

\usepackage[cmtip, all]{xy}
\usepackage{tikz}
\usepackage{tikz-cd}

\begingroup\expandafter\expandafter\expandafter\endgroup
\expandafter\ifx\csname IncludeInRelease\endcsname\relax
\usepackage{fixltx2e}
\fi

\newtheorem{theorem}[equation]{{Theorem}}
\newtheorem{corollary}[equation]{{Corollary}}
\newtheorem{lemma}[equation]{{Lemma}}
\newtheorem{proposition}[equation]{{Proposition}}
\newtheorem{example}[equation]{{Example}}
\newtheorem{construction}[equation]{{Construction}}
\newtheorem{method}[equation]{{Method}}

\theoremstyle{definition}
\newtheorem{definition}[equation]{{Definition}}
\newtheorem{condition}[equation]{{Condition}}

\theoremstyle{remark}
\newtheorem{remark}[equation]{{Remark}}

\newcommand{\BBdRp}{\bB_\dR^+}
\newcommand{\BBdR}{\bB_\dR}

\newcommand{\OBdlp}{\cO\bB_{\dR, \log}^+}
\newcommand{\OBdl}{\cO\bB_{\dR, \log}}

\newcommand{\Ddlalg}{D_{\dR, \log}^\alg}

\newcommand{\ho}{\widehat{\otimes}}

\newcommand{\GrSh}[1]{\underline{#1}}

\newcommand{\bA}{\mathbb{A}}
\newcommand{\bB}{\mathbb{B}}
\newcommand{\bC}{\mathbb{C}}

\newcommand{\bF}{\mathbb{F}}
\newcommand{\bG}{\mathbb{G}}
\newcommand{\bH}{\mathbb{H}}

\newcommand{\bL}{\mathbb{L}}
\newcommand{\bM}{\mathbb{M}}

\newcommand{\bP}{\mathbb{P}}
\newcommand{\bQ}{\mathbb{Q}}
\newcommand{\bR}{\mathbb{R}}

\newcommand{\bZ}{\mathbb{Z}}

\newcommand{\cD}{\mathcal{D}}

\newcommand{\cF}{\mathcal{F}}
\newcommand{\cG}{\mathcal{G}}
\newcommand{\cH}{\mathcal{H}}
\newcommand{\cI}{\mathcal{I}}
\newcommand{\cJ}{\mathcal{J}}

\newcommand{\cL}{\mathcal{L}}

\newcommand{\cO}{\mathcal{O}}

\newcommand{\cU}{\mathcal{U}}
\newcommand{\cV}{\mathcal{V}}
\newcommand{\cW}{\mathcal{W}}
\newcommand{\cX}{\mathcal{X}}
\newcommand{\cY}{\mathcal{Y}}
\newcommand{\cZ}{\mathcal{Z}}

\newcommand{\OP}[1]{\operatorname{#1}}

\newcommand{\Ext}{\OP{Ext}}
\newcommand{\Spec}{\OP{Spec}}
\newcommand{\Spa}{\OP{Spa}}
\newcommand{\Proj}{\OP{Proj}}

\newcommand{\Em}{\hookrightarrow}                   % embedding
\newcommand{\Surj}{\twoheadrightarrow}              % surjection
\newcommand{\Mi}{\stackrel{\sim}{\to}}              % mapping isomorphically
\newcommand{\Mapn}[1]{\xrightarrow{#1}}             % mapping with name

\newcommand{\bAi}{{\bA^\infty}}                     % finite adeles
\newcommand{\bAip}{{\bA^{\infty, p}}}               % finite adeles away from p

\newcommand{\Grp}[1]{\mathrm{#1}}                   % algebraic group we consider
\newcommand{\homom}{\varrho}                        % group homomorphism
\newcommand{\Lie}{\OP{Lie}}                         % Lie algebra
\newcommand{\Flalg}{\mathrm{Fl}}                    % flag variety (algebraic)
\newcommand{\Fl}{{\mathcal{F}\ell}}                 % flag variety (p-adic analytic)

\newcommand{\ad}{{\Utext{ad}}}                      % adjoint quotient
\newcommand{\der}{{\Utext{der}}}                    % derived group
\newcommand{\scc}{{\Utext{sc}}}                     % simply-connected cover

\newcommand{\Shdom}{\mathsf{X}}                     % Hermitian symmetric domain
\newcommand{\Dom}{\mathsf{Y}}                       % Hermitian symmetric domain
\newcommand{\Dompt}{y}                              % point
\newcommand{\MaxCpt}{K_\infty}                      % maximal compact (mod center)
\newcommand{\Lquot}{\backslash}                     % left quotient
\newcommand{\LSV}{\mathrm{S}}                       % locally symmetric varieties
\newcommand{\Bd}{\mathsf{F}}                        % boundary component

\newcommand{\hd}{h}                                 % h
\newcommand{\hc}{\mu}                               % Hodge cocharacter

\newcommand{\ReFl}{E}                               % reflex field
\newcommand{\Sh}{\mathrm{Sh}}                       % Shimura variety
\newcommand{\aShEv}{\mathcal{S}\textit{h}}          % adic Shimura variety over E_v
\newcommand{\aSh}{\mathcal{S}}                      % adic Shimura variety over C
\newcommand{\levcp}{K}                              % level defined by open compact
\newcommand{\level}{\varGamma}                      % level in general
\newcommand{\Min}{{\Utext{min}}}                    % minimal compactification
\newcommand{\Tor}{{\Utext{tor}}}                    % toroidal compactification
\newcommand{\TorMap}{\textstyle\oint}               % map from toroidal to minimal

\newcommand{\GL}{\mathrm{GL}}                       % general linear group
\newcommand{\Grpmu}{\boldsymbol{\mu}}               % group scheme symbol for \mu
\newcommand{\Mt}{\mathrm{M}}                        % matrix algebra

\newcommand{\im}{\OP{im}}                           % image
\newcommand{\sgn}{\OP{sgn}}                         % sign

\newcommand{\Id}{\OP{Id}}                           % identity

\newcommand{\alg}{\Utext{alg}}                      % algebraization
\newcommand{\an}{\Utext{an}}                        % analytification
\newcommand{\cons}{\Utext{cons}}                    % constructible
\newcommand{\la}{\Utext{la}}                        % locally analytic
\newcommand{\red}{\Utext{red}}                      % reduced
\newcommand{\std}{\Utext{std}}                      % standard

\newcommand{\Ind}{\OP{Ind}}                         % induced representation
\newcommand{\Res}{\OP{Res}}                         % restriction of scalars
\newcommand{\Sym}{\OP{Sym}}			                % symmetric algebra

\newcommand{\sC}{\mathscr{C}}			            % space of functions

\newcommand{\dualsign}{{\vee}}                      % notation for dual
\newcommand{\dual}[1]{{#1}^\dualsign}               % dual
\newcommand{\Ex}{\wedge}                            % exterior power

\newcommand{\Rep}{\OP{Rep}}                         % representations
\newcommand{\rV}{V}                                 % representation V of G
\newcommand{\rW}{W}                                 % representation W of M

\newcommand{\canext}{\Utext{can}}                   % canonical extension

\newcommand{\pl}{\upsilon}                          % archimedean places
\newcommand{\Pl}{\Upsilon}                          % set of all archimedean places

\newcommand{\SG}{S}                                 % permutation group
\newcommand{\WG}{\mathrm{W}}                        % Weyl group
\newcommand{\RT}{\Phi}                              % set of roots
\newcommand{\wt}{\lambda}                           % weight
\newcommand{\wtalt}{\mu}                            % weight (alternative)
\newcommand{\rt}{\alpha}                            % root
\newcommand{\hsum}{\rho}                            % half-sum of positive roots
\newcommand{\HC}{\gamma}                            % Harish-Chandra isomorphism

\newcommand{\typeA}{\mathrm{A}}                     % type A
\newcommand{\typeB}{\mathrm{B}}                     % type B
\newcommand{\typeC}{\mathrm{C}}                     % type C
\newcommand{\typeD}{\mathrm{D}}                     % type D
\newcommand{\typeE}{\mathrm{E}}                     % type E

\newcommand{\dR}{\Utext{dR}}                        % de Rham
\newcommand{\HT}{\Utext{HT}}                        % Hodge--Tate

\newcommand{\Fil}{\Utext{Fil}}                      % filtration
\newcommand{\gr}{\OP{gr}}

\newcommand{\et}{\Utext{\'et}}                      % etale
\newcommand{\ket}{\Utext{k\'et}}                    % Kummer etale
\newcommand{\proet}{\Utext{pro\'et}}                % pro-etale
\newcommand{\proket}{\Utext{prok\'et}}              % pro-Kummer etale

\newcommand{\etSh}[1]{{}_{\et}\GrSh{#1}}            % etale local system
\newcommand{\ketSh}[1]{{}_{\ket}\GrSh{#1}}          % Kummer etale local system
\newcommand{\proetSh}[1]{{}_{\proet}\GrSh{#1}}      % pro-etale local system
\newcommand{\proketSh}[1]{{}_{\proket}\GrSh{#1}}    % pro-Kummer etale local system

\newcommand{\dRSh}[1]{{}_{\dR}\GrSh{#1}}            % de Rham local system

\newcommand{\adic}[1]{\breve{#1}}                   % decoration for morphisms of adic spaces

\newcommand{\AC}[1]{\overline{#1}}                  % algebraic closure or geometric point

\newcommand{\Utext}[1]{\text{\rm #1}}               % upright math text

\newcommand{\Refenum}[1]{\Pth{\textrm{#1}}}
\newcommand{\Refeq}[1]{\Pth{#1}}

\newcommand{\Pth}[1]{{\rm (}#1{\rm )}}              % parenthesis ( )
\newcommand{\Qtn}[1]{``#1''}                        % quotation `` ''

\newcommand{\apage}{p.\@\xspace}                    % page
\newcommand{\aCh}{Ch.\@\xspace}                     % chapter
\newcommand{\aSec}{Sec.\@\xspace}                   % section
\newcommand{\aDef}{Def.\@\xspace}                   % definition
\newcommand{\aDefs}{Def.\@\xspace}                  % definitions
\newcommand{\aLem}{Lem.\@\xspace}                   % lemma
\newcommand{\aProp}{Prop.\@\xspace}                 % proposition
\newcommand{\aThm}{Thm.\@\xspace}                   % theorem
\newcommand{\aThms}{Thm.\@\xspace}                  % theorems
\newcommand{\aCor}{Cor.\@\xspace}                   % corollary
\newcommand{\aRem}{Rem.\@\xspace}                   % remark
\newcommand{\aEx}{Ex.\@\xspace}                     % example
\newcommand{\resp}{resp.\@\xspace}                  % resp.
\newcommand{\ie}{i.e.\@\xspace}                     % i.e.
\newcommand{\eg}{e.g.\@\xspace}                     % e.g.

\newcommand{\Refcf}{cf.\@\xspace}                   % 'compare' for the purpose of reference

\begin{document}

\begin{abstract}
    Based on an almost Kodaira-type vanishing result in mixed characteristics of Bhatt, we show that, in the locally analytic completed cohomology of a general Shimura variety, sufficiently regular infinitesimal weights can only show up in the middle degree.
\end{abstract}

\maketitle

\tableofcontents

\numberwithin{equation}{subsection}

\section{Introduction}\label{sec-intro}

\subsection{Main result}

Let $p$ be a prime number.  Emerton introduced the \Pth{$p$-adically} completed cohomology of locally symmetric spaces, and described it as \Qtn{a suitable surrogate for a space of $p$-adic automorphic forms} \cite{Emerton:2014-ccplg}.  This paper studies the completed cohomology of Shimura varieties below the middle degrees.  Let us briefly recall Emerton's construction.  Let $(\Grp{G}, \Shdom)$ be a Shimura datum.  For a neat open compact subgroup $\levcp \subset \Grp{G}(\bAi)$ with $\bAi$ denoting the ring of finite adeles, let
\[
    \Sh_{\levcp, \bC}^\an = \Grp{G}(\bQ) \big\Lquot \bigl(\Shdom \times \Grp{G}(\bAi)\bigr) \big/ \levcp
\]
be the \Pth{complex analytic} Shimura variety of level $\levcp$.  Fix an open compact subgroup $\levcp^p$ of $\Grp{G}(\bAip)$, where $\bAip$ is the ring of finite adeles away from $p$.  For each $i \geq 0$, the $i$-th completed cohomology is defined as
\[
    \tilde{H}^i := \varprojlim_n \varinjlim_{\levcp_p} H^i(\Sh_{\levcp^p \levcp_p, \bC}^\an, \bZ / p^n).
\]
It is $p$-adically complete and carries a natural continuous action of $\Grp{G}(\bQ_p)$.  Let $d$ be the dimension of $\Shdom$ as a complex manifold.  Similar to the cohomology of usual automorphic local systems, it is believed \Pth{see \cite{Calegari/Emerton:2012-ccs}} that $i = d$ is the most interesting degree.  Our main result supports this belief in terms of the infinitesimal character of the $\Grp{G}(\bQ_p)$-action.  To state it, we need to introduce some notation.  Let $C$ be the $p$-adic completion of an algebraic closure of $\bQ_p$.  Let $\mathfrak{g} = \Lie \Grp{G}(C)$, and let $Z(U(\mathfrak{g}))$ denote the center of the universal enveloping algebra of $\mathfrak{g}$.  Fix a Cartan subalgebra $\mathfrak{h}$ of $\mathfrak{g}$, denote by $\WG_{\mathfrak{g}}$ the Weyl group of $\mathfrak{g}$ with respect to $\mathfrak{h}$, and identify each character of $Z(U(\mathfrak{g}))$ with a $\WG_{\mathfrak{g}}$-orbit in the weight space $\mathfrak{h}^*$ via the Harish-Chandra isomorphism.  \Pth{Our convention is that the $\WG_{\mathfrak{g}}$-action is centered at zero.}  We say that a character $Z(U(\mathfrak{g})) \to C$ or its kernel \Pth{which defines a $C$-point of $\Spec Z(U(\mathfrak{g}))$} is \emph{sufficiently regular} if, for every weight $\wt'$ in the corresponding $\WG_{\mathfrak{g}}$-orbit $[\wt]$ and every root $\rt$ of $\mathfrak{g}$ with respect to $\mathfrak{h}$, we have
\begin{equation}\label{eq-cond-suff-reg-intro}
    (\wt', \dual{\rt}) \not\in
    \begin{cases}
        \{ 1, 2 \}, & \text{if $\rt$ is a shorter root, as in Definition \ref{def-shorter-rt}}; \\
        \{ 1 \}, & \text{otherwise}.
    \end{cases}
\end{equation}
Note that the shorter root can only appear in $C$-simple factors of $\mathfrak{g}$ of types $\typeB$ and $\typeC$.  When all $C$-simple factors of $\mathfrak{g}$ are of types $\typeA$, $\typeD$, and $\typeE$, an irreducible representation of $\mathfrak{g}$ of highest weight $\lambda$ has sufficiently regular infinitesimal character if and only if $\lambda$ is \emph{regular} in the usual sense; \ie, not fixed by any nontrivial element of $\WG_{\mathfrak{g}}$.  The set of non-sufficiently regular characters underlies a \Pth{reduced} closed subvariety of $\Spec Z(U(\mathfrak{g}))$ and defines an ideal $\cI \subset Z(U(\mathfrak{g}))$.

Let $\tilde{H}_C^i := \tilde{H}^i \ho_{\bQ_p} C$, and denote by $\tilde{H}_C^{i, \la}$ its subspace of $\Grp{G}(\bQ_p)$-locally analytic vectors.  There is a natural Lie algebra action of $\mathfrak{g}$ on $\tilde{H}_C^{i, \la}$ via derivation.

\begin{theorem}\label{thm-main-intro}
    Let $\cI$ be as above.  Then $\cI^n$ annihilates \Pth{\ie, acts by zero on} $\tilde{H}_C^{< d, \la}$, for some $n > 0$.  In particular, if $[\wt]: Z(U(\mathfrak{g})) \to C$ is sufficiently regular, then the $[\wt]$-isotypic part $\tilde{H}_{C, [\wt]}^{< d, \la}$ is zero.
\end{theorem}

\begin{remark}
    In fact, Theorem \ref{thm-main-intro} will follow from a similar vanishing result in Corollary \ref{cor-thm-main} under a slightly weaker regularity condition on the infinitesimal characters of $\Lie \Grp{G}$ \Pth{rather than $\mathfrak{g} = \Lie \Grp{G}(C)$}, called \Qtn{$\hc_\hd$-sufficient regularity}---See Definition \ref{def-suff-reg-wt}\Refenum{\ref{def-suff-reg-wt-g}} for this notion, and see Definition \ref{def-suff-reg-wt-id} for the corresponding improvement of $\cI \subset Z(U(\mathfrak{g}))$.  As the name suggests, this regularity condition will depend on the choice of a Hodge cocharacter $\hc_\hd$.  As we shall see in Example \ref{ex-suff-reg-type-A-1}, the $\hc_\hd$-sufficient regularity condition for a Shimura curve will be weaker than the one for the Hilbert modular variety associated with the same totally real field.
\end{remark}

This result allows us to reprove the vanishing result for cohomology of automorphic local systems obtained by the first-named author in \cite{Lan:2016-vtcac}.  See Remarks \ref{rem-suff-reg-wt-old-vs-new} and \ref{rem-LS-van-reprove} below.  As noted in \cite{Lan:2016-vtcac}, by using automorphic methods, Li and Schwermer \cite{Li/Schwermer:2004-ecag} were able to prove the vanishing result only assuming that the highest weight of the local system is regular.  It is not clear to us whether our method can reprove the same result when $\mathfrak{g}$ has $C$-simple factors of types $\typeB$ and $\typeC$.  On the other hand, our method can also handle \Qtn{big} local systems which are not necessarily finite-dimensional.  See Corollary \ref{cor-van-la-sys}.

We shall explain in Section \ref{sec-hor-act} that the action of $Z(U(\mathfrak{g}))$ on $\tilde{H}_C^{i, \la}$ can be extended to an action of $Z(U(\mathfrak{m}))$, where $\mathfrak{m}$ denotes a Levi subalgebra determined by the Hodge cocharacter, after fixing an isomorphism $\iota: \bC \Mi C$ of fields.  The extra symmetry comes from the \Pth{partial} Sen operators on $\tilde{H}_C^{i, \la}$.  Our actual main result, Theorem \ref{thm-main}, has a more refined version regarding this enlarged symmetry. Very roughly speaking, if we fix an infinitesimal character $[\wt]$, it determines a list of possible Hodge--Tate--Sen weights corresponding to $\iota$, and the \Qtn{sufficiently regular ones} among them will not show up in $\tilde{H}_{C, [\wt]}^{< d, \la}$.

In a follow-up paper \cite{Pan:2025-cchmd}, the second-named author shows that an extension of our method in the case of Hilbert modular varieties allows one to produce ideals of $U(\mathfrak{g})$ \Pth{not just ideals of the center} annihilating $\tilde{H}_C^{< d, \la}$. We believe similar results should hold for more general Shimura varieties.

\begin{remark}
    For each $i \leq d$, there should be a notion of $i$-sufficiently regularity for infinitesimal characters such that $\tilde{H}_{C, [\wt]}^{< i, \la}=0$ when $[\wt]$ is $i$-sufficiently regular.  Our method can give results in this direction, but we shall not purse them in this paper.  Similar to how we work out the sufficiently regular condition in Proposition \ref{prop-key-obs} below, they should be reduced to combinatorial problems in Lie theory.
\end{remark}

\begin{remark}
    Since the pioneering work of Caraiani--Scholze \cite{Caraiani/Scholze:2017-gccus, Caraiani/Scholze:2024-gcnus}, there has been a lot of powerful vanishing results \cite{Koshikawa:2021-gclgs, Hamann/Lee:2023-tvssv, Yang/Zhu:2025-gcsva} concerning the cohomology of Shimura varieties below degree.  There are two major differences between our work and these works.  Firstly, these works allow torsion and hence integral coefficients, while our work is completely rational.  Secondly, the genericity conditions in these works arise from the study of $p$-adic \Pth{or $\bmod~p$} representations of $\Grp{G}(\bQ_l)$ at a place $l \neq p$, while our work studies $p$-adic representations of $\Grp{G}(\bQ_p)$.
\end{remark}

\begin{remark}
    $\tilde{H}^{> d} = 0$.  This is part of the Calegari--Emerton conjecture in \cite{Calegari/Emerton:2012-ccs}, and the first major breakthrough is made by Scholze in \cite{Scholze:2013-pss, Scholze:2015-tclsv}.  Recently, the rational analogue $\tilde{H}_{\bQ_p}^{> d} = 0$ for general Shimura varieties is proved in \cite{RodriguezCamargo:2022-lacc}, and the original $\tilde{H}^{> d} = 0$ for general Shimura varieties is proved in \cite{He:2025-psasv}.  See the references in these paper for important partial results in special cases.
\end{remark}

\subsection{Overview of the proof}

For simplicity of exposition, let us assume that the Shimura variety is compact.  \Pth{Otherwise, we need to introduce toroidal compactifications.}  Let us also assume that the derived group $\Grp{G}^\der$ is simply-connected, and that the center $Z(\Grp{G})$ has the same rank over $\bQ$ and over $\bR$.  Fix a level $\levcp_p$ at $p$, and let $\levcp = \levcp^p \levcp_p$.  After choosing any isomorphism $\bC \Mi C$, consider the associated adic space $\aSh_\levcp$ of the Shimura variety over $C$. The tower of Shimura varieties $\{ \aSh_{\levcp^p \levcp_p'} \}_{\levcp_p' \subset \levcp_p}$ forms a pro-\'etale $\levcp_p$-torsor over $\aSh_\levcp$.  Denote by $\sC^\la(\levcp_p, \bQ_p)$ the space of $\bQ_p$-valued locally analytic functions on $\levcp_p$, viewed as a $\levcp_p$-representation via left translation.  It gives rise to an \Pth{infinite-dimensional} pro-\'etale local system $\proetSh{\sC^\la(\levcp_p, \bQ_p)}$ over $\aSh_{\levcp, \proet}$.  By using the primitive comparison theorem, we obtain a natural isomorphism
\[
    \tilde{H}_C^{i, \la} \cong H_\proet^i(\aSh_\levcp, \proetSh{\sC^\la(\levcp_p, \bQ_p)} \ho_{\bQ_p} \hat{\cO}_{\aSh_{\levcp, \proet}}).
\]

Here comes the first main ingredient of the proof.  Let $\cL$ be an ample line bundle on $\aSh_\levcp$.  We apply a Kodaira-type vanishing result in mixed characteristics of Bhatt's \cite[\aSec 11.2]{Bhatt:2025-apaht} and deduce that
\[
    H_\proet^{< d}(\aSh_\levcp, \proetSh{\sC^\la(\levcp_p, \bQ_p)} \ho_{\bQ_p} \hat{\cO}_{\aSh_{\levcp, \proet}} \otimes_{\cO_{\aSh_\levcp}} \cL^{-1}) = 0.
\]

Now the question is how to untwist this $\cL^{-1}$.  For simplicity, let us temporarily write $\cF = \proetSh{\sC^\la(\levcp_p, \bQ_p)} \ho_{\bQ_p} \hat{\cO}_{\aSh_{\levcp, \proet}}$.  We use an argument essentially due to Beilinson--Bernstein in their work on localization \cite{Beilinson/Bernstein:1981-ldgm}.  Choose a finite-dimensional representation $\rV$ of $\Grp{G}$, and consider
\[
    H_\proet^{< d}(\aSh_\levcp, \cF \otimes_{\cO_{\aSh_\levcp}} \cL^{-1} \otimes \rV) \cong H_\proet^{< d}(\aSh_\levcp, \cF \otimes_{\cO_{\aSh_\levcp}} \cL^{-1}) \otimes \rV = 0,
\]
where the last isomorphism uses $\sC^\la(\levcp_p, \bQ_p) \otimes \rV \cong \sC^\la(\levcp_p, \bQ_p) \otimes \rV$, with $\levcp_p$ acting on every term except for the last $\rV$.  By making suitable choices of $\cL$ and $\rV$ under our simplifying assumptions, we can arrange that $\cF$ appears as a \emph{direct summand} of $\cF \otimes_{\cO_{\aSh_\levcp}} \cL^{-1} \otimes \rV$ on the sufficiently regular locus of $\Spec Z(U(\mathfrak{g}))$.  A more careful study of the $Z(U(\mathfrak{g}))$-action gives our main theorem.

Such an argument is probably unsurprising for people who are familiar with the Beilinson--Bernstein localization: there is a close relationship between twisting of a line bundle and the translation functor \Pth{defined as a direct summand of the functor of tensoring with a finite dimensional representation \cite{Bernstein/Gelfand:1980-tprsl}} in the localization \cite{Beilinson/Ginzrburg:1999-wcfdm}.  We remark that it was first observed in \cite{Pan:2022-laccm} and further generalized in \cite{RodriguezCamargo:2022-lacc} that the sheaf $\cF$ is annihilated by certain first-order differential operators coming from geometric Sen theory, and hence can be regarded after pushing forward along the Hodge--Tate period map as some sort of localization of $\tilde{H}_C^{i, \la}$ on the partial flag variety defined by the Hodge cocharacter.

The paper is organized as follows.  In Section \ref{sec-akv}, we discuss Bhatt's Kodaira-type vanishing theorem, and make a slight generalization using some standard cyclic cover technique. In Section \ref{sec-lsv}, we recollect various standard facts about locally symmetric varieties and their compactifications, and about automorphic vector bundles and their canonical extensions.  In Section \ref{sec-geom-Sen-Sh}, we summarize some results of \cite{RodriguezCamargo:2022-lacc} on locally analytic functions over Shimura varieties at infinite levels.  In Section \ref{sec-trans}, under the simplifying assumptions that the derived group $\Grp{G}^\der$ is simply-connected, and that the center $Z(\Grp{G})$ has the same rank over $\bQ$ and over $\bR$, we carry out the translation functor argument sketched above, and explain where the sufficiently regular condition comes from.  Finally, in Section \ref{sec-van-res}, we explain how to remove the simplifying assumptions by working with towers of closely related locally symmetric varieties.

\subsection*{Notation and conventions}

We shall fix a prime number $p$.  Let $C$ be $p$-adic completion of an algebraic closure of $\bQ_p$, with ring of integers $\cO_C$.  We shall denote by $D(\cO_C)^a$ the almost derived category of quasi-coherent sheaves on $\cO_C$ with almost structure with respect to the maximal ideal of $\cO_C$, as in \cite[\aSec 3.3]{Bhatt:2025-apaht}.

For any algebraic group $\Grp{H}$ over a base field $k$ and any $k$-algebra $l$, we shall denote by $\Rep_l(\Grp{H})$ the category of algebraic representations of $\Grp{H}$ \Pth{or rather its base extension $\Grp{H} \otimes_k l$} over $l$.

For simplicity, we will often omit \Qtn{$\otimes$} when denoting \Pth{possibly negative} tensor powers of line bundles.

\section{An almost Kodaira vanishing theorem of Bhatt}\label{sec-akv}

\subsection{Bhatt's mixed characteristic Kodaira vanishing result}\label{sec-Bhatt}

Recently, Bhargav Bhatt proved an almost Kodaira vanishing result over a mixed characteristic base.  See \cite[\aSec 11.2]{Bhatt:2025-apaht}.  We will state \Pth{a special case of} his result below.  Let us begin by introducing some notation.

Let $X$ be a finitely presented flat $\cO_C$-scheme.  We shall denote by $X_C$ its generic fiber; by $D_\cons^b(X_C, \bZ / p^n)$ the usual constructible derived category of \'etale $\bZ / p^n$-sheaves on $X_C$; by $\widehat{X}$ the $p$-adic completion of $X$, viewed as a $p$-adic formal scheme; and by $\cX$ its adic generic fiber.  We have a natural \Qtn{support} map $\eta': \cX \to X_C$ induced by taking the supports of valuations, and a nearby cycles map $\nu': \cX \to \widehat{X}$ defined by taking the centers of valuations.  We shall denote by $\eta: \cX_\et \to X_{C, \et}$ the morphism of \'etale sites induced by $\eta'$, and by $\nu: \cX_\et \to \widehat{X}$ the composition of $\nu'$ with the projection $\cX_\et \to \cX$.

\begin{theorem}[Bhatt]\label{thm-akv}
    Let $X$ be a flat projective $\cO_C$-scheme of relative dimension $d$, and $L$ an ample line bundle on $X$.  For any $F \in {^p}D_\cons^{\geq 0}(X_C, \bZ / p^n)$ \Pth{\ie, coconnective with respect to the perverse $t$-structure for middle perversity, as in \cite[\S 4]{Beilinson/Bernstein/Deligne/Gabber:2018-FP-2}}, we have
    \[
        R\Gamma\bigl(\cX_\et, (\eta^* F) \otimes_{\bZ_p} \nu^* \widehat{L}^{-1} \bigr) \in D^{\geq d}(\cO_C)^a,
    \]
	where the pullback $\nu^*$ is understood as the pullback of a line bundle with respect to the integral structure sheaf $\cO_{\cX_\et}^+$ on $\cX_\et$; and where $\widehat{L}$ denotes the $p$-adic formal completion of $L$, which is a line bundle on $\widehat{X}$.
\end{theorem}
\begin{proof}
    When $n = 1$, this is Bhatt's \Qtn{reformulation of almost Kodaira vanishing via $\cO_X^+$} in \cite[\aSec 11.2]{Bhatt:2025-apaht}.  The general case then follows by induction on $n \geq 1$.
\end{proof}

\begin{corollary}\label{cor-akv-lc}
    Let $X$ and $L$ be as in Theorem \ref{thm-akv}.  Let $j: U \Em X_C$ be an open immersion from a smooth subvariety.  For every \'etale finite $\bZ / p^n$-locally constant sheaf $F$ on $U$, we have
	\[
        R\Gamma\bigl(\cX_\et, (\eta^* R j_{\et, *} F) \otimes_{\bZ_p} \nu^* \widehat{L}^{-1} \bigr) \in D^{\geq d}(\cO_C)^a.
    \]
\end{corollary}
\begin{proof}
    Since $j: U \to X_C$ is an open immersion, by \cite[\aProp 2.1.6]{Beilinson/Bernstein/Deligne/Gabber:2018-FP-2}, $R j_{\et, *}$ is perverse $t$-left exact.  Therefore, $R j_{\et, *} F[d] \in {^p}D_\cons^{\geq 0}(X_C, \bZ / p^n)$, and we can apply Bhatt's theorem.
\end{proof}

We will need a version of this corollary \Qtn{completely on the generic fiber}.  Let us explain the arguments as follows.

\begin{construction}\label{constr-proet-ls}
    Let $X_C$ be an algebraic variety over $C$, and $j: U \Em X_C$ an open immersion.  Suppose $\levcp$ is a profinite group, and $\tilde{U}$ is a pro-\'etale $K$-torsor over $U$. Concretely, this means that, for every open normal subgroup $H$ of $K$, we have an \'etale $K / H$-torsor $U_H$ over $U$; and for open normal subgroups $H_1$ and $H_2$ of $K$ such that $H_1 \subset H_2$, we have a $K$-equivariant map $U_{H_1} \to U_{H_2}$ which satisfies usual compatibilities when having a tower of open normal subgroups.

    Let $\rV$ be a unitary $p$-adic Banach space representation of $K$.  Then the unit ball $\rV^\circ$ is $p$-adically complete and $K$-stable.  For $n \geq 1$, every vector of $\rV_n := \rV^\circ / p^n \rV^\circ$ is fixed by an open subgroup of $K$.  Hence, we can write
    \[
        \rV_n = \cup_{M \subset \rV_n} \, M,
    \]
    where $M$ runs over all finite $K$-stable subgroups of $\rV_n$, and such $M$'s form a direct system.  For each $M$, the action of $K$ on it factors through a finite quotient $K / H$.  By descending the constant \'etale sheaf associated with $M$ along the $K / H$-torsor $U_H \to U$, we obtain a locally constant \'etale sheaf $\etSh{M}$ on $U$, which is independent of the choice of $H$.  We define an \'etale sheaf
    \[
        \etSh{\rV_n} := \varinjlim_{M \subset \rV_n} \etSh{M},
    \]
    which we view as the \'etale local system associated with $\rV_n$.  Note that the open immersion $j: U \to X_C$ is quasi-compact.  By abuse of notation, let us define $\etSh{M}_{X_C} := R j_{\et, *} \, \etSh{M}$ \Pth{even though this is not a sheaf in general}, and define
    \[
        \etSh{\rV_n}_{X_C} := R j_{\et, *} \, \etSh{\rV_n} \cong \varinjlim_{M \subset \rV_n} R j_{\et, *} \, \etSh{M}.
    \]
    Then $\{ \etSh{\rV_n}_{X_C} \}_n$ form a projective system as $\{ \etSh{\rV_n} \}_n$ does.

    Let $\cX_\proet$ denote its pro-\'etale site, introduced in \cite{Scholze:2013-phtra}.  Let
    \[
        \eta_\proet: \cX_\proet \to \cX_\et \to X_{C, \et}
    \]
    denote the composition of natural projections of sites.  Still by abuse of notation, consider
    \begin{equation}\label{eq-constr-proet-ls-tor}
        \proetSh{\rV_n}_\cX := \eta_\proet^* \, \etSh{\rV_n}_{X_C},
    \end{equation}
    and consider the following derived limit
    \begin{equation}\label{eq-constr-proet-ls}
        \proetSh{\rV^\circ}_\cX := R \lim_n \proetSh{\rV_n}_\cX = R \lim_n \eta_\proet^* R j_{\et, *} \, \etSh{\rV_n}.
    \end{equation}

    Since $\rV^\circ$ is $p$-adically complete and torsionfree, it follows from the construction that there are exact sequences
    \begin{equation}\label{eq-constr-proet-ls-seq}
        0 \to \etSh{\rV_m} \Mapn{\cdot p^n} \etSh{\rV_{n + m}} \to \etSh{\rV_n} \to 0.
    \end{equation}
    By applying $\eta_\proet^* R j_{\et, *}$ to this, and by taking derived limit over $m$, we obtain
    \[
        \proetSh{\rV^\circ}_\cX / p^n \cong \proetSh{\rV_n}_\cX,
    \]
    for each $n \geq 1$, where $/ p^n$ is understood in the derived sense; \ie, $\otimes_{\bZ_p}^\bL (\bZ / p^n)$.
\end{construction}

\begin{remark}
    We emphasize that we are pushing forward along the algebraic open immersion $j$, which is quasi-compact, instead of the induced open immersion $\adic{j}$ between adic spaces, which is not quasi-compact in general.  Let $\eta: \cX_\et \to X_{C, \et}$ and $\eta_U: \cU_\et \to U_\et$ denote the canonical morphisms.  Since $\etSh{M}$ is constructible for each finite $M$, by Huber's comparison theorem \cite[\aThm 3.8.1]{Huber:1996-ERA}, we have
    \[
        \etSh{\rV_n}_\cX := \eta^* \etSh{\rV_n}_{X_C} = \eta^* R j_{\et, *} \, \etSh{\rV_n} \cong \varinjlim_{M \subset \rV_n} \eta^* R j_{\et, *} \, \etSh{M} = \varinjlim_{M \subset \rV_n} R \adic{j}_{\et, *} \eta_U^* \, \etSh{M}.
    \]
    Since $\adic{j}$ might not be quasi-compact and $\etSh{\rV_n}$ might not be constructible either, we do not know whether $\etSh{\rV_n}_\cX$ is equal to $R \adic{j}_{\et, *} \eta_U^* \, \etSh{\rV_n} \cong R \adic{j}_{\et, *} \varinjlim_{M \subset \rV_n} \eta_U^* \, \etSh{M}$.
\end{remark}

\begin{theorem}\label{thm-kv}
    Let $X_C$ be a normal projective variety over $C$ of dimension $d$, and $L_C$ a line bundle on $X_C$.  Suppose that $X_C$ and $L_C$ are defined over a finite extension of $\bQ_p$.  Let $U \subset X_C$ be a smooth open subvariety, $K$ a profinite group, and $\tilde{U}$ a pro-\'etale $K$-torsor of $U$ \Pth{as in Construction \ref{constr-proet-ls}}.  Assume that:
    \begin{equation}\label{eq-thm-kv-cond}
        \parbox{0.85\textwidth}{$L_C^m$ is globally generated, for some $m \in \bZ_{\geq 1}$ \Pth{\ie, $L_C$ is semiample}, whose global sections define a morphism $\phi_m: X_C \to \bP H^0(X_C, L_C^m)$; and the restriction $\phi_m|_U: U \to \bP H^0(X_C, L_C^m)$ is an immersion.}
    \end{equation}
    Then, for every unitary $p$-adic Banach space representation $\rV$ of $K$, we have
    \[
        R\Gamma(\cX_\proet, \proetSh{\rV^\circ}_\cX \ho_{\bZ_p} \eta_\cX^* L_C^{-1}) \in D^{\geq d}(C),
    \]
    where $\eta_\cX: (\cX_\proet, \widehat{\cO}_{\cX_\proet}) \to (X_C, \cO_{X_C})$ is the composition of the natural projection of ringed sites $(\cX_\proet, \widehat{\cO}_{\cX_\proet}) \to (\cX_\et, \cO_{\cX_\et})$ and $\eta$, and where $\ho_{\bZ_p}$ denotes the $p$-adically completed tensor product.  Then we can write:
    \[
        \proetSh{\rV^\circ}_\cX \ho_{\bZ_p} \eta_\cX^* L_C^{-1} \cong \bigl(R\lim_n(\proetSh{\rV^\circ}_\cX \otimes_{\bZ_p} (\widehat{\cO}_{\cX_\proet}^+ / p^n))\bigr) \otimes_{\widehat{\cO}_{\cX_\proet}^+} \eta_\cX^* L_C^{-1}.
    \]
\end{theorem}

The assumption \Refeq{\ref{eq-thm-kv-cond}} in Theorem \ref{thm-kv} holds if $L_C$ is ample.  The essential improvement here is that $L_C$ is no longer assumed to be extendable to a \Pth{semiample} line bundle on some integral model of $X_C$ \Pth{over $\cO_C$}.  The rest of this subsection will be devoted to the proof of this theorem.  Unsurprisingly, we will use some technique of cyclic coverings to reduce it to Bhatt's result.

\begin{construction}[{\Refcf{} \cite[\aProp 4.1.6]{Lazarsfeld:2004-PAG-1}}]\label{constr-cyc-cov}
    Let $X$ be a variety over a field of characteristic zero, and $L$ a line bundle on $X$.  Suppose $m \in \bZ_{\geq 1}$, and $s \in \Gamma(X, L^m)$ is a nonzero section, which defines a divisor $D \subset X$.  Then there exists a finite flat covering $\pi: Y \to X$ of degree $m$, where $Y$ is a scheme such that $L' := \pi^* L$ admits a section $s' \in \Gamma(Y, L')$ satisfying $(s')^m = \pi^*(s)$.  Explicitly, $Y$ is the relative spectrum of the quotient of the quasi-coherent $\cO_X$-algebra
    \[
        \Sym_{\cO_X}(L^{-1}) := \oplus_{i \in \bZ_{\geq 0}} L^{-i} = \cO_X \oplus L^{-1} \oplus L^{-2} \oplus \cdots
    \]
    by the $\cO_X$-ideal generated by $s(x) - x$, for all $x \in L^{-m}$, where $s$ is viewed as a morphism $L^{-m} \to \cO_X$.  By construction, there is a natural isomorphism
    \[
        \pi_* \cO_Y \cong \cO_X \oplus L^{-1} \oplus L^{-2} \oplus \cdots \oplus L^{-m + 1}
    \]
    of $\cO_X$-modules.  Moreover, there is a natural action of $\Grpmu_m$ on $Y$, defined by
    \[
        u \cdot a = u^i a,
    \]
    for all $u \in \Grpmu_m$ and $a \in L^i$, for all $i = 0, -1, \ldots, -m + 1$.  Clearly, $\Grpmu_m$ acts transitively on each fiber of $\pi$.  If $X$ and $D$ are nonsingular, then so is $Y$.
\end{construction}

The following proposition roughly says that it suffices to prove Theorem \ref{thm-kv} after passing to a cyclic cover.

\begin{proposition}\label{prop-cyc-cov-coh-inj}
    Same setup as in Theorem \ref{thm-kv}.  Let $m \in \bZ_{\geq 1}$ and $0 \neq s \in \Gamma(X_C, L_C^m)$, so that we have a cyclic covering $\pi: Y \to X_C$ as in Construction \ref{constr-cyc-cov}.  Let $\cY$ denote the adic space associate with $Y$, and let $\adic{\pi}: \cY \to \cX$ denote the induced morphism.  Let $U' := \pi^{-1}(U)$, and let $j': U' \to Y$ denote the canonical open immersion.  \Pth{We shall use superscripts $'$ to denote objects on $Y$.}  By forming fiber products, the pro-\'etale $K$-torsor $\tilde{U}$ over $U$ pulls back to a pro-\'etale $K$-torsor $\tilde{U}'$ over $U'$.  Let $\etSh{\rV_n}$ over $U_\et$ and $\proetSh{\rV^\circ}_\cX \cong R\lim_n \proetSh{\rV_n}_\cX$ over $\cX_\proet$ be as in Construction \ref{constr-proet-ls}.  By applying Construction \ref{constr-proet-ls} to $j'$, $\tilde{U}'$, and $\rV' = \rV$ instead, we obtain $\etSh{\rV_n'}$ over $U_\et'$ and $\proetSh{\rV^{\prime \circ}}_\cY = R\lim_n \proetSh{\rV_n'}_\cY$ over $\cY_\proet$.  Then $\Grpmu_m$ naturally acts on $\proetSh{\rV^{\prime \circ}}_\cY$ and $\pi^* L_C^{-1}$; and there is a natural isomorphism
    \[
        H^i(\cX_\proet, \proetSh{\rV^\circ}_\cX \ho_{\bZ_p} \eta_\cX^* L_C^{-1}) \cong H^i(\cY_\proet, \proetSh{\rV^{\prime \circ}}_\cY \ho_{\bZ_p} \adic{\pi}_\proet^* \eta_\cX^* L_C^{-1})^{\Grpmu_m},
    \]
    for each $i \geq 0$.  In particular, if $H^i(\cY_\proet, \proetSh{\rV^{\prime \circ}}_\cY \ho_{\bZ_p} \adic{\pi}_\proet^* \eta_\cX^* L_C^{-1}) = 0$, then $H^i(\cX_\proet, \proetSh{\rV^\circ}_\cX \ho_{\bZ_p} \eta_\cX^* L_C^{-1}) = 0$.
\end{proposition}
\begin{proof}
    By the projection formula, we have
    \[
    \begin{split}
        & H^i(\cY_\proet, \proetSh{\rV^{\prime \circ}}_\cY \ho_{\bZ_p} \adic{\pi}_\proet^* \eta_\cX^* L_C^{-1}) \\
        & \cong H^i(\cX_\proet, R \adic{\pi}_{\proet, *}(\proetSh{\rV^{\prime \circ}}_\cY \ho_{\bZ_p} \adic{\pi}_\proet^* \eta_\cX^* L_C^{-1})) \\
        & \cong H^i(\cX_\proet, R \adic{\pi}_{\proet, *}(\proetSh{\rV^{\prime \circ}}_\cY \ho_{\bZ_p} \widehat{\cO}_{\cY_\proet}) \otimes_{\widehat{\cO}_{\cX_\proet}} \eta_\cX^* L_C^{-1}).
    \end{split}
    \]
    Since taking $\Grpmu_m$-invariants is exact over characteristic zero base fields, and it suffices to show that the canonical morphism
    \[
         \proetSh{\rV^\circ}_\cX \ho_{\bZ_p} \widehat{\cO}_{\cX_\proet} \to \bigl(R \adic{\pi}_{\proet, *}(\proetSh{\rV^{\prime \circ}}_\cY \ho_{\bZ_p} \widehat{\cO}_{\cY_\proet})\bigr)^{\Grpmu_m}
    \]
    is an isomorphism.  Let us also abusively denote by $(\,\cdot\,)^{\Grpmu_m}$ the operation of taking derived invariants, even when the coefficients are not over characteristic zero fields.  Since $\pi$ is proper, so is the associated morphism $\adic{\pi}$ of adic spaces, and $R \adic{\pi}_{\proet, *}$ commutes with inverting $p$.  Therefore, it suffices to prove that the integral version
    \[
        \proetSh{\rV^\circ}_\cX \ho_{\bZ_p} \widehat{\cO}_{\cX_\proet}^+ \to \bigl(R \adic{\pi}_{\proet, *}(\proetSh{\rV^{\prime \circ}}_\cY \ho_{\bZ_p} \widehat{\cO}_{\cY_\proet}^+)\bigr)^{\Grpmu_m}
    \]
    is an \emph{almost isomorphism}, which then implies the above assertion by inverting $p$.

    Since $R \adic{\pi}_{\proet, *}$ commutes with derived limits by \cite[\href{https://stacks.math.columbia.edu/tag/0A07}{Tag 0A07}]{Stacks-Project}, since $\proetSh{\rV^\circ}_\cX = R\lim_n \proetSh{\rV_n}_\cX$ and $\proetSh{\rV^{\prime \circ}}_\cY = R\lim_n \proetSh{\rV_n'}_\cY$ by construction, and since $\proetSh{\rV_n}_\cX$ and $\proetSh{\rV_n'}_\cY$ are filtered by copies of $\proetSh{\rV_1}_\cX$ and $\proetSh{\rV_1'}_\cY$, respectively, it suffices to show that the canonical morphism
    \[
        \proetSh{\rV_1}_\cX \otimes_{\bF_p} (\widehat{\cO}_{\cX_\proet}^+ / p) \to \bigl(R \adic{\pi}_{\proet, *}(\proetSh{\rV_1'}_\cY \otimes_{\bF_p} (\widehat{\cO}_{\cY_\proet}^+ / p))\bigr)^{\Grpmu_m}
    \]
    is an almost isomorphism.  Since both $R \adic{\pi}_{\proet, *}$ and tensor products commute with filtered colimits because $\pi$ is proper, by writing $\rV_1' = \rV_1$ as a union of $K$-stable finite-dimensional $\bF_p$-subspaces, we may assume from now on that $\rV_1' = \rV_1$ is finite.

    Let $\lambda: \cX_\proet \to \cX_\et$ denote the canonical projection of sites, and let $\cO_{\cX_\et}^+ / p$ denote the mod $p$ structure sheaf on $\cX_\et$.  Then we have $\proetSh{\rV_1}_\cX \cong \lambda^* \, \etSh{\rV_1}_\cX$ and $\widehat{\cO}_{\cX_\proet}^+ / p \cong \lambda^*(\cO_{\cX_\et}^+ / p)$, by \cite[\aLem 4.2(iii)]{Scholze:2013-phtra}.  Similar statements hold over $\cY$.  By \cite[\aCor 3.17]{Scholze:2013-phtra}, it suffices to show that the \'etale version
    \[
        \etSh{\rV_1}_\cX \otimes_{\bF_p} (\cO_{\cX_\et}^+ / p) \to \bigl(R \pi_{\et, *}(\etSh{\rV_1'}_\cY \otimes_{\bF_p} (\cO_{\cY_\et}^+ / p))\bigr)^{\Grpmu_m}
    \]
    is an almost isomorphism.  Let $\eta': \cY_\et \to Y_\et$ denote the canonical projection of sites.  By construction, $\etSh{\rV_1'}_\cY \cong \eta^{\prime *} \, \etSh{\rV_1'}_Y$, where $\etSh{\rV_1'}_Y := R j_{\et, *}' \, \etSh{\rV_1'}$ is a constructible $\bF_p$-sheaf on $Y_\et$ by our assumption that $\rV_1' = \rV_1$ is finite.  Since $\pi$ is finite, by the primitive comparison theorem \cite[\aThm 3.13]{Scholze:2013-pss} \Pth{although the almost version is sufficient for our purpose, see also \cite[\aCor 7.2.9]{Zavyalov:2025-ACS} for a non-almost version in this case}, the canonical morphism
    \[
        (R \adic{\pi}_{\et, *} \, \etSh{\rV_1'}_\cY) \otimes_{\bF_p} (\cO_{\cX_\et}^+ / p) \to R \adic{\pi}_{\et, *}(\etSh{\rV_1'}_\cY \otimes_{\bF_p} (\cO_{\cY_\et}^+ / p))
    \]
    is an almost isomorphism.  Thus, it remains to show that the canonical morphism
    \[
        \etSh{\rV_1}_\cX \to (R \adic{\pi}_{\et, *} \etSh{\rV_1'}_\cY)^{\Grpmu_m}
    \]
    is an isomorphism.  Write $\pi' = \pi|_{U'}: U' \to U$, so that $\pi \circ j' = j \circ \pi': U' \to X_C$.  By Huber's comparison theorem \Pth{see \cite[\aProp 2.1.4 and \aThm 3.8.1]{Huber:1996-ERA}},
    \[
        R \adic{\pi}_{\et, *} \etSh{\rV_1'}_\cY = R \adic{\pi}_{\et, *} \eta^{\prime *} R j_{\et, *}' \, \etSh{\rV_1'} \cong \eta^* R \pi_{\et, *} R j_{\et, *}' \, \etSh{\rV_1'} \cong \eta^* R j_{\et, *} R \pi_{\et, *}' \, \etSh{\rV_1'}.
    \]
    Since $\etSh{\rV_1}_\cX = \eta^* R j_{\et, *} \, \etSh{\rV_1}$, it suffices to show that the canonical morphism
    \[
        \etSh{\rV_1} \to (R \pi_*' \, \etSh{\rV_1'})^{\Grpmu_m}
    \]
    over $U_\et$ is an isomorphism.  By construction, $\rV' = \rV$ and $\etSh{\rV_1'} \cong \pi_\et^{\prime *} \, \etSh{\rV_1}$.  Let $\etSh{\bF_p}_U$ \Pth{\resp $\etSh{\bF_p}_{U'}$} denote the constant sheaves associated with $\bF_p$ over $U$ \Pth{\resp $U'$}.  Since $\pi'$ is finite, by \cite[\href{https://stacks.math.columbia.edu/tag/0A4K}{Tag 0A4K}]{Stacks-Project} and the projection formula \cite[\href{https://stacks.math.columbia.edu/tag/0GL5}{Tag 0GL5}]{Stacks-Project}, $(\pi_{\et, *}' \, \etSh{\bF_p}_{U'}) \otimes_{\bF_p} \etSh{\rV_1} \cong \pi_{\et, *}' \pi_\et^{\prime *} \, \etSh{\rV_1} \cong R \pi_{\et, *}' \pi_\et^{\prime *} \, \etSh{\rV_1}$.  Thus, to complete the proof, it suffices to note that the canonical morphism $\etSh{\bF_p}_U \to \pi_{\et, *}' \, \etSh{\bF_p}_{U'}$ is an isomorphism, because $\Grpmu_m$ acts transitively on each fiber of $\pi$, by construction.
\end{proof}

\begin{proof}[Proof of Theorem \ref{thm-kv}]
    First assume the following:
    \begin{equation}\label{eq-thm-kv-very-amp}
        \parbox{0.85\textwidth}{There is a finite morphism $\phi: X_C \to \bP_C^N$ such that $\phi^* \cO_{\bP_C^N}(1) \cong L_C$.}
    \end{equation}
    We claim that there exists a finitely presented flat $\cO_C$-integral model $X$ of $X_C$, and $L_C$ can be extended to an ample line bundle $L$ on $X$.  To see this, since $X_C$ and $L_C$ are defined over a finite extension $E$ of $\bQ_p$, we may assume that $X_C \cong X_E \otimes_E C$ for some variety $X_E$ over $E$, and that $L_C$ comes from an ample line bundle $L_E$ over $X_E$.  It suffices to produce integral models of $X_E$ and $L_E$ \Pth{over $\cO_E$}.

    It follows from our assumption that there is finite morphism $\phi_E: X_E \to \bP_E^N$.  Consider the standard integral model of $\bP^N_{\cO_E}$ of $\bP^N_E$, and take its normalization $X_{\cO_E}$ in $X_E$, as in \cite[\href{https://stacks.math.columbia.edu/tag/035H}{Tag 035H}]{Stacks-Project}.  Since $X_E$ is reduced and $\bP^N_{\cO_E}$ is Nagata \Pth{by \cite[\href{https://stacks.math.columbia.edu/tag/0335}{Tag 0335} and \href{https://stacks.math.columbia.edu/tag/033S}{Tag 033S}]{Stacks-Project}}, the natural morphism $\bar{\phi}: X_{\cO_E} \to \bP_{\cO_E}^N$ is finite \Pth{by \cite[\href{https://stacks.math.columbia.edu/tag/03GH}{Tag 03GH}]{Stacks-Project}}.  In particular, the pullback $L_{\cO_E} := \bar{\phi}^* \cO_{\bP^N_{\cO_E}}(1)$ is ample on $X_{\cO_E}$ and extends $L_E$.  Then we can take $X := X_{\cO_E} \otimes_{\cO_E} C$ and $L := L_{\cO_E} \otimes_{\cO_E} C$.

    Recall that, in the beginning of the section, there are the morphism of ringed sites $\eta': (\cX, \cO_\cX) \to (X_C, \cO_{X_C})$ defined by taking the support of a valuation, and the nearby cycles map $\nu': (\cX, \cO_\cX^+) \to (\widehat{X}, \cO_{\widehat{X}})$ defined by taking the center of a valuation.  Then there is a natural isomorphism
    \begin{equation}\label{eq-isom-inv-sh-cmpl}
        \nu^{\prime *} \widehat{L} \otimes_{\cO_{\cX}^+} \cO_{\cX} \cong \eta^{\prime *} L_C
    \end{equation}
    of invertible $\cO_\cX$-modules on $\cX$.  Indeed, the case of $X = \bP_{\cO_C}^N$ and $L = \cO_X(1)$ can be checked directly.  The general case follows by pullback along $\bar{\phi}$.

    As before, write $\rV_n = \cup_{M \subset \rV_n} \, M$ as a union of finite $K$-stable subgroups.  By Corollary \ref{cor-akv-lc}, for each such $M$, we have
    \[
        R\Gamma(\cX_\et, (\eta^* R j_{\et, *} \, \etSh{M}) \otimes_{\bZ_p} \nu^* \widehat{L}^{-1}) \in D^{\geq d}(\cO_C)^a.
    \]
    By taking the directed limit over all such $M$ and noting that both functors $\eta^*$ and $R j_{\et, *}$ commute with filtered colimits, we obtain
    \[
        R\Gamma(\cX_\et, (\eta^* R j_{\et, *} \, \etSh{\rV_n}) \otimes_{\bZ_p} \nu^* \widehat{L}^{-1}) \in D^{\geq d}(\cO_C)^a.
    \]
    Recall that $\proetSh{\rV^\circ}_\cX = R\lim_n \proetSh{\rV_n}_\cX$, and each $\proetSh{\rV_n}_\cX$ is the pullback of $\eta^* R j_{\et, *} \, \etSh{\rV_n}$ via $\lambda: \cX_\proet \to \cX_\et$ \Pth{see \Refeq{\ref{eq-constr-proet-ls-tor}} and \Refeq{\ref{eq-constr-proet-ls}}}.  By \cite[\aCor 3.17]{Scholze:2013-phtra} and the above, and by taking derived limit over $n$, we obtain
    \[
        R\Gamma(\cX_\proet, \proetSh{\rV^\circ}_\cX \ho_{\bZ_p} \nu_\proet^* \widehat{L}^{-1}) \in D^{\geq d}(\cO_C)^a,
    \]
    where $\nu_\proet := \lambda \circ \nu: (\cX_\proet, \widehat{\cO}_{\cX_\proet}^+) \to (\widehat{X}, \cO_{\widehat{X}})$.  The pullback of the isomorphism \Refeq{\ref{eq-isom-inv-sh-cmpl}} to the pro-\'etale site gives $\nu_\proet^* \widehat{L}^{-1}[\frac{1}{p}] \cong \eta_\cX^* L_C^{-1}$.  Since $\cX$ is quasi-compact,
    \[
        R\Gamma(\cX_\proet, \proetSh{\rV^\circ}_\cX \ho_{\bZ_p} \eta_\cX^* L_C^{-1}) \cong R\Gamma(\cX_\proet, \proetSh{\rV^\circ}_\cX \ho_{\bZ_p} \nu_\proet^* \widehat{L}^{-1})[\tfrac{1}{p}]
    \]
    belongs to $D^{\geq d}(C)$.

    In general, we will use the following lemma to put ourselves in the situation where $\Gamma(X_C, L_C)$ defines a finite morphism of $X_C$ to $\bP_C^N$, for some $N \geq 0$.

    \begin{lemma}\label{lem-BG-cov}
        Let $X$ be a projective variety over a field $E$ of characteristic zero, and $U \subset X$ a smooth open subvariety.  Let $L$ be a line bundle on $X$ such that $L^m$ is globally generated and $\phi_m|_U: U \to \bP H^0(X, L^m)$ is an immersion, for some $m > 0$.  Then there exists a basis $s_0, \cdots, s_N$ of $\Gamma(X, L^m)$ such that, if we denote by $\pi_i: Y_i \to X$ the cyclic cover constructed from $s_i$, as in Construction \ref{constr-cyc-cov}, for $i = 0, \cdots, N$, and by $\pi: Y := Y_0 \times_X Y_1 \times_X \cdots \times_X Y_N \to X$ the fiber product of all $\pi_i$, then $\pi^{-1}(U)$ is smooth and $Y$ is reduced \Pth{\ie, $Y$ is a variety}.
    \end{lemma}

    Assuming this lemma, let us finish the proof of Theorem \ref{thm-kv}.  By assumption, $X_C$ and $L_C$ are defined over a finite extension $E$ of $\bQ_p$.  Let $m$ be as in the theorem, so that we have global sections $s_0, \cdots, s_N$ of $L_C^m$ defined over $E$, and we can construct $Y_0, \cdots, Y_N$ and $\pi: Y \to X_C$ as in the lemma.  Let $L_C' = \pi^* L_C$. By Construction \ref{constr-cyc-cov}, $L_C'$ carries global sections $s_0', \ldots, s_N'$ defined over $E$ such that $(s_i')^m = \pi^* s_i$, for $i= 0, \ldots, N$.  In particular, $s_0', \ldots, s_N'$ induces a morphism $\phi': Y \to \bP_C^N$ which fits into the commutative diagram
    \[
        \xymatrix{ {Y} \ar[r]^-{\phi'} \ar[d]_-\pi & {\bP_C^N} \ar[d]^-{\nu_m} \\
        {X_C} \ar[r]^-{\phi_m} & {\bP_C^N} }
    \]
    in which $\bP_C^N$ is identified with $\bP H^0(X_C, L_C^m)$ using $s_0, \cdots, s_N$ and $\nu_m$ is defined by raising the homogeneous coordinates associated with $s_0, \cdots, s_N$ to their $m$-th powers.  Consider the open immersion $j': U' := \pi^{-1}(U) \to Y$.  Since $\phi_m|_U$ is an immersion and $\nu_m$ is finite, $\phi'|_{U'}$ is quasi-finite.  Consider the Stein factorization of $\phi'$ given by $Y \Mapn{\phi''} Z \Mapn{g} \bP^N$, where $g$ is finite and $\phi''$ has connected fibers.  By \cite[\href{https://stacks.math.columbia.edu/tag/03H0}{Tag 03H0} and \href{https://stacks.math.columbia.edu/tag/03GW}{Tag 03GW}]{Stacks-Project}, $j'' := \phi''|_{U'}: U' \to Z$ is an open immersion.  Moreover, since $L_C' \cong \phi^{\prime *} \cO_{\bP_C^N}(1)$, it descends to the ample line bundle $L_C'' := g^* \cO_{\bP_C^N}(1)$ on $Z$.  Note that $Y$ and $\phi'$ are still defined over a finite extension of $\bQ_p$.  Let $\cZ$ denote the adic space associated with $Z$; define $\eta_\cZ: (\cZ_\proet, \widehat{\cO}_{\cZ_\proet}) \to (Z, \cO_Z)$ as in Theorem \ref{thm-kv}; and define $\proetSh{\rV^{\prime\prime \circ}}_\cZ$ by applying Construction \ref{constr-proet-ls} to $j''$, the pullback $\tilde{U}''$ of $\tilde{U}$ to $Z$, and $\rV'' = \rV$.  By what we have proved under the assumption \Refeq{\ref{eq-thm-kv-very-amp}},
    \begin{equation}\label{eq-thm-kv-Stein}
        R\Gamma(\cZ_\proet, \proetSh{\rV^{\prime\prime \circ}}_\cZ \ho_{\bZ_p} \eta_\cZ^* (L_C'')^{-1}) \in D^{\geq d}(C),
    \end{equation}
    On the other hand, since $\phi'': Y \to Z$ is proper \Pth{and $\rV'' = \rV' = \rV$}, by the same arguments based on Scholze's primitive comparison theorem \cite[\aThm 3.13]{Scholze:2013-pss} in the proof of Proposition \ref{prop-cyc-cov-coh-inj}, there is a natural isomorphism
    \[
        \proetSh{\rV^{\prime\prime \circ}}_\cZ \ho_{\bZ_p} \eta_\cZ^* (L_C'')^{-1} \cong R\phi_{\proet, *}''(\proetSh{\rV^{\prime \circ}}_\cY \ho_{\bZ_p} \eta_\cY^* (L_C')^{-1}).
    \]
    over $\cZ_\proet$.  Thus, it follows from \Refeq{\ref{eq-thm-kv-Stein}} that
    \[
        R\Gamma(\cY_\proet, \proetSh{\rV^{\prime \circ}}_\cY \ho_{\bZ_p} \eta_\cY^* (L_C')^{-1}) \in D^{\geq d}(C).
    \]
    Since $\pi: Y \to X$ is the composition $Y = Y_0 \times_X Y_1 \times_X \cdots \times_X Y_N \to Y_0 \times_X Y_1 \times_X \cdots \times_X Y_{N - 1} \to \cdots \to Y_0 \to X$ of cyclic covering maps pulled back from the $\pi_i$'s, we can finish the proof of Theorem \ref{thm-kv} by repeatedly applying Proposition \ref{prop-cyc-cov-coh-inj}.
\end{proof}

\begin{proof}[Proof of Lemma \ref{lem-BG-cov}]
    One can argue as in the proof of \cite[\aThm 4.1.10]{Lazarsfeld:2004-PAG-1} \Pth{the so-called \emph{Bloch--Gieseker coverings}}.  For the convenience of the reader, let us give a more direct proof in our case.  For any nonzero $s \in \Gamma(\bP_E^N, \cO_{\bP_E^N}(1))$, we shall denote by $H_s$ the hyperplane it defines in $\bP_E^N$.  By repeatedly applying Bertini's theorem, there exist nonzero $s_0, \cdots, s_N \in \Gamma(\bP_E^N, \cO_{\bP_E^N}(1))$ such that, for $i = 0, \cdots, N$:
    \begin{itemize}
        \item $H_{s_i}$ does not contain any irreducible component of $\phi_m(X)$; and

        \item $H_{s_i}$ intersects $\phi_m(U) \cap \big(\cap_{j \in I} H_{s_j}\bigr)$ transversally, for any $I \subset \{ 0, \cdots, i - 1 \}$.
    \end{itemize}

    Construct the finite cover $\pi: Y \to X$ as in the lemma.  We first prove that $Y$ is reduced. Since this is a local property, we may assume that $X = \Spec(A)$ is affine and that $L$ is trivial, and identify $s_0, \cdots, s_N$ with elements in $A$.  Then $Y \cong \Spec(B)$, where $B = A[x_0, \ldots, x_N] / (x_0^m - s_0, \ldots, x_N^m - s_N)$.  We may enlarge $E$ so that $x^m - 1$ splits completely in $E$.  Note that $G = \Grpmu_m \times \cdots \times \Grpmu_m$ \Pth{$(N + 1)$-copies of $\Grpmu_m$} acts on $B$ via $(a_0, \cdots, a_N) \cdot x_i = a_i x_i$.  Let $f \in B$ be a nilpotent element which is also an eigenvector of $G$.  Then $f = a x_0^{i_0} \cdots x_N^{i_N}$ for some $a \in A$ and $i_0, \cdots, i_N \in \{ 0, \cdots, m - 1 \}$.  Since $f^m = a^m s_0^{i_0} \cdots s_N^{i_N} \in A$ is nilpotent, but $A$ is reduced and $s_0, \cdots, s_N$ are nonzero, we must have $a = 0$ and hence $f = 0$.

    To see the smoothness of $\pi^{-1}(U)$, we may similarly assume that $U = \Spec(A)$ and $\pi^{-1}(U) \cong \Spec(B)$, where $B = A[x_0, \cdots, x_N] / (x_0^m - s_0, \ldots, x_N^m - s_N)$.  Let $\mathfrak{m}$ be a maximal ideal of $B$, and $\mathfrak{m}_A := \mathfrak{m} \cap A$.  Let $I$ be the set of $i$'s such that $s_i \in \mathfrak{m}_A$.  By our transversality assumption, we can find elements $t_1, \cdots, t_l \in \mathfrak{m}_A$, with $l = \dim(A) - |I|$, such that the images of $\{ t_1, \cdots, t_l \}$ and $\{ s_i \}_{i \in I}$ form a basis of $\mathfrak{m}_A / \mathfrak{m}_A^2$.  Since $x_i^m - s_i$ in $B$, for each $i \in I$, and since $m \in E^\times$, the images of $\{ t_1, \cdots, t_l \}$ and $\{ x_i \}_{i \in I}$ also form a basis of $\mathfrak{m} / \mathfrak{m}^2$.  Thus, $B$ is smooth at $\mathfrak{m}$.
\end{proof}

\subsection{Reformulation via Kummer \'etale site}\label{sec-akv-ket}

Instead of working with the complexes $\etSh{\rV^\circ}_{X_C}$ and $\proetSh{\rV^\circ}_\cX$ introduced in Construction \ref{constr-proet-ls}, we shall make use of the Kummer \'etale and pro-Kummer \'etale sites introduced in \cite{Diao/Lan/Liu/Zhu:2023-lasfr} and work instead with sheaves \Pth{\ie, complexes concentrated in single degrees}.  Our setup is as follows.  Let $X_C$ denote a smooth variety over $C$, and $D \subset X_C$ an effective divisor with  \emph{normal crossings}, with complementary open immersion $j: U \to X$.  Let $\cX$ and $\cD$ denote the rigid analytic varieties over $C$ associated with $X_C$.  Then $\cD$ is also a normal crossings divisor on $\cX$, with complementary open immersion $\adic{j}: \cU \to \cX$, which defines an fs log structure on $\cX$ whose restriction to $\cU$ is the trivial one, as explained in \cite[\aEx 2.3.16]{Diao/Lan/Liu/Zhu:2023-lasfr}.  Let $\cX_\ket$ denote its Kummer \'etale site with integral structure sheaf $\cO_{\cX_\ket}^+$, as in \cite[\aDefs 4.1.16 and 4.3.1]{Diao/Lan/Liu/Zhu:2023-lasfr}, and let $g: \cX_\ket \to \cX_\et$ the canonical morphism induced by forgetting the log structure.  The restriction $g_\cU: \cU_\ket \to \cU_\et$ is trivially an isomorphism, and the canonical morphisms of sites $\adic{j}_\et: \cU_\et \to \cX_\et$ and $\adic{j}_\ket: \cU_\et \to \cX_\ket$ satisfy $g \circ \adic{j}_\ket = \adic{j}_\et$.

\begin{proposition}\label{prop-ket}
    Let $F$ be an \'etale $\bZ / p^n$-local system of finite rank over $U$.
    \begin{enumerate}
        \item\label{prop-ket-1} Purity: $R j_{\ket, *} F = j_{\ket, *} F$ is a torsion Kummer \'etale local system.

        \item\label{prop-ket-2} There is a natural isomorphism
            \[
                (R j_{\et, *} F) \otimes_{\bZ_p} \cO_{\cX_\et}^+ \cong (R g_* j_{\ket, *} F) \otimes_{\bZ_p} \cO_{\cX_\et}^+ \cong R g_*(j_{\ket, *} F \otimes_{\bZ_p} \cO_{\cX_\ket}^+).
            \]
    \end{enumerate}
\end{proposition}
\begin{proof}
    See \cite[\aLem 4.5.8 and \aThm 4.6.1]{Diao/Lan/Liu/Zhu:2023-lasfr}.  Note that $(R j_{\et, *} F) \otimes_{\bZ_p} \cO_{\cX_\et}^+ \cong (R j_{\et, *} F) \otimes_{\bZ / p^n} (\cO_{\cX_\et}^+ / p^n)$, because $\cO_{\cX_\et}^+$ is flat over $\bZ_p$.
\end{proof}

\begin{construction}\label{constr-proket-ls}
    Suppose $K$ is a profinite group, and $\tilde{U}$ is a pro-\'etale $K$-torsor over $U$, as in Construction \ref{constr-proet-ls}.  Let $\rV$ be a unitary $p$-adic Banach space representation of $K$.  As in Construction \ref{constr-proet-ls}, $\rV_n = \rV^\circ / p^n = \cup_{M \subset \rV_n} \, M$, where $M$ runs over all finite $K$-stable subgroups of $\rV_n$; and for each $M$, we constructed an \'etale sheaf $\etSh{M}$ over $U$.  Let $\eta_U: \cU_\et \to U_\et$ denote the canonical morphism of sites.  Over $\cX_\ket$, we define $\ketSh{M}_\cX := R \adic{j}_{\ket, *} \eta_U^* \, \etSh{M}$, for each $M$ as above; and define $\ketSh{\rV_n}_\cX := \varinjlim_{M \subset \rV_n} \ketSh{M}_\cX$, which can be thought as the associated Kummer \'etale local systems.  \Pth{These are sheaves, not just complexes.}  Let $\cX_\proket$ denote the pro-Kummer-\'etale site introduced in \cite[\aDef 5.1.2]{Diao/Lan/Liu/Zhu:2023-lasfr}, and let $\lambda_\cX: \cX_\proket \to \cX_\ket$ denote the projection of sites.  We define pro-Kummer \'etale sheaves $\proketSh{\rV_n}_\cX := \lambda_\cX^* \, \ketSh{\rV_n}_\cX$, $\proketSh{\rV^\circ}_\cX := \varprojlim_n \proketSh{\rV_n}_\cX$, and $\proketSh{\rV}_\cX := \proketSh{\rV^\circ}_\cX \otimes_\bZ \bQ$, which are the pro-Kummer \'etale local systems associated with $\rV_n$, $\rV^\circ$, and $\rV$, respectively.
\end{construction}

The projective limit in Construction \ref{constr-proket-ls} agrees with the derived limit, by the following lemma.

\begin{lemma}\label{lem-proket-ls}
    In Construction \ref{constr-proket-ls}, assume moreover that everything is defined over a finite extension of $\bQ_p$ \Pth{as in \cite[\aSec 5.3]{Diao/Lan/Liu/Zhu:2023-lasfr}}.  Then:
    \begin{enumerate}
        \item\label{lem-proket-ls-1} $R^i\lim_n \proketSh{\rV_n}_\cX = 0$, for $i>0$.
        \item\label{lem-proket-ls-2} $\proketSh{\rV^\circ}_\cX / p^n \cong \proketSh{\rV_n}_\cX$, for each $n \geq 1$.
    \end{enumerate}
\end{lemma}
\begin{proof}
    Starting with the short exact sequence \Refeq{\ref{eq-constr-proet-ls-seq}}, by Proposition \ref{prop-ket}\Refenum{\ref{prop-ket-1}} and by expressing $\rV_{n + m}$ as the direct limit of its finite stable $K$-subgroups, we obtain a short exact sequence
    \[
        0 \to \ketSh{\rV_m}_\cX \Mapn{\cdot p^n} \ketSh{\rV_{n + m}}_\cX \to \ketSh{\rV_n}_\cX \to 0.
    \]
    The same holds for its pullback by $\lambda_\cX$.  Taking the inverse limit over $m$, we see that \Refenum{\ref{lem-proket-ls-1}} implies \Refenum{\ref{lem-proket-ls-2}}.  It remains to prove \Refenum{\ref{lem-proket-ls-1}}.  Let $\widetilde{\cX}$ be a universal cover of a log affinoid perfectoid object of $\cX_\proket$, as in the proof of \cite[\aProp 5.3.13]{Diao/Lan/Liu/Zhu:2023-lasfr}.  Since $\widetilde{\cX}$ is quasi-compact, for all $i \geq 0$, by using \cite[\href{https://stacks.math.columbia.edu/tag/0739}{Tag 0739}]{Stacks-Project}, we have
    \begin{equation}\label{eq-lem-proket-ls-H-i}
    \begin{split}
        H^i({\cX_\proket}_{/\widetilde{\cX}}, \proketSh{\rV_n}_\cX) & \cong H^i({\cX_\proket}_{/\widetilde{\cX}}, \varinjlim_{M \subset \rV_n} \lambda_\cX^* \adic{j}_{\ket, *} \eta_U^* \, \etSh{M}) \\
        & \cong \varinjlim_{M \subset \rV_n} H^i({\cX_\proket}_{/\widetilde{\cX}}, \lambda_\cX^* \adic{j}_{\ket, *} \eta_U^* \, \etSh{M}),
    \end{split}
    \end{equation}
    where $M$ runs through all finite $K$-stable subgroups of $\rV_n$.  By \cite[\aProp 5.3.13]{Diao/Lan/Liu/Zhu:2023-lasfr}, \Refeq{\ref{eq-lem-proket-ls-H-i}} is zero when $i > 0$.  Therefore, for all $n \geq 1$, the natural map
    \[
        H^0({\cX_\proket}_{/\widetilde{\cX}}, \proketSh{\rV_{n + 1}}_\cX) \to H^0({\cX_\proket}_{/\widetilde{\cX}}, \proketSh{\rV_n}_\cX)
    \]
    is surjective; and $R^1\lim_n H^0({\cX_\proket}_{/\widetilde{\cX}}, \proketSh{\rV_n}_\cX) = 0$.  Since such $\widetilde{\cX}$ form a basis of $\cX_\proket$, by \cite[\aProp 5.3.12]{Diao/Lan/Liu/Zhu:2023-lasfr}, \Refenum{\ref{lem-proket-ls-1}} follows, by \cite[\aLem 3.18]{Scholze:2013-phtra}.
\end{proof}

The following is the Kummer version of Theorem \ref{thm-kv}.
\begin{corollary}\label{cor-kv-proket}
    Let $X_C$ be a smooth equidimensional projective variety over $C$ of dimension $d$, and $L_C$ a line bundle on $X_C$.  Let $U \subset X_C$ be a smooth open subvariety whose complement $D$ is an effective divisor of normal crossings, so that we have the associated fs log adic spaces $\cX$ etc as in the beginning of this Section \ref{sec-akv-ket}.  Let $K$ be a profinite group, and $\tilde{U}$ a pro-\'etale $K$-torsor over $U$.  Assume that the same condition \Refeq{\ref{eq-thm-kv-cond}} as in Theorem \ref{thm-kv} holds.  Assume moreover that $X_C$, $U$, and $\tilde{U}$ are defined over a finite extension of $\bQ_p$.  Then, for every unitary $p$-adic Banach space representation $\rV$ of $K$, to which Construction \ref{constr-proket-ls} applies, we have
    \[
    \begin{split}
        & R\Gamma(\cX_\proket, \proketSh{\rV}_\cX \ho_{\bQ_p} \nu_\proket^* L_C^{-1} ) \\
        & \cong R\Gamma(\cX_\proket, \proketSh{\rV^\circ}_\cX \ho_{\bZ_p} \nu_\proket^* L_C^{-1}) \in D^{\geq d}(C),
    \end{split}
    \]
    where $\nu_\proket: (\cX_\proket, \widehat{\cO}_{\cX_\proket}) \to (X_C, \cO_{X_C})$ is the natural projection of ringed sites, where $\widehat{\cO}_{\cX_\proket}$ denotes the completed structure sheaf \Pth{as in \cite[\aDef 5.4.1]{Diao/Lan/Liu/Zhu:2023-lasfr}}, and where $\ho_{\bZ_p}$ denotes $p$-adically completed tensor product.
\end{corollary}

\begin{remark}
    The condition \Refeq{\ref{eq-thm-kv-cond}} in Theorem \ref{thm-kv} and Corollary \ref{cor-kv-proket} holds if there is a morphism $\pi: X_C \to Y_C$ of projective varieties such that $\pi|_U: U \to Y_C$ is an open immersion and such that $L_C^m$ descends to an ample line bundle on $Y_C$.  In applications, $X_C$ will be some projective smooth toroidal compactification of a Shimura variety, and $Y_C$ will be the associated minimal compactification.
\end{remark}

\begin{proof}[Proof of Corollary \ref{cor-kv-proket}]
    In view of Theorem \ref{thm-kv}, it suffices to show that
    \[
        R g_{\proket, *}(\proketSh{\rV^\circ}_\cX \ho_{\bZ_p} \nu_\proket^* L_C^{-1}) \cong  \proetSh{\rV^\circ}_\cX \ho_{\bZ_p} \eta_\cX^* L_C^{-1},
    \]
    where $g_\proket$ denotes the canonical projection $\cX_\proket \to \cX_\proet$.  Since $R g_{\proket, *}$ commutes with derived limit, by the projection formula, it suffices to show that
    \begin{equation}\label{eq-cor-kv-proket-red}
        R g_{\proket, *}\bigl((\lambda_\cX^* \, \ketSh{\rV_n}_\cX) \otimes_{\bZ / p^n} (\widehat{\cO}_{\cX_\proket}^+ / p^n)\bigr) \cong \proetSh{\rV_n}_\cX \otimes_{\bZ / p^n} (\widehat{\cO}_{\cX_\proet}^+ / p^n).
    \end{equation}
    By definition \Pth{see \cite[\aDef 5.4.1]{Diao/Lan/Liu/Zhu:2023-lasfr}}, we have $(\lambda_\cX^* \, \ketSh{\rV_n}_\cX) \otimes_{\bZ / p^n} (\widehat{\cO}_{\cX_\proket}^+ / p^n) \cong \lambda_\cX^*\bigl(\ketSh{\rV_n}_\cX \otimes_{\bZ_p} (\cO_{\cX_\ket}^+ / p^n)\bigr) \cong \lambda_\cX^*(\ketSh{\rV_n}_\cX \otimes_{\bZ_p} \cO_{\cX_\ket}^+)$.  By \cite[\aProp 5.2.1]{Diao/Lan/Liu/Zhu:2023-lasfr}, the left-hand side of \Refeq{\ref{eq-cor-kv-proket-red}} is the pullback of $R g_*(\ketSh{\rV_n}_\cX \otimes_{\bZ_p} \cO_{\cX_\ket}^+)$ via the natural projection $\lambda: \cX_\proet \to \cX_\et$.  On the other hand, the right-hand side of \Refeq{\ref{eq-cor-kv-proket-red}} is the pullback of $(\eta^* R j_{\et, *} \, \etSh{\rV_n}) \otimes_{\bZ_p} \cO_{\cX_\et}^+$ via $\lambda$.  Hence, it suffices to prove that
    \[
        R g_*(\ketSh{\rV_n}_\cX \otimes_{\bZ_p} \cO_{\cX_\ket}^+) \cong (\eta^* R j_{\et, *} \, \etSh{\rV_n}) \otimes_{\bZ_p} \cO_{\cX_\et}^+.
    \]
    When $\rV_n$ is finite, this follows from Proposition \ref{prop-ket} and Huber's comparison theorem \cite[\aThm 3.8.1]{Huber:1996-ERA}.  In general, we have $R j_{\et, *} \, \etSh{\rV_n} = \varinjlim_{M \subset \rV_n} R j_{\et, *} \, \etSh{M}$, with $M$ running through all finite $K$-stable subgroups of $\rV_n$, as usual; and it remains to observe that $R g_*$ commutes with filtered colimits, as $g: \cX_\ket \to \cX_\et$ is qcqs, and that tensor product commutes with colimits \cite[\href{https://stacks.math.columbia.edu/tag/03EP}{Tag 03EP}]{Stacks-Project}.
\end{proof}

\section{Locally symmetric varieties and automorphic vector bundles}\label{sec-lsv}

In this section, we study \Pth{disjoint unions of} locally symmetric varieties and their minimal and toroidal compactifications, and explain materials that will be useful for our study of completed cohomology of Shimura varieties in later sections.

\subsection{General setup}\label{sec-lsv-setup}

We shall consider pairs $(\Grp{H}, \Dom)$ of the following type.
\begin{itemize}
    \item $\Grp{H}$ is a connected reductive algebraic group over $\bQ$.  We shall denote its center by $Z(\Grp{H})$, and consider the associated adjoint group $\Grp{H}^\ad := \Grp{H} / Z(\Grp{H})$.

    \item $\Dom$ is a finite disjoint union of Hermitian symmetric domains such that $\Grp{H}(\bR)$ acts transitively on $\Dom$ via its image under $\Grp{H}(\bR) \to \Grp{H}^\ad(\bR)$, and such that the \emph{neutral component} \Pth{or \emph{identity component}; \ie, the connected component containing the identity element} $\Grp{H}^\ad(\bR)^+$ of $\Grp{H}^\ad(\bR)$ acts with maximal compact stabilizer subgroups.  \Pth{Since $\Grp{H}^\ad(\bR)^+$ is a connected semisimple Lie group, this is consistent with the classification of their associated Hermitian symmetric domains, which we allow to be trivial when the Lie group is compact, as in \cite{Helgason:2001-DLS}.  See also \cite[\aSec 1.2]{Deligne:1979-vsimc} and \cite[\aSec 1]{Milne:2005-isv}.}
\end{itemize}
Let $\Dompt_0 \in \Dom$ be any fixed choice of a point of $\Dom$, and let $\Dom^+$ be the connected component of $\Dom$ containing $\Dompt_0$.  For any group $A$ acting continuous on $\Dom$, we shall denote by $A_+$ the stabilizer $\Dom^+$ in $A$; \ie, the subgroup of $A$ consisting of elements of $A$ stabilizing $\Dom^+$.  In particular, we have $\Grp{H}(\bR)_+$, the stabilizer of $\Dom^+$ in $\Grp{H}(\bR)$, so that $\Dom^+ \cong \Grp{H}(\bR)_+ / \MaxCpt \cong \Grp{H}^\ad(\bR)^+ / \MaxCpt^\ad$, where $\MaxCpt$ \Pth{\resp $\MaxCpt^\ad$, by abuse of notation} is the stabilizer of $\Dompt_0$ in $\Grp{H}(\bR)_+$ \Pth{\resp $\Grp{H}^\ad(\bR)^+$}.  By our assumption above, $\MaxCpt^\ad$ is a maximal compact subgroup of $\Grp{H}^\ad(\bR)^+$.  We shall extend this notation system by adding subscripts \Qtn{$+$} to stabilizers of $\Dom^+$ in other groups acting on $\Dom$.

Since the underlying Riemannian symmetric space of $\Dom^+$ parameterizes Cartan decompositions, the point $\Dompt_0 \in \Dom^+$ induces a decomposition
\begin{equation}\label{eq-Cartan-decomp-C}
    \Lie \Grp{H}_\bC \Mi (T_{\Dompt_0} \Dom^+)_\bC^- \oplus (\Lie \MaxCpt)_\bC \oplus (T_{\Dompt_0} \Dom^+)_\bC^+,
\end{equation}
where $(T_{\Dompt_0} \Dom^+)_\bC^+$ \Pth{\resp $(T_{\Dompt_0} \Dom^+)_\bC^-$} is the holomorphic \Pth{\resp anti-holomorphic} tangent space of $\Dom^+$ at $\Dompt_0$; and the complex structure of $\Dom^+$ at the point $\Dompt_0$ fixed by $\MaxCpt^\ad$ defines a homomorphism $u: \Grp{U}_1 := \{ z \in \bC^\times : |z|=1 \} \to \MaxCpt^\ad \subset \Grp{H}^\ad(\bR)^+$ \Pth{as in \cite[\aThm 1.21 and \aCor 1.22]{Milne:2005-isv}}, whose complexification gives a homomorphism
\begin{equation}\label{eq-hc-ad}
    \hc^\ad: \bG_{m, \bC} \to \Grp{H}_\bC^\ad.
\end{equation}
By construction, \Refeq{\ref{eq-Cartan-decomp-C}} is a decomposition into spaces of weights $-1$, $0$, and $1$ for the representation of $\bC$ defined by the composition of $\Lie \hc^\ad: \bC \to \Lie \Grp{H}_\bC^\ad$ with the canonical splitting $\Lie \Grp{H}_\bC^\ad \cong [\Lie \Grp{H}_\bC, \Lie \Grp{H}_\bC] \Em \Lie \Grp{H}_\bC$; and there are no other weights.  This implies, in particular, that both $(T_{\Dompt_0} \Dom^+)_\bC^+$ and $(T_{\Dompt_0} \Dom^+)_\bC^-$ are abelian unipotent Lie subalgebras of $\Lie \Grp{H}_\bC$ normalized by the adjoint actions of $(\Lie \MaxCpt)_\bC$.

Let $\Grp{M}$, $\Grp{N}$, and $\Grp{N}^\std$ be connected subgroups of $\Grp{H}_\bC$ such that \Refeq{\ref{eq-Cartan-decomp-C}} induce $\Lie \Grp{M} \Mi (\Lie \MaxCpt)_\bC$, $\Lie \Grp{N} \Mi (T_{\Dompt_0} \Dom^+)_\bC^+$, and $\Lie \Grp{N}^\std \Mi (T_{\Dompt_0} \Dom^+)_\bC^-$.  Then
\begin{equation}\label{eq-def-P}
    \Grp{P} := \Grp{N} \Grp{M} \cong \Grp{N} \rtimes \Grp{M} \qquad \text{and} \qquad \Grp{P}^\std := \Grp{M} \Grp{N}^\std \cong \Grp{M} \ltimes \Grp{N}^\std
\end{equation}
are two opposite parabolic subgroups with unipotent radicals $\Grp{N}$ and $\Grp{N}^\std$, respectively; and $\Grp{M}$ is a common Levi subgroup of $\Grp{P}$ and $\Grp{P}^\std$.  In $\Grp{H}(\bC)$, we have $\MaxCpt = \Grp{H}(\bR)_+ \cap \Grp{P}^\std(\bC) = \Grp{H}(\bR)_+ \cap \Grp{P}(\bC)$.  Then we also have the following:
\begin{remark}\label{rem-fil-by-hc-ad}
    Given any $\rW$ in $\Rep_\bC(\Grp{P})$ \Pth{\resp $\Rep_\bC(\Grp{P}^\std)$}, the action of the composition of $\Lie \hc^\ad$ with $(\Lie \MaxCpt^\ad)_\bC \subset (\Lie \MaxCpt)_\bC \cong \Lie \Grp{M}$ defines an increasing \Pth{\resp decreasing} filtration on $\rW$ by subrepresentations of $\Grp{P}$ \Pth{\resp $\Grp{P}^\std$}.
\end{remark}

Over $\bC$, we have two \emph{partial flag varieties}
\begin{equation}\label{eq-def-fl}
    \Flalg := \Grp{P} \Lquot \Grp{H}_\bC \qquad \text{and} \qquad \Flalg^\std := \Grp{H}_\bC / \Grp{P}^\std,
\end{equation}
with natural transitive right and left actions of $\Grp{G}_\bC$, respectively.  \Pth{It is a matter of conventions that we choose to have right and left actions as above.  We can easily flip our choices by taking inverses.}  By definition, the assignments of the decomposition \Refeq{\ref{eq-Cartan-decomp-C}} and of parabolic subgroups to $\Dompt_0 \in \Dom^+$ are $\Grp{H}(\bR)_+$-equivariant.  Accordingly, we have the $\Grp{H}(\bR)_+$-equivariant \emph{Borel embedding}
\begin{equation}\label{eq-Borel-emb}
    \Dom^+ \Em \Grp{H}(\bC) / \Grp{P}^\std(\bC) \cong (\Flalg^\std)^\an
\end{equation}
from a Hermitian symmetric domain to its compact dual \Pth{\Refcf{} \cite[\aProp I.5.24]{Borel/Ji:2006-CSL}}.  Note that it is the complex analytic $(\Flalg^\std)^\an$ that serves as targets of Borel embeddings, while the $p$-adic analytification \Pth{\ie, associated adic space} of $\Flalg_C$ will serve as targets of Hodge--Tate period morphisms \Pth{to be explained in Section \ref{sec-HT-mor}}.

Let $\Grp{H}^c$ be the quotient of $\Grp{H}$ by the minimal subtorus $Z_s(\Grp{H})$ of $Z(\Grp{H})$ such that the torus $Z(\Grp{H})^\circ / Z_s(\Grp{H})$ has the same split ranks over $\bQ$ and $\bR$.  For any closed algebraic subgroup $\Grp{A}$ of $\Grp{H}$, we denote its closed \Pth{schematic} image in $\Grp{H}^c$ by $\Grp{A}^c$.  In particular, we have algebraic subgroups $\Grp{M}^c$, $\Grp{P}^c$, and $\Grp{P}^{\std, c}$ of $\Grp{H}^c$.  For $? = \emptyset$ and $\std$, since $\Grp{N}^?$ intersects $\ker(\Grp{H} \to \Grp{H}^c) \subset Z(\Grp{H})$ trivially, the canonical short exact sequence $1 \to \Grp{N}^? \to \Grp{P}^? \to \Grp{M} \to 1$ induces a canonical short exact sequence $1 \to \Grp{N}^? \to \Grp{P}^{?, c} \to \Grp{M}^c \to 1$.  Similarly, for any subgroup $A$ of $\Grp{H}(R)$, where $R$ is any $\bQ$-algebra, we denote its image in $\Grp{H}^c(R)$ by $A^c$.  In particular, the image of any open compact subgroup $\levcp \subset \Grp{H}(\bAi)$ in $\Grp{H}^c(\bAi)$ is denoted by $\levcp^c$.

\subsection{Arithmetic quotients and double coset spaces}\label{sec-dcs-an}

Let $(\Grp{H}, \Dom)$ be any pair as in Section \ref{sec-lsv-setup}, together with $\Grp{H}^\ad$, $\Dom^+$, etc introduced there.

Let $\Gamma$ be any \emph{neat} arithmetic subgroup of $\Grp{H}(\bQ)_+$ \Pth{see, \eg, \cite[7.11, 17.1, and 17.3]{Borel:2019-IAG}}, whose action on $\Dom^+$ factors through its image $\Gamma^\ad$ in $\Grp{H}^\ad(\bQ)$.  By \cite[\aThm 8.9, \aRem 8.11, and 17.3]{Borel:2019-IAG}, $\Gamma^\ad$ is again a neat arithmetic subgroup.  Note that $\Gamma$ acts on $\Dom^+$ via $\Gamma^\ad$, and the neatness of $\Gamma^\ad$ implies that the finite discrete group $\Gamma^\ad \cap \MaxCpt^\ad$ is trivial, in which case $\Gamma^\ad$ acts freely on $\Dom^+$.  Hence, the quotient
\begin{equation}\label{eq-def-lsv-an-Gamma}
    \LSV_\Gamma^\an := \Gamma \Lquot \Dom^+ \cong \LSV^\an_{\Gamma^\ad} := \Gamma^\ad \Lquot \Dom^+
\end{equation}
is a complex manifold \Pth{which we shall view as a smooth complex analytic space} as $\Dom^+$ is.  Since $\Grp{H}(\bR)$ acts on $\Dom$ via its image in $\Grp{H}^\ad(\bR)$, the whole center $Z(\Grp{H})(\bR)$ acts trivially and hence lies in $\Grp{H}(\bR)_+$.  Therefore, we have $\Gamma^\ad \cong \Gamma / \bigl(\Gamma \cap (Z(\Grp{H})(\bQ))\bigr)$.

For any open compact subgroup $\levcp \subset \Grp{H}(\bAi)$, consider the double coset space
\begin{equation}\label{eq-def-dcs-an}
    \LSV_\levcp^\an := \LSV_\levcp^\an(\Grp{H}, \Dom) := \Grp{H}(\bQ) \big\Lquot \bigl(\Dom \times \Grp{H}(\bAi)\bigr) \big/ \levcp,
\end{equation}
where $\Grp{H}(\bQ)$ acts diagonally on $\Dom \times \Grp{H}(\bAi)$ from the left-hand side, and where $\levcp$ acts only on $\Grp{H}(\bAi)$ from the right-hand side.  Since $\Grp{H}$ is connected as an algebraic group over $\bQ$, by a special case of weak approximation \Pth{see \cite[\aSec 7.3, \aThm 7.7]{Platonov/Rapinchuk:1994-AGN}}, $\Grp{H}(\bQ)$ is dense in $\Grp{H}(\bR)$ in the real topology.  Consider $\Grp{H}(\bQ)_+ := \Grp{H}(\bQ) \cap \Grp{H}(\bR)_+$, which is the stabilizer of $\Dom^+$ in $\Grp{H}(\bQ)$.  Then we can rewrite \Refeq{\ref{eq-def-dcs-an}} as
\begin{equation}\label{eq-dcs-an-plus}
    \LSV_\levcp^\an = \Grp{H}(\bQ)_+ \big\Lquot \bigl(\Dom^+ \times \Grp{H}(\bAi)\bigr) \big/ \levcp.
\end{equation}

For each $h \in \Grp{H}(\bAi)$, let us introduce the subspace
\begin{equation}\label{eq-def-dcs-an-K-h}
    \LSV_{\levcp, h}^\an := \Grp{H}(\bQ)_+ \big\Lquot \bigl(\Dom^+ \times (\Grp{H}(\bQ)_+ h \levcp)\bigr) \big/ \levcp
\end{equation}
of $\LSV_\levcp^\an$.  Since $\gamma h \levcp = h \levcp$ for $\gamma \in \Grp{H}(\bQ)_+$ exactly when $\gamma \in h \levcp h^{-1}$, we can write
\begin{equation}\label{eq-def-dcs-an-K-h-equiv}
    \LSV_{\levcp, h}^\an = \LSV_{\Gamma_{\levcp, h}}^\an = \Gamma_{\levcp, h} \Lquot \Dom^+
\end{equation}
instead of \Refeq{\ref{eq-def-dcs-an-K-h}}, which is a connected topological space as $\Dom^+$ is, where
\begin{equation}\label{eq-def-Gamma-K-h}
    \Gamma_{\levcp, h} := \Grp{H}(\bQ)_+ \cap (h \levcp h^{-1})
\end{equation}
is an \emph{arithmetic subgroup} of $\Grp{H}(\bQ)$, whose action on $\Dom^+$ factors through its image
\begin{equation}\label{eq-Gamma-K-h-ad}
\begin{split}
    \Gamma_{\levcp, h}^\ad & \cong \Gamma_{\levcp, h} \big/ \bigl(\Gamma_{\levcp, h} \cap (Z(\Grp{H})(\bQ))\bigr) \\
    & \cong \bigl(\Grp{H}(\bQ)_+ \cap (h \levcp h^{-1})\bigr) \big/ \bigl((Z(\Grp{H})(\bQ)) \cap (h \levcp h^{-1})\bigr)
\end{split}
\end{equation}
in $\Grp{H}^\ad(\bQ)$, which is an arithmetic subgroup of $\Grp{H}^\ad(\bQ)$.  For simplicity, when $h = 1$, we shall suppress $1$ from the notation, and simply write $\Gamma_\levcp = \Gamma_{\levcp, 1}$ and $\Gamma_\levcp^\ad = \Gamma_{\levcp, 1}^\ad$.

By \cite[\aThm 5.1]{Borel:1963-fpagn}, $\#(\Grp{H}(\bQ)_+ \Lquot \Grp{H}(\bAi) / \levcp) < \infty$; \ie, there exists a finite set of representatives $\{ h_i \}_{i \in I}$ such that
\begin{equation}\label{eq-dc-fin}
    \Grp{H}(\bAi) = \coprod_{i \in I} (\Grp{H}(\bQ) h_i \levcp).
\end{equation}
By substituting this into \Refeq{\ref{eq-def-dcs-an}}, and by using \Refeq{\ref{eq-def-dcs-an-K-h-equiv}}, we obtain disjoint unions
\begin{equation}\label{eq-dcs-decomp-an}
    \LSV_\levcp^\an = \coprod_{i \in I} \LSV_{\levcp, h_i}^\an = \coprod_{i \in I} \LSV_{\Gamma_{\levcp, h_i}}^\an,
\end{equation}
which show that $\{ \LSV_{\levcp, h_i}^\an \}_{i \in I}$ are the \Pth{finitely many} connected components of $\LSV_\levcp^\an$.  The component labeled by $h = 1$ is called the \emph{neutral component}
\begin{equation}\label{eq-dsc-circ-an}
    \LSV_\levcp^{\circ, \an} := \LSV_{\levcp, 1}^\an = \Gamma_\levcp \Lquot \Dom^+ = \Gamma_\levcp^\ad \Lquot \Dom^+.
\end{equation}

When $\levcp \subset \Grp{H}(\bAi)$ is \emph{neat} as in \cite[0.6]{Pink:1989-Ph-D-Thesis}, $\Gamma_{\levcp, h_i} \subset \Grp{H}(\bQ)$ and their images $\Gamma_{\levcp, h_i}^\ad \subset \Grp{H}^\ad(\bQ)$ are neat as in \cite[17.1 and 17.3]{Borel:2019-IAG}.  So $\LSV_{\Gamma_{\levcp, h_i}}^\an \cong \LSV_{\Gamma_{\levcp, h_i}^\ad}^\an$ and $\LSV_\levcp^\an$ are \emph{complex manifolds} \Pth{viewed as smooth complex analytic spaces} as $\Dom^+$ is.  For simplicity, we will only consider neat levels $\levcp$'s, unless otherwise stated.

\subsection{Locally symmetric varieties and their compactifications}\label{sec-lsv-cpt}

In this subsection, we shall simultaneously work with complex manifolds of the forms $\LSV_\Gamma^\an$, $\LSV_\levcp^\an$, and $\LSV_{\levcp, h}^\an$ introduced in Section \ref{sec-dcs-an}.  To simplify the exposition, when the statements apply to all three kinds of manifolds above, we shall write $\LSV_\level^\an$, where $\level$ can stand for either $\Gamma$, or $\levcp$, or a pair of $\levcp$ and $h$.  Later, when referring to results stated this way, we shall freely replace $\level$ with $\Gamma$ etc depending on the context.

By \cite{Baily/Borel:1966-caqbs} \Pth{applied to all connected components}, the complex analytic space $\LSV_\level^\an$ is Zariski open dense in its so-called \emph{Satake--Baily--Borel} or \emph{minimal compactification} $\LSV_\level^{\Min, \an}$, which is the analytification of a normal projective variety $\LSV_\level^\Min$ \Pth{over $\bC$}.  Hence, $\LSV_\level^\an$ is the analytification of a \Pth{necessarily smooth} quasi-projective variety $\LSV_\level$.  By \cite[\aThm 3.7 and its proof]{Borel:1972-maset} and GAGA \cite{Serre:1955-1956-gaga}, any holomorphic map from the analytification of an algebraic variety over $\bC$ to $\LSV_\level^\an$ or $\LSV_\level^{\Min, \an}$ uniquely algebraizes, and it follows that the algebraic structures of $\LSV_\level$ and $\LSV_\level^\Min$ are unique.  Varieties over $\bC$ of the form $\LSV_\level$ are called \emph{locally symmetric varieties} \Pth{although the notion is often reserved for the special case where $\level = \Gamma$}.

When $h = 1$, we have the neutral components $\LSV_\levcp^\circ := \LSV_{\levcp, 1}$, $\LSV_\levcp^{\circ, \Min, \an} := \LSV_{\levcp, 1}^{\Min, \an}$, and $\LSV_\levcp^{\circ, \Min} := \LSV_{\levcp, 1}^\Min$.  Therefore, the disjoint union \Refeq{\ref{eq-dcs-decomp-an}} is the analytification of
\begin{equation}\label{eq-dcs-decomp}
    \LSV_\levcp = \LSV_\levcp(\Grp{H}, \Dom) = \coprod_{i \in I} \LSV_{\levcp, h_i} = \coprod_{i \in I} \LSV_{\Gamma_{\levcp, h_i}}
\end{equation}
\Pth{over $\bC$}, and its minimal compactification
\begin{equation}\label{eq-dcs-decomp-min-an}
    \LSV_\levcp^{\Min, \an} := \LSV_\levcp^{\Min, \an}(\Grp{H}, \Dom) = \coprod_{i \in I} \LSV_{\levcp, h_i}^{\Min, \an}
\end{equation}
is the analytification of the canonical normal projective variety
\begin{equation}\label{eq-dcs-decomp-min}
    \LSV_\levcp^\Min := \LSV_\levcp^\Min(\Grp{H}, \Dom) = \coprod_{i \in I} \LSV_{\levcp, h_i}^\Min
\end{equation}

\begin{remark}\label{rem-lsv-sc}
    When $\Grp{H}$ is simply-connected \Pth{and semisimple}, $\Grp{H}(\bR)$ is connected in the real topology, by \cite[\aSec 7.2, \aProp 7.6]{Platonov/Rapinchuk:1994-AGN}, in which case $\Dom = \Dom^+$ is connected, $\Grp{H}(\bR)_+ = \Grp{H}(\bR)$, and $\Grp{H}(\bQ)_+ = \Grp{H}(\bQ)$.  Moreover, by strong approximation \Pth{see \cite[\aSec 7.4, \aThm 7.12]{Platonov/Rapinchuk:1994-AGN}}, $\Grp{H}(\bAi) = \Grp{H}(\bQ) \levcp$, for any open compact subgroup $\levcp$ of $\Grp{H}(\bAi)$.  Thus, in \Refeq{\ref{eq-dc-fin}}, the index set $I$ is a singleton and can be chosen to be $\{ 1 \}$.  Therefore, $\LSV_\levcp^\an = \LSV_\levcp^{\circ, \an} = \Gamma_\levcp \Lquot \Dom$, where $\Gamma_\levcp = \Gamma_{\levcp, 1}$, is connected as a complex manifold.  This forces $\LSV_\levcp$ to be connected as an algebraic variety \Pth{over $\bC$}.  In this case, we have $\LSV_\levcp = \LSV_{\levcp, h} = \LSV_{\levcp, 1} = \LSV_\levcp^\circ$, for all $h \in \Grp{H}(\bAi)$.
\end{remark}

By \cite[\aCh III, \aThm 5.2, and \aCh IV, \aThm 2.2]{Ash/Mumford/Rapoport/Tai:2010-SCL-2} \Pth{again, applied to all connected components}, there is a collection $\{ \LSV_{\level, \Sigma}^\Tor \}_\Sigma$ of normal projective varieties called \emph{toroidal compactifications}---indexed by certain combinatorial data $\Sigma$ called \Pth{collections of} \emph{cone decompositions}, and we shall assume that all $\Sigma$'s are \emph{projective} as in \cite[\aCh IV, \aDef 2.1]{Ash/Mumford/Rapoport/Tai:2010-SCL-2}---such that the canonical open immersion $\LSV_\level \Em \LSV_\level^\Min$ with dense image factors through $\LSV_\level \Em \LSV_{\level, \Sigma}^\Tor \to \LSV_\level^\Min$, where the first morphism is a canonical open immersion with dense image, whose \Pth{reduced closed} complement $\partial \LSV_{\level, \Sigma}^\Tor := (\LSV_{\level, \Sigma}^\Tor - \LSV_\level)_\red$ is an effective divisor; and where the second morphism is proper \Pth{and necessarily surjective, by density of $\LSV_\level$ in both its source and target}.  The analytifications of all of these are denoted similarly, with additional superscripts \Qtn{$\an$}.  When $\level$ is an open compact subgroup $\levcp \subset \Grp{H}(\bAi)$, any $\Sigma$ as above is a collection $\Sigma = \{ \Sigma_h \}_{h \in \Grp{H}(\bAi)}$, where each $\Sigma_h$ is a cone decomposition for the component $\LSV_{\levcp, h}$.  Without explaining the meaning of cone decompositions, let us still summarize the following qualitative facts:
\begin{proposition}\label{prop-lsv-tor-facts}
    \begin{enumerate}
        \item\label{prop-lsv-tor-facts-ext} For each $\Sigma$, the above $\LSV_\level^\an \Em \LSV_{\level, \Sigma}^{\Tor, \an}$ and hence its algebraization $\LSV_\level \Em \LSV_{\level, \Sigma}^\Tor$ are uniquely characterized by the extension property defined by $\Sigma$ as in \cite[\aCh III, \aThm 7.5]{Ash/Mumford/Rapoport/Tai:2010-SCL-2}.

        \item\label{prop-lsv-tor-facts-min} Let $\TorMap_\Sigma: \LSV_{\level, \Sigma}^\Tor \to \LSV_\level^\Min$ be the canonical morphism, which is an isomorphism over the open dense $\LSV_\level$.  Hence, it induces $\cO_{\LSV_\level^\Min} \Mi \TorMap_{\Sigma, *}(\cO_{\LSV_{\level, \Sigma}^\Tor})$, by the normality of $\LSV_\level^\Min$ and Zariski's main theorem.

        \item\label{prop-lsv-tor-facts-loc} $\LSV_\level^\an \Em \LSV_{\level, \Sigma}^{\Tor, \an}$ is locally isomorphic to analytifications of affine toroidal embeddings as in \cite[\aCh I]{Kempf/Knudsen/Mumford/Saint-Donat:1973-TE-1}, and $\Sigma$ tells us which affine toroidal embeddings are needed in such local descriptions.  Concretely, when we work componentwise with some $\LSV_\Gamma^\an = \Gamma \Lquot \Dom^+ \cong \Gamma^\ad \Lquot \Dom^+ \Em \LSV_{\Gamma, \Sigma}^{\Tor, \an}$ \Pth{where, by abuse of notation, $\Sigma$ means the cone decomposition for this particular component}, based on \cite[\aCh III]{Ash/Mumford/Rapoport/Tai:2010-SCL-2}, we have the following:
            \begin{enumerate}
                \item\label{prop-lsv-tor-facts-loc-parab} Let $\Bd$ be a rational boundary component of $\Dom^+$, which is a subset of $\Flalg^\std(\bC) \cong \Grp{H}(\bC) / \Grp{P}^\std(\bC)$ \Pth{see \Refeq{\ref{eq-Borel-emb}}} stabilized by a \emph{rational parabolic subgroup} of $\Grp{H}^\ad(\bR)^+$, by \cite[\aCh III, \aProp 3.6]{Ash/Mumford/Rapoport/Tai:2010-SCL-2}.  By \cite[\aCh III, \aDef 3.12]{Ash/Mumford/Rapoport/Tai:2010-SCL-2}, this \Pth{real Lie} subgroup is of the form $\Grp{P}_\Bd^\ad(\bR) \cap \Grp{H}^\ad(\bR)^+$, for a uniquely associated parabolic subgroup $\Grp{P}_\Bd^\ad$ of the algebraic group $\Grp{H}^\ad$ over $\bQ$.  Let $\Grp{W}_\Bd^\ad$ \Pth{\resp $\Grp{U}_\Bd^\ad$, \resp $\Grp{V}_\Bd^\ad$} denote the unipotent radical of $\Grp{P}_\Bd^\ad$ \Pth{\resp center of $\Grp{W}_\Bd^\ad$, \resp{} $\Grp{W}_\Bd^\ad / \Grp{U}_\Bd^\ad$}.

                \item\label{prop-lsv-tor-facts-loc-unip} Consider $\Dom^+ \Em \Dom^+(\Bd) := \Grp{U}_\Bd^\ad(\bC) \cdot \Dom^+$ in $\Flalg^\std(\bC)$, and $\Dom^+(\Bd)' := \Dom^+(\Bd) / \Grp{U}_\Bd^\ad(\bC)$.  Then $\Dom^+(\Bd) \to \Dom^+(\Bd)'$ is canonically a $\Grp{U}_\Bd^\ad(\bC)$-torsor, and $\Dom^+(\Bd)' \to \Bd$ is canonically a $\Grp{V}_\Bd^\ad(\bR)$-torsor.  \Pth{See the summary near the end of \cite[\aCh III, \aSec 4.3]{Ash/Mumford/Rapoport/Tai:2010-SCL-2}.}

                \item\label{prop-lsv-tor-facts-loc-arith} Consider $\Gamma_{\Grp{P}_\Bd}^\ad := \Gamma^\ad \cap \Grp{P}_\Bd^\ad(\bQ)$ and $\Gamma_{\Grp{U}_\Bd}^\ad := \Gamma^\ad \cap \Grp{U}_\Bd^\ad(\bQ)$.

                \item\label{prop-lsv-tor-facts-loc-tor-quot} Consider $\Grp{T}_\Bd^\an := \Grp{U}_\Bd^\ad(\bC) / \Gamma_{\Grp{U}_\Bd}^\ad$, which is the analytification of an algebraic torus $\Grp{T}_\Bd$ over $\bC$ with cocharacter group $\Gamma_{\Grp{U}_\Bd}^\ad$.  Then $\Gamma_{\Grp{U}_\Bd}^\ad \Lquot \Dom^+(\Bd) \to \Dom^+(\Bd)'$ is canonically a $\Grp{T}_\Bd^\an$-torsor.  We shall denote $\Grp{T}_\Bd$ by ${}^{\Gamma^\ad} \Grp{T}_\Bd$ when we want to emphasize its dependence on $\Gamma^\ad$.

                \item\label{prop-lsv-tor-facts-loc-tor-emb} Consider the algebraic toroidal embedding $\Grp{T}_\Bd \Em \Grp{T}_{\Bd, {\{\sigma_\alpha\}}}$, where $\{\sigma_\alpha\}$ is given by $\Sigma$ and admits an action of $\Gamma_{\Grp{P}_\Bd}^\ad$ \Pth{with a finite number of orbits}.  Then the pushout of the $\Grp{T}_\Bd^\an$-torsor $\Gamma_{\Grp{U}_\Bd}^\ad \Lquot \Dom^+(\Bd) \to \Dom^+(\Bd)'$ via the analytification $\Grp{T}_\Bd^\an \Em \Grp{T}_{\Bd, \{\sigma_\alpha\}}^\an$ gives an analytic toroidal embedding $\bigl(\Gamma_{\Grp{U}_\Bd}^\ad \Lquot \Dom^+(\Bd)\bigr) \Em \bigl(\Gamma_{\Grp{U}_\Bd}^\ad \Lquot \Dom^+(\Bd)\bigr) \times^{\Grp{T}_\Bd^\an} \Grp{T}_{\Bd, \{\sigma_\alpha\}}^\an$ over $\Dom^+(\Bd)'$.

                \item\label{prop-lsv-tor-facts-loc-bd} Finally, form the quotient of $\bigl(\Gamma_{\Grp{U}_\Bd}^\ad \Lquot \Dom^+(\Bd)\bigr) \times^{\Grp{T}_\Bd^\an} \Grp{T}_{\Bd, \{\sigma_\alpha\}}^\an \to \Dom^+(\Bd)' \to \Bd$ by $\Gamma_{\Grp{P}_\Bd}^\ad / \Gamma_{\Grp{U}_\Bd}^\ad$.  Then $\Gamma_{\Grp{P}_\Bd}^\ad \Lquot \Bd$ gives a boundary stratum of $\LSV_\Gamma^{\Min, \an}$, and an open subspace of $\bigl(\Gamma_{\Grp{P}_\Bd}^\ad / \Gamma_{\Grp{U}_\Bd}^\ad\bigr) \Lquot \bigl(\bigl(\Gamma_{\Grp{U}_\Bd}^\ad \Lquot \Dom^+(\Bd)\bigr) \times^{\Grp{T}_\Bd^\an} \Grp{T}_{\Bd, \{\sigma_\alpha\}}^\an\bigr)$ is what forms the an \emph{open neighborhood} in $\LSV_{\Gamma, \Sigma}^{\Tor, \an}$ of the preimage of $\Gamma_{\Grp{P}_\Bd}^\ad \Lquot \Bd$ via the canonical map $\TorMap_\Sigma^\an: \LSV_{\Gamma, \Sigma}^{\Tor, \an} \to \LSV_\Gamma^{\Min, \an}$.
            \end{enumerate}
            By passing to formal completions of both $\LSV_{\Gamma, \Sigma}^\Tor$ and $\partial \LSV_{\Gamma, \Sigma}^\Tor$, which we can compare with their analytic analogues, and by Artin's approximation \Pth{see \cite[\aThm 1.12, and the proof of the corollaries in \aSec 2]{Artin:1969-aascl}}, we see that $\LSV_\level \Em \LSV_{\level, \Sigma}^\Tor$ is \'etale locally products of toroidal embeddings of the form $\Grp{T}_\Bd \Em \Grp{T}_{\Bd, \{\sigma_\alpha\}}$ with identity morphisms of smooth schemes.

        \item\label{prop-lsv-tor-facts-sm} There is a condition on $\Sigma$, depending only on the subgroups $\Gamma_{\Grp{U}_\Bd}^\ad$ of $\Grp{U}_\Bd^\ad(\bQ)$ as in \Refenum{\ref{prop-lsv-tor-facts-loc-arith}}, that ensures that all relevant toroidal embeddings above have smooth targets, in which case the underlying scheme of $\LSV_{\level, \Sigma}^\Tor$ is also smooth; and that $\partial \LSV_{\level, \Sigma}^\Tor$ is a \emph{normal crossings divisor}.

        \item\label{prop-lsv-tor-facts-ref} The collection $\Sigma$ is partially ordered by \emph{refinements}.  For any finite number of given collections of cone decompositions $\Sigma_i$'s, we can find a $\Sigma$ refining all of the $\Sigma_i$'s and satisfying any subset of the above conditions.

        \item\label{prop-lsv-tor-facts-log-sm} By \Refenum{\ref{prop-lsv-tor-facts-loc}} and \cite[7.3, (b)--(d)]{Illusie:2002-fknle}, if we equip $\LSV_{\level, \Sigma}^\Tor$ with the fs log structure induced by the divisor $\partial \LSV_{\level, \Sigma}^\Tor$, or equivalently the direct image of the trivial log structure on the open subvariety $\LSV_\level$, then $\LSV_{\level, \Sigma}^\Tor$ is \emph{log smooth} over $\bC$.  Then the sheaf of log differentials $\Omega_{\LSV_{\level, \Sigma}^\Tor}^{\log, 1}$ \Pth{with the base field $\bC$ omitted from notation}, which is denoted by $\omega_{\LSV_{\level, \Sigma}^\Tor / \Spec(\bC)}^1$ in \cite[(1.7)]{Kato:1989-lsfi}, and its exterior powers $\Omega_{\LSV_{\level, \Sigma}^\Tor}^{\log, \bullet} = \Ex^\bullet \Omega_{\LSV_{\level, \Sigma}^\Tor}^{\log, 1}$, are defined.  In the top degree $d := \dim_\bC(\Dom^+)$, we have the \emph{log canonical bundle} $\Omega_{\LSV_{\level, \Sigma}^\Tor}^{\log, d}$.  When $\Sigma$ is smooth and $\partial \LSV_{\level, \Sigma}^\Tor$ is a normal crossings divisor, we have $\Omega_{\LSV_{\level, \Sigma}^\Tor}^{\log, 1} \cong \Omega_{\LSV_{\level, \Sigma}^\Tor}^1(\log \partial \LSV_{\level, \Sigma}^\Tor)$, the sheaf of differentials with log poles along $\partial \LSV_{\level, \Sigma}^\Tor$.  In this case, $\Omega_{\LSV_{\level, \Sigma}^\Tor}^{\log, d}$ is the usual log canonical bundle.  \Pth{This justifies the above terminologies.}

        \item\label{prop-lsv-tor-facts-min-can} Suppose $\Sigma$ is smooth.  By \cite[\aProp 3.4]{Mumford:1977-hptnc} \Pth{and GAGA \cite{Serre:1955-1956-gaga}}, $\Omega_{\LSV_{\level, \Sigma}^\Tor}^{\log, \bullet}$ are the \emph{canonical extensions} \Pth{constructed for more general automorphic vector bundles in \cite{Mumford:1977-hptnc}} of $\Omega_{\LSV_\level}^\bullet$, and $\Omega_{\LSV_{\level, \Sigma}^\Tor}^{\log, d}$ descends to an ample line bundle on $\LSV_\level^\Min$, which we shall abusively denote by $\Omega_{\LSV_\level^\Min}^d$ and call the \emph{canonical bundle} on $\LSV_\level^\Min$.  Thus, $\Omega_{\LSV_{\level, \Sigma}^\Tor}^{\log, d} \cong \TorMap_\Sigma^* \Omega_{\LSV_\level^\Min}^d$, and so $(\Omega_{\LSV_\level^\Min}^d)^{\otimes m} \cong \TorMap_{\Sigma, *} \TorMap_\Sigma^* \bigl((\Omega_{\LSV_\level^\Min}^d)^{\otimes m}\bigr) \cong \TorMap_{\Sigma, *} \bigl((\Omega_{\LSV_{\level, \Sigma}^\Tor}^{\log, d})^{\otimes m}\bigr)$, by \Refenum{\ref{prop-lsv-tor-facts-min}} and the projection formula, for all $m \in \bZ$.  Since $\Omega_{\LSV_\level^\Min}^d$ is ample over the proper scheme $\LSV_\level^\Min$ over $\Spec(\bC)$, we obtain $\LSV_\level^\Min \cong \Proj\bigl(\oplus_{m \geq 0} H^0(\LSV_\level^\Min, (\Omega_{\LSV_\level^\Min}^d)^{\otimes m})\bigr) \cong \Proj\bigl(\oplus_{m \geq 0} H^0(\LSV_{\level, \Sigma}^\Tor, (\Omega_{\LSV_{\level, \Sigma}^\Tor}^{\log, d})^{\otimes m})\bigr)$.

        \item\label{prop-lsv-tor-facts-prod} The formation of minimal and toroidal compactifications is naturally compatible with direct products of the data involved.
    \end{enumerate}
\end{proposition}

\subsection{Compatibility with central isogenies}\label{sec-lsv-funct}

\begin{lemma}\label{lem-comp-ad-isog}
    Let $\homom: \Grp{H}_1 \to \Grp{H}_2$ be any central isogeny of connected algebraic groups over $\bQ$.  Then it induces an isomorphism $\homom^\ad: \Grp{H}_1^\ad \Mi \Grp{H}_2^\ad$.  Equivalently, $Z(\Grp{H}_1)$ is the schematic preimage of $Z(\Grp{H}_2)$; \ie, $Z(\Grp{H}_1) \cong Z(\Grp{H}_2) \times_{\Grp{H}_2} \Grp{H}_1$.
\end{lemma}
\begin{proof}
    This is well known, but let us still include some explanation, for the convenience of the readers.  Since the base field $\bQ$ is of characteristic zero, $\homom$ is finite \'etale.  Therefore, we may verify the lemma by base change from $\bQ$ to any algebraic closure $\AC{\bQ}$ of $\bQ$, and compare the points of $Z(\Grp{H}_1)$ and $Z(\Grp{H}_2) \times_{\Grp{H}_2} \Grp{H}_1$ over $\AC{\bQ}$.  Since $\homom$ is a central isogeny between connected algebraic groups, it is surjective and induces a surjection homomorphism $\homom(\AC{\bQ}): \Grp{H}_1(\AC{\bQ}) \to \Grp{H}_2(\AC{\bQ})$, which maps $Z(\Grp{H}_1)(\AC{\bQ})$ to $Z(\Grp{H}_2)(\AC{\bQ})$.  Conversely, suppose $h_0 \in (Z(\Grp{H}_2) \times_{\Grp{H}_2} \Grp{H}_1)(\AC{\bQ})$, or equivalently a point of $\Grp{H}_1(\AC{\bQ})$ mapped to $Z(\Grp{H}_2)(\AC{\bQ})$ in $\Grp{H}_2(\AC{\bQ})$.  Consider the commutator morphism $\Grp{H}_{1, \AC{\bQ}} \to \Grp{H}_{1, \AC{\bQ}}: h \mapsto h_0 h h_0^{-1} h^{-1}$.  Since $h_0$ is mapped to $Z(\Grp{H}_2)(\AC{\bQ})$ in $\Grp{H}_2(\AC{\bQ})$, the morphism factors through the discrete $\ker(\homom({\AC{\bQ}}))$, or rather the identity, by the connectedness of $\Grp{H}_1$ \Pth{and hence of $\Grp{H}_{1, \AC{\bQ}}$}.  Thus, $h_0 \in \Grp{H}_1(\AC{\bQ})$, as desired.
\end{proof}

Let $\Grp{H}^\der$ denote the derived group of $\Grp{H}$ \Pth{see \cite[\aDef A.1.14]{Conrad/Gabber/Prasad:2010-PRG} and the references there}.  Then $\Grp{H}^\der$ is connected and semisimple.  By \cite[\aCor A.4.11]{Conrad/Gabber/Prasad:2010-PRG}, there is a central isogeny $\Grp{H}^\scc \to \Grp{H}^\der$ from a simply-connected connected semisimple algebraic group $\Grp{H}^\scc$, which we shall call the \emph{simply-connected cover} of $\Grp{H}^\der$, whose formation is compatible with base change from $\bQ$ to any extension field.  Let $\Grp{H}^{\scc, \ad} := \Grp{H}^\scc / Z(\Grp{H}^\scc)$ and $\Grp{H}^{\der, \ad} := \Grp{H}^\der / Z(\Grp{H}^\der)$, as in the case of $\Grp{H}^\ad = \Grp{H} / Z(\Grp{H})$.

\begin{lemma}\label{lem-comp-ad}
    \begin{enumerate}
        \item\label{lem-comp-ad-sc-vs-der} The canonical homomorphism $\Grp{H}^{\scc, \ad} \to \Grp{H}^{\der, \ad}$ induced by $\Grp{H}^\scc \to \Grp{H}^\der$ is an isomorphism.

        \item\label{lem-comp-ad-der-vs-ad} The canonical homomorphism $\Grp{H}^{\der, \ad} \to \Grp{H}^\ad$ induced by the composition of $\Grp{H}^\der \to \Grp{H} \to \Grp{H}^\ad$ is an isomorphism.  Equivalently, $Z(\Grp{H}^\der) = Z(\Grp{H}) \cap \Grp{H}^\der$ in $\Grp{H}$; \ie, $Z(\Grp{H}^\der) \cong Z(\Grp{H}) \times_{\Grp{H}} \Grp{H}^\der$.
    \end{enumerate}
\end{lemma}
\begin{proof}
    Again, these are well known, but let us still include some explanations, for the convenience of the readers.  As for \Refenum{\ref{lem-comp-ad-sc-vs-der}}, it is a special case of Lemma \ref{lem-comp-ad-isog}.  As for \Refenum{\ref{lem-comp-ad-der-vs-ad}}, since $H^\der$ is a subgroup of $\Grp{H}$, we have $Z(\Grp{H}) \cap \Grp{H}^\der \subset Z(\Grp{H}^\der)$, essentially by definition.  Conversely, since $\Grp{H}$ is reductive, the homomorphism $Z(\Grp{H}) \times \Grp{H}^\der \to \Grp{H}$ of algebraic groups over $\bQ$ is surjective, by comparing Lie algebras.  Therefore, $\Grp{H} = Z(\Grp{H}) \cdot \Grp{H}^\der$, and we obtain $Z(\Grp{H}^\der) \subset Z(\Grp{H}) \cap \Grp{H}^\der$, as desired.
\end{proof}

In the following constructions, let us assume that the following holds:
\begin{condition}\label{cond-lsv-funct}
    Suppose we have pairs $(\Grp{H}_1, \Dom_1)$ and $(\Grp{H}_2, \Dom_2)$ as in Section \ref{sec-lsv-setup}, together with a group homomorphism $\homom: \Grp{H}_1 \to \Grp{H}_2$ inducing a central isogeny $\homom^\der: \Grp{H}_1^\der \to \Grp{H}_2^\der$ and, by Lemma \ref{lem-comp-ad-isog}, also an isomorphism $\homom^\ad: \Grp{H}_1^\ad \Mi \Grp{H}_2^\ad$.  Suppose we also have a holomorphic map $\homom_\Dom: \Dom_1 \to \Dom_2$ equivariant with $\homom^\ad(\bR): \Grp{H}_1^\ad(\bR) \Mi \Grp{H}_2^\ad(\bR)$, which necessarily induces an isomorphism from any fixed choice of a connected component $\Dom_1^+$ of $\Dom_1$ to a connected component $\Dom_2^+$ of $\Dom_2$.
\end{condition}

\begin{construction}\label{constr-lsv-funct}
    Let $\Gamma_1 \subset \Grp{H}_1(\bQ)_+$ and $\Gamma_2 \subset \Grp{H}_2(\bQ)_+$ be neat arithmetic subgroups of $\Grp{H}_1(\bQ)$ and $\Grp{H}_2(\bQ)$, respectively, such that $\homom^\ad(\bQ)(\Gamma_1^\ad) \subset \Gamma_2^\ad$ in $\Grp{H}_2^\ad(\bQ)$.  Since $\Dom_1^+ \Mi \Dom_2^+$, the canonical morphism $[1]_{\Gamma_1, \Gamma_2}^\an: \LSV_{\Gamma_1}^\an = \Gamma_1 \Lquot \Dom_1^+ = \Gamma_1^\ad \Lquot \Dom_1^+ \to \LSV_{\Gamma_2}^\an = \Gamma_1 \Lquot \Dom_2^+ = \Gamma_2^\ad \Lquot \Dom_2^+$ is finite \'etale surjective, which uniquely algebraizes to a finite \'etale surjective morphism $[1]_{\Gamma_1, \Gamma_2}: \LSV_{\Gamma_1} \to \LSV_{\Gamma_2}$, by \cite{Borel:1972-maset} and GAGA \cite{Serre:1955-1956-gaga}.
\end{construction}

\begin{construction}\label{constr-lsv-tor-funct}
    In Construction \ref{constr-lsv-funct}, suppose moreover that $\Sigma_1$ and $\Sigma_2$ are projective cone decompositions for $\LSV_{\Gamma_1}$ and $\LSV_{\Gamma_2}$, respectively, such that $\Sigma_1$ refines the pullback of $\Sigma_2$.  By comparing the extension properties of $\LSV_{\Gamma_1, \Sigma_1}^{\Tor, \an}$ and $\LSV_{\Gamma_2, \Sigma_2}^{\Tor, \an}$ as in \cite[\aCh III, \aThm 7.5]{Ash/Mumford/Rapoport/Tai:2010-SCL-2} and by GAGA \cite{Serre:1955-1956-gaga}, for $? = \emptyset$ and $\an$, the morphism $[1]_{\Gamma_1, \Gamma_2}^?$ \Pth{uniquely} extends to a proper surjective morphism $[1]_{(\Gamma_1, \Sigma_1), (\Gamma_2 \Sigma_2)}^{\Tor, ?}: \LSV_{\Gamma_1, \Sigma_1}^{\Tor, ?} \to \LSV_{\Gamma_2, \Sigma_2}^{\Tor, ?}$ such that $\LSV_{\Gamma_1}^?$ is the open preimage of $\LSV_{\Gamma_2}^?$.
\end{construction}

\begin{construction}\label{constr-dcs-funct}
    Let $g \in \Grp{H}_2(\bAi)$.  Let $\levcp_1 \subset \Grp{H}_1(\bAi)$ and $\levcp_2 \subset \Grp{H}_2(\bAi)$ be neat open compact subgroups such that $\homom(\bAi)(\levcp_1) \subset g \levcp_2 g^{-1}$.  Then we have a canonical morphism $[g]_{\levcp_1, \levcp_2}^\an: \LSV_{\levcp_1}^\an = \LSV_{\levcp_1}^\an(\Grp{H}_1, \Dom_1) \to \LSV_{\levcp_2}^\an = \LSV_{\levcp_2}^\an(\Grp{H}_2, \Dom_2)$ defined by right multiplication by $g$, which uniquely algebraize to a canonical morphism $[g]_{\levcp_1, \levcp_2}: \LSV_{\levcp_1} = \LSV_{\levcp_1}(\Grp{H}_1, \Dom_1) \to \LSV_{\levcp_2} = \LSV_{\levcp_2}(\Grp{H}_2, \Dom_2)$, by \cite{Borel:1972-maset} and GAGA \cite{Serre:1955-1956-gaga}.  If $\homom$ is an isomorphism, then $[g]_{\levcp_1, \levcp_2}$ and $[g]_{\levcp_1, \levcp_2}^\an$ are surjective.
\end{construction}

\begin{construction}\label{constr-dcs-comp-funct}
    In Construction \ref{constr-dcs-funct}, suppose $\homom(\bAi)$ maps $h_1 \in \Grp{H}_1(\bAi)$ to $h_2 \in \Grp{H}_2(\bAi)$.  For $? = \emptyset$ and $\an$, the morphism $[g]_{\levcp_1, \levcp_2}^?$ there induces a finite \'etale surjective morphism $[g]_{(\levcp_1, h_1), (\levcp_2, h_2 g)}^?: \LSV_{\levcp_1, h_1}^? \to \LSV_{\levcp_2, h_2 g}^?$, which can be identified with the morphism $[1]_{\Gamma_{\levcp_1, h_1}, \Gamma_{\levcp_2, h_2 g}}^?: \LSV_{\Gamma_{\levcp_1, h_1}}^? \to \LSV_{\Gamma_{\levcp_2, h_2 g}}^?$ in Construction \ref{constr-lsv-funct}.  Consequently, for $? = \emptyset$ and $\an$, the morphism $[g]_{\levcp_1, \levcp_2}^?$ in Construction \ref{constr-dcs-funct} is also finite \'etale, because it is a disjoint union of morphisms $[g]_{(\levcp_1, h_1), (\levcp_2, h_2 g)}^?$ here.
\end{construction}

\begin{construction}\label{constr-dcs-comp-tor-funct}
    In Construction \ref{constr-dcs-comp-funct}, suppose moreover that $\Sigma_1$ and $\Sigma_2$ are cone decompositions for $\LSV_{\levcp_1, h_1}$ and $\LSV_{\levcp_2, h_2 g}$, respectively, such that $\Sigma_1$ refines the pullback of $\Sigma_2$ via $[g]_{(\levcp_1, h_1), (\levcp_2, h_2 g)}$.  By Construction \ref{constr-lsv-tor-funct}, with $\Gamma_1 = \Gamma_{\levcp_1, h_1}$ and $\Gamma_2 = \Gamma_{\levcp_2, h_2 g}$, we see that, for $? = \emptyset$ and $\an$, the morphism $[g]_{(\levcp_1, h_1), (\levcp_2, h_2 g)}^?$ extends to a proper surjective morphism $[g]_{(\levcp_1, h_1, \Sigma_1), (\levcp_2, h_2 g, \Sigma_2)}^?: \LSV_{\levcp_1, h_1, \Sigma_1}^{\Tor, ?} \to \LSV_{\levcp_2, h_2 g, \Sigma_2}^{\Tor, ?}$ such that $\LSV_{\levcp_1, h_1}^?$ is the open preimage of $\LSV_{\levcp_2, h_2 g}^?$.
\end{construction}

\begin{construction}\label{constr-dcs-tor-funct}
    In Construction \ref{constr-dcs-funct}, suppose moreover that $\Sigma_1$ and $\Sigma_2$ are cone decompositions for $\LSV_{\levcp_1}$ and $\LSV_{\levcp_2}$, respectively, such that $\levcp_1$ refines the pullback of $\levcp_2$ in the sense that $\Sigma_{1, h_1}$ refines the pullback of $\Sigma_{2, h_2 g}$ as in Construction \ref{constr-dcs-comp-tor-funct}, for each $h_1 \in \Grp{H}_1(\bAi)$ mapped to $h_2 \in \Grp{H}_2(\bAi)$ via $\homom(\bAi)$.  By applying Construction \ref{constr-dcs-comp-tor-funct} to all $h_1 \in \Grp{H}_1(\bAi)$, we see that, for $? = \emptyset$ and $\an$, the finite morphism $[g]_{\levcp_1, \levcp_2}^?$ in Construction \ref{constr-dcs-funct} extends to a proper morphism $[g]_{(\levcp_1, \Sigma_1), (\levcp_2, \Sigma_2)}^?: \LSV_{\levcp_1, \Sigma_1}^{\Tor, ?} \to \LSV_{\levcp_2, \Sigma_2}^{\Tor, ?}$ such that $\LSV_{\levcp_1}^?$ is the open preimage of $\LSV_{\levcp_2}^?$.  For $? = \emptyset$ and $\an$, when $[g]_{\levcp_1, \levcp_2}^?$ is surjective, so is $[g]_{(\levcp_1, \Sigma_1), (\levcp_2, \Sigma_2)}^?$.
\end{construction}

\begin{construction}\label{constr-dcs-tower}
    Let $(\Grp{H}, \Dom)$ be any pair as in Section \ref{sec-lsv-setup}.  Consider the tower $\{ \LSV_\levcp^\an \}_\levcp$, indexed by neat open compact subgroups $\levcp$ of $\Grp{H}(\bAi)$.  For $g \in \Grp{H}(\bAi)$ and open compact subgroups $\levcp_1$ and $\levcp_2$ of $\Grp{H}(\bAi)$ such that $\levcp_1 \subset g \levcp_2 g^{-1}$, right multiplication by $g$ on the second factor of $\Dom \times \Grp{H}(\bAi)$ induces $[g]_{\levcp_1, \levcp_2}^\an: \LSV_{\levcp_1}^\an \to \LSV_{\levcp_2}^\an$, which uniquely algebraizes to $[g]_{\levcp_1, \levcp_2}: \LSV_{\levcp_1} \to \LSV_{\levcp_2}$, as in Construction \ref{constr-dcs-funct} \Pth{with $(\Grp{H}_1, \Dom_1) = (\Grp{H}_2, \Dom_2) = (\Grp{H}, \Dom)$ and $\homom = \Id_{\Grp{H}}$ in Condition \ref{cond-lsv-funct}}.  Since $\homom_\Dom = \Id_\Dom$, all morphisms $[g]_{\levcp_1, \levcp_2}^\an$ and hence $[g]_{\levcp_1, \levcp_2}$ as above are finite \'etale surjective, which compatibly define canonical right $\Grp{H}(\bAi)$-actions on $\{ \LSV_\levcp \}_\levcp$, for $? = \emptyset$ and $\an$.
\end{construction}

\begin{construction}\label{constr-dcs-normal}
    In Construction \ref{constr-dcs-tower}, for each $g \in \Grp{H}(\bAi)$ normalizing \Pth{\resp contained in} a particular $\levcp$, we have $h \levcp g = hg \levcp$ \Pth{\resp $h \levcp g = h \levcp$}, and hence right multiplication on $\Grp{H}(\bAi)$ defines an action \Pth{\resp trivial action} of $\levcp$ on the double coset space $\LSV_\levcp^\an = \Grp{H}(\bQ)_+ \big\Lquot \bigl(\Dom^+ \times \Grp{H}(\bAi)\bigr) \big/ \levcp$ which permutes \Pth{\resp fixes} its connected components $\LSV_{\levcp, h}^\an$.  Therefore, when $\levcp'$ is a normal subgroup of $\levcp$, for $? = \emptyset$ and $\an$, the action of $\levcp$ on $[1]_{\levcp', \levcp}^?: \LSV_{\levcp'}^? \to \LSV_\levcp^?$ factors through $\levcp / \levcp'$, fixes the target $\LSV_\levcp^?$, and induces actions of $\levcp / \levcp'$ on the fibers.
\end{construction}

\begin{proposition}\label{prop-KZ-mod-K'Z}
    In the setting of Construction \ref{constr-dcs-normal}, let us denote closures of subgroups of $\Grp{H}(\bAi)$ by overlines.  Note that $\levcp \cdot (Z(\Grp{H})(\bQ)) = \levcp \cdot (\overline{Z(\Grp{H})(\bQ)})$ and $\overline{\levcp \cap (Z(\Grp{H})(\bQ))} = \levcp \cap \overline{(Z(\Grp{H})(\bQ))}$, for any open compact subgroup $\levcp$ of $\Grp{H}(\bQ)$.  Suppose $\levcp' \subset \levcp$ are neat open compact subgroups of $\Grp{H}(\bAi)$.  Then, for $? = \emptyset$ and $\an$, the actions of $\levcp / \levcp'$ in Construction \ref{constr-dcs-normal} factor through
    \begin{equation}\label{eq-prop-KZ-mod-K'Z}
        \levcp / \bigl(\levcp' \cdot (\levcp \cap \overline{(Z(\Grp{H})(\bQ))})\bigr) \cong \bigl(\levcp \cdot (Z(\Grp{H})(\bQ))\bigr) \big/ \bigl(\levcp' \cdot (Z(\Grp{H})(\bQ))\bigr)
    \end{equation}
    and makes $[1]_{\levcp', \levcp}^?: \LSV_{\levcp'}^? \to \LSV_\levcp^?$ an \'etale torsor under this finite group \Refeq{\ref{eq-prop-KZ-mod-K'Z}}.
\end{proposition}
\begin{proof}
    Since $[1]_{\levcp', \levcp}$ is finite \'etale, it suffices to prove this proposition in the analytic case.  Let us study the fiber of $[1]_{\levcp', \levcp}^\an: \LSV_{\levcp'}^\an \to \LSV_{\levcp}^\an$ over a double coset $\Grp{H}(\bQ)_+ (\Dompt, h) \levcp$ in $\LSV_\levcp^\an$, where $\Dompt \in \Dom^+$ and $h \in \Grp{H}(\bAi)$.  By definition, any member of this fiber is a double coset $\Grp{H}(\bQ)_+ (\Dompt', h') \levcp'$ in $\LSV_{\levcp'}^\an$, where $\Dompt' \in \Dom^+$ and $h' \in \Grp{H}(\bAi)$, such that $\Dompt' = \gamma \Dompt$ and $h' = \gamma h k$ for some $\gamma \in \Grp{H}(\bQ)_+$ and $k \in \levcp$.  Up to replacing $(\Dompt', h')$ with $(\gamma^{-1} \Dompt', \gamma^{-1} h')$, we may assume that $\Dompt' = \Dompt$ and $h' \in h \levcp$, without changing $\Grp{H}(\bQ)_+ (\Dompt', h') \levcp'$.  In particular, $\levcp / \levcp'$ acts transitively on the fiber.  Since every $\gamma' \in Z(\Grp{H})(\bQ)$ satisfies $(\Dompt, h' \gamma') = (\gamma' \Dompt, \gamma' h') = \gamma' (\Dompt, h')$, right multiplication by $\gamma'$ acts trivially on the fiber.  Hence, the action of $\levcp / \levcp'$ factors through \Refeq{\ref{eq-prop-KZ-mod-K'Z}}.  If $k', k'' \in \levcp$ and $\Grp{H}(\bQ)_+ (\Dompt, h k') \levcp' = \Grp{H}(\bQ)_+ (\Dompt, h k'') \levcp'$, then $\Dompt = \gamma' \Dompt$ and $h k' = \gamma' h k'' k'''$, for some $\gamma' \in \Grp{H}(\bQ)_+$ and $k''' \in \levcp'$.  In this case, $\gamma' \in \Grp{H}(\bQ)_+ \cap (h \levcp h^{-1}) = \Gamma_{\levcp, h}$.  Since $\levcp$ is neat, $\Gamma_{\levcp, h}^\ad$ acts freely on $\Dom$.  Since $\gamma'$ fixes $\Dompt$, its image in $\Gamma_{\levcp, h}^\ad$ is trivial, and hence $\gamma' \in \ker(\Gamma_{\levcp, h} \to \Gamma_{\levcp, h}^\ad) \subset Z(\Grp{H})(\bQ)$.  Since $\gamma' \in h \levcp h^{-1}$, we have $\gamma' = h^{-1} \gamma' h \in \levcp$.  Thus, $k' = h^{-1} \gamma' h k'' k''' = k'' k''' \gamma' \in k'' \levcp' (\levcp \cap \overline{(Z(\Grp{H})(\bQ))})$, and \Refeq{\ref{eq-prop-KZ-mod-K'Z}} acts faithfully \Pth{and transitively} on the fiber, as desired.
\end{proof}

\begin{proposition}\label{prop-lsv-funct}
    Suppose Condition \ref{cond-lsv-funct} holds, so that we have the constructions following it.  Suppose that $(g, \level_1, \level_2)$ is one of the following:
    \begin{enumerate}[(A)]
        \item\label{prop-lsv-funct-case-arith} $(1, \Gamma_1, \Gamma_2)$ for some $\Gamma_1$ and $\Gamma_2$ as in Construction \ref{constr-lsv-funct};

        \item\label{prop-lsv-funct-case-dcs} some $(g, \levcp_1, \levcp_2)$ as in Construction \ref{constr-dcs-funct};

        \item\label{prop-lsv-funct-case-dcs-comp} some $(g, (\levcp_1, h_1), (\levcp_2, h_2 g))$ as in Construction \ref{constr-dcs-comp-funct}.
    \end{enumerate}
    We shall refer to these as Cases \ref{prop-lsv-funct-case-arith}, \ref{prop-lsv-funct-case-dcs}, and \ref{prop-lsv-funct-case-dcs-comp}, respectively.  Then:
    \begin{enumerate}
        \item\label{prop-lsv-funct-op} The canonical finite \'etale morphism $[g]_{\level_1, \level_2}: \LSV_{\level_1} \to \LSV_{\level_2}$ of smooth quasi-projective varieties is surjective in Cases \ref{prop-lsv-funct-case-arith} and \ref{prop-lsv-funct-case-dcs-comp}; and, when $\homom$ is an isomorphism, also in Case \ref{prop-lsv-funct-case-dcs}.

        \item\label{prop-lsv-funct-min} The morphism $[g]_{\level_1, \level_2}$ extends to a canonical finite morphism
            \begin{equation}\label{eq-prop-lsv-funct-min}
                [g]_{\level_1, \level_2}^\Min: \LSV_{\level_1}^\Min \to \LSV_{\level_2}^\Min
            \end{equation}
            of normal projective varieties, which is surjective when $[g]_{\level_1, \level_2}$ is.  Hence, $\LSV_{\level_1}^\Min$ is the normalization of $\LSV_{\level_2}^\Min$ via $\LSV_{\level_1} \to \LSV_{\level_2} \Em \LSV_{\level_2}^\Min$, by Zariski's main theorem.

        \item\label{prop-lsv-funct-tor} If $\Sigma_1$ and $\Sigma_2$ are as in Constructions \ref{constr-lsv-tor-funct}, \ref{constr-dcs-comp-tor-funct}, and \ref{constr-dcs-tor-funct}, then the canonical proper morphism $[g]_{(\level_1, \Sigma_1), (\level_2, \Sigma_2)}^\Tor$ is a \emph{log \'etale} morphism of fs log schemes whose underlying schemes are normal projective varieties, which is surjective when $[g]_{\level_1, \level_2}$ is.  Hence, $[g]_{(\level_1, \Sigma_1), (\level_2, \Sigma_2)}^{\Tor, *} \, \Omega_{\LSV_{\level_2, \Sigma_2}^\Tor}^{\log, \bullet} \cong \Omega_{\LSV_{\level_1, \Sigma_1}^\Tor}^{\log, \bullet}$.  By the arguments in \cite[\aCh I, \aSec 3, especially \apage 44, \aCor 2]{Kempf/Knudsen/Mumford/Saint-Donat:1973-TE-1}, by the local freeness of $\Omega_{\LSV_{\level_2, \Sigma_2}^\Tor}^{\log, \bullet}$ in all degrees, and by the projection formula, it follows that the canonical \Pth{adjunction} morphism
            \begin{equation}\label{eq-prop-lsv-funct-Omega-pushforward}
                (\Omega_{\cO_{\LSV_{\level_2, \Sigma_2}^\Tor}}^{\log, \bullet})^{\otimes m} \to R [g]_{(\level_1, \Sigma_1), (\level_2, \Sigma_2), *}^\Tor \bigl((\Omega_{\LSV_{\Gamma_1, \Sigma_1}^\Tor}^{\log, \bullet})^{\otimes m}\bigr)
            \end{equation}
            is an isomorphism \Pth{\resp the zero morphism} over the connected components of $\LSV_{\level_2, \Sigma_2}^\Tor$ that are in \Pth{\resp not in} the image of $[g]_{(\level_1, \Sigma_1), (\level_2, \Sigma_2)}$, for all $m \in \bZ$ \Pth{with negative tensor powers of $\Omega_{\cO_{\LSV_{\level_2, \Sigma_2}^\Tor}^{\log, \bullet}}$ defined by duality}.

        \item\label{prop-lsv-funct-tor-ket} If the above $\Sigma_1$ is exactly the pullback of $\Sigma_2$, then $[g]_{(\level_1, \Sigma_1), (\level_2, \Sigma_2)}^\Tor$ is \emph{finite Kummer \'etale}, and $\LSV_{\level_1, \Sigma_1}^\Tor$ is the normalization of $\LSV_{\level_2, \Sigma_2}^\Tor$ via $\LSV_{\level_1} \to \LSV_{\level_2} \Em \LSV_{\level_2, \Sigma_2}^\Tor$, by Zariski's main theorem again.  \Pth{The assumption makes sense because the pullback $\Sigma_1$ of a general projective $\Sigma_2$ is projective.  On the contrary, the pullback $\Sigma_1$ of a general smooth $\Sigma_2$ is not necessarily smooth.}

        \item\label{prop-lsv-funct-Gal} Suppose that one of the following holds:
            \begin{enumerate}
                \item[(A$^\prime$)]\label{prop-lsv-funct-Gal-case-arith} $(g, \level_1, \level_2) = (1, \Gamma_1, \Gamma_2)$ is as in Case \ref{prop-lsv-funct-case-arith}, and $\homom^\ad(\bQ)(\Gamma_1^\ad)$ is a normal subgroup of $\Gamma_2^\ad$, in which case $\overline{\level} := \Gamma_2^\ad \big/ \bigl(\homom^\ad(\bQ)(\Gamma_1^\ad)\bigr)$;

                \item[(B$^\prime$)]\label{prop-lsv-funct-Gal-case-dcs} $(g, \level_1, \level_2) = (1, \levcp_1, \levcp_2)$ is as in Case \ref{prop-lsv-funct-case-dcs}, $\homom$ is an isomorphism, and $\homom(\bAi)(\levcp_1)$ is a normal subgroup of $\levcp_2$, in which case $\overline{\level} := \levcp_2 / \bigl(\homom(\bAi)(\levcp_1) \cdot (\levcp_2 \cap \overline{(Z(\Grp{H}_2)(\bQ))})\bigr)$ \Pth{\Refcf{} \Refeq{\ref{eq-prop-KZ-mod-K'Z}}}.
            \end{enumerate}
            In both cases, $\overline{\level}$ is a finite group.  Then:
            \begin{enumerate}
                \item\label{prop-lsv-funct-Gal-op} The natural action of $\level_2$ on $\LSV_{\level_1}$ factors through $\overline{\level}$ and makes the finite \'etale surjective morphism $[1]_{\level_1, \level_2}: \LSV_{\level_1} \to \LSV_{\level_2}$ an \'etale torsor under $\overline{\level}$.  In this case, $\LSV_{\level_1} / \overline{\level} \Mi \LSV_{\level_2}$.

                \item\label{prop-lsv-funct-Gal-min} The action of $\overline{\level}$ on $[1]_{\level_1, \level_2}: \LSV_{\level_1} \to \LSV_{\level_2}$ \Pth{necessarily uniquely} extends to the finite surjective morphism $[1]_{\level_1, \level_2}^\Min: \LSV_{\level_1}^\Min \to \LSV_{\level_2}^\Min$ in \Refeq{\ref{eq-prop-lsv-funct-min}}, and induces $\LSV_{\level_1}^\Min / \overline{\level} \Mi \LSV_{\level_2}^\Min$.

                \item\label{prop-lsv-funct-Gal-tor-ket} If $\Sigma_1$ and $\Sigma_2$ are as in \Refenum{\ref{prop-lsv-funct-tor-ket}}, then the action of $\overline{\level}$ on $[1]_{\level_1, \level_2}: \LSV_{\level_1} \to \LSV_{\level_2}$ \Pth{necessarily uniquely} extends to the finite Kummer \'etale surjective morphism $[1]_{(\level_1, \Sigma_1), (\level_2, \Sigma_2)}^\Tor: \LSV_{\level_1, \Sigma_1}^\Tor \to \LSV_{\level_2, \Sigma_2}^\Tor$, and makes the latter a Kummer \'etale torsor under $\overline{\level}$.  In this case, $\LSV_{\level_1, \Sigma_1}^\Tor / \overline{\level} \Mi \LSV_{\level_2, \Sigma_2}^\Tor$.  The actions of $\overline{\level}$ on $[1]_{(\level_1, \Sigma_1), (\level_2, \Sigma_2)}^\Tor$ and $[1]_{\level_1, \level_2}^\Min$ are compatible.
            \end{enumerate}
    \end{enumerate}
\end{proposition}
\begin{proof}
    Firstly, \Refenum{\ref{prop-lsv-funct-op}} follows from Constructions \ref{constr-lsv-funct}, \ref{constr-dcs-funct}, and \ref{constr-dcs-comp-funct}.  In the latter two constructions, if $\homom(\bAi)(h_1) = h_2$, then $[g]_{\levcp_1, \levcp_2}^\an: \LSV_{\levcp_1}^\an \to \LSV_{\levcp_2}^\an$ maps the connected component $\LSV_{\levcp_1, h_1}^\an = \LSV_{\Gamma_{\levcp_1, h_1}}^\an \cong \LSV_{\Gamma_{\levcp_1, h_1}^\ad}^\an$ of $\LSV_{\levcp_1}^\an$ to $\LSV_{\levcp_2, h_2 g}^\an = \LSV_{\Gamma_{\levcp_2, h_2 g}}^\an \cong \LSV_{\Gamma_{\levcp_2, h_2 g}^\ad}^\an$.  Hence, it suffices to prove \Refenum{\ref{prop-lsv-funct-min}}--\Refenum{\ref{prop-lsv-funct-tor-ket}} in Case \ref{prop-lsv-funct-case-arith}.

    If $\Sigma_1$ and $\Sigma_2$ are as in Construction \ref{constr-lsv-tor-funct}, then $[1]_{(\Gamma_1, \Sigma_1), (\Gamma_2, \Sigma_2)}^{\Tor, \an}: \LSV_{\Gamma_1, \Sigma_1}^{\Tor, \an} \to \LSV_{\Gamma_2, \Sigma_2}^{\Tor, \an}$ is locally given by ${}^{\Gamma_1^\ad} \Grp{T}_{\Bd, \{\sigma_{\alpha'}'\}}^\an \to {}^{\Gamma_2^\ad} \Grp{T}_{\Bd, \{\sigma_\alpha\}}^\an$, in the notation of Proposition \ref{prop-lsv-tor-facts}\Refenum{\ref{prop-lsv-tor-facts-loc}}, where $\{\sigma_{\alpha'}'\}$ is a refinement of $\{\sigma_\alpha\}$, for the same rational boundary component $\Bd$ of $\Dom_1^+ \Mi \Dom_2^+$.  Note that ${}^{\Gamma_1^\ad} \Grp{T}_\Bd \to {}^{\Gamma_2^\ad} \Grp{T}_\Bd$ is an isogeny of tori, because the corresponding homomorphism of cocharacter groups $\Gamma_1^\ad \cap \Grp{U}_\Bd^\ad(\bQ) \to \Gamma_2^\ad \cap \Grp{U}_\Bd^\ad(\bQ)$ is between two arithmetic subgroups of $\Grp{U}_\Bd^\ad(\bQ)$.  By passing to formal completions and by Artin's approximation as in Proposition \ref{prop-lsv-tor-facts}\Refenum{\ref{prop-lsv-tor-facts-loc}}, $[1]_{(\Gamma_1, \Sigma_1), (\Gamma_2, \Sigma_2)}^\Tor$ is \'etale locally of the form ${}^{\Gamma_1^\ad} \Grp{T}_{\Bd, \{\sigma_{\alpha'}'\}} \to {}^{\Gamma_2^\ad} \Grp{T}_{\Bd, \{\sigma_\alpha\}}$, which is \'etale locally some standard log \'etale morphisms as in \cite[\aProp 3.4]{Kato:1989-lsfi}.  Hence, $[1]_{(\Gamma_1, \Sigma_1), (\Gamma_2, \Sigma_2)}^\Tor$ is log \'etale as a morphism of fs log schemes.  The remaining statements in \Refenum{\ref{prop-lsv-funct-tor}} then follow.

    Up to refining both $\Sigma_1$ and $\Sigma_2$, suppose moreover that $\Sigma_1$ and $\Sigma_2$ are not only projective but also smooth.  Since $\Dom_1^+ \Mi \Dom_2^+$, we have $d = \dim_\bC(\Dom_1^+) = \dim_\bC(\Dom_2^+)$.  By Proposition \ref{prop-lsv-tor-facts}\Refenum{\ref{prop-lsv-tor-facts-min-can}}, we obtain a canonical morphism
    \begin{equation}\label{eq-prop-lsv-funct-min-proj}
    \begin{split}
        \LSV_{\Gamma_2}^\Min & \cong \Proj\bigl(\oplus_{m \geq 0} H^0(\LSV_{\Gamma_2, \Sigma_2}^\Tor, (\Omega_{\LSV_{\Gamma_2, \Sigma_2}^\Tor}^{\log, d})^{\otimes m})\bigr) \\
        \to \LSV_{\Gamma_1}^\Min & \cong \Proj\bigl(\oplus_{m \geq 0} H^0(\LSV_{\Gamma_1, \Sigma_1}^\Tor, (\Omega_{\LSV_{\Gamma_1, \Sigma_1}^\Tor}^{\log, d})^{\otimes m})\bigr)
    \end{split}
    \end{equation}
    which is necessarily proper and pulls $\Omega_{\LSV_{\Gamma_2}^\Min}^d$ back to $\Omega_{\LSV_{\Gamma_1}^\Min}^d$.  Since the ample line bundle $\Omega_{\LSV_{\Gamma_2}^\Min}^d$ on $\LSV_{\Gamma_2}^\Min$ pulls back to the ample line bundle $\Omega_{\LSV_{\Gamma_1}^\Min}^d$ on $\LSV_{\Gamma_1}^\Min$, this proper morphism \Refeq{\ref{eq-prop-lsv-funct-min}} is quasi-affine and hence finite.  Since any finite number of cone decompositions admit a common projective smooth refinement, it follows from \Refeq{\ref{eq-prop-lsv-funct-Omega-pushforward}} that the morphism \Refeq{\ref{eq-prop-lsv-funct-min-proj}} is independent of the choices of $\Sigma_1$ and $\Sigma_2$.  Thus, \Refeq{\ref{eq-prop-lsv-funct-min-proj}} gives the desired finite morphism \Refeq{\ref{eq-prop-lsv-funct-min}}, and the remaining statements in \Refenum{\ref{prop-lsv-funct-min}} are self-explanatory.

    Again assuming that $\Sigma_1$ and $\Sigma_2$ are projective but not necessarily smooth, if $\Sigma_1$ is exactly the pullback of $\Sigma_2$, then $[1]_{(\Gamma_1, \Sigma_1), (\Gamma_2, \Sigma_2)}^\Tor$ is locally given by ${}^{\Gamma_1^\ad} \Grp{T}_{\Bd, \{\sigma_{\alpha}\}}^\an \to {}^{\Gamma_2^\ad} \Grp{T}_{\Bd, \{\sigma_\alpha\}}^\an$, as before, but with the \emph{same} $\{\sigma_\alpha\}$ on both sides.  In this case, since $[1]_{(\Gamma_1, \Sigma_1), (\Gamma_2, \Sigma_2)}^{\Tor, \an}$ has finite fibers, the proper morphism $[1]_{(\Gamma_1, \Sigma_1), (\Gamma_2, \Sigma_2)}^\Tor$ is quasi-finite and hence \emph{finite}.  By Artin's approximation for pullbacks of the finite morphism to formal completions on the target, $[1]_{(\Gamma_1, \Sigma_1), (\Gamma_2, \Sigma_2)}^\Tor$ is \'etale locally on the target of the form ${}^{\Gamma_1^\ad} \Grp{T}_{\Bd, \{\sigma_{\alpha}\}} \to {}^{\Gamma_2^\ad} \Grp{T}_{\Bd, \{\sigma_\alpha\}}$, which is \'etale locally on the target some standard log \'etale morphisms as in \cite[\aProp 3.4]{Kato:1989-lsfi} associated with Kummer homomorphisms of monoids as in \cite[1.4]{Illusie:2002-fknle}.  Thus, $[1]_{(\Gamma_1, \Sigma_1), (\Gamma_2, \Sigma_2)}^\Tor$ is finite Kummer \'etale as a morphism of fs log schemes, and the remaining statements in \Refenum{\ref{prop-lsv-funct-tor-ket}} follow.

    Finally, in the setting of \Refenum{\ref{prop-lsv-funct-Gal}}, the statements in \Refenum{\ref{prop-lsv-funct-Gal-op}} are self-explanatory in Case \ref{prop-lsv-funct-Gal-case-arith}, and follow from Proposition \ref{prop-KZ-mod-K'Z} in Case \ref{prop-lsv-funct-Gal-case-dcs}.  As for \Refenum{\ref{prop-lsv-funct-Gal-min}} and \Refenum{\ref{prop-lsv-funct-Gal-tor-ket}}, the action of $\overline{\level}$ extends to the minimal and toroidal compactifications and have the desired properties by taking normalizations and resorting to Zariski's main theorem, as in the last sentences of \Refenum{\ref{prop-lsv-funct-min}} and \Refenum{\ref{prop-lsv-funct-tor-ket}}.  The extended actions of $\overline{\level}$ to the compactifications are necessarily unique and compatible with each other, because they are determined by their restrictions to $[1]_{\level_1, \level_2}: \LSV_{\level_1} \to \LSV_{\level_2}$ by density.
\end{proof}

\begin{remark}\label{rem-dcs-funct-sc}
    A useful special case of Case \ref{prop-lsv-funct-case-dcs} of Proposition \ref{prop-lsv-funct} is when $\Grp{H}_1$ is semisimple and simply-connected, so that $\Grp{H}_1 = \Grp{H}_1^\der = \Grp{H}_1^\scc$.  Then $\Dom_1 = \Dom_1^+$, $\Grp{H}_1(\bR)_+ = \Grp{H}_1(\bR)$, $\Grp{H}_1(\bQ)_+ = \Grp{H}_1(\bQ)$, and $\LSV_{\levcp_1} = \LSV_{\levcp_1, h_1} = \LSV_{\levcp_1}^\circ$, for all $h_1 \in \Grp{H}_1(\bAi)$, as explained in Remark \ref{rem-lsv-sc}.  In this case, we can study $\LSV_{\levcp_1}$ and its compactifications as in Section \ref{sec-lsv-cpt}, and deduce results for $\LSV_{\levcp_2, h_2}$ for all $h_2 \in \Grp{H}_2(\bAi)$, including $\LSV_{\levcp_2}^\circ = \LSV_{\levcp_2, 1}$, and their compactifications.
\end{remark}

\begin{remark}\label{rem-dcs-funct-sc-der}
    In Case \ref{prop-lsv-funct-case-dcs} of Proposition \ref{prop-lsv-funct}, if we assume that $\Grp{H}_1$ is semisimple and simply-connected, as in Remark \ref{rem-dcs-funct-sc}; that the central isogeny $\homom^\der: \Grp{H}_1^\der \to \Grp{H}_2^\der$ is an \emph{isomorphism}; and that $\levcp_1 = \homom(\bAi)^{-1}(g \levcp_2 g^{-1})$.  Then we obtain an isomorphism $[g]_{\levcp_1, \levcp_2}: \LSV_{\levcp_1, h_1} \to \LSV_{\levcp_2, h_2 g}$ of connected components, which is the restriction of the open-and-closed immersion $[g]_{\levcp_1, \levcp_2}: \LSV_{\levcp_1} \to \LSV_{\levcp_2}$ from the connected $\LSV_{\levcp_1}$ into the whole \Pth{generally disconnected} $\LSV_{\levcp_2}$.
\end{remark}

By Case \ref{prop-lsv-funct-case-dcs} in Proposition \ref{prop-lsv-funct}, in the setting in Construction \ref{constr-dcs-tower}, we have:
\begin{corollary}\label{cor-dcs-act}
    The tower $\{ \LSV_\levcp \}_\levcp$ in Construction \ref{constr-dcs-tower}, with its canonical $\Grp{H}(\bAi)$-action given by finite \'etale surjective morphisms as in Proposition \ref{prop-lsv-funct}\Refenum{\ref{prop-lsv-funct-op}}, admits a tower of minimal compactifications $\{ \LSV_\levcp^\Min \}_\levcp$, which are normal projective varieties, with a canonical right $\Grp{H}(\bAi)$-action extending that on $\{ \LSV_\levcp \}_\levcp$, given by finite surjective morphisms as in Proposition \ref{prop-lsv-funct}\Refenum{\ref{prop-lsv-funct-min}}.  It also admits a double tower of toroidal compactifications $\{ \LSV_{\levcp, \Sigma}^\Tor \}_{(\levcp, \Sigma)}$, which are projective normal varieties, with a canonical right $\Grp{H}(\bAi)$-action compatible with those on $\{ \LSV_\levcp \}_\levcp$ and $\{ \LSV_\levcp^\Min \}_\levcp$, given by proper log \'etale morphisms \Pth{\resp finite Kummer \'etale} as in Proposition \ref{prop-lsv-funct}\Refenum{\ref{prop-lsv-funct-tor}} \Pth{\resp{} \Refenum{\ref{prop-lsv-funct-tor-ket}}}, when cone decompositions on the sources refine \Pth{\resp are exactly} the pullbacks of those on the targets.
\end{corollary}

\begin{corollary}\label{cor-KZ-mod-KZ'-tower}
    In the setting of Corollary \ref{cor-dcs-act}, given any neat open compact subgroup $\levcp$ of $\Grp{H}(\bAi)$, and any normal subgroup $\levcp_0$ of $\levcp$, the canonical action of $\levcp$ on the tower $\{ \LSV_{\levcp'} \}_{\levcp_0 \subset \levcp' \subset \levcp}$ \Pth{\resp $\{ \LSV_{\levcp', \Sigma'}^\Tor \}_{\levcp_0 \subset \levcp' \subset \levcp}$} indexed by open subgroups $\levcp'$ of $\levcp$ containing $\levcp_0$ \Pth{where $\Sigma'$ abusively denotes the pullback of $\Sigma$ for each level $\levcp'$} factors through a faithful action of the following quotient group
    \begin{equation}\label{eq-cor-KZ-mod-KZ'-tower}
        \levcp / \bigl(\overline{\levcp_0 \cdot \bigl(\levcp \cap (Z(\Grp{H})(\bQ))\bigr)}\bigr) \cong \bigl(\levcp \cdot (Z(\Grp{H})(\bQ))\bigr) \big/ \overline{\bigl(\levcp_0 \cdot (Z(\Grp{H})(\bQ))\bigr)},
    \end{equation}
    where overlines denote closures of subgroups of $\Grp{H}(\bAi)$ as in Proposition \ref{prop-KZ-mod-K'Z}.
\end{corollary}
\begin{proof}
    Note that
    \[
    \begin{split}
        \overline{\levcp_0 \cdot \bigl(\levcp \cap (Z(\Grp{H})(\bQ))\bigr)} & = \cap_{\levcp_0 \subset \levcp' \subset \levcp} \overline{\levcp' \cdot \bigl(\levcp \cap (Z(\Grp{H})(\bQ))\bigr)} \\
        & = \cap_{\levcp_0 \subset \levcp' \subset \levcp} \levcp' \cdot \bigl(\levcp \cap \overline{(Z(\Grp{H})(\bQ))}\bigr)
    \end{split}
    \]
    and
    \[
        \overline{\levcp_0 \cdot (Z(\Grp{H})(\bQ))} = \cap_{\levcp_0 \subset \levcp' \subset \levcp} \bigl(\levcp' \cdot (Z(\Grp{H})(\bQ))\bigr).
    \]
    Then the assertions follow from Proposition \ref{prop-lsv-funct}\Refenum{\ref{prop-lsv-funct-Gal}} and Corollary \ref{cor-dcs-act}.
\end{proof}

\begin{example}\label{ex-KZ-mod-KZ'-tower-p-infty}
    Let $\levcp^p \subset \Grp{H}(\bAip)$ and $\levcp_p \subset \Grp{H}(\bQ_p)$ be open compact subgroups.  In Corollary \ref{cor-KZ-mod-KZ'-tower}, let $\levcp_0 = \levcp^p$ and $\levcp = \levcp^p \levcp_p$.  Then \Refeq{\ref{eq-cor-KZ-mod-KZ'-tower}} becomes
    \begin{equation}\label{eq-ex-KZ-mod-KZ'-tower-p-infty}
        (\levcp^p \levcp_p) / \bigl(\levcp^p \cdot \bigl((\levcp^p \levcp_p) \cap \overline{(Z(\Grp{H})(\bQ))}\bigr)\bigr).
    \end{equation}
    This is consistent with \cite[2.1.9]{Deligne:1979-vsimc} in the setting of Shimura varieties \Pth{which we shall consider later in Section \ref{sec-Sh-var}}.  In \cite[\aSec 4.1]{RodriguezCamargo:2022-lacc}, this is denoted by $\widetilde{\levcp}_p$, with $\Grp{H}$ and $Z(\Grp{H})$ here denoted by $\mathbf{G}$ and $\mathbf{Z}$ there, respectively.
\end{example}

\begin{remark}\label{rem-KZ-mod-KZ'-tower-cf-K-c}
    Since $Z(\Grp{H})^\circ / Z_s(\Grp{H})$ has the same split ranks over $\bQ$ and $\bR$, the closure $\overline{Z(\Grp{H})(\bQ)}$ of $Z(\Grp{H})(\bQ)$ in $Z(\Grp{H})(\bAi)$ \Pth{and hence in $\Grp{H}(\bAi)$} is contained in $Z_s(\Grp{H})(\bAi)$.  Thus, we have surjections $\levcp \Surj \levcp / (\levcp \cap (\overline{Z(\Grp{H})(\bQ)})) \Surj \levcp^c$, where $\levcp / (\levcp \cap (\overline{Z(\Grp{H})(\bQ)}))$ is as in \Refeq{\ref{eq-cor-KZ-mod-KZ'-tower}} \Pth{with $\levcp_0 = \{1\}$ in the notation there}.
\end{remark}

\subsection{Automorphic vector bundles and their canonical extensions}\label{sec-lsv-aut-bdl-can-ext}

Let $(\Grp{H}, \Dom)$ be any pair as in Section \ref{sec-lsv-setup}, and let $\level$ be as in Section \ref{sec-lsv-cpt}, so that $\LSV_\level$ and its minimal and toroidal compactifications are defined.

Recall that, by our convention, $\Rep_\bC(\Grp{P}^{\std, c})$ and $\Rep_\bC(\Grp{M}^c)$ denote the categories of finite-dimensional representations of $\Grp{P}^{\std, c}$ and $\Grp{M}^c$, respectively, over $\bC$.  Each $\rW \in \Rep_\bC(\Grp{P}^{\std, c})$ defines a $\Grp{H}_\bC$-equivariant vector bundle on $\Flalg^\std$ whose analytification pulls back via \Refeq{\ref{eq-Borel-emb}} to $\Dom^+$, naturally descends to a vector bundle on any $\LSV_\level^\an$ we consider, and canonically algebraizes to a vector bundle on $\LSV_\level$, called the \emph{automorphic vector bundle} associated with $\rW$.  Moreover, by Remark \ref{rem-fil-by-hc-ad}, $\rW$ is equipped with a canonical decreasing filtration $\Fil^\bullet \rW$ by subrepresentations of $\Grp{P}^{\std, c}$.  By abuse of notation, we shall denote by $\GrSh{\rW}$ all of these vector bundles \Pth{including their analytifications} associated with $\rW$ in $\Rep_\bC(\Grp{P}^{\std, c})$, and denote the filtrations on them by $\Fil^\bullet = \Fil^\bullet \, \GrSh{\rW}$.  The assignment $\rW \mapsto (\GrSh{\rW}, \Fil^\bullet)$ defines a tensor functor from $\Rep_\bC(\Grp{P}^{\std, c})$ to the category of filtered vector bundles over $\LSV_\level$.  For each \Pth{projective} cone decomposition $\Sigma$ for $\LSV_\levcp$, the analytic $(\GrSh{\rW}, \Fil^\bullet)$ on $\LSV_\level^\an$ admits a \emph{canonical extension} $(\GrSh{\rW}^\canext, \Fil^\bullet)$ over $\LSV_{\level, \Sigma}^{\Tor, \an}$, which algebraizes by GAGA \cite{Serre:1955-1956-gaga} to a vector bundle which we still denote by $(\GrSh{\rW}^\canext, \Fil^\bullet)$ over $\LSV_{\level, \Sigma}^\Tor$, whose restriction to $\LSV_\level$ is the above algebraic $(\GrSh{\rW}, \Fil^\bullet)$.  \Pth{See \cite[\aThms 1.4.1 and 4.2]{Harris:1989-ftcls} for the above statements in the context of Shimura varieties, which we will introduce later in Section \ref{sec-Sh-var}.  The arguments in the proofs there apply here, because they are over individual connected components and use nothing more than general properties of locally symmetric varieties.}  We shall say that the algebraic $(\GrSh{\rW}^\canext, \Fil^\bullet)$ over $\LSV_{\level, \Sigma}^\Tor$ is the canonical extension of the algebraic $(\GrSh{\rW}, \Fil^\bullet)$ over $\LSV_\level$.  When the context is clear, we will often suppress $\Fil^\bullet$ from the notation.  For $\rW \in \Rep_\bC(\Grp{M}^c)$, we shall view it as a representation of $\Grp{P}^{\std, c}$ via the canonical homomorphism $\Grp{P}^{\std, c} \to \Grp{M}^c$, and all the above still applies.

\begin{remark}\label{rem-can-ext-funct}
    It follows immediately from the characterizing property of canonical extensions as in \cite[(4.1.1)]{Harris:1989-ftcls} that the formation of canonical extensions as above is compatible with all morphisms between toroidal compactifications in Proposition \ref{prop-lsv-funct}\Refenum{\ref{prop-lsv-funct-tor}}.  See \cite[4.3.1, 4.3.2, and 4.3.3]{Harris:1989-ftcls} for some explicitly stated cases.
\end{remark}

\begin{remark}\label{rem-can-ext-log-diff}
    As in Proposition \ref{prop-lsv-tor-facts}\Refenum{\ref{prop-lsv-tor-facts-min-can}} \Pth{based on \cite[\aProp 3.4]{Mumford:1977-hptnc}}, for $\rW = \Lie \Grp{N}^\std$ in $\Rep_\bC(\Grp{M})$ \Pth{with adjoint action}, we have $\GrSh{\Lie \Grp{N}^\std} \cong \Omega_{\LSV_\level}^1$ over $\LSV_\level$, which extends to $(\GrSh{\Lie \Grp{N}^\std})^\canext \cong \Omega_{\LSV_{\level, \Sigma}^\Tor}^{\log, 1} \cong \Omega_{\LSV_{\level, \Sigma}^\Tor}^1(\log \partial \LSV_{\level, \Sigma}^\Tor)$ over $\LSV_{\level, \Sigma}^\Tor$ when $\Sigma$ is smooth.
\end{remark}

For $\rV \in \Rep_\bC(\Grp{G}^c)$, the filtered vector bundle $(\GrSh{\rV}^\canext, \Fil^\bullet)$ associated with the restriction $\rV|_{\Grp{P}^{\std, c}}$ as above admits the additional structure of a canonical integrable log connection $\nabla: \GrSh{\rV}^\canext \to \GrSh{\rV}^\canext \otimes_{\cO_{\LSV_{\level, \Sigma}^\Tor}} \Omega_{\LSV_{\level, \Sigma}}^{\log, 1}$ satisfying Griffiths transversality; \ie, $\nabla(\Fil^\bullet) \subset (\Fil^{\bullet - 1}) \otimes_{\cO_{\LSV_{\level, \Sigma}^\Tor}} \Omega_{\LSV_{\level, \Sigma}}^{\log, 1}$.  By restriction to $\LSV_\level$, we obtain a canonical integral connection $\nabla$ with regular singularity \Pth{still satisfying Griffiths transversality}.  In order to emphasize that such constructions are meant to be useful for studying the de Rham cohomology, we shall denote by $(\dRSh{\rV}, \nabla, \Fil^\bullet)$ and $(\dRSh{\rV}^\canext, \nabla, \Fil^\bullet)$ the associated filtered connections and log connections, respectively.

\begin{remark}\label{rem-can-ext-conn}
    By \cite[(4.2.2)]{Harris:1989-ftcls}, the above $(\dRSh{\rV}^\canext, \nabla)$ above is canonically isomorphic to the canonical extension of $(\dRSh{\rV}, \nabla)$ introduced by Deligne in \cite{Deligne:1970-EDR}.
\end{remark}

\subsection{Automorphic line bundles of positive parallel weights}\label{sec-lsv-pos}

Let us continue with the setting of Section \ref{sec-lsv-aut-bdl-can-ext}.  Let $\Grp{H}^\scc$ be the simply-connected cover of $\Grp{H}^\der$, as in Section \ref{sec-lsv-cpt}.  Let $\Grp{H}^{\der, c}$ abusively denote the image of $\Grp{H}^\der$ in $\Grp{H}^c$.
\begin{construction}\label{constr-sc-ad-Levi}
    By Lemma \ref{lem-comp-ad}, we have central isogenies $\Grp{H}^\scc \to \Grp{H}^\der \to \Grp{H}^{\der, c} \to \Grp{H}^\ad$.  Let us abusively and temporarily denote \Pth{in this subsection} the induced central isogenies of Levi subgroups by $\Grp{M}^\scc \to \Grp{M}^\der \to \Grp{M}^{\der, c} \to \Grp{M}^\ad$.
\end{construction}

\begin{construction}\label{constr-sc-ad-decomp}
    Since $\Grp{H}^\scc$ is connected, semisimple, and simply-connected, it decomposes into a product $\Grp{H}^\scc \cong \prod_{j \in J} \Grp{H}_j$ of its $\bQ$-simple factors.  Accordingly, $\Dom^+ \cong \prod_{j \in J} \Dom^+_j$, and each $(\Grp{H}_j, \Dom^+_j)$ is still a pair as in Section \ref{sec-lsv-setup}.  By Lemma \ref{lem-comp-ad}, we have a compatible product $\Grp{H}^\ad \cong \prod_{j \in J} \Grp{H}_j^\ad$ of adjoint quotients.
\end{construction}

\begin{construction}\label{constr-sc-ad-decomp-det-N}
    For each $j \in J$ \Pth{as above}, denote by $\Grp{M}_j$, $\Grp{N}_j^\std$, etc the respective pullbacks of the subgroups $\Grp{M}$, $\Grp{N}^\std$, etc of $\Grp{H}_\bC$ to $\Grp{H}_{j, \bC}$; and abusively denote by $\Grp{M}_j^\ad$ the pullback of the \Pth{abusively denoted} subgroup $\Grp{M}^\ad$ of $\Grp{H}^\ad$ \Pth{as in Construction \ref{constr-sc-ad-Levi}} to $\Grp{H}_j^\ad$.  The adjoint action of $\Grp{M}$ on $\Grp{N}^\std$ factors through $\Grp{M}^c$ and $\Grp{M}^\ad$ \Pth{because $\ker(\Grp{M} \to \Grp{M}^c) \subset \ker(\Grp{M} \to \Grp{M}^\ad) = Z(\Grp{H})$ by definition} and defines $\Lie \Grp{N}^\std$ in $\Rep_\bC(\Grp{M}^c)$ and $\Rep_\bC(\Grp{M}^\ad)$, which pulls back to $\Lie \Grp{N}_j^\std$ in $\Rep_\bC(\Grp{M}_j)$ and $\Rep_\bC(\Grp{M}_j^\ad)$, for each $j \in J$.  \Pth{We have similar statements with \Qtn{$\std$} omitted.}
\end{construction}

\begin{definition}\label{def-pos}
    We say an irreducible $\rW$ in $\Rep_\bC(\Grp{M}^c)$ or $\Rep_\bC(\Grp{P}^{\std, c})$ is of \emph{positive parallel weight} if, for each $j \in J$ \Pth{as above}, and for some \Pth{and equivalently every} choice of maximal torus $\Grp{T}_j$ of $\Grp{M}_j$, one weight of the pullback $\rW_j$ of $\rW$ to $\Rep_\bC(\Grp{M}_j)$ \Pth{with respect to $\Grp{T}_j$} is a positive multiple of the weight of $\det(\Lie \Grp{N}_j^\std)$.
\end{definition}

\begin{remark}\label{rem-def-pos}
    By the theory of highest weights, the condition in Definition \ref{def-pos} forces $\rW_j$, for all $j \in J$, and hence $\rW$ to be \emph{one-dimensional} representations.  In this case, there exist $a_j, b_j \in \bZ_{\geq 1}$ such that $\rW_j^{\otimes a_j} \cong \det(\Lie \Grp{N}_j^\std)^{\otimes b_j}$, for all $j \in J$.
\end{remark}

\begin{proposition}\label{prop-lsv-semi-ample}
    Let $\rW \in \Rep_\bC(\Grp{M}^c)$ be irreducible of positive parallel weight, as in Definition \ref{def-pos}, which is necessarily one-dimensional, by Remark \ref{rem-def-pos}.  Let $\TorMap_\Sigma: \LSV_{\level, \Sigma}^\Tor \to \LSV_\level^\Min$ be as in Proposition \ref{prop-lsv-tor-facts}\Refenum{\ref{prop-lsv-tor-facts-min}}.  Let $\GrSh{\rW}$ over $\LSV_\level$ and $\GrSh{\rW}^\canext$ over $\LSV_{\level, \Sigma}^\Tor$ be automorphic line bundles as in Section \ref{sec-lsv-aut-bdl-can-ext}.  \Pth{We shall omit \Qtn{$\otimes$} when denoting tensor powers of line bundles.}  Then there exists $m_0 \in \bZ_{\geq 1}$ such that $(\GrSh{\rW}^\canext)^{m_0} \cong \TorMap_\Sigma^* L$ for some \emph{ample} line bundle $L$ on $\LSV_\level^\Min$.  In this case, $\GrSh{\rW}^\canext$ is semiample, and there exists \Pth{by general theory} some $r \in \bZ_{\geq 1}$ such that $L^r$ is very ample, so that $(\GrSh{\rW}^\canext)^m$ is globally generated over $\LSV_{\level, \Sigma}^\Tor$ for all positive multiples $m$ of $r m_0$, and so that the induced morphism $\phi_m: \LSV_\level^\Tor \to \bP H^0(\LSV_{\level, \Sigma}^\Tor, (\GrSh{\rW}^\canext)^m) \cong \bP H^0(\LSV_\level^\Min, L^{\frac{m}{m_0}})$ factors through $\TorMap_\Sigma$ and the canonical closed immersion $\LSV_\level^\Min \Em \bP H^0(\LSV_\level^\Min, L^{\frac{m}{m_0}})$.  In particular, $\phi_m|_{\LSV_\level}$ is an immersion as $\TorMap_\Sigma|_{\LSV_\level}$ is, and $\GrSh{\rW}$ is ample over $\LSV_\level$.
\end{proposition}
\begin{proof}
    This is essentially \cite[\aLem 3.2]{Lan:2016-vtcac}, up to some minor changes of notation and conventions.  But let us spell out the details, for the sake of clarity.

    By working over connected components, it suffices to prove Proposition \ref{prop-lsv-semi-ample} in the special case where $\level = \Gamma$ is a neat arithmetic subgroup of $\Grp{H}(\bQ)_+$.  By Proposition \ref{prop-lsv-funct}\Refenum{\ref{prop-lsv-funct-tor}}, we may also replace $\Sigma$ with any projective smooth refinement.

    By Remark \ref{rem-def-pos}, there exists $m_0 \in \bZ_{\geq 1}$ such that $\rW_j^{\otimes m_0} \cong \det(\Lie \Grp{N}_j^\std)^{\otimes c_j}$, for some $c_j \in \bZ_{\geq 1}$, which can be viewed as a representation in $\Rep_\bC(\Grp{M}_j^\ad)$, for each $j \in J$.  Let $\Gamma_j^\ad$ denote the image of $\Gamma^\ad$ in $\Grp{H}_j^\ad(\bQ)_+$, and let $\Sigma_j$ be any projective smooth cone decomposition for $\LSV_{\Gamma_j^\ad} = \Gamma_j^\ad \Lquot \Dom_j^+$, for each $j \in J$.  Consider $\overline{\Gamma} := \prod_{j \in J} \Gamma_j^\ad \subset \Grp{H}^\ad(\bQ)_+$ and the cone decomposition $\overline{\Sigma} := \prod_{j \in J} \Sigma_j^\ad$ for $\LSV_{\overline{\Gamma}}$.  Let $\Sigma$ be any \Pth{projective} refinement of the pullback of $\overline{\Sigma}$ to $\LSV_\Gamma$.  By Propositions \ref{prop-lsv-tor-facts} and \ref{prop-lsv-funct}, and by Remark \ref{rem-can-ext-log-diff}, we have a commutative diagram
    \[
        \xymatrix{ {\LSV_{\Gamma, \Sigma}^\Tor} \ar[rr]^-{[1]_{(\Gamma, \Sigma), (\overline{\Gamma}, \overline{\Sigma})}^\Tor} \ar[d]_-{\TorMap_\Sigma} & & {\LSV_{\overline{\Gamma}, \overline{\Sigma}}^\Tor} \ar[rr]^-\cong \ar[d]_-{\TorMap_{\overline{\Sigma}}} & & {\prod_{j \in J} \LSV_{\Gamma_j^\ad, \Sigma_j^\ad}^\Tor} \ar[d]_-{\prod_{j \in J} \TorMap_{\Sigma_j^\ad}} \\
        {\LSV_\Gamma^\Min} \ar[rr]^-{[1]_{\Gamma, \overline{\Gamma}}^\Min}_-{\Utext{finite}} & & {\LSV_{\overline{\Gamma}}^\Min} \ar[rr]^-\cong & & {\prod_{j \in J} \LSV_{\Gamma_j^\ad}^\Min,} }
    \]
    together with canonical isomorphisms of line bundles
    \[
    \begin{split}
        (\GrSh{\rW}^\canext)^{\otimes m_0} & \cong [1]_{(\Gamma, \Sigma), (\overline{\Gamma}, \overline{\Sigma})}^{\Tor, *} \bigl(\otimes_{j \in J} \bigl((\LSV_{\overline{\Gamma}, \overline{\Sigma}}^\Tor \to \LSV_{\Gamma_j^\ad, \Sigma_j^\ad}^\Tor)^* (\GrSh{\det(\Lie \Grp{N}_j^\std)}^\canext)^{\otimes c_j}\bigr)\bigr) \\
        & \cong \otimes_{j \in J} \bigl((\LSV_{\Gamma, \Sigma}^\Tor \to \LSV_{\Gamma_j^\ad, \Sigma_j^\ad}^\Tor)^* (\Omega_{\LSV_{\Gamma_j^\ad, \Sigma_j^\ad}^\Tor}^{\log, d_j})^{\otimes c_j}\bigr) \\
        & \cong \otimes_{j \in J} \bigl((\LSV_{\Gamma, \Sigma}^\Tor \to \LSV_{\Gamma_j^\ad, \Sigma_j^\ad}^\Tor)^* \TorMap_{\Sigma_j^\ad}^* (\Omega_{\LSV_{\Gamma_j}^\Min}^{d_j})^{\otimes c_j}\bigr) \cong \TorMap_\Sigma^* \, L,
    \end{split}
    \]
    where $L := (\LSV_\Gamma^\Min \to \LSV_{\overline{\Gamma}}^\Min)^* \bigl(\otimes_{j \in J} \bigl((\LSV_{\overline{\Gamma}}^\Min \to \LSV_{\Gamma_j^\ad}^\Min)^* (\Omega_{\LSV_{\Gamma_j}^\Min}^{d_j})^{\otimes c_j}\bigr)\bigr)$ is ample over $\LSV_\Gamma^\Min$ because each of $\Omega_{\LSV_{\Gamma_j}^\Min}^{d_j}$ is ample over $\LSV_{\Gamma_j^\ad}^\Min$, by Proposition \ref{prop-lsv-tor-facts}\Refenum{\ref{prop-lsv-tor-facts-min-can}}, as desired.
\end{proof}

When we apply Proposition \ref{prop-lsv-semi-ample} later in Section \ref{sec-trans}, we will want the ratios $\frac{b_j}{a_j}$ in Remark \ref{rem-def-pos} to be as small as possible.  Let us introduce the following:
\begin{condition}\label{cond-der-c-sc}
    $\Grp{H}^{\der, c}$ is simply-connected.
\end{condition}
Since the canonical homomorphism $\Grp{H}^\der \to \Grp{H}^{\der, c}$ is a central isogeny, our assumption forces $\Grp{H}^\der$ to be also simply-connected.  Then we have the following:
\begin{proposition}\label{prop-lsv-pos-smallest}
    Suppose Condition \ref{cond-der-c-sc} holds.  Then there exists an irreducible $\rW_0 \in \Rep_\bC(\Grp{M}^c)$ of positive parallel weight such that, for each $j \in J$, the pullback $\rW_{0, j} \in \Rep_\bC(\Grp{M}_j)$ of $\rW_0$ satisfies $\rW_{0, j}^{\otimes \dual{h}_j} \cong \det(\Lie \Grp{N}_j^\std)$, where $\dual{h}_j \in \bZ_{\geq 1}$ is the \emph{dual Coxeter number} for the root system associated with any $\bC$-simple factor of the complex semisimple Lie algebra $\Lie \Grp{H}_{j, \bC}$ \Pth{and any choice of Cartan subalgebra}.  \Pth{Since $\Grp{H}_j$ is $\bQ$-simple, this $\dual{h}_j$ is well defined, because the $\bC$-simple factors of $\Lie \Grp{H}_{j, \bC}$ are all abstractly isomorphic to each other.}  For each $j \in J$, this ratio $\frac{1}{\dual{h}_j}$ is the \emph{smallest} among all $\frac{b_j}{a_j}$ such that there exists some $\rW_j \in \Rep_\bC(\Grp{M}_j)$ such that $\rW_j^{\otimes a_j} \cong \det(\Lie \Grp{N}_j^\std)^{\otimes b_j}$ for some $a_j, b_j \in \bZ_{\geq 1}$ \Pth{\Refcf{} Remark \ref{rem-def-pos}}.
\end{proposition}
\begin{proof}
    This follows from \cite[\aThm 3.8]{Lan:2016-vtcac}.
\end{proof}

\section{Geometric Sen theory for towers of Shimura varieties}\label{sec-geom-Sen-Sh}

In this section, we collect some known results about the $p$-adic geometry of towers of Shimura varieties, notably the differential equation satisfied by the locally analytic functions at the infinite level. This is a consequence of some calculation of the canonical Higgs field for towers of Shimura varieties, which was first discovered by the second author in the case of modular curves \cite{Pan:2022-laccm} and in general done by Juan Esteban Rodr{\'\i}guez Camargo \cite{RodriguezCamargo:2022-lacc}.

\subsection{Shimura varieties and their compactifications}\label{sec-Sh-var}

Let $(\Grp{G}, \Shdom)$ be a Shimura datum.  In particular, $\Grp{G}$ is a connected reductive group over $\bQ$, and $\Shdom$ is a conjugacy class of homomorphisms
\begin{equation}\label{eq-hd}
    \hd: \Res_{\bC / \bR} \bG_{m, \bC} \to \Grp{G}_\bR
\end{equation}
satisfying a list of axioms, which imply that the pair $(\Grp{G}, \Shdom)$ satisfies the assumptions in Section \ref{sec-lsv-setup}.  \Pth{See, \eg, \cite[2.1.1]{Deligne:1979-vsimc}, \cite[\aDef 5.5 and \aCor 5.8]{Milne:2005-isv}, or \cite[\aSec 2.3]{Lan:ebisv}.  See also other parts of these references for many standard facts we review below.}  In particular, we have the associated double coset spaces, their minimal compactifications, their toroidal compactifications, and the canonical algebraizations of all of these, as in Sections \ref{sec-dcs-an} and \ref{sec-lsv-cpt}, and they enjoy the properties we have shown in Sections \ref{sec-lsv-funct}--\ref{sec-lsv-pos}.  We shall change the notation from $\LSV_\levcp$ etc to $\Sh_{\levcp, \bC}$ etc in this subsection.  By the theory of canonical models, $\Sh_{\levcp, \bC}$ admits a canonical model $\Sh_\levcp$ over a number field $\ReFl \subset \bC$, called the \emph{reflex field} of the Shimura datum $(\Grp{G}, \Shdom)$.  We shall call $\Sh_\levcp$ the canonical model of the Shimura variety associated with $(\Grp{G}, \Shdom)$ at level $\levcp$.  For simplicity, we shall also call various base changes of $\Sh_\levcp$ and their analytifications Shimura varieties.  The precise meaning of our terminologies should be clear in the context.

By \cite[12.4]{Pink:1989-Ph-D-Thesis}, for a cofinal system of choices of $\Sigma$'s for $\Sh_{\levcp, \bC}$, the corresponding toroidal compactifications $\Sh_{\levcp, \Sigma, \bC}^\Tor$ also admit canonical models $\Sh_{\levcp, \Sigma}^\Tor$ over $\ReFl$, and we shall call these the toroidal compactifications of $\Sh_\levcp$.  Without reviewing the conditions on $\Sigma$'s in \cite[12.4]{Pink:1989-Ph-D-Thesis} in detail, let us just note the following:
\begin{remark}\label{rem-Pink-cone-decomp-higher-lev}
    Given any projective $\Sigma$ for $\Sh_\levcp$, its pullback $\Sigma'$ to any higher level $\levcp' \subset \levcp$ \Pth{which is, a priori, just for $\Sh_{\levcp', \bC}$} also qualifies as a projective cone decomposition for $\Sh_{\levcp'}$, so that $\Sh_{\levcp', \Sigma'}^\Tor$ is defined.
\end{remark}

Each $\hd \in \Shdom$ as in \Refeq{\ref{eq-hd}} induces a homomorphism $\hd_\bC: \bG_{m, \bC} \times \bG_{m, \bC} \to \Grp{G}_\bC$, whose restriction to the first factor defines the so-called \emph{Hodge cocharacter}
\begin{equation}\label{eq-hc}
    \hc_\hd: \bG_{m, \bC} \to \Grp{G}_\bC
\end{equation}
\Pth{over $\bC$}.  Essentially by the definition of a Shimura datum, the composition of \Refeq{\ref{eq-hc}} with the adjoint action of $\Grp{G}_\bC$ on $\Lie \Grp{G}_\bC$ induces the same decomposition as in \Refeq{\ref{eq-Cartan-decomp-C}}, and the composition of \Refeq{\ref{eq-hc}} with the canonical homomorphism $\Grp{G}_\bC \to \Grp{G}_\bC^\ad$ recovers the $\hc^\ad$ \Pth{defined by $\Dompt_0 = \hd_0 \in \Shdom^+$} in \Refeq{\ref{eq-hc-ad}}.  Hence, we can define $\Grp{M}$, $\Grp{N}$, $\Grp{N}^\std$, $\Grp{P}$, and $\Grp{P}^\std$ as in \Refeq{\ref{eq-def-P}} and the sentences preceding it; and define $\Flalg := \Grp{P} \Lquot \Grp{G}_\bC$ and $\Flalg^\std := \Grp{H}_\bC / \Grp{P}^\std$ as in \Refeq{\ref{eq-def-fl}}.

\begin{remark}\label{rem-refl-def}
    The reflex field $\ReFl$ is the field of definition of the conjugacy class of $\hc_\hd$, and both $\Flalg$ and $\Flalg^\std$ naturally have models over $\ReFl$ because of this.
\end{remark}

\begin{remark}\label{rem-fil-by-hc}
    Given a representation $\rW$ in $\Rep_\bC(\Grp{P})$ \Pth{\resp $\Rep_\bC(\Grp{P}^\std)$}, the action of $\Lie \hc_\hd$ defines an increasing \Pth{\resp decreasing} filtration $\Fil^\bullet$ on $\rW$ by subrepresentations of $\Grp{P}$ \Pth{\resp $\Grp{P}^\std$}.  However, the numbering of this filtration might differ from one defined by $\Lie \hc^\ad$ in Remark \ref{rem-fil-by-hc-ad} by a shifting, because the composition of $\Lie \hc^\ad$ with $(\Lie \MaxCpt^\ad)_\bC \subset (\Lie \MaxCpt)_\bC \cong \Lie \Grp{M}$ might differ from $\Lie \hc_\hd$ by a nonzero homomorphism from $\bC$ to $\Lie Z(\Grp{G}_\bC) \cong \ker(\Lie \Grp{G}_\bC \to \Lie \Grp{G}_\bC^\ad)$.
\end{remark}

\begin{remark}\label{rem-Sh-var-aut-bdl-can-ext}
    Everything in Sections \ref{sec-lsv-aut-bdl-can-ext}--\ref{sec-lsv-pos}, including importantly Propositions \ref{prop-lsv-semi-ample} and \ref{prop-lsv-pos-smallest}, specializes to the setting of Shimura varieties.  We shall use the same notation $\GrSh{\rW}$, $\GrSh{\rW}^\canext$, etc, without further explanation.
\end{remark}

In the remainder of this Section \ref{sec-geom-Sen-Sh}, we fix a neat open compact subgroup $\levcp = \levcp^p \levcp_p$ of $\Grp{G}(\bAi)$, with $\levcp^p \subset \Grp{G}(\bAip)$ and $\levcp_p \subset \Grp{G}(\bQ_p)$.  Moreover, we fix a projective smooth cone decomposition $\Sigma$ for $\Sh_\levcp$, and we shall write $\Sh_\levcp^\Tor = \Sh_{\levcp, \Sigma}^\Tor$, without emphasizing the choice of $\Sigma$.  By Remark \ref{rem-Pink-cone-decomp-higher-lev}, for each open subgroup $\levcp' \subset \levcp$, the pullback $\Sigma'$ to level $\levcp'$ is projective \Pth{but not necessarily smooth}, and we shall similarly write $\Sh_{\levcp'}^\Tor = \Sh_{\levcp', \Sigma'}^\Tor$, without emphasizing $\Sigma'$.  By Corollary \ref{cor-dcs-act}, the tower $\{ \Sh_{\levcp^p \levcp_p'} \}_{\levcp_p'}$ indexed by open subgroups $\levcp_p'$ of $\levcp_p$, together with its canonical $\levcp_p$-action, extends to a tower $\{ \Sh_{\levcp^p \levcp_p'}^\Tor \}_{\levcp_p'}$; and the transition morphisms of this latter tower are finite Kummer \'etale surjective morphisms.  Let $\widetilde{\levcp_p}$ be the quotient of $\levcp_p$ given by \Refeq{\ref{eq-ex-KZ-mod-KZ'-tower-p-infty}} \Pth{with $\Grp{H}$ there replaced with $\Grp{G}$ here}.  By Corollary \ref{cor-KZ-mod-KZ'-tower}, the canonical actions of $\levcp_p$ on the towers $\{ \Sh_{\levcp^p \levcp_p'} \}_{\levcp_p'}$ and $\{ \Sh_{\levcp^p \levcp_p'}^\Tor \}_{\levcp_p'}$ factor through faithful actions of $\widetilde{\levcp_p}$.  \Pth{In other words, $\widetilde{\levcp_p}$ is the image of $\levcp_p$ in the automorphism groups of these towers.}

\subsection{Hodge--Tate period morphisms}\label{sec-HT-mor}

Now let us make the transition to $p$-adic geometry.  Let $C$ be the $p$-adic completion of an algebraic closure of $\bQ_p$.  From now on, we fix an isomorphism $\iota: \bC \Mi C$, whose restriction to $\ReFl \subset \bC$ induces an embedding $\ReFl \Em C$ and hence a $p$-adic place $v$ of $\ReFl$.  Let $\ReFl_v$ denote the corresponding extension of $\bQ_p$.  We shall denote by $\Sh_{\levcp, \ReFl_v}$ etc \Pth{\resp $\Sh_{\levcp, C}$ etc} the base changes of $\Sh_\levcp$ etc to $\ReFl_v$ \Pth{\resp $C$}.  Note that these base changes \Pth{and therefore the constructions in the remainder of section} depend on $\iota$.  We shall denote by $\aShEv_\levcp$ \Pth{\resp $\aSh_\levcp$} the adic space associated with $\Sh_{\levcp, \ReFl_v}$ \Pth{\resp $\Sh_{\levcp, C}$} over $\Spa(\ReFl_v, \cO_{\ReFl_v})$ \Pth{\resp $\Spa(C, \cO_C)$}, and similarly denote other analytifications.  Then we have compatible towers $\{ \aShEv_{\levcp^p \levcp_p'} \}_{\levcp_p'}$ and $\{ \aShEv_{\levcp^p \levcp_p'}^\Tor \}_{\levcp_p'}$ over $\Spa(\ReFl_v, \cO_{\ReFl_v})$, which base change to compatible towers $\{ \aSh_{\levcp^p \levcp_p'} \}_{\levcp_p'}$ and $\{ \aSh_{\levcp^p \levcp_p'}^\Tor \}_{\levcp_p'}$ over $\Spa(C, \cO_C)$.  Note that $\{ \Sh_{\levcp^p \levcp_p', C} \}_{\levcp_p'}$ satisfies the requirement as a $\widetilde{\levcp_p}$-torsor over $\Sh_{\levcp, C}$ in Construction \ref{constr-proet-ls}, and $\Sh_{\levcp, C} \Em \Sh_{\levcp, C}^\Tor$ satisfies the requirement for $j$ in Construction \ref{constr-proket-ls}.

By \cite[\aEx 2.3.17]{Diao/Lan/Liu/Zhu:2023-lasfr}, each $\aShEv_{\levcp^p \levcp_p'}^\Tor$ has a natural fs log structure defined by its boundary $\partial \aShEv_{\levcp^p \levcp_p'}^\Tor$ divisor.  By comparing \'etale local descriptions of log smooth and log \'etale morphisms in the algebraic setup in \cite{Kato:1989-lsfi} with those in the $p$-adic analytic setup in \cite{Diao/Lan/Liu/Zhu:2023-lasfr}, we see that the transition morphisms of the tower $\{ \aShEv_{\levcp^p \levcp_p'}^\Tor \}_{\levcp_p'}$ are finite Kummer \'etale surjective morphisms.  Hence, $\varprojlim_{\levcp_p' \subset \levcp_p} \aShEv_{\levcp^p \levcp_p'}^\Tor$ is a $\widetilde{\levcp_p}$-torsor which defines a pro-Kummer \'etale cover of $\aShEv_\levcp^\Tor$ in $\aShEv_{\levcp, \proket}^\Tor$.  By base change to $\Spa(C, \cO_C)$, we obtain a $\widetilde{\levcp_p}$-torsor $\varprojlim_{\levcp_p' \subset \levcp_p} \aSh_{\levcp^p \levcp_p'}^\Tor$, which defines a pro-Kummer \'etale cover of $\aSh_\levcp^\Tor$ in $\aSh_{\levcp, \proket}^\Tor$.

This last $\widetilde{\levcp_p}$-torsor over $\Spa(C, \cO_C)$ has an underlying topological space
\[
    \aSh_{\levcp_p}^\Tor := \varprojlim_{\levcp_p' \subset \levcp_p} |\aSh_{\levcp^p \levcp_p'}^\Tor|
\]
equipped with the inverse limit topology, where $|\aSh_{\levcp^p \levcp_p'}^\Tor|$ denotes the underlying topological space of $\aSh_{\levcp^p \levcp_p'}^\Tor$.  Concretely, the topology of $\aSh_{\levcp^p}^\Tor$ has a basis consisting of $\pi_{\levcp_p'}^{-1}(\cU) = \varprojlim_{\levcp_p'' \subset \levcp_p'} |(\aSh_{\levcp^p \levcp_p''}^\Tor \to \aSh_{\levcp^p \levcp_p'}^\Tor)^{-1}(\cU)|$, with $\levcp_p'$ an open subgroup of $\levcp_p$ and $\cU$ an affinoid open subspace of $\aSh_{\levcp^p \levcp_p'}^\Tor$, where $\pi_{\levcp_p'}: \aSh_{\levcp^p}^\Tor \to |\aSh_{\levcp^p \levcp_p'}^\Tor|$ denotes the natural projection morphism.  There is a natural continuous action of $\levcp_p$ on $\aSh_{\levcp^p}^\Tor$ which factors through the quotient $\widetilde{\levcp_p}$.

\begin{definition}\label{def-O-plus}
    We define the sheaf $\widehat{\cO}_\aSh^+ := \widehat{\cO}_{\aSh_{\levcp^p}^\Tor}^+$ on $\aSh_{\levcp^p}^\Tor$ to be the restriction of the integral completed structure sheaf $\widehat{\cO}_{\aSh_{\levcp, \proket}^\Tor}^+$ to the underlying topological space of $\varprojlim_{\levcp_p' \subset \levcp_p} \aSh_{\levcp^p \levcp_p'}^\Tor$, and define $\widehat{\cO}_\aSh := \widehat{\cO}_{\aSh_{\levcp^p}^\Tor} := \widehat{\cO}_{\aSh_{\levcp^p}^\Tor}^+[\frac{1}{p}]$.  Explicitly, $\widehat{\cO}_\aSh^+$ sends any $\varprojlim_{\levcp_p''} |\pi_{\levcp_p'', \levcp_p'}^{-1}(\cU)|$ as above to $\widehat{\cO}_{\aSh_{\levcp, \proket}^\Tor}^+(\varprojlim_{\levcp_p''} \pi_{\levcp_p'', \levcp_p'}^{-1}(\cU))$, where $\varprojlim_{\levcp_p''} \pi_{\levcp_p'', \levcp_p'}^{-1}(\cU)$ is viewed as an object in the pro-Kummer \'etale site $\aSh_{\levcp, \proket}^\Tor$.  There is a similar description of $\widehat{\cO}_\aSh$ in terms of the completed structure sheaf $\widehat{\cO}_{\aSh_{\levcp, \proket}^\Tor}$.  Both sheaves are $\widetilde{\levcp_p}$-equivariant, and there is a natural morphism of ringed sites $\pi_{\levcp_p}: (\aSh_{\levcp^p}^\Tor, \widehat{\cO}_\aSh) \to (\aSh_\levcp^\Tor, \cO_{\aSh_\levcp^\Tor})$.
\end{definition}

\begin{remark}\label{rem-Sh-infty}
    Intuitively, we think of the ringed site $(\aSh_{\levcp^p}^\Tor, \widehat{\cO}_\aSh)$ as a \Qtn{\emph{Shimura variety of infinite level at $p$}}.  One caveat is that, while this ad hoc definition will be enough for our purpose, the definition of the structural sheaf here might not be the most correct one in general.  From the point of view of \emph{perfectization} in \cite[\aSec 4.1]{Bhatt:2025-apaht}, it is unclear whether there should be some derived structure on the sheaf $\widehat{\cO}_\aSh$, unless we know that $\varprojlim_{\levcp_p' \subset \levcp_p} \aSh_{\levcp^p \levcp_p'}^\Tor$ is perfectoid.
\end{remark}

It was first observed by Scholze in \cite{Scholze:2015-tclsv} that any Hodge-type Shimura variety of finite level at $p$ has a Hodge--Tate period morphism towards the associated flag variety, which played a fundamental role in his study of the $p$-adic geometry of Shimura varieties.  \Pth{See \cite{Pilloni/Stroh:2016-ccerg, Lan:2022-citcs} for the analogous statement for toroidal compactifications.}  Using Diao--Lan--Liu--Zhu's work on \Pth{$p$-adic} logarithmic Riemann--Hilbert correspondence and its comparison with Deligne's \Pth{complex} Riemann--Hilbert correspondence on all Shimura varieties \Pth{see \cite[\aThm 1.5]{Diao/Lan/Liu/Zhu:2023-lrhrv}}, one can define Hodge--Tate period morphism for general Shimura varieties \Pth{\Refcf{} \cite[\aThm 4.2.1]{RodriguezCamargo:2022-lacc}}.  To state the result, we introduce some notation.

Let $\Fl$ denote the adic space associated with the partial flag variety $\Flalg_C = \Grp{P}_C \Lquot \Grp{G}_C$ \Pth{\Refcf{} \Refeq{\ref{eq-def-fl}}}.  From this quotient expression, with any representation $\rW$ of $\Grp{P}_C$, increasingly filtered as in Remark \ref{rem-fil-by-hc}, we can canonically associate an increasingly filtered $\Grp{G}_C$-equivariant vector bundle $(\GrSh{\rW}_{\Flalg_C}, \Fil_\bullet)$ on $\Flalg_C$.  By \Pth{$p$-adic} analytification, we obtain the associated \Pth{increasingly} filtered $\Grp{G}(C)$-equivariant vector bundle $(\cW_\Fl, \Fil_\bullet)$ on $\Fl$.  The assignment $\rW \mapsto (\rW_\Fl, \Fil_\bullet)$ defines a tensor functor from $\Rep_C(\Grp{P}_C)$ to the category of increasingly filtered vector bundles on $\Fl$.  Such a construction extends naturally to representations of quotient \Pth{algebraic} groups of $\Grp{P}_C$, including, in particular, $\Grp{P}_C^c$, $\Grp{M}_C$, and $\Grp{M}_C^c$.

As in Remark \ref{rem-Sh-var-aut-bdl-can-ext}, with each $\rW \in \Rep_\bC(\Grp{P}^{\std, c})$, we also have an associated \Pth{decreasingly} filtered vector bundle $(\GrSh{\rW}, \Fil^\bullet)$ over $\Sh_{\levcp, \bC}$, together with its \emph{canonical extension} $(\GrSh{\rW}^\canext, \Fil^\bullet)$ over $\Sh_{\levcp, \bC}^\Tor$.  By using the chosen isomorphism $\iota: \bC \Mi C$, and by $p$-adic analytification, we obtain a filtered vector bundle $(\cW^\canext, \Fil^\bullet)$ on $\aSh_\levcp^\Tor$.  The assignment $\rW \mapsto (\cW^\canext, \Fil^\bullet)$ defines a tensor functor from $\Rep_C(\Grp{P}_C^{\std, c})$ to the category of decreasingly filtered vector bundles on $\aSh_\levcp^\Tor$.

\begin{theorem}\label{thm-HT-mor}
    There is a $\levcp_p$-equivariant morphism of ringed sites
    \[
        \pi_\HT^\Tor: (\aSh_{\levcp^p}^\Tor, \widehat{\cO}_\aSh) \to (\Fl, \cO_\Fl),
    \]
    called the \emph{Hodge--Tate period morphism}, such that, for each $\rW \in \Rep_C(\Grp{M}^c)$, there is a $\levcp_p$-equivariant isomorphism of locally free $\widehat{\cO}_\aSh$-modules
    \[
        \pi_\HT^{\Tor, *}(\cW_\Fl) \Mi \pi_{\levcp_p}^*(\cW^\canext).
    \]
    Moreover, $\pi_\HT^\Tor$ is deduced from a $\levcp_p$-equivariant morphism $(\aSh_{\levcp^p}^\Tor, \widehat{\cO}_\aSh^+) \to (\Fl, \cO_\Fl^+)$ by inverting $p$.
\end{theorem}

\begin{remark}\label{rem-thm-HT-mor}
    To make the isomorphism also equivariant with respect to the action of local Galois group of $\ReFl_v$, one needs to put appropriate Tate twists by the Hodge cocharacter $\hc_\hd$ in the isomorphism \Pth{\Refcf{} \cite[\aThm 4.2.1]{RodriguezCamargo:2022-lacc} and also Remark \ref{rem-fil-by-hc}}.
\end{remark}

\begin{proof}[Proof of Theorem \ref{thm-HT-mor}]
    This is essentially \cite[\aThm 4.2.1]{RodriguezCamargo:2022-lacc}.  See also \cite[\aThm 4.4.40]{Boxer/Pilloni:2021-hct}.  Let us still give a sketch here.  As explained in \cite[\aSec 5.2]{Diao/Lan/Liu/Zhu:2023-lrhrv}, any finite-dimensional $\rV \in \Rep_{\bQ_p}(\Grp{G}^c)$, viewed as a representation of $\widetilde{\levcp_p}$ via $\widetilde{\levcp_p} \to \levcp_p^c \subset \Grp{G}^c(\bQ_p)$, defines an \Pth{automorphic} $\bQ_p$-\'etale local system $\etSh{\rV}$ on $\Sh_\levcp$, which has \emph{unipotent geometric monodromy} along the boundary \Pth{in any toroidal compactification} and is \emph{de Rham} by \cite[\aThm 1.2]{Liu/Zhu:2017-rrhpl}.  Therefore, by extending $\etSh{\rV}$ to a Kummer \'etale $\bQ_p$-local system $\ketSh{\rV}$ over $\aShEv_\levcp^\Tor$, by applying the algebraized $p$-adic logarithmic Riemann--Hilbert functor $\Ddlalg$ as in \cite[\aSec 4.1]{Diao/Lan/Liu/Zhu:2023-lrhrv}, and by base change from $\ReFl_v$ to $C$, we obtain a filtered log connection over $\Sh_{\levcp, C}^\Tor$.  By \cite[\aThm 5.3.1]{Diao/Lan/Liu/Zhu:2023-lrhrv}, this last filtered log connection is canonically isomorphic to $(\dRSh{\rV_\bC}^\canext, \nabla, \Fil^\bullet) \otimes_{\bC, \iota} C$, where $\rV_\bC := \rV \otimes_{\bQ_p, \iota^{-1}} \bC$ and $(\dRSh{\rV_\bC}^\canext, \nabla, \Fil^\bullet)$ is as in the last paragraph of Section \ref{sec-lsv-aut-bdl-can-ext}.  Moreover, such canonical isomorphisms are compatible with tensor products and duals.  In the following, we shall denote by $(\cV^\canext, \nabla, \Fil^\bullet)$ the \Pth{$p$-adic} analytification of $(\dRSh{\rV_\bC}^\canext, \nabla, \Fil^\bullet) \otimes_{\bC, \iota} C$ over $\aSh_\levcp^\Tor$.  Let $\proketSh{\rV}$ denote the pullback of $\ketSh{\rV}$ to $\aShEv_{\levcp, \proket}^\Tor$.  \Pth{Its further pullback to $\aSh_{\levcp, \proket}^\Tor$ is the $\proketSh{\rV}_{\aSh_\levcp^\Tor}$ constructed in Construction \ref{constr-proket-ls}.}  Since $\etSh{\rV}$ is de Rham and has unipotent geometric monodromy along the boundary, $\proketSh{\rV}$ on $\aShEv_{\levcp, \proket}^\Tor$ and $(\cV^\canext, \nabla, \Fil^\bullet)$ are associated in the sense that there is a natural isomorphism
    \[
        \proketSh{\rV} \otimes_{\bQ_p} \OBdl \cong \cV^\canext \otimes_{\cO_{\aShEv_\levcp^\Tor}} \OBdl
    \]
    over $\aShEv_{\levcp, \proket}^\Tor$, which is compatible with filtrations and log connections on both sides \Pth{see the paragraph preceding \cite[\aSec 3.5]{Diao/Lan/Liu/Zhu:2023-lrhrv} and the notation for period sheaves there}.  For simplicity, let us write $X = \aShEv_{\levcp, \proket}^\Tor$.  Similar to \cite[\aProp 7.9]{Scholze:2013-phtra}, one defines an increasing \emph{Hodge--Tate} filtration on $\proketSh{\rV} \otimes_{\bQ_p} \widehat{\cO}_X$ using the relative position of the $\BBdRp$-lattices $\bM = \proketSh{\rV} \otimes_{\bQ_p} \BBdRp$ and $\bM_0 = (\dRSh{\rV} \otimes_{\cO_{\aShEv_\levcp^\Tor}} \OBdlp)^{\nabla = 0}$ in $\proketSh{\rV}\otimes_{\bQ_p} \BBdR$; \ie,
    \[
        \Fil_{-j}(\proketSh{\rV} \otimes_{\bQ_p} \widehat{\cO}_X) := (\bM \cap \Fil^j \bM_0) / (\Fil^1 \bM \cap \Fil^j \bM_0),
    \]
    where both $\bM$ and $\bM_0$ are equipped with the $\ker(\theta)$-adic topology, and where $\theta: \BBdRp \to \widehat{\cO}_X$ denotes the usual $\theta$ map.  Moreover, we have
    \[
        \gr_j(\proketSh{\rV} \otimes_{\bQ_p} \widehat{\cO}_X) = (\gr^j \dRSh{\rV}) \otimes_{\cO_{\aShEv_\levcp^\Tor}} \widehat{\cO}_X(-j).
    \]
    Since the Hodge filtration $\Fil^\bullet$ on $\cV^\canext$ is induced \Pth{via $\iota: \bC \Mi C$} by the decreasing filtration on $\rV_\bC|_{\Grp{P}_\bC^{\std, c}}$, which is defined up to conjugation by $\Grp{G}^c$ by the action of $\Lie \hc_\hd$, or rather just $\hc_\hd$, as in Remark \ref{rem-fil-by-hc}, it follows from the above that the Hodge--Tate filtration on $\proketSh{\rV} \otimes_{\bQ_p} \widehat{\cO}_X$ is associated with the opposite parabolic subgroup $\Grp{P}_C^c$ \Pth{up to conjugation}.  Moreover, the formation of such filtrations are compatible with the formation of tensor products and duals.

    Consequently, the Tannakian $\Grp{G}_C^c$-torsor $\cG^c$ on $(X, \widehat{\cO}_X)$ given by the tensor functor $\bigl(\rV \in \Rep_{\bQ_p}(\Grp{G}^c)\bigr) \mapsto \proketSh{\rV}\otimes_{\bQ_p} \widehat{\cO}_X$ admits a $\Grp{P}_C^c$-reduction, by \cite[\aProp 1.7]{Milne:1990-cmsab} \Pth{based on \cite[\aCh IV, especially IV.2.2.5]{SaavedraRivano:1972-CT}}.  \Pth{Here the torsors are only defined by the Tannakian formalism, which makes sense over any ringed site.}  Since $\proketSh{\rV}$ is canonically trivialized on the pro-Kummer \'etale cover $Y := \varprojlim_{\levcp_p' \subset \levcp_p} \aSh_{\levcp^p \levcp_p'}^\Tor$, so is the restriction of $\cG^c$ to $Y$.  By the usual universal property of the algebraic partial flag variety $\Flalg_C$, the $\Grp{P}_C^{c, \an}$-reduction induces a morphism of ringed sites
    \[
        (Y, \widehat{\cO}_X|_Y) \to (\Flalg_C, \cO_{\Flalg_C}).
    \]
    Since $\Fl \cong \Flalg_C \times_{\Spec(C)} \Spa(C, \cO_C)$ \Pth{as in \cite[\aProp 3.8 and \aRem 4.6]{Huber:1994-gfsra}}, this factors through a morphism of ringed sites
    \[
        (Y, \widehat{\cO}_X|_Y) \to (\Fl, \cO_\Fl),
    \]
    whose restriction to the underlying topological space $|Y| = \aSh_{\levcp^p}^\Tor$ gives the desired $\pi_\HT^\Tor$, which is deduced from a morphism $(Y, \widehat{\cO}_X^+|_Y) \to (\Fl, \cO_\Fl^+)$ by inverting $p$.  All the claimed properties then follow from the construction.
\end{proof}

\begin{remark}\label{rem-HT-mor-comp}
    In the construction of $\pi_\HT$, one can use the associated adjoint Shimura datum $(\Grp{G}^\ad, \Shdom^\ad)$ instead of $\Grp{G}^c$.  Let $\levcp^\ad$ denote the image of $\levcp$ in $\Grp{G}^\ad(\bAi)$, which is neat as $\levcp$ is.  Up to refinement of cone decompositions for $\Sh_\levcp$, as in Proposition \ref{prop-lsv-funct}, we can choose smooth projective toroidal compactifications $\Sh_\levcp^\Tor$ of $\Sh_\levcp$ and $\Sh_{\levcp^\ad}^\Tor$ of $\Sh_{\levcp^\ad}$ with normal crossing boundary divisors as before such that the finite map $\Sh_\levcp \to \Sh_{\levcp^\ad}$ extends to a proper morphism between toroidal compactifications $\Sh_\levcp^\Tor \to \Sh_{\levcp^\ad}^\Tor$.  \Pth{We shall adopt a similar notation system for various objects defined by the similarly defined tower over $\Sh_{\levcp^\ad}^\Tor$.}  Then we obtain a natural morphism of ringed sites
    \[
        \phi: (\aSh_{\levcp^p}^\Tor, \widehat{\cO}_\aSh) \to (\aSh_{\levcp^{\ad, p}}^\Tor, \widehat{\cO}_{\aSh^\ad}),
    \]
    where $(\aSh_{\levcp^{\ad, p}}^\Tor, \widehat{\cO}_{\aSh^\ad})$ is the analogue of $(\aSh_{\levcp^p}^\Tor, \widehat{\cO}_\aSh)$ for $\Sh_{\levcp^\ad}^\Tor$.  The Hodge--Tate period morphisms for both the source and target of $\phi$ have the same target.  We claim that the Hodge--Tate period morphism for $(\aSh_{\levcp^p}^\Tor, \widehat{\cO}_\aSh)$ factors through $\phi$.  This follows from the construction and from the compatibility of the logarithmic $p$-adic Riemann--Hilbert functor for local systems with unipotent geometric monodromy under pull-back, by \cite[\aThm 3.2.3(4)]{Diao/Lan/Liu/Zhu:2023-lrhrv}.
\end{remark}

We will need a \Qtn{\emph{locally analytic}} version of the Hodge--Tate period morphism.

\begin{definition}\label{def-HT-mor-la}
    We define the sheaf $\cO_\aSh^\la$ on $\aSh_{\levcp^p}^\Tor$ as the subsheaf of $\widehat{\cO}_\aSh$ of $\widetilde{\levcp_p}$-locally analytic sections.  Concretely, for any open subgroup $\levcp_p' \subset \levcp_p$ and any affinoid $\cU \subset \aSh_{\levcp^p \levcp_p'}^\Tor$, we know that $\widehat{\cO}_\aSh\bigl(\pi_{\levcp_p'}^{-1}(\cU)\bigr) = \widehat{\cO}_\aSh^+(\pi_{\levcp_p'}^{-1}(\cU))[\frac{1}{p}]$ is a $p$-adic Banach space representation of $\levcp_p'$, and we define $\cO_\aSh^\la\bigl(\pi_{\levcp_p'}^{-1}(\cU)\bigr)$ to be the subspace of $\levcp_p'$-locally analytic vectors of $\widehat{\cO}_\aSh(\pi_{\levcp_p'}^{-1}(\cU))$.  This defines a $\widetilde{\levcp_p}$-equivariant subsheaf $\cO_\aSh^\la$ of $\widehat{\cO}_\aSh$.  \Pth{See also \cite[\aDef 6.2.1]{RodriguezCamargo:2022-lacc}.}
\end{definition}

Note that $\pi_{\levcp_p}: (\aSh_{\levcp^p}^\Tor, \widehat{\cO}_\aSh) \to (\aSh_\levcp^\Tor, \cO_{\aSh_\levcp^\Tor})$ induces a morphism
\[
    \pi_{\levcp_p}^\la: (\aSh_{\levcp^p}^\Tor, \cO_\aSh^\la) \to (\aSh_\levcp^\Tor, \cO_{\aSh_\levcp^\Tor}),
\]
because all the sections of $\cO_{\aSh_\levcp^\Tor}$ are $\levcp_p$-invariant, and hence $\levcp_p$-locally analytic.

\begin{theorem}\label{thm-HT-mor-la}
    The Hodge--Tate period morphism $\pi_\HT^\Tor$ in Theorem \ref{thm-HT-mor} induces a $\levcp_p$-equivariant morphism of ringed sites
    \[
        \pi_\HT^\la: (\aSh_{\levcp^p}^\Tor, \cO_\aSh^\la) \to (\Fl, \cO_\Fl)
    \]
    such that, for $\rW \in \Rep_C(\Grp{M}^c)$, there is a $\levcp_p$-equivariant isomorphism of locally free $\cO_\aSh^\la$-modules
    \[
        \pi_\HT^{\la, *}(\cW_\Fl) \cong \pi^{\la, *}_{\levcp_p}(\cW^\canext).
    \]
\end{theorem}
\begin{proof}
    The morphism $\pi_\HT^\Tor: (\aSh_{\levcp^p}^\Tor, \widehat{\cO}_\aSh) \to (\Fl, \cO_\Fl)$ gives a morphism of sheaves of rings $(\pi_\HT^\Tor)^\#: (\pi_\HT^\Tor)^{-1}(\cO_\Fl) \to \widehat{\cO}_\aSh$.  We claim that this morphism factors through $\cO_\aSh^\la$ and therefore defines a morphism $\pi_\HT^\la: (\aSh_{\levcp^p}^\Tor, \cO_\aSh^\la) \to (\Fl, \cO_\Fl)$.  This is a purely local question.  Suppose that $\tilde{\cU} = \pi_{\levcp_p'}^{-1}(\cU)$ is an open subset of $\aSh_{\levcp^p}^\Tor$ with $\cU \subset \aSh_{\levcp^p \levcp_p'}^\Tor$ affinoid and $\levcp_p' \subset \levcp_p$ an open subgroup, and that $\cV \subset \Fl$ is an affinoid containing $\pi_\HT^\Tor(\tilde{\cU})$.  Note that such $\tilde{\cU}$'s form a basis of the topology of $\aSh_{\levcp^p}^\Tor$.  Since $\cV$ is quasi-compact, it is stable under the action of an open subgroup of $\levcp_p$, which we may assume to be $\levcp_p'$ up to shrinking it.  Then $(\pi_\HT^\Tor)^\#$ gives a continuous map of $p$-adic Banach spaces
    \[
        \cO_\Fl(\cV) \to \widehat{\cO}_\aSh(\tilde{\cU}),
    \]
    which is $\levcp_p'$-equivariant.  Since $\Fl$ is of finite type over $C$, the action of $\levcp_p'$ on $\cO_\Fl^+(\cV) / p$ factors through a finite quotient.  Hence, the action of $\levcp_p'$ on $\cO_\Fl(\cV)$ is locally analytic, by \cite[\aProp 2.7]{Pan:2024-nspah}.

    Now let $\rW \in \Rep_C(\Grp{M}^c)$.  By Theorem \ref{thm-HT-mor}, there is a $\levcp_p$-equivariant isomorphism
    \[
        (\pi_\HT^\Tor)^{-1}(\cW_\Fl) \otimes_{(\pi_\HT^\Tor)^{-1}(\cO_\Fl)} \widehat{\cO}_\aSh \cong \pi_{\levcp_p}^{-1}(\cW^\canext) \otimes_{\pi_{\levcp_p}^{-1}(\cO_{\aSh_\levcp^\Tor})} \widehat{\cO}_\aSh.
    \]
    We need to show this remains an isomorphism if we replace $\widehat{\cO}_\aSh$ with $\cO_\aSh^\la$.  Again, this is a local question, and we shall use the same notation as above.  Up to shrinking $\cU$ and $\cV$ if necessary, we may assume that $\cW^\canext|_\cU$ is a trivial vector bundle, and that $\GrSh{\rW}_{\Flalg_C}$ is trivial on a Zariski open subset $\tilde{V}$ of $\Flalg_C$ whose associated adic space contains $\cV$.  Choose a basis of $\cW^\canext(\cU)$ and a basis of $\GrSh{\rW}_{\Flalg_C}|_{\tilde{V}}$, viewed as a basis of $\cW_\Fl(\cV)$.  Then the above isomorphism corresponds to a matrix $M \in \GL_n(\widehat{\cO}_\aSh(\tilde{\cU}))$, where $n = \dim_C(\rW)$, and it suffices to show that $M \in \GL_n(\cO_\aSh^\la(\tilde{\cU}))$.  Note that the actions of $\levcp_p'$ on the two bases we choose are locally algebraic.  Therefore, $M$ and $M^{-1}$ have entries in the subspace of $\levcp_p'$-locally algebraic vectors of $\widehat{\cO}_\aSh(\tilde{\cU})$, and hence in $\cO_\aSh^\la(\tilde{\cU})$.
\end{proof}

\begin{remark}\label{rem-HT-mor-la-comp}
    Similar to Remark \ref{rem-HT-mor-comp}, $\pi_\HT^\la$ factors through its analogue for the adjoint Shimura datum as well.
\end{remark}

\begin{remark}
    The ringed site $(\aSh_{\levcp^p}^\Tor, \cO_\aSh^\la)$ should be thought of as a \Qtn{\emph{toroidal compactification of the locally analytic Shimura variety of infinite level at $p$}}.  Even though this ad hoc definition looks very similar to the one in Definition \ref{def-O-plus}, we will see in Theorem \ref{thm-O-la-prim-comp} below that the issue mentioned in Remark \ref{rem-Sh-infty} will disappear when we replace $\widehat{\cO}_\aSh$ with $\cO_\aSh^\la$.
\end{remark}

\subsection{Geometric Sen theory}\label{sec-geom-Sen}

The goal of this subsection is to describe the differential equation satisfied by $\cO_\aSh^\la$ in terms of the geometry of $\Fl$ and relate $\cO_\aSh^\la$ to the construction in the previous section.

Consider the sequence of inclusions of Lie algebras
\[
    \Lie \Grp{N}_C \subset \Lie \Grp{P}_C \subset \Lie \Grp{G}_C,
\]
where $\Grp{N}_C$ is the unipotent radical of $\Grp{P}_C$, as in Sections \ref{sec-lsv-setup} and \ref{sec-Sh-var}.  The group $\Grp{P}_C$ acts on each term by the adjoint action.  Hence, this sequence of Lie of algebras defines a sequence of $\Grp{G}(C)$-equivariant vector bundles
\[
    \mathfrak{n}^0 \subset \mathfrak{p}^0 \subset \mathfrak{g}^0 = \cO_\Fl \otimes_C \mathfrak{g}
\]
on $\Fl$.  Here $\mathfrak{g} := \Lie \Grp{G}_C$, and the last equality follows as the $\Grp{P}_C$-action on $\Lie \Grp{G}_C$ can be extended to the adjoint action of $\Grp{G}_C$. Concretely, for any $C$-point $x$ of $\Fl$ defining a parabolic subgroup $\Grp{P}_x$ of $\Grp{G}_C$ conjugate to $\Grp{P}_C$, with unipotent radical $\Grp{N}_x$, the subspaces $\mathfrak{n}_x^0 \subset \mathfrak{p}_x^0 \subset \mathfrak{g}_x^0 = \mathfrak{g}$ are given by $\Lie \Grp{N}_x \subset \Lie \Grp{P}_x \subset \Lie \Grp{G}_C = \mathfrak{g}$.  By taking the derivation of the $G$-action on the flag variety, we see that $\mathfrak{g}$ acts on $\cO_\Fl$.  Therefore, $\mathfrak{g}^0 = \cO_\Fl \otimes_C \mathfrak{g}$ has a nature structure of a Lie algebroid on $\Fl$, with Lie bracket
\[
    [f_1 \otimes l_1, f_2 \otimes l_2] = f_1 l_1(f_2) \otimes l_2 - f_2 l_2(f_1) \otimes l_1 + f_1 f_2 \otimes [l_1, l_2],
\]
for any $f_1, f_2 \in \cO_\Fl$ and any $l_1, l_2 \in \mathfrak{g}$.  The kernel of the natural anchor map from $\mathfrak{g}^0$ to the tangent bundle of $\Fl$ is exactly $\mathfrak{p}^0$.

Since the action of $\levcp_p$ on $\cO_\aSh^\la$ is locally analytic, it induces a $C$-linear action of the Lie algebra $\mathfrak{g}$ on $\cO_\aSh^\la$.  By extending this action $\cO_\aSh^\la$-linearly, we obtain an action of $\cO_\aSh^\la \otimes_C \mathfrak{g} = \pi_\HT^{\la, *}(\mathfrak{g}^0)$ on $\cO_\aSh^\la$.  We will focus on restrictions of this action to $\pi_\HT^{\la, *}(\mathfrak{n}^0)$ and $\pi_\HT^{\la, *}(\mathfrak{p}^0)$.  The following result was first established in the case of modular curves by the second-named author \cite{Pan:2022-laccm} and in general by Juan Esteban Rodr{\'\i}guez Camargo \cite{RodriguezCamargo:2022-lacc}.

\begin{theorem}\label{thm-n-0-zero}
    The action of $\pi_\HT^{\la, *}(\mathfrak{n}^0)$ on $\cO_\aSh^\la$ is zero.
\end{theorem}
\begin{proof}
    This is \cite[\aCor 6.2.13]{RodriguezCamargo:2022-lacc}.  The basic idea is that, by the general result in geometric Sen theory, the locally analytic vectors $\cO_\aSh^\la$ are annihilated by the geometric Sen operators of the tower $\{ \aSh_{\levcp^p \levcp_p'}^\Tor \}_{\levcp_p' \subset \levcp_p}$ of Shimura varieties, which can be computed via the $p$-adic variations of Hodge structures associated with automorphic local systems.  The main results of \cite{Diao/Lan/Liu/Zhu:2023-lrhrv} identify such $p$-adic variations of Hodge structures with the tautological complex variations of Hodge structures over Shimura varieties, and from this identification, it follows that the geometric Sen operators are essentially given by the actions of $\pi_\HT^{\la, *}(\mathfrak{n}^0)$.
\end{proof}

Next, we relate $\cO_\aSh^\la$ to the constructions in Section \ref{sec-HT-mor}.  Recall that the tower $\{ \aSh_{\levcp^p \levcp_p'} \}_{\levcp_p' \subset \levcp_p}$ defines a $\widetilde{\levcp_p}$-torsor in $\aSh_{\levcp, \proet}$.  Let $\sC^\la(\widetilde{\levcp_p}, \bQ_p)$ denote the LB space of $\bQ_p$-valued locally analytic functions on the $p$-adic Lie group $\widetilde{\levcp_p}$.  It is a representation of $\widetilde{\levcp_p}$ by the left translation action, and can be written as an inductive limit of unitary $p$-adic Banach space representations $\sC^\la(\widetilde{\levcp_p}, \bQ_p) = \varinjlim_n \sC_n$.  This can be seen by taking a normal uniform pro-$p$ subgroup $\levcp_0$ of finite index in $\widetilde{\levcp_p}$, which exists by \cite[\aCor 8.34]{Dixon/DuSautoy/Mann/Segal:1999-APG-2}; and by considering $\sC^\la(\widetilde{\levcp_p})$ as the direct limit of $\levcp_0^{p^n}$-analytic functions on $\widetilde{\levcp_p}$.  By applying Construction \ref{constr-proket-ls} to each $\sC_n$, and by taking direct limit, we obtain a pro-Kummer \'etale local system
\[
    \proketSh{\sC^\la(\widetilde{\levcp_p}, \bQ_p)} = \varinjlim_n \proketSh{\sC_n}
\]
on $\aSh_\levcp^\Tor$, which is independent of the choice of $\sC_n$.  Accordingly, consider
\[
    \proketSh{\sC^\la(\widetilde{\levcp_p}, \bQ_p)} \ho_{\bQ_p} \widehat{\cO}_{\aSh_{\levcp, \proket}^\Tor} := \varinjlim_n \bigl(\proketSh{\sC_n} \ho_{\bQ_p} \widehat{\cO}_{\aSh_{\levcp, \proket}^\Tor}\bigr),
\]
where the completed tensor product means $p$-adically completed tensor product.  Denote by $\eta_{\levcp_p}: \aSh_{\levcp, \proket}^\Tor \to \aSh_\levcp^\Tor$ the natural projection morphism of sites.

\begin{theorem}\label{thm-O-la-prim-comp}
    There is a natural quasi-isomorphism
    \[
        R \eta_{\levcp_p, *}\bigl(\proketSh{\sC^\la(\widetilde{\levcp_p}, \bQ_p)} \ho_{\bQ_p} \widehat{\cO}_{\aSh_{\levcp, \proket}^\Tor}\bigr) \cong \pi_{\levcp_p, *}^\la (\cO_\aSh^\la)
    \]
    of sheaves of $\cO_{\aSh_\levcp^\Tor}$-modules on $\aSh_\levcp^\Tor$.
\end{theorem}
\begin{proof}
    This is \cite[(6.23)]{RodriguezCamargo:2022-lacc}.  Essentially, one needs to show that the left-hand side is concentrated in degree zero.  The key point is that $H^*$ of the left-hand side can be calculated via an explicit log Higgs complex involving geometric Sen operators.
\end{proof}

\begin{corollary}\label{cor-kv-Sh-infty}
    Let $\rW \in \Rep_C(\Grp{M}^c)$ be one-dimensional and of positive parallel weight, as in Definition \ref{def-pos} and Remark \ref{rem-def-pos}.  Consider the associated $\GrSh{\rW}^\canext$ over $\Sh_{\levcp, C}^\Tor$, as in Remark \ref{rem-Sh-var-aut-bdl-can-ext}, with analytification $\cL$ over $\aSh_\levcp^\Tor$.  Then we have
    \[
        R\Gamma\bigl(\aSh_{\levcp^p}^\Tor, \pi_{\levcp_p}^{\la, *}(\cL^{-1})\bigr) \in D^{\geq d}(C),
    \]
    where $d = \dim_\bC(\Shdom)$.
\end{corollary}
\begin{proof}
    Since $\aSh_{\levcp_p}^\Tor = \varprojlim_{\levcp_p' \subset \levcp_p} |\aSh_{\levcp^p \levcp_p'}^\Tor|$, we can apply \cite[\href{https://stacks.math.columbia.edu/tag/0EXJ}{Tag 0EXJ}]{Stacks-Project} and write $\pi_{\levcp_p}^{\la, *}(\cL^{-1})$ as the colimit of the pullback sheaves of $\pi_{\levcp'_p, *}^\la \pi_{\levcp_p}^{\la, *}(\cL^{-1})$, for $\levcp'_p \subset \levcp_p$ open.  Hence, $R\Gamma\bigl(\aSh_{\levcp^p}^\Tor, \pi_{\levcp_p}^{\la, *}(\cL^{-1})\bigr)$ is the colimit of $R\Gamma\bigl(\aSh_{\levcp^p \levcp'_p}^\Tor, \pi_{\levcp_p', *}^\la \pi_{\levcp_p}^{\la, *}(\cL^{-1})\bigr)$, by \cite[\href{https://stacks.math.columbia.edu/tag/09YP}{Tag 09YP}]{Stacks-Project}; and it suffices to show that, for any $\levcp'_p$,
    \[
        R\Gamma\bigl(\aSh_{\levcp^p \levcp'_p}^\Tor, \pi_{\levcp_p', *}^\la \pi_{\levcp_p}^{\la, *}(\cL^{-1})\bigr) \in D^{\geq d}(C).
    \]
    In fact, we may assume that $\levcp'_p = \levcp_p$, because $\pi_{\levcp_p}^{\la, *}(\cL^{-1})$ is independent of the level $\levcp_p$, by Remark \ref{rem-can-ext-funct}.  By Theorem \ref{thm-O-la-prim-comp} and by applying the projection formulas twice, we obtain
    \[
    \begin{split}
        & R\Gamma\bigl(\aSh_{\levcp}^\Tor, \pi_{\levcp_p,*}^{\la}\pi_{\levcp_p}^{\la, *}(\cL^{-1})\bigr)\cong R\Gamma\bigl(\aSh_\levcp^\Tor, \pi_{\levcp_p, *}^\la(\cO_\aSh^\la) \otimes_{\cO_{\aSh_\levcp^\Tor}} \cL^{-1}\bigr) \\
        & \cong R\Gamma\bigl(\aSh_\levcp^\Tor, R \eta_{\levcp_p, *}(\proketSh{\sC^\la(\widetilde{\levcp_p}, \bQ_p)} \ho_{\bQ_p} \widehat{\cO}_{\aSh_{\levcp, \proket}^\Tor}) \otimes_{\cO_{\aSh_\levcp^\Tor}} \cL^{-1}\bigr) \\
        & \cong R\Gamma\bigl(\aSh_{\levcp, \proket}^\Tor, \proketSh{\sC^\la(\widetilde{\levcp_p}, \bQ_p)} \ho_{\bQ_p} \eta_{\levcp_p}^* \cL^{-1}\bigr).
    \end{split}
    \]
    Thus, this corollary follows, by applying Corollary \ref{cor-kv-proket} with $X_C = \Sh_{\levcp, C}^\Tor$, $L_C = \GrSh{\rW}^\canext$, $U = \Sh_{\levcp, C}$, and $\rV = \sC_n$, where the condition \Refeq{\ref{eq-thm-kv-cond}} in Theorem \ref{thm-kv} and Corollary \ref{cor-kv-proket} are satisfied because of Proposition \ref{prop-lsv-semi-ample}; and by taking direct limit over $n$, which is applicable because $\aSh_\levcp^\Tor$ is proper over $C$.
\end{proof}

The following summarizes what we have learned so far from studying both the complex and $p$-adic geometry of Shimura varieties:
\begin{theorem}\label{thm-geom-summary}
    Let $\pi_\HT^\la$ be as in Theorem \ref{thm-HT-mor-la}, and let
    \begin{equation}\label{eq-thm-geom-summary}
        \cO^\la := R \pi_{\HT, *}^\la(\cO_\aSh^\la)
    \end{equation}
    in $D^b(\cO_\Fl)$, the bounded derived category of sheaves of $\cO_\Fl$-modules on $\Fl$.  The Lie algebra $\mathfrak{g}$ acts naturally on $\cO^\la$ when the latter is viewed as an object in the bounded derived category of sheaves of abelian groups on $\Fl$, which is compatible with the $\cO_\Fl$-module structure; \ie, we have $l \cdot (f s) = l(f) s + f(l \cdot s)$, for all $l \in \mathfrak{g}$, $f \in \cO_\Fl$, and $s \in \cO^\la$.  This action extends $\cO_\Fl$-linearly to an action of $\mathfrak{g}^0$ on $\cO^\la$.  Moreover, we have the following:
    \begin{enumerate}
        \item\label{thm-geom-summary-1} The action of $\mathfrak{n}^0$ on $\cO^\la$ \Pth{defined by restricting that of $\mathfrak{g}^0$} is zero.

        \item\label{thm-geom-summary-2} Let $\rW \in \Rep_C(\Grp{M}^c)$ be one-dimensional and of \emph{positive parallel weight}, as in Definition \ref{def-pos} and Remark \ref{rem-def-pos}, with associated $\Grp{G}(C)$-equivariant line bundle $\cL_\Fl := \cW_\Fl$ over $\Fl$.  Then we have
            \[
                H^{< d}(\Fl, \cO^\la \otimes_{\cO_\Fl}^\bL \cL_\Fl^{-1}) = 0.
            \]
    \end{enumerate}
    From now on, we will drop the superscript \Qtn{$\bL$} in the notation of derived tensor product of $\cO^\la$ and vector bundles on $\Fl$.
\end{theorem}
\begin{proof}
    Part \Refenum{\ref{thm-geom-summary-1}} follows from \Refeq{\ref{eq-thm-geom-summary}} and Theorem \ref{thm-n-0-zero}.  Also, by \Refeq{\ref{eq-thm-geom-summary}} and the projection formula, we have $R\Gamma(\Fl, \cO^\la \otimes_{\cO_\Fl} \cL_\Fl^{-1}) \cong R\Gamma\bigl(\aSh_{\levcp^p}^\Tor, \pi_\HT^{\la, *}(\cL_\Fl^{-1})\bigr)$.  Based on this, Part \Refenum{\ref{thm-geom-summary-2}} follows from Theorem \ref{thm-HT-mor-la} and Corollary \ref{cor-kv-Sh-infty}.
\end{proof}

\subsection{Horizontal action}\label{sec-hor-act}

Finally, we discuss the induced action of the Levi quotient $\mathfrak{m}^0 \cong \mathfrak{p}^0 / \mathfrak{n}^0$ on $\cO^\la$.  Note that $\mathfrak{p}^0$ annihilates $\cO_\Fl$, because it is the kernel of the anchor map of the Lie algebroid $\mathfrak{g}^0$.  Hence, the restriction of the Lie bracket of $\mathfrak{g}^0$ to $\mathfrak{p}^0$ is $\cO_\Fl$-bilinear.  Equivalently, $\mathfrak{p}^0$ is the $\Grp{G}(C)$-equivariant bundle associated with the $\Grp{P}_C$-representation $\Lie \Grp{P}_C$, and this Lie bracket on $\mathfrak{p}^0$ agrees with the one coming from the Lie bracket on $\Lie \Grp{P}_C$ \Pth{viewed as a map $\Ex^2 \Lie \Grp{P}_C \to \Lie \Grp{P}_C$}.  Let $U(\mathfrak{m}^0)$ be the universal enveloping algebra of $\mathfrak{m}^0$ over $\cO_\Fl$.  Equivalently, $U(\mathfrak{m}^0)$ is the $\Grp{G}(C)$-equivariant bundle associated with the universal enveloping algebra $U(\mathfrak{m}) \in \Rep_C(\Grp{M}_C)$, where $\mathfrak{m} := \Lie \Grp{M}_C$.  The action of $\mathfrak{m}^0$ on $\cO^\la$ extends naturally to $U(\mathfrak{m}^0)$.  In particular, we get an action on $\cO^\la$ of the center
\[
    Z(U(\mathfrak{m})) \cong U(\mathfrak{m})^{\Grp{M}_C} \subset H^0(\Fl, U(\mathfrak{m}^0)).
\]
of $U(\mathfrak{m})$.  On the other hand, since the Lie algebra $\mathfrak{g}$ acts on $\cO^\la$, the center of the universal enveloping algebra $Z(U(\mathfrak{g}))$ also acts on $\cO^\la$.  To describe the relation between these two actions, we recall the Harish-Chandra homomorphism.  \Pth{We also take this opportunity to explain our choices of conventions.}  Fix a Cartan subalgebra $\mathfrak{h}$ of $\Lie \Grp{G}_C$ inside of $\Lie \Grp{P}_C$, and let $\WG_{\mathfrak{g}}$ denote the Weyl group of $\mathfrak{g}$ with respect to $\mathfrak{h}$.  We choose compatible positive roots of $\mathfrak{g}$ and $\mathfrak{m}$, and denote by $\hsum$ and $\hsum_{\mathfrak{m}}$ the usual half-sums of positive roots of $\mathfrak{g}$ and $\mathfrak{m}$, respectively.  There is an isomorphism of $C$-algebras
\begin{equation}\label{eq-HC-hom}
    \HC: Z(U(\mathfrak{g})) \Mi U(\mathfrak{h})^{\WG_{\mathfrak{g}}},
\end{equation}
called the \emph{Harish-Chandra homomorphism}, where the action of $\WG_{\mathfrak{g}}$ on $U(\mathfrak{h})$ is centered at $0$.  It has the property that:
\begin{equation}\label{eq-HC-hom-prop}
    \parbox{0.85\textwidth}{for any irreducible representation $\rV \in \Rep_C(\Grp{G}_C)$ of highest weight $\wt$, the center $Z(U(\mathfrak{g}))$ acts on $\rV$ via $Z(U(\mathfrak{g})) \Mapn{\HC} U(\mathfrak{h})^{\WG_{\mathfrak{g}}} \to U(\mathfrak{h}) \Mapn{\wt + \hsum} C$.}
\end{equation}
Similarly, there is the Harish-Chandra isomorphism of $C$-algebras
\[
    \HC_{\mathfrak{m}}: Z(U(\mathfrak{m})) \Mi U(\mathfrak{h})^{\WG_{\mathfrak{m}}},
\]
where $\WG_{\mathfrak{m}}$ denotes the Weyl group of $\mathfrak{m}$ with respect to $\mathfrak{h}$, viewed as a subgroup of $\WG_{\mathfrak{g}}$.  Consider the representation $\det \mathfrak{n}$ of $\Grp{M}_C$, whose weight we shall denote by $\wt_{\det \mathfrak{n}}$.  Since $\det \mathfrak{n}$ is one-dimensional, by the theory of highest weights, all rational multiples of $\wt_{\det \mathfrak{n}}$, including $\frac{1}{2} \wt_{\det \mathfrak{n}}$, are $\WG_{\mathfrak{m}}$-invariant.  Let $T_{- \frac{1}{2} \wt_{\det \mathfrak{n}}}: U(\mathfrak{h}) \to U(\mathfrak{h})$ denotes the translation by $-\frac{1}{2} \wt_{\det \mathfrak{n}}$; \ie, the induced map $T_{- \frac{1}{2} \wt_{\det \mathfrak{n}}}^*: \Spec U(\mathfrak{h}) \to \Spec U(\mathfrak{h})$ is the translation by $- \frac{1}{2} \wt_{\det \mathfrak{n}}$ when restricted to points defined by $\mathfrak{h}^*$.  Then there is an injection $\HC^{\mathfrak{m}}: Z(U(\mathfrak{g})) \to Z(U(\mathfrak{m}))$ which makes the following diagram commute:
\begin{equation}\label{eq-diag-gamma-m}
    \xymatrix{ {Z(U(\mathfrak{g}))} \ar[rr]^-{\HC^{\mathfrak{m}}} \ar[d]_-\HC & & {Z(U(\mathfrak{m}))} \ar[d]^-{\HC_{\mathfrak{m}}} \\
    {U(\mathfrak{h})^{\WG_{\mathfrak{g}}}} \ar[rr]^-{T_{- \frac{1}{2} \wt_{\det \mathfrak{n}}}} & & {U(\mathfrak{h})^{\WG_{\mathfrak{m}}}} }
\end{equation}

\begin{theorem}\label{thm-comp-Z(U(g))-Z(U(m))}
    The action of $Z(U(\mathfrak{g}))$ on $\cO^\la$ factors through $Z(U(\mathfrak{m}))$ via $\HC^{\mathfrak{m}}$.
\end{theorem}
\begin{proof}
    This result was previously obtained by Juan Esteban Rodr{\'\i}guez Camargo, and probably well known to experts.  We may choose the positive roots such that all the roots in $\mathfrak{n}$ are positive.  In this case, we have $\frac{1}{2} \wt_{\det \mathfrak{n}} = \hsum - \hsum_{\mathfrak{m}}$.

    The first step is to produce a map $Z(U(\mathfrak{g})) \to Z(U(\mathfrak{m}))$ compatible with both actions on $\cO^\la$.  Consider the quotient of $U(\mathfrak{g})$ by its right ideal $\mathfrak{n} U(\mathfrak{g})$.  Let $\mathfrak{p} := \Lie \Grp{P}_C$.  By the Poincar\'e--Birkhoff--Witt theorem, the natural homomorphism of algebras $U(\mathfrak{m}) \cong U(\mathfrak{p}) / \mathfrak{n} U(\mathfrak{p}) \to U(\mathfrak{g}) / \mathfrak{n} U(\mathfrak{g})$ is injective.  We shall view $U(\mathfrak{m})$ as a subalgebra of $U(\mathfrak{g}) / \mathfrak{n} U(\mathfrak{g})$ via this injection.

    \begin{lemma}\label{lem-ident-gamma-m}
        The image of the composition of $\Grp{P}_C$-equivariant maps
        \[
            Z(U(\mathfrak{g})) \to U(\mathfrak{g}) \to U(\mathfrak{g}) / \mathfrak{n} U(\mathfrak{g})
        \]
        lies in $U(\mathfrak{m})$.  The induced map $Z(U(\mathfrak{g})) \to U(\mathfrak{m})^{\Grp{M}_C} = Z(U(\mathfrak{m}))$ agrees with $\HC^{\mathfrak{m}}$.
    \end{lemma}
    \begin{proof}[Proof of Lemma \ref{lem-ident-gamma-m}]
        For the first part, it suffices to show that
        \[
            Z(U(\mathfrak{g})) \subset U(\mathfrak{m}) + \mathfrak{n} U(\mathfrak{g}).
        \]
        Consider the action of the Hodge cocharacter $\hc_\hd$ as in \Refeq{\ref{eq-hc}} which gives \Pth{via $\iota: \bC \Mi C$} a weight decomposition $\mathfrak{g} = \mathfrak{n}^\std \oplus \mathfrak{m} \oplus \mathfrak{n}$ \Pth{compatible with \Refeq{\ref{eq-Cartan-decomp-C}}, as explained in Section \ref{sec-Sh-var}}, with $\mathfrak{m}$ being the subspace of weight $0$.  Using this decomposition, we can see that the weight-$0$ part of $U(\mathfrak{g})$ is contained in $U(\mathfrak{m}) + \mathfrak{n} U(\mathfrak{g})$.  The desired claim then follows as $Z(U(\mathfrak{g}))$ is fixed by $\hc_\hd$.

        For the second part, we need to recall the construction of the Harish-Chandra isomorphism $\HC$.  Our choice of the positive roots defines a Borel subalgebra $\mathfrak{b}$ of $\mathfrak{g}$.  Let $\mathfrak{u} := [\mathfrak{b}, \mathfrak{b}]$ be its unipotent radical.  Then Harish-Chandra showed that
        \[
            Z(U(\mathfrak{g})) \subset U(\mathfrak{h}) + \mathfrak{u} U(\mathfrak{g}).
        \]
        Moreover, the right-hand side is a direct sum and induces a map $Z(U(\mathfrak{g})) \to U(\mathfrak{h})$ which gives $T_{- \hsum} \circ \HC$.  Similarly, the choice of positive roots defines a Borel subalgebra $\mathfrak{b}_{\mathfrak{m}} = \mathfrak{b} \cap \mathfrak{m}$ of $\mathfrak{m}$ with unipotent radical $\mathfrak{u_{\mathfrak{m}}} := [\mathfrak{b}_{\mathfrak{m}}, \mathfrak{b}_{\mathfrak{m}}] = \mathfrak{u} \cap \mathfrak{m}$.  We have
        \[
            Z(U(\mathfrak{m})) \subset U(\mathfrak{h}) + \mathfrak{u_{\mathfrak{m}}} U(\mathfrak{m}),
        \]
        which induces $T_{- \hsum_{\mathfrak{m}}} \circ \HC_{\mathfrak{m}}$ in a similar way.  By putting these together, we obtain
        \[
            Z(U(\mathfrak{g})) \subset Z(U(\mathfrak{m})) + \mathfrak{n} U(\mathfrak{g}) \subset U(\mathfrak{h}) + \mathfrak{u_{\mathfrak{m}}} U(\mathfrak{m}) + \mathfrak{n} U(\mathfrak{g}) \subset U(\mathfrak{h}) + \mathfrak{u} U(\mathfrak{g}),
        \]
        which induces $\HC': Z(U(\mathfrak{g})) \to Z(U(\mathfrak{m}))$ in the following commutative diagram:
        \[
            \xymatrix{ {Z(U(\mathfrak{g}))} \ar[rr]^-{\HC'} \ar[d]_-{T_{- \hsum} \circ \HC} & & {Z(U(\mathfrak{m}))} \ar[d]^-{T_{- \hsum_{\mathfrak{m}}} \circ \HC_{\mathfrak{m}}} \\
            {U(\mathfrak{h})} \ar@{=}[rr] & & {U(\mathfrak{h})} }
        \]
        Since $\hsum = \frac{1}{2} \wt_{\det \mathfrak{n}} + \hsum_{\mathfrak{m}}$, this induced map $\HC'$ indeed agrees with $\HC^{\mathfrak{m}}$, as desired.
    \end{proof}

    Back to the proof of Theorem \ref{thm-comp-Z(U(g))-Z(U(m))}.  Let $U(\mathfrak{g}^0) := \cO_\Fl \otimes_C U(\mathfrak{g})$.  We can define a natural $C$-algebra structure on it, which extends the algebra structures of $\cO_\Fl$ and $U(\mathfrak{g})$, and satisfies the usual relation
    \[
        l f = f l + l(f),
    \]
    for $l \in \mathfrak{g}$ and $f \in \cO_\Fl$.  Since $\mathfrak{n}^0 \subset \mathfrak{g}^0$ is a Lie ideal and annihilates $\cO_\Fl$, we see that
    \[
        U(\mathfrak{g}^0) \mathfrak{n}^0 U(\mathfrak{g}^0) = \mathfrak{n}^0 U(\mathfrak{g}^0) = \mathfrak{n}^0 U(\mathfrak{g})
    \]
    is an ideal of $U(\mathfrak{g}^0)$.  It is $\Grp{G}_C$-equivariant, and under the dictionary between $\Grp{G}(C)$-equivariant vector bundles over $\Fl$ and representations of $\Grp{P}_C$, the inclusion $\mathfrak{n}^0 U(\mathfrak{g}) \subset U(\mathfrak{g}^0)$ corresponds to $\mathfrak{n} U(\mathfrak{g}) \subset U(\mathfrak{g})$.

    From the universal property of universal enveloping algebras, we obtain a natural morphism $U(\mathfrak{p}^0) \to U(\mathfrak{g}^0)$, which is nothing but the morphism of $\Grp{G}(C)$-equivariant bundles associated with the natural map $U(\mathfrak{p}) \subset U(\mathfrak{g})$ of $\Grp{P}_C$-representations.  By Lemma \ref{lem-ident-gamma-m}, the image of
    \[
        Z(U(\mathfrak{g})) \to \cO_\Fl \otimes_C Z(U(\mathfrak{g})) \to \cO_\Fl \otimes_C U(\mathfrak{g}) \to U(\mathfrak{g^0}) / \mathfrak{n^0} U(\mathfrak{g})
    \]
    is contained in the image of the injective map
    \[
        Z(U(\mathfrak{m})) \to U(\mathfrak{m^0}) \cong U(\mathfrak{p^0}) / \mathfrak{n^0} U(\mathfrak{p^0}) \subset U(\mathfrak{g^0}) / \mathfrak{n^0} U(\mathfrak{g}),
    \]
    and the induced map $Z(U(\mathfrak{g})) \to Z(U(\mathfrak{m}))$ agrees with $\gamma^{\mathfrak{m}}$.  This proves the theorem because, by Theorem \ref{thm-geom-summary}\Refenum{\ref{thm-geom-summary-1}}, the natural action of $U(\mathfrak{g}^0)$ on $\cO^\la$ \Pth{extending that of $\mathfrak{g}^0$} factors through the quotient $U(\mathfrak{g}^0) / \mathfrak{n}^0 U(\mathfrak{g})$.
\end{proof}

\section{Translation functors}\label{sec-trans}

In this section, we explain how to deduce a vanishing result from Theorem \ref{thm-geom-summary}, by using translation functors to \Qtn{untwist $\cL_\Fl$}.  The argument is pretty standard and goes back to the first paper of Beilinson--Bernstein on localizations \cite{Beilinson/Bernstein:1981-ldgm}.

We shall proceed with the same setting as in Section \ref{sec-geom-Sen-Sh}.  In particular, we shall continue to identity algebraic objects over $\bC$ with algebraic objects over $C$ via the chosen isomorphism $\iota: \bC \Mi C$.

\subsection{Sufficiently regular weights}\label{sec-suff-reg-wt}

Recall that we have fixed a Cartan subalgebra $\mathfrak{h}$ of $\mathfrak{g} = \Lie \Grp{G}_C$ inside of $\mathfrak{p} = \Lie \Grp{P}_C$, and denoted by $\WG_{\mathfrak{g}}$ and $\WG_{\mathfrak{m}}$ the Weyl groups of $\mathfrak{g}$ and $\mathfrak{m}$, respectively, with respect to $\mathfrak{h}$.  Moreover, we have the Harish-Chandra isomorphism $\gamma: Z(U(\mathfrak{g})) \Mi U(\mathfrak{h})^{\WG_{\mathfrak{g}}}$ with respect to the natural action of $\WG_{\mathfrak{g}}$ \Pth{centered at $0 \in \mathfrak{h}$}, as in \Refeq{\ref{eq-HC-hom}}.  Based on $\gamma$, a character $Z(U(\mathfrak{g})) \to C$ \Pth{\ie, a $C$-homomorphism} will be identified with a $\WG_{\mathfrak{g}}$-orbit $[\wt]$ in the space $\mathfrak{h}^*$ of weights $\wt: \mathfrak{h} \to C$.  Similarly, based on the Harish-Chandra isomorphism $\gamma_{\mathfrak{m}}: Z(U(\mathfrak{m})) \Mi U(\mathfrak{h})^{\WG_{\mathfrak{m}}}$, a character $Z(U(\mathfrak{m})) \to C$ will be identified with a $\WG_{\mathfrak{m}}$-orbit $[\wt]_{\mathfrak{m}}$ in $\mathfrak{h}^*$.

For simplicity, let us introduce the following:
\begin{definition}\label{def-shorter-rt}
    We say a root of $\mathfrak{g}$ is a \emph{shorter root} if it is shorter than some other root in the same $C$-simple factor of $\mathfrak{g}$.
\end{definition}

\begin{remark}\label{rem-shorter-rt}
    Shorter roots only exist on $C$-simple factors of $\mathfrak{g}$ of type $B$ or $C$.
\end{remark}

\begin{definition}\label{def-non-cpt-factor}
    We say that a $\bC$-simple factor of $\Lie \Grp{G}_\bC$, or rather of its \Pth{semisimple} derived subalgebra $[\Lie \Grp{G}_\bC, \Lie \Grp{G}_\bC] \cong \Lie \Grp{G}_\bC^\scc \cong \Lie \Grp{G}_\bC^\der \cong \Lie \Grp{G}_\bC^\ad$ \Pth{see Lemma \ref{lem-comp-ad} and the paragraph preceding it}, is \emph{compact} \Pth{\resp \emph{noncompact}} if the pullback of $\Lie \Grp{N}$ to it is zero \Pth{\resp nonzero}.  Accordingly, we say that a $\bC$-simple factor of $\Grp{G}_\bC^\scc$ or $\Grp{G}_\bC^\ad$ is \emph{compact} \Pth{\resp \emph{noncompact}} if its Lie algebra is a compact \Pth{\resp noncompact} factor of $\Lie \Grp{G}_\bC$.  With any choice of $\iota: \bC \Mi C$, we will say that the corresponding $C$-simple factor of $\mathfrak{g} = \Lie \Grp{G}_C$, its derived subalgebra $\mathfrak{g}^\der = [\mathfrak{g}, \mathfrak{g}] \cong \Lie \Grp{G}_C^\scc \cong \Lie \Grp{G}_C^\ad$, $\Grp{G}_C^\scc$, or $\Grp{G}_\bC^\ad$ is \emph{$\iota$-compact} \Pth{\resp \emph{$\iota$-noncompact}}.
\end{definition}

\begin{remark}\label{rem-non-cpt-factor}
    By the definition of $\Grp{N}$ in Section \ref{sec-lsv-setup}, a $\bC$-simple factor of $\Lie \Grp{G}_\bC$, $\Lie \Grp{G}_\bC^\scc \cong \Lie \Grp{G}_\bC^\ad$, $\Grp{G}_\bC^\scc$, or $\Grp{G}_\bC^\ad$ is compact if and only if it acts trivially on $\Flalg_\bC$.  However, since $\Flalg$ has a natural model over $\ReFl \subset \bC$ \Pth{see Remark \ref{rem-refl-def}}, the notions of $\iota$-compact and $\iota$-noncompact factors of $\mathfrak{g}$ depend on the embedding $\ReFl \Em C$ induced by $\iota$ in general.  We will take advantage of this flexibility in the following.
\end{remark}

\begin{definition}\label{def-suff-reg-wt}
    Let $\iota: \bC \Mi C$ be any isomorphism.
    \begin{enumerate}
        \item\label{def-suff-reg-wt-m-iota} We say a character $Z(U(\mathfrak{m})) \to C$ or its kernel is \emph{$\hc_\hd^\iota$-sufficiently regular} if, for every weight $\wt'$ in the corresponding $\WG_{\mathfrak{m}}$-orbit $[\wt]_{\mathfrak{m}}$ and every root $\rt$ in $\mathfrak{n}$, with its coroot denoted by $\dual{\rt}$ as usual, we have
            \[
                (\wt' - \tfrac{1}{2} \wt_{\det \mathfrak{n}}, \dual{\rt}) \not\in
                \begin{cases}
                    \{ -1, -2 \}, & \text{if $\rt$ is a shorter root, as in Definition \ref{def-shorter-rt}}; \\
                    \{ -1 \}, & \text{otherwise}.
                \end{cases}
            \]

        \item\label{def-suff-reg-wt-g-iota} Similarly, we say a character $Z(U(\mathfrak{g}))\to C$ or its kernel is \emph{$\hc_\hd^\iota$-sufficiently regular} if, for every weight $\wt'$ in the corresponding $\WG_{\mathfrak{g}}$-orbit $[\wt]$ and every root $\rt$ in $\mathfrak{n}$ \Pth{or, equivalently, in all $\iota$-noncompact $C$-simple factors of $\mathfrak{g}$---see Remark \ref{rem-root-orbit-n} below}, we have
            \[
                (\wt', \dual{\rt}) \not\in
                \begin{cases}
                    \{ -1, -2 \}, & \text{if $\rt$ is a shorter root, as in Definition \ref{def-shorter-rt}}; \\
                    \{ -1 \}, & \text{otherwise}.
                \end{cases}
            \]

       \item\label{def-suff-reg-wt-g} We say a character $Z(U(\Lie \Grp{G})) \to C$ or its kernel is \emph{$\hc_\hd$-sufficiently regular} if its $C$-linear extension $Z(U(\Lie \Grp{G}))\otimes_\bQ C \cong Z(U(\mathfrak{g})) \to C$ is $\hc_\hd^{\iota'}$-sufficiently regular for some isomorphism $\iota': \bC \Mi C$.
    \end{enumerate}
\end{definition}

\begin{remark}\label{rem-root-orbit-n}
    It will follow from our explicit descriptions in Sections \ref{sec-key-obs} and \ref{sec-key-obs-ex} that every root $\rt$ in an $\iota$-noncompact $C$-simple factor of $\mathfrak{g}$ \Pth{see Definition \ref{def-non-cpt-factor}} is in the $\WG_{\mathfrak{g}}$-orbit of some root of some noncompact factor of $\mathfrak{n}$, because each nonzero factor of $\mathfrak{n}$ contains both the longest and shortest roots.  Note that, once we consider such $\WG_{\mathfrak{g}}$-orbits, the signs in Definition \ref{def-suff-reg-wt} are no longer important, and this is why we did not use negative signs in the condition \Refeq{\ref{eq-cond-suff-reg-intro}} in the introduction.
\end{remark}

\begin{example}\label{ex-suff-reg-type-A-1}
    Definition \ref{def-suff-reg-wt} \Refenum{\ref{def-suff-reg-wt-m-iota}} and \Refenum{\ref{def-suff-reg-wt-g}} in some cases of type $\typeA_1$:
    \begin{enumerate}
        \item\label{ex-suff-reg-type-A-1-mc} \Pth{Modular curves.}  $(\Grp{G}, \Shdom) = (\GL_2, \cH^\pm)$, where $\cH^\pm = \bC \setminus \bR$ is the union of the upper and lower Poincar\'e half-planes.  In this case, $\WG_{\mathfrak{m}}$ is trivial, and a character $\wt'$ of $Z(U(\mathfrak{m})) = U(\mathfrak{m})$ is $\hc_\hd^\iota$-sufficiently regular if and only if $\wt'|_{\mathfrak{m} \cap \mathfrak{sl}_2(C)} \neq 0$, where $\mathfrak{sl}_2(C) = [\mathfrak{gl}_2(C), \mathfrak{gl}_2(C)]$ is the derived subalgebra.  A character $[\wt]$ of $Z(U(\mathfrak{gl}_2(\bQ)))$ is $\hc_\hd$-sufficiently regular if and only if $[\wt]|_{Z(U(\mathfrak{sl}_2(\bQ)))}$ is not the character $\chi_0$ of the trivial representation of $\mathfrak{sl}_2(\bQ)$.

        \item\label{ex-suff-reg-type-A-1-sc} \Pth{Shimura curves.}  Let $F$ be a finite totally real field extension of $\bQ$, and $D / F$ a quaternion algebra split at exactly one place $v_0 | \infty$.  Fix any isomorphism $D \otimes_F F_{v_0} \Mi \Mt_2(\bR)$, and consider the Shimura datum $(\Grp{G}, \Shdom) = (\Res_{F / \bQ} D^\times, \cH^\pm)$, where $(\Res_{F / \bQ} D^\times)(\bR) \cong \prod_{v | \infty} (D \otimes_F F_v)^\times$ acts on $\cH^\pm$ via the isomorphism $(D \otimes_F F_{v_0})^\times \Mi \GL_2(\bR)$ induced by the above fixed one.  In this case, $\mathfrak{g} \cong \oplus_{\tau: F \to C} (D \otimes_{F, \tau} C) \cong \oplus_{\tau: F \to C} \, \mathfrak{gl}_2(C)^\tau$, where each factor $\mathfrak{gl}_2(C)^\tau$ is a copy of $\mathfrak{gl}_2(C)$ indexed by $\tau$, containing its derived subalgebra $\mathfrak{sl}_2(C)^\tau$.  Via any $\iota: \bC \Mi C$, the embeddings $F \Em F_{v_0} \cong \bR \Em \bC$ induce an embedding $\tau_0: F \to C$ \Pth{depending on $\iota$}; and a character $\wt'$ of $Z(U(\mathfrak{m}))$ is $\hc_\hd^\iota$-sufficiently regular if and only if $\wt'|_{\mathfrak{m} \cap \mathfrak{sl}_2(C)^{\tau_0}} \neq 0$.  A character $[\wt]$ of $Z(U(\Lie \Grp{G}))$ is $\hc_\hd$-sufficiently regular if and only if its $C$-linear extension $[\wt] \otimes_\bQ C: Z(U(\Lie \Grp{G})) \otimes_\bQ C \cong \otimes_{\tau: F \to C} \, Z(U(\mathfrak{gl}_2(C)^\tau)) \to C$ is not $\chi_0 \otimes_\bQ C$ when restricted to $Z(U(\mathfrak{sl}_2(C)^\tau))$ in \emph{at least one} factor $Z(U(\mathfrak{gl}_2(C)^\tau))$.  \Pth{This last notion makes no use of any isomorphism $\iota: \bC \Mi C$, and hence there is no $\tau_0$ associated with the distinguished place $v_0 | \infty$ of $F$.}

        \item\label{ex-suff-reg-type-A-1-hmv} \Pth{Hilbert modular varieties.}  Let $F$ be a finite totally real field extension of $\bQ$.  Consider the Shimura datum $(\Grp{G}, \Shdom) = (\Res_{F / \bQ} \GL_{2, F}, \prod_{\tau: F \to \bR} \cH^\pm)$.  As in \Refenum{\ref{ex-suff-reg-type-A-1-sc}}, $\mathfrak{g} \cong \oplus_{\tau: F \to C} \, \mathfrak{gl}_2(C)^\tau$; and we have $\mathfrak{sl}_2(C)^\tau \subset \mathfrak{gl}_2(C)^\tau$, for each $\tau$.  A character $\wt'$ of  $Z(U(\mathfrak{m}))$ is $\hc_\hd^\iota$-sufficiently regular if and only if $\wt'|_{\mathfrak{m} \cap \mathfrak{sl}_2(C)^\tau} \neq 0$ for all $\tau$, and hence this notion is independent of the choice of $\iota$.  A character $[\wt]$ of $Z(U(\Lie \Grp{G}))$ is $\hc_\hd$-sufficiently regular if and only if $[\wt] \otimes_\bQ C: Z(U(\Lie \Grp{G})) \otimes_\bQ C \cong \otimes_{\tau: F \to C} \, Z(U(\mathfrak{gl}_2(C)^\tau)) \to C$ is not $\chi_0 \otimes_\bQ C$ when restricted to $Z(U(\mathfrak{sl}_2(C)^\tau))$ in \emph{every} factor $Z(U(\mathfrak{gl}_2(C)^\tau))$.
    \end{enumerate}
\end{example}

\begin{definition}\label{def-suff-reg-wt-id}
    The \emph{non-$\hc_\hd^\iota$-sufficiently regular} \emph{locus} of $\Spec Z(U(\mathfrak{m}))$ \Pth{\resp $\Spec Z(U(\mathfrak{g}))$} is the \Pth{necessarily reduced} Zariski closure of the closed points corresponding to non-$\hc_\hd^\iota$-sufficiently regular kernels of characters as in Definition \ref{def-suff-reg-wt}.  We denote by $\cI_{\mathfrak{m}}^\iota \subset Z(U(\mathfrak{m}))$ \Pth{\resp $\cI^\iota \subset Z(U(\mathfrak{g}))$} the \Pth{necessarily radical} ideal defining this locus.  Equivalently, it is the intersection of all the non-$\hc_\hd^\iota$-sufficiently regular kernels of characters $Z(U(\mathfrak{m})) \to C$ \Pth{\resp $Z(U(\mathfrak{g})) \to C$}.  Similarly, we define the ideal $\cI$ of $Z(U(\Lie \Grp{G}))$ as the intersection of the kernel of all the non-$\hc_\hd$-sufficiently regular characters.
\end{definition}

\begin{remark}\label{rem-suff-reg-wt-old-vs-new}
    Let us choose positive roots of $\mathfrak{g}$, with usual half-sum of positive roots $\hsum$. Then $Z(U(\mathfrak{g}))$ acts on any irreducible $\rV_\wt \in \Rep_C(\Grp{G}^c)$ of highest weight $\wt$ via the infinitesimal character $[\wt + \hsum]$.  Since $(\hsum, \dual{\rt}) = 1$ for all positive roots $\rt$ of $\mathfrak{g}$, it follows that, if $\wt$ is sufficiently regular \Pth{via $\iota^{-1}: C \Mi \bC$} in the sense of \cite[\aThm 4.10]{Lan:2016-vtcac}, then $[\wt + \hsum]$ is $\hc_\hd^\iota$-sufficient regular in the sense of Definition \ref{def-suff-reg-wt}\Refenum{\ref{def-suff-reg-wt-g-iota}}.  The converse is true in most cases, except for the following two issues:
    \begin{enumerate}
        \item On $\iota$-noncompact factors of types $\typeB$ and $\typeC$, the condition in \cite[\aThm 4.10]{Lan:2016-vtcac} does not distinguish between shorter and other \Pth{longer} roots, while the condition in Definition \ref{def-suff-reg-wt}\Refenum{\ref{def-suff-reg-wt-g-iota}} does.  Nevertheless, the statement and proof in \cite[\aThm 4.10]{Lan:2016-vtcac} can be improved to take care of this difference.

        \item We have the flexibility in choosing $\iota: \bC \Mi C$ and hence permuting the $C$-simple factors of $\mathfrak{g}$, which provides additional flexibility when $\Lie \Grp{G}_\bC$ admits nontrivial compact factors \Pth{see Definition \ref{def-non-cpt-factor} and Remark \ref{rem-non-cpt-factor}}.
    \end{enumerate}
\end{remark}

\begin{proposition}\label{prop-suff-reg-m-vs-g}
    Consider the finite homomorphism $\gamma^{\mathfrak{m}}: Z(U(\mathfrak{g})) \to Z(U(\mathfrak{m}))$ in \Refeq{\ref{eq-diag-gamma-m}}.
    \begin{enumerate}
        \item\label{prop-suff-reg-m-vs-g-1} $\cI^\iota = (\gamma^{\mathfrak{m}})^{-1}(\cI_{\mathfrak{m}}^\iota)$.  Equivalently, $f: Z(U(\mathfrak{g})) \to C$ is $\hc_\hd^\iota$-sufficiently regular if and only if every extension $f': Z(U(\mathfrak{m})) \to C$ of $f$ via $\gamma^{\mathfrak{m}}$ is $\hc_\hd^\iota$-sufficiently regular.

        \item\label{prop-suff-reg-m-vs-g-2} Every homomorphism $Z(U(\mathfrak{g})) \to C$ can be extended to a $\hc_\hd^\iota$-sufficiently regular character $Z(U(\mathfrak{m})) \to C$ via $\gamma^{\mathfrak{m}}$.

        \item\label{prop-suff-reg-m-vs-g-3} $\cI$ is the radical ideal of $Z(U(\Lie \Grp{G})) \cap \sum_{\iota'} \cI^{\iota'}$, where $\iota'$ runs through all isomorphisms $\iota': \bC \Mi C$ and the sum is a finite sum.
    \end{enumerate}
\end{proposition}
\begin{proof}
    As for \Refenum{\ref{prop-suff-reg-m-vs-g-1}}, just note that the shift by $- \frac{1}{2} \wt_{\det \mathfrak{n}}$ in Definition \ref{def-suff-reg-wt} cancels the corresponding shift in the construction of $\gamma^{\mathfrak{m}}$.  We omit the details here.  As for \Refenum{\ref{prop-suff-reg-m-vs-g-2}}, it suffices to show that, given a $\WG_{\mathfrak{g}}$-orbit $[\wt]$ in $\mathfrak{h}^*$, we can find a $\WG_{\mathfrak{m}}$-orbit $[\wt]_{\mathfrak{m}} \subset [\wt]$ such that, for every $\wt' \in [\wt]$ and every root $\rt$ in $\mathfrak{n}$, we have
    \[
        (\wt', \dual{\rt}) \not \in
        \begin{cases}
            \{ -1, -2 \}, & \text{if $\rt$ is a shorter root}; \\
            \{ -1 \}, & \text{otherwise}.
        \end{cases}
    \]
    We may assume that $(\wt', \dual{\rt}) \in \bQ$, for all roots $\rt$ in $\mathfrak{g}$ and $\wt' \in [\wt]$, by choosing a $\bQ$-basis $e_1 = 1, e_2, \cdots, e_r$ of the $\bQ$-vector space spanned by all $(\wt', \dual{\rt})$ in $C$, and by considering the projection of $[\wt]$ onto its $e_1$-component.  Let us choose the positive roots of $\mathfrak{g}$ such that the roots in $\mathfrak{n}$ are positive, and choose $\wt' \in [\wt]$ in the dominant Weyl chamber.  Hence, $(\wt', \dual{\rt}) \geq 0$, for every positive root $\rt$.  Since roots in $\mathfrak{n}$ are $\WG_{\mathfrak{m}}$-invariant, $(w(\wt'), \dual{\rt}) \geq 0$, for all $w \in \WG_{\mathfrak{m}}$ and roots $\rt$ in $\mathfrak{n}$.  Thus, we can take $[\wt]_{\mathfrak{m}} = \WG_{\mathfrak{m}}(\wt')$.  Finally, for the last part, let $L \subset \bC$ be a sufficiently large finite extension of $E$ over which all $\bC$-simple factors are defined.  Then it follows that $\cI^\iota$ only depends on $\iota|_L: L \to C$.  Then our claim follows from Definition \ref{def-suff-reg-wt}.
\end{proof}

We now explain where the sufficient regularity condition comes from and how we are going to use it.  For technical reasons, let us assume that the following holds:
\begin{condition}\label{cond-der-sc-no-need-c}
    $\Grp{G}^\der$ is simply-connected, and $\Grp{G} = \Grp{G}^c$.
\end{condition}
In this case, $\Grp{G}^\scc \Mi \Grp{G}^\der \Mi \Grp{G}^{\der, c}$, and so Condition \ref{cond-der-c-sc} holds, with $\Grp{H}$ there replaced with $\Grp{G}$ here.  Let $\rW_0 \in \Rep_\bC(\Grp{M}^c)$ be any one-dimensional representation as in Proposition \ref{prop-lsv-pos-smallest}, and let $\wt_0$ denote the weight of $\rW_0$.  Concretely, in the notation of \cite[\aSec 3.3]{Lan:2016-vtcac} \Pth{via $\iota: \bC \Mi C$}, we have the following pullbacks of $\wt_0$ on individual \emph{$\iota$-noncompact} $C$-simple factors of $\Grp{G}_C^\scc$ \Pth{see Definition \ref{def-non-cpt-factor}}:
\begin{itemize}
    \item Type $\typeA$: $(1, \ldots, 1; 0, \ldots, 0) \pmod{(1, \ldots, 1; 1, \ldots, 1)}$.

    \item Types $\typeB$ and $\typeD^\bR$: $(1; 0, \ldots, 0)$.

    \item Type $\typeC$: $(1, 1, \ldots, 1)$.

    \item Type $\typeD^\bH$: $(\frac{1}{2}, \frac{1}{2}, \ldots, \frac{1}{2})$.

    \item Types $\typeE_6$ and $\typeE_7$: $(0, \ldots, 0, \frac{2\sqrt{3}}{3})$ and $(1, 0, \ldots, 0, \frac{\sqrt{2}}{2})$.
\end{itemize}
In \cite[\aSec 3.3]{Lan:2016-vtcac}, it was also shown \Pth{by working case-by-case on all possible $C$-simple factors of $\Grp{G}_C^\scc$} that $(\wt_0, \dual{\rt}) = -1$ for every root $\rt$ in $\mathfrak{n}$ unless $\rt$ is a shorter root as in Definition \ref{def-shorter-rt}, in which case $(\wt_0, \dual{\rt}) = -2$.  This explains the appearance of $-2$ in the definition of $\hc_\hd$-sufficiently regular weights.

Still assuming Condition \ref{cond-der-sc-no-need-c}, consider the finite-dimensional irreducible representation $\rV_0 \in \Rep_C(\Grp{G}_C)$ of extremal weight $\wt_0$; \ie, $\wt_0$ is in the $\WG_{\mathfrak{g}}$-orbit of the highest weight of $\rV_0$.  We shall also denote by $\rV_0$ the induced $\mathfrak{g}$-representation.

\begin{proposition}\label{prop-key-obs}
    Let $\wt_0$ and $\rV_0$ be as above.  Then a character of $Z(U(\mathfrak{m}))$ is $\hc_\hd$-\emph{sufficiently regular} as in Definition \ref{def-suff-reg-wt} if and only if the corresponding $\WG_{\mathfrak{m}}$-orbit $[\wt]_{\mathfrak{m}}$ satisfies that, for all weights $\wt_0'$ of $\rV_0$ other than $\wt_0$, we have
    \begin{equation}\label{eq-prop-key-obs}
        ([\wt]_{\mathfrak{m}} - \tfrac{1}{2} \wt_{\det \mathfrak{n}} - \wt_0 + \wt_0') \cap \WG_{\mathfrak{g}}([\wt]_{\mathfrak{m}} - \tfrac{1}{2} \wt_{\det \mathfrak{n}}) = \emptyset.
    \end{equation}
    Equivalently, $(\Id - w) (\wt - \frac{1}{2} \wt_{\det \mathfrak{n}}) \neq \wt_0 - \wt_0'$ for all $w \in \WG_{\mathfrak{g}}$, all $\wt \in [\wt]_{\mathfrak{m}}$, and all weights $\wt_0' \neq \wt_0$ of $\rV_0$.
\end{proposition}
The proof of Proposition \ref{prop-key-obs} will be given in Sections \ref{sec-key-obs} and \ref{sec-key-obs-ex}.

\subsection{Vanishing results}\label{sec-sc-van-res}

Recall that, in Theorem \ref{thm-geom-summary}, we defined $\cO^\la \in D^b(\cO_\Fl)$, with natural actions of $Z(U(\mathfrak{m}))$ and $\mathfrak{g}$.  Let $d = \dim_C(\Fl) = \dim(\Sh_\levcp) = \dim_\bC(\Shdom)$.

\begin{theorem}\label{thm-sc-van-m}
    Assume that Condition \ref{cond-der-sc-no-need-c} and Proposition \ref{prop-key-obs} hold.  Then $H^{< d}(\Fl, \cO^\la)$ is annihilated by $(\cI_{\mathfrak{m}}^\iota)^n$, for some $n > 0$.
\end{theorem}

\begin{corollary}\label{cor-sc-van-g}
    Under the same assumptions, $H^{< d}(\Fl, \cO^\la)$ is annihilated by $(\cI^\iota)^n$, for some $n > 0$.
\end{corollary}
\begin{proof}
    Combine Theorem \ref{thm-sc-van-m} with Proposition \ref{prop-suff-reg-m-vs-g} and Theorem \ref{thm-comp-Z(U(g))-Z(U(m))}.
\end{proof}

The proof of Theorem \ref{thm-sc-van-m} uses an action of $Z(U(\mathfrak{m})) \otimes_C Z(U(\mathfrak{g}))$.  For simplicity, let us drop the subscript $C$ in the remainder of this subsection.  In Section \ref{sec-hor-act}, we constructed an action of $Z(U(\mathfrak{m}))$ on $\cO^\la$, which is $\cO_\Fl$-linear and commutes with the $U(\mathfrak{g})$-action.  Let $\rW$ be a finite-dimensional representation of $\Grp{P}_C$, with associated $\mathfrak{g}$-equivariant vector bundle $\cW_\Fl$ over $\Fl$, as in Section \ref{sec-HT-mor}.  We define an action of $Z(U(\mathfrak{m})) \otimes Z(U(\mathfrak{g}))$ on $\cW_\Fl \otimes_{\cO_\Fl} \cO^\la$ with:
\begin{itemize}
    \item $Z(U(\mathfrak{m}))$ acting only on $\cO^\la$; and

    \item $Z(U(\mathfrak{g}))$ acting diagonally on the tensor product.
\end{itemize}
The upshot is that, for any map $\rW \to \rW'$ of $\Grp{P}_C$-representations, the induced map $\cW_\Fl \otimes_{\cO_\Fl} \cO^\la \to \cW_\Fl' \otimes_{\cO_\Fl} \cO^\la$ is equivariant with respect to this action.

If $\rW$ is a representation of $\Grp{M}_C$, then $\mathfrak{n}^0$ annihilates $\cW_\Fl$, and we obtain a diagonal action of $\mathfrak{m}^0 \cong \mathfrak{p}^0 / \mathfrak{n}^0$ on $\cW_\Fl \otimes_{\cO_\Fl} \cO^\la$, and this gives an action of $Z(U(\mathfrak{m})) \subset H^0(\Fl, U(\mathfrak{m}^0))$, as explained in Section \ref{sec-hor-act}.  Accordingly, we define an action of $Z(U(\mathfrak{m})) \otimes Z(U(\mathfrak{m}))$ on $\cW_\Fl \otimes_{\cO_\Fl} \cO^\la$, where the first \Pth{\resp second} $Z(U(\mathfrak{m}))$ acts only on $\cO^\la$ \Pth{\resp diagonally}.  It follows from the proof of Theorem \ref{thm-comp-Z(U(g))-Z(U(m))} that the previous $Z(U(\mathfrak{m})) \otimes Z(U(\mathfrak{g}))$-action factors through this via
\[
    \Id \otimes \gamma^{\mathfrak{m}}: Z(U(\mathfrak{m})) \otimes Z(U(\mathfrak{g})) \to Z(U(\mathfrak{m})) \otimes Z(U(\mathfrak{m})).
\]
When $\rW = C$ is the trivial representation, the action of $Z(U(\mathfrak{m})) \otimes Z(U(\mathfrak{m}))$ factors through the diagonal quotient $Z(U(\mathfrak{m}))$.  For general $\rW$, by \cite[\aThm 2.5]{Bernstein/Gelfand:1980-tprsl}, which generalizes an earlier work of Kostant's, the action of $Z(U(\mathfrak{m})) \otimes Z(U(\mathfrak{m}))$ on $\cW_\Fl \otimes_{\cO_\Fl} \cO^\la$ is zero when restricted to the ideal $\cJ_\rW^\Delta \subset Z(U(\mathfrak{m})) \otimes Z(U(\mathfrak{m}))$ defined as follows.  Under the Harish-Chandra isomorphism $\gamma_{\mathfrak{m}}: Z(U(\mathfrak{m})) \Mi U(\mathfrak{h})^{\WG_{\mathfrak{m}}}$, elements in $Z(U(\mathfrak{m})) \otimes Z(U(\mathfrak{m}))$ can be regarded as $\WG_{\mathfrak{m}} \times \WG_{\mathfrak{m}}$-invariant polynomial functions on $\mathfrak{h}^* \times \mathfrak{h}^*$.  Then $\cJ_\rW^\Delta$ consists of functions $Q$ such that:
\begin{itemize}
    \item $Q(\chi + \wtalt, \chi) = 0$ for any weight $\wtalt$ of $\rW$ and $\chi \in \mathfrak{h}^*$.
\end{itemize}
Denote by $\cJ_\rW \subset Z(U(\mathfrak{m})) \otimes Z(U(\mathfrak{g}))$ the ideal $(\Id \otimes \gamma^{\mathfrak{m}})^{-1}(\cJ_\rW^\Delta)$.  Recall that we introduced the representation $\rV_0 \in \Rep_C(\Grp{G}_C)$ of extremal weight $\wt_0$ in the paragraph preceding Proposition \ref{prop-key-obs}.  Let $\rW_0$ be the one-dimensional representation of $\Grp{M}_C$ of \Pth{highest} weight $\wt_0$, as before.
\begin{lemma}\label{lem-cons-of-Kostant}
    Suppose $\rW$ is any nontrivial representation of $\Grp{M}_C$ appearing as a subquotient of $\rV_0 \otimes_C \dual{\rW}_0$ \Pth{viewed as a representation of $\Grp{P}_C$}.  Then the composition of $Z(U(\mathfrak{m})) \Mapn{\Id \otimes 1} Z(U(\mathfrak{m})) \otimes Z(U(\mathfrak{g})) \to \bigl(Z(U(\mathfrak{m})) \otimes Z(U(\mathfrak{g}))\bigr) / (\cJ_\rW + \cJ_C)$ factors through $Z(U(\mathfrak{m})) / (\cI_{\mathfrak{m}}^\iota)^n$, for some $n > 0$.
\end{lemma}
\begin{proof}
    Recall that $\cI_{\mathfrak{m}}^\iota$ was introduced in Definition \ref{def-suff-reg-wt-id} as the ideal defining the non-$\hc_\hd$-sufficiently regular locus of $\Spec Z(U(\mathfrak{m}))$.  By Hilbert's Nullstellensatz, it suffices to prove the assertion at the level of maximal ideals. Via the Harish-Chandra isomorphism, a closed point of $\Spec\bigl(\bigl(Z(U(\mathfrak{m})) \otimes Z(U(\mathfrak{g}))\bigr) / (\cJ_\rW + \cJ_C)\bigr)$ corresponds to a $\WG_{\mathfrak{m}}$-orbit $[\wt]_{\mathfrak{m}}$ and a $\WG_{\mathfrak{g}}$-orbit $[\chi]$ in $\mathfrak{h}^*$ satisfying:
    \begin{itemize}
        \item $\wt' + \wtalt - \frac{1}{2} \wt_{\det \mathfrak{n}} \in [\chi]$, for some $\wt' \in [\wt]_{\mathfrak{m}}$ and weight $\wtalt$ of $\rW$; and

        \item $\wt' - \frac{1}{2} \wt_{\det \mathfrak{n}} \in [\chi]$, for some $\wt' \in [\wt]_{\mathfrak{m}}$.
    \end{itemize}
    Hence,
    \[
        ([\wt]_{\mathfrak{m}} - \tfrac{1}{2} \wt_{\det \mathfrak{n}} + \wtalt) \cap \WG_{\mathfrak{g}}([\wt]_{\mathfrak{m}} - \tfrac{1}{2} \wt_{\det \mathfrak{n}}) \neq \emptyset,
    \]
    for some weight $\wtalt$ of $\rW$.  By our assumption, $\wtalt = - \wt_0 + \wt_0'$, for some weight $\wt_0'$ of $\rV_0$.  Thus, by Proposition \ref{prop-key-obs}, $[\wt]_{\mathfrak{m}}$ is not $\hc_\hd$-sufficiently regular.
\end{proof}

\begin{proof}[Proof of Theorem \ref{thm-sc-van-m}]
    Consider the $\Grp{P}_C$-equivariant surjection
    \[
        \rV_0 \Surj \rW_0,
    \]
    with kernel $\rW'$.  Let $\cL_\Fl$ \Pth{\resp $\cW_\Fl'$} be the line bundle \Pth{\resp vector bundle} on $\Fl$ associated with $\rW_0$ \Pth{\resp $\rW'$} in $\Rep_C(\Grp{P}_C)$, as in Section \ref{sec-HT-mor}.  Then we have a $\mathfrak{g}$-equivariant exact sequence of the corresponding vector bundles on $\Fl$
    \[
        0 \to \cW_\Fl' \to \cO_\Fl \otimes_C \rV_0 \to \cL_\Fl \to 0.
    \]
    By taking the tensor product with $\cO^\la \otimes_{\cO_\Fl} \cL_\Fl^{-1}$, we obtain an exact triangle
    \[
        \cO^\la \otimes_{\cO_\Fl} \cL_\Fl^{-1} \otimes_{\cO_\Fl} \cW_\Fl' \to \cO^\la \otimes_{\cO_\Fl} \cL_\Fl^{-1} \otimes_C \rV_0 \to \cO^\la \Mapn{+1}
    \]
    equivariant with respect to the $Z(U(\mathfrak{m})) \otimes Z(U(\mathfrak{g}))$-action introduced before.  By Proposition \ref{prop-lsv-pos-smallest}, Remark \ref{rem-Sh-var-aut-bdl-can-ext}, and Theorem \ref{thm-geom-summary}, we have
    \[
        H^{< d}(\Fl, \cO^\la \otimes_{\cO_\Fl} \cL_\Fl^{-1} \otimes_C \rV_0) \cong H^{< d}(\Fl, \cO^\la \otimes_{\cO_\Fl} \cL_\Fl^{-1}) \otimes_C \rV_0 = 0.
    \]
    Hence, for $i< d$, the cohomology of the above exact triangle gives a short exact sequence
    \[
        0 \to H^i(\Fl, \cO^\la) \to H^{i + 1}(\Fl, \cO^\la \otimes_{\cO_\Fl} \cL_\Fl^{-1} \otimes_{\cO_\Fl} \cW_\Fl').
    \]
    Note that $\cL_\Fl^{-1} \otimes_{\cO_\Fl} \cW_\Fl'$ is filtered by equivariant vector bundles on $\Fl$ associated with nontrivial irreducible $\Grp{P}_C$-subquotients $\rW$ of $\rV_0 \otimes \dual{\rW}_0$.  Therefore, its cohomology $H^{i + 1}(\Fl, \cO^\la \otimes_{\cO_\Fl} \cL_\Fl^{-1} \otimes_{\cO_\Fl} \cW_\Fl')$ is filtered by $Z(U(\mathfrak{m})) \otimes Z(U(\mathfrak{g}))$-modules annihilated by $\cJ_\rW$, for some such $\rW$.  Since $\cJ_C$ annihilates $H^i(\Fl, \cO^\la)$, the above short exact sequence implies that $H^i(\Fl, \cO^\la)$ is filtered by $Z(U(\mathfrak{m})) \otimes Z(U(\mathfrak{g}))$-modules annihilated by $\cJ_\rW + \cJ_C$.  Thus, by Lemma \ref{lem-cons-of-Kostant}, the $Z(U(\mathfrak{m}))$-action on $H^i(\Fl, \cO^\la)$ factors through $Z(U(\mathfrak{m})) / (\cI_{\mathfrak{m}}^\iota)^n$, for some $n > 0$, as desired.
\end{proof}

\begin{remark}\label{rem-pf-thm-sc-van-m}
    The proof shows that, to make the exponent of $\cI_{\mathfrak{m}}^\iota$ in Theorem \ref{thm-sc-van-m} explicit, it suffices to understand the exponent of $\cI_{\mathfrak{m}}^\iota$ in Lemma \ref{lem-cons-of-Kostant}.
\end{remark}

\subsection{Justification of the key observation}\label{sec-key-obs}

In order to prove Proposition \ref{prop-key-obs}, we need to show that a weight $\wt: \mathfrak{h} \to C$ satisfies $\wt - w(\wt) = \wt_0 - \wt_0'$, for some $1 \neq w \in \WG_{\mathfrak{g}}$ and weight $\wt_0'$ of $\rV_0$ other than $\wt_0$, if and only if there exists some root $\rt$ of $\mathfrak{n}$ such that $(\wt, \dual{\rt}) \in \{ -1, -2 \}$ or $\{ -1 \}$ depending on whether $\rt$ is a shorter root as in Definition \ref{def-shorter-rt} or not.  Since the roots in $\mathfrak{n}$ are chosen to be negative in \cite{Lan:2016-vtcac}, to reduce the number of negative signs, we shall switch from $\mathfrak{n}$ to the opposite algebra $\mathfrak{n}^\std$, and consider the equivalent condition that there exists some root $\rt$ of $\mathfrak{n}^\std$ such that $(\wt, \dual{\rt}) \in \{ 1, 2 \}$ or $\{ 1 \}$ depending on whether $\rt$ is a shorter root as in Definition \ref{def-shorter-rt} or not.  Since $\Grp{G}$ is reductive and $\mathfrak{g} = \Lie \Grp{G}(C)$, we have $\mathfrak{g} \cong \mathfrak{z} \oplus \mathfrak{g}^\der$, where $\mathfrak{z} = Z(\mathfrak{g})$ is the center of $\mathfrak{g}$, and $\mathfrak{g}^\der = [\mathfrak{g}, \mathfrak{g}] \cong \oplus_{\pl \in \Pl} \, \mathfrak{g}_\pl$, where the $\mathfrak{g}_\pl$'s are its $C$-simple factors \Pth{\Refcf{} Definition \ref{def-non-cpt-factor}}. Accordingly, we have decompositions $\mathfrak{m} \cong \mathfrak{z} \oplus (\oplus_{\pl \in \Pl} \, \mathfrak{m}_\pl)$, $\mathfrak{h} \cong \mathfrak{z} \oplus (\oplus_{\pl \in \Pl} \, \mathfrak{h}_\pl)$, $\WG_{\mathfrak{g}} \cong \prod_{\pl \in \Pl} \WG_{\mathfrak{g}_\pl}$, $\WG_{\mathfrak{m}} \cong \prod_{\pl \in \Pl} \WG_{\mathfrak{m}_\pl}$, etc.  \Pth{We shall use similar subscripts \Qtn{$\pl$} for other similar decompositions, without explicitly introducing them.}  In particular, we have $\mathfrak{m} \cap \mathfrak{g}^\der \cong \oplus_{\pl \in \Pl} \, \mathfrak{h}_\pl$ and $\mathfrak{h} \cap \mathfrak{g}^\der \cong \oplus_{\pl \in \Pl} \, \mathfrak{h}_\pl$.  Based on these decompositions, let us write $\wt|_{\mathfrak{g}^\der} = (\wt_\pl)_{\pl \in \Pl}$, $\wt_0|_{\mathfrak{m} \cap \mathfrak{g}^\der} = (\wt_{0, \pl})_{\pl \in \Pl}$, and $\wt_0'|_{\mathfrak{m} \cap \mathfrak{g}^\der} = (\wt_{0, \pl}')_{\pl \in \Pl}$.  Since $\rV_0$ is irreducible, $\wt_0 \neq \wt_0'$ \Pth{as weights of $\mathfrak{m} \cong \mathfrak{z} \oplus (\oplus_{\pl \in \Pl} \, \mathfrak{m}_\pl)$} if and only if $\wt_{0, \pl} \neq \wt_{0, \pl}'$ \Pth{as weights of $\mathfrak{m}_\pl$} for some $\pl \in \Pl$.  So it suffices to prove the following:

\begin{proposition}\label{prop-key-obs-factor}
    In the above, for each $\pl \in \Pl$, a weight $\wt_\pl: \mathfrak{h}_\pl \to C$ satisfies $\wt_\pl - w(\wt_\pl) = \wt_{0, \pl} - \wt_{0, \pl}'$, for some $1 \neq w \in \WG_{\mathfrak{g}_\pl}$ and weight $\wt_{0, \pl}'$ of $\rV_0|_{\mathfrak{m}_\pl}$ other than $\wt_{0, \pl}$, if and only if there exists some root $\rt$ of $\mathfrak{n}_\pl^\std$ such that $(\wt, \dual{\rt}) \in \{ 1, 2 \}$ or $\{ 1 \}$ depending on whether $\rt$ is a shorter root as in Definition \ref{def-shorter-rt} or not.
\end{proposition}

On each $C$-simple factor $\mathfrak{m}_\pl$ of $\mathfrak{m} \cap \mathfrak{g}^\der$ as above, the weights $\wt_{0, \pl}'$ of $\rV_{0, \pl} := \rV_0|_{\mathfrak{m}_\pl}$ are integral in the sense that it paired integrally with the coroots of $\mathfrak{m}_\pl$.  Moreover, when we view them as subsets of the corresponding real vector spaces spanned by all integral weights, they lie in the convex hull of the $\WG_{\mathfrak{g}_\pl}$-orbit of $\wt_{0, \pl}$.  Given the explicit description of $\wt_{0, \pl}$, there is a very short list of all weights $\wt_{0, \pl}'$'s of $\rV_{0, \pl}$.

On classical factors, given our explicit knowledge of $\WG_{\mathfrak{g}_\pl}$ as combinations of sign changes and permutations, we can easily verify by case-by-case arguments, which we shall explain based on the types of $\mathfrak{g}_\pl$ in the following Sections \ref{sec-key-obs-type-A-n}--\ref{sec-key-obs-type-D-n-H}.  As for the exceptional factors, since there are only two possible cases, we shall resort to brute force computation, and we shall explain our methods in Section \ref{sec-key-obs-ex}.

\subsubsection{Type $\typeA_n$}\label{sec-key-obs-type-A-n}

In this case, by using the same notation system as in \cite[\aSec 3.3.1]{Lan:2016-vtcac}, we have the following:
\begin{itemize}
    \item $\wt_{0, \pl} = \varpi_{r_\pl} \equiv (1, \ldots, 1; 0, \ldots, 0) \pmod{(1, \ldots, 1; 1, \ldots, 1)}$, where the semicolon \Qtn{$;$} occurs between the $r_\pl$-th and $(r_\pl + 1)$-th entries;

    \item $\WG_{\mathfrak{g}_\pl} \cong \SG_{n + 1}$ \Pth{the permutation group on $n + 1$ elements}; and

    \item $\wt_{0, \pl} - \wt_{0, \pl}'$ can be any vector with entries in $\{ 0, \pm 1 \}$, with $+1$'s only occurring before \Qtn{$;$} and the $-1$'s only occurring after \Qtn{$;$}, and with the number of $+1$'s and $-1$'s being nonzero and equal to each other.  \Pth{See, \eg, the explicit descriptions on representations of $\mathfrak{sl}_{n + 1}(\bC)$ in \cite[\S\S 15.2--15.3]{Fulton/Harris:1991-RT}.}
\end{itemize}
Hence, $\wt_\pl - w(\wt_\pl) = \wt_{0, \pl} - \wt_{0, \pl}'$, where $\wt_\pl = (\wt_{\pl, 1}, \ldots, \wt_{\pl, n + 1})$, exactly when all the following hold:
\begin{itemize}
    \item $(\wt_\pl, e_i - e_{w(i)}) = \wt_{\pl, i} - \wt_{\pl, w(i)} \in \{ 0, 1 \}$, for $1 \leq i \leq r_\pl$;

    \item $(\wt_\pl, e_i - e_{w(i)}) = \wt_{\pl, i} - \wt_{\pl, w(i)} \in \{ 0, -1 \}$, for $r_\pl < i \leq n + 1$; and

    \item the numbers of $+1$'s and $-1$'s in the above are nonzero and equal to each other.
\end{itemize}
Assuming that all of these hold, we would like to show that $(\wt_\pl, \dual{\rt}) = 1$ for some root $\rt$ of $\mathfrak{n}_\pl^\std$.  Note that the roots of $\mathfrak{n}_\pl^\std$ are $e_{i_1} - e_{i_2}$, for all $1 \leq i_1 \leq r_\pl < i_2 \leq n + 1$, and there are no shorter roots.  Hence, it suffices to verify the claim that there exist some $1 \leq i_1 \leq r_\pl < i_2 \leq n + 1$ such that $\wt_{\pl, i_1} - \wt_{\pl, i_2} = (\wt_\pl, e_{i_1} - e_{i_2}) = 1$.

Let $i_0$ be any integer such that $1 \leq i_0 \leq r_\pl$ and $\wt_{\pl, i_0} - \wt_{\pl, w(i_0)} = 1$, which exists by assumption.  Consider its orbit under the action of powers of $w$ \Pth{as elements of $\SG_n$}.  Let $N_0$ denote the cardinality of this orbit.  For any integer $N$ such that $0 \leq N < N_0$, we define the \emph{integer-valued} function
\[
    F(N) := \sum_{j = 0}^N (\wt_{\pl, w^j(i_0)} - \wt_{\pl, w^{j + 1}(i_0)}) = \wt_{\pl, i_0} - \wt_{\pl, w^{N + 1}(i_0)}.
\]
By the choice of $i_0$, we have $F(0) = 1$.  Since $w^{N_0}(i_0) = i_0$, we have $F(N_0 - 1) = 0$.  Note that $F(N) - F(N - 1) \in \{ 0, 1 \}$ when $1 \leq w^N(i_0) \leq r_\pl$, and $F(N) - F(N - 1) \in \{ 0, -1 \}$ when $r_\pl < w^N(i_0) \leq n + 1$.  In particular, when the variable $N$ increases by $1$ from $N - 1$, the value of $F$ can either increase by $1$, stay the same, or decrease by $1$.  The value can increase only when $1 \leq w^N(i_0) \leq r_\pl$, and can decrease only when $r_\pl < w^N(i_0) \leq n + 1$.  Hence, we may assume that there exists some $N_1$ such that $1 \leq w^{N_1}(i_0) \leq r_\pl$ and $F(N_1) \geq F(N)$, for all $0 \leq N < N_0$.  In particular, $F(N_1) \geq F(0) = 1$.  Since $F(N_0 - 1) = 0 < F(N_1)$, the value of $F$ has to decrease by $1$ at some point; \ie, there exists some $N_2 > N_1$ such that $r_\pl < w^{N_2}(i_0) \leq n + 1$, $\wt_{\pl, w^{N_2}(i_0)} - \wt_{\pl, w^{N_2 + 1}(i_0)} = -1$, and $F(N_2) = F(N_1) - 1$.  If $1 \leq w^{N_2 + 1}(i_0) \leq r_\pl$, we set $i_1 := w^{N_2 + 1}(i_0)$ and $i_2 := w^{N_2}(i_0)$.  Otherwise, and we set $i_1 := w^{N_1}(i_0)$ and $i_2 := w^{N_2 + 1}(i_0)$.  In both cases, we have $1 \leq i_1 \leq r_\pl < i_2 \leq n + 1$ and $\wt_{\pl, i_1} - \wt_{\pl, i_2} = (\wt_\pl, e_{i_1} - e_{i_2}) = 1$, justifying the above claim.

Conversely, suppose $\rt = e_i - e_j$, for some $1 \leq i \leq r_\pl < j \leq n + 1$, is a root of $\mathfrak{n}_\pl^\std$, with coroot $\dual{\rt} = e_i - e_j$; and suppose $(\wt_\pl, \dual{\rt}) = \wt_{\pl, i} - \wt_{\pl, j} = 1$.  Let $w \in \WG_{\mathfrak{g}_\pl}$ be the simple reflection with respect to $\rt$.  Then $\wt_\pl - w(\wt_\pl) = (\wt_\pl, \dual{\rt}) \rt = \rt = e_i - e_j = \wt_{0, \pl} - \wt_{0, \pl}'$, if we take $\wt_{0, \pl}' = w(\wt_{0, \pl})$ in this case.

We shall write $\wt_\pl = (\wt_{\pl, 1}, \ldots, \wt_{\pl, n})$ in the remaining classical cases.

\subsubsection{Type $\typeB_n$}\label{sec-key-obs-type-B-n}

In this case, by using the same notation system as in \cite[\aSec 3.3.2]{Lan:2016-vtcac}, we have the following:
\begin{itemize}
    \item $\wt_{0, \pl} = \varpi_1 = (1; 0, \ldots, 0)$;

    \item $\WG_{\mathfrak{g}_\pl} \cong \SG_n \ltimes \{ \pm 1 \}^n$, and we shall denote by $\sgn(w, i)$ the sign of $w \in \WG_{\mathfrak{g}_\pl}$ at the $i$-th factor of $\{ \pm 1 \}^n$; and

    \item $\wt_{0, \pl} - \wt_{0, \pl}'$ is either $(2; 0, \ldots, 0)$ or $(1; 0, \ldots, 0)$ or $(1; 0, \ldots, 0, \pm 1, 0, \ldots, 0)$.  \Pth{See, \eg, the explicit description on the so-called \emph{standard representation} of $\mathfrak{so}_{2n + 1}(\bC)$ in \cite[\S 18.1 and \S 19.4]{Fulton/Harris:1991-RT}.}
\end{itemize}
Hence, $\wt_\pl - w(\wt_\pl) = \wt_{0, \pl} - \wt_{0, \pl}'$ exactly when one of the following three cases holds:
\begin{enumerate}
    \item\label{case-B-n-1}
        \begin{itemize}
            \item $\wt_{\pl, 1} - \sgn(w, 1) \, \wt_{\pl, w(1)} = 2$; and

            \item $\wt_{\pl, i} - \sgn(w, i) \, \wt_{\pl, w(i)} = 0$, for all $1 < i \leq n$.
        \end{itemize}

    \item\label{case-B-n-2}
        \begin{itemize}
            \item $\wt_{\pl, 1} - \sgn(w, 1) \, \wt_{\pl, w(1)} = 1$; and

            \item $\wt_{\pl, i} - \sgn(w, i) \, \wt_{\pl, w(i)} = 0$, for all $1 < i \leq n$.
        \end{itemize}

    \item\label{case-B-n-3}
        \begin{itemize}
            \item $\wt_{\pl, 1} - \sgn(w, 1) \, \wt_{\pl, w(1)} = 1$;

            \item $\wt_{\pl, i_0} - \sgn(w, i_0) \, \wt_{\pl, w(i_0)} = \pm 1$, for exactly one $1 < i_0 \leq n$; and

            \item $\wt_{\pl, i} - \sgn(w, i) \, \wt_{\pl, w(i)} = 0$, for all $i \neq 1, i_0$.
        \end{itemize}
\end{enumerate}
Assuming that one of the above cases holds, we would like to show that $(\wt_\pl, \dual{\rt}) \in \{ 1, 2 \}$ or $\{ 1 \}$ depending on whether $\rt$ is a shorter root or not.  Note that the roots of $\mathfrak{n}_\pl^\std$ are $e_1$ and $e_1 \pm e_i$, for all $1 < i \leq n$, and $e_1$ is the only shorter root.

In case \Refenum{\ref{case-B-n-1}}, consider the orbit of $1$ under the action of powers of $w$ \Pth{as elements of $\SG_n$}.  Let $N_1$ denote the cardinality of this orbit.  For all $1 \leq j < N_1$, since $w^j(1) \neq 1$, we have $\wt_{\pl, w^j(1)} - \sgn(w, w^j(1)) \, \wt_{\pl, w^{j + 1}(1)} = 0$, or rather $\wt_{\pl, w^j(1)} = \sgn(w, w^j(1)) \, \wt_{\pl, w^{j + 1}(1)}$.  Since we cannot have $\wt_{\pl, 1} - \wt_{\pl, 1} = 2$, we must have $\wt_{\pl, 1} - \sgn(w, 1) \, \wt_{\pl, w(1)} = \wt_{\pl, 1} - \prod_{0 \leq j < N_1} \sgn(w, w^j(1)) \, \wt_{\pl, 1} = \wt_{\pl, 1} + \wt_{\pl, 1} = 2$, in which case $\sgn(w, w^j(1)) = -1$ for an odd number of exponents $0 \leq j < N_1$.  If we consider the root $\rt = e_1$ of $\mathfrak{n}_\pl^\std$, with coroot $\dual{\rt} = 2 e_1$, then $(\wt_\pl, \dual{\rt}) = 2 \wt_{\pl, 1} = 2$.

In cases \Refenum{\ref{case-B-n-2}} and \Refenum{\ref{case-B-n-3}}, if $w(1) = 1$, in which case $\sgn(w, 1) = -1$ and $2 \wt_{\pl, 1} = 1$, and if we consider the root $\rt = e_1$ of $\mathfrak{n}_\pl^\std$, with coroot $\dual{\rt} = 2 e_1$, then $(\wt_\pl, \dual{\rt}) = 2 \wt_{\pl, 1} = 1$.  Otherwise, if $w(1) \neq 1$, and if we consider the root $\rt = e_1 - \sgn(w, 1) \, e_{w(1)}$ of $\mathfrak{n}_\pl^\std$, with coroot $\dual{\rt} = e_1 - \sgn(w, 1) \, e_{w(1)}$, then $(\wt_\pl, \dual{\rt}) = \wt_{\pl, 1} - \sgn(w, 1) \, \wt_{\pl, w(1)} = 1$.

In all three cases, we have $(\wt_\pl, \dual{\rt}) \in \{ 1, 2 \}$ or $\{ 1 \}$ depending on whether $\rt$ is a shorter root or not.

Conversely, suppose $\rt$ is a root of $\mathfrak{n}_\pl^\std$.
\begin{itemize}
    \item Suppose $\rt = e_1$, with $\dual{\rt} = 2 e_1$.  Let $w \in \WG_{\mathfrak{g}_\pl}$ denote the simple reflection with respect to $\rt$.  Then $\wt_\pl - w(\wt_\pl) = (\wt_\pl, \dual{\rt}) \rt$.  If $(\wt_\pl, \dual{\rt}) = (\wt_\pl, 2 e_1) = 2$, then $\wt_\pl - w(\wt_\pl) = (2; 0, \ldots, 0)$, and we are in case \Refenum{\ref{case-B-n-1}} above.  If $(\wt_\pl, \dual{\rt}) = (\wt_\pl, 2 e_1) = 1$, then $\wt_\pl - w(\wt_\pl) = (1; 0, \ldots, 0)$, and we are in case \Refenum{\ref{case-B-n-2}} above.

    \item Suppose $\rt = e_1 \pm e_i$ for some $1 < i \leq n$ and some sign $\pm$, with $\dual{\rt} = e_1 \pm e_i$.  Suppose $(\wt_\pl, \dual{\rt}) = 1$.  Let $w \in \WG_{\mathfrak{g}_\pl}$ denote the simple reflection with respect to $\rt$.  Then $\wt_\pl - w(\wt_\pl) = (\wt_\pl, \dual{\rt}) \rt = \rt = e_1 \pm e_i = \wt_{0, \pl} - \wt_{0, \pl}'$, if we take $\wt_{0, \pl}' = w(\wt_{0, \pl})$ in this case, and we are in case \Refenum{\ref{case-B-n-3}} above.
\end{itemize}
These show that, if $(\wt_\pl, \dual{\rt}) \in \{ 1, 2 \}$ or $\{ 1 \}$ depending on whether $\rt$ is a shorter root or not, then we are in one of the three cases \Refenum{\ref{case-B-n-1}}, \Refenum{\ref{case-B-n-2}}, and \Refenum{\ref{case-B-n-3}} above, for some $w \in \WG_{\mathfrak{g}_\pl}$ and weight $\wt_{0, \pl}'$ of $\rV_{0, \pl}$ other than $\wt_{0, \pl}$.

\subsubsection{Type $\typeC_n$}\label{sec-key-obs-type-C-n}

In this case, by using the same notation system as in \cite[\aSec 3.3.3]{Lan:2016-vtcac}, we have the following:
\begin{itemize}
    \item $\wt_{0, \pl} = \varpi_n = (1, 1, \ldots, 1)$;

    \item $\WG_{\mathfrak{g}_\pl} \cong \SG_n \ltimes \{ \pm 1 \}^n$, and we shall denote by $\sgn(w, i)$ the sign of $w \in \WG_{\mathfrak{g}_\pl}$ at the $i$-th factor of $\{ \pm 1 \}^n$; and

    \item $\wt_{0, \pl} - \wt_{0, \pl}'$ is any nonzero vector with entries in $\{ 0, 1, 2 \}$, with an even \Pth{and possibly zero} number of occurrences of $1$'s.  \Pth{See, \eg, the explicit descriptions on representations of $\mathfrak{sp}_{2 n}(\bC)$ in \cite[\S\S 17.2--17.3]{Fulton/Harris:1991-RT}.}
\end{itemize}
Hence, $\wt_\pl - w(\wt_\pl) = \wt_{0, \pl} - \wt_{0, \pl}'$ exactly when both of the following holds:
\begin{itemize}
    \item $\wt_{\pl, i} - \sgn(w, i) \, \wt_{\pl, w(i)} \in \{ 0, 1, 2 \}$, with an even number of occurrences of $1$'s as $i$ runs over all integers from $1$ to $n$; and

    \item there exists some integer $i_0$ such that $\wt_{\pl, i_0} - \sgn(w, i_0) \, \wt_{\pl, w(i_0)} \neq 0$.
\end{itemize}
Note that the roots of $\mathfrak{n}_\pl^\std$ are $e_{i_1} + e_{i_2}$, for all $1 \leq i_1, i_2 \leq n$ \Pth{including the case where $i_1 = i_2$ and $e_{i_1} + e_{i_2} = 2 e_{i_1}$}, and the shorter roots are those with $i_1 \neq i_2$.  The corresponding coroots are $e_{i_1} + e_{i_2}$ when $i_1 \neq i_2$, and $e_{i_1}$ when $i_1 = i_2$.

Assuming that both of the above hold, we claim that either $(\wt_\pl, e_{i_1} + e_{i_2}) \in \{ 1, 2 \}$, for some $1 \leq i_1, i_2 \leq n$ such that $i_1 \neq i_2$; or $(\wt_\pl, e_{i_1}) = 1$, for some $1 \leq i_1 \leq n$.  Let $i_0$ be any integer such that $1 \leq i_0 \leq n$ and $\wt_{\pl, i_0} - \sgn(w, i_0) \, \wt_{\pl, w(i_0)} \in \{ 1, 2 \}$.

Firstly, suppose $w(i_0) \neq i_0$.  Consider its orbit under the action of powers of $w$ \Pth{as elements of $\SG_n$}.  Let $N_0$ denote the cardinality of this orbit.  Since $w(i_0) \neq i_0$, we have $N_0 > 1$.  Moreover, $w^{j + 1}(i_0) \neq w^j(i_0)$, for all $j \in \bZ$.  If $\sgn(w, w^j(i_0)) = 1$, for all $0 \leq j < N_0$, then $\wt_{\pl, w^j(i_0)} - \sgn(w, w^j(i_0)) \, \wt_{\pl, w^{j + 1}(i_0)} = \wt_{\pl, w^j(i_0)} - \wt_{\pl, w^{j + 1}(i_0)} \in \{ 0, 1, 2 \}$, for all $0 \leq j < N_0$; and \Pth{as integers} $0 = \sum_{j = 0}^{N_0 - 1} (\wt_{\pl, w^j(i_0)} - \wt_{\pl, w^{j + 1}(i_0)}) \geq \wt_{\pl, i_0} - \wt_{\pl, w(i_0)} > 0$, which is a contradiction.  Hence, there exists some $0 \leq j_0 < N_0$ such that $\sgn(w, w^j(i_0)) = 1$, for all $0 \leq j < j_0$; but $\sgn(w, w^{j_0}(i_0)) = -1$.  In this case, $\wt_{\pl, w^{j_0}(i_0)} - \sgn(w, w^{j_0}(i_0)) \, \wt_{\pl, w^{j_0 + 1}(i_0)} = \wt_{\pl, w^{j_0}(i_0)} + \wt_{\pl, w^{j_0 + 1}(i_0)} \in \{ 0, 1, 2 \}$.  If $\wt_{\pl, w^{j_0}(i_0)} + \wt_{\pl, w^{j_0 + 1}(i_0)} \in \{ 1, 2 \}$, then we can verify the claim by setting $i_1 = w^{j_0}(i_0)$ and $i_2 = w^{j_0 + 1}(i_0)$.  Otherwise, $\wt_{\pl, w^{j_0}(i_0)} + \wt_{\pl, w^{j_0 + 1}(i_0)} = 0$, or rather $\wt_{\pl, w^{j_0 + 1}(i_0)} = - \wt_{\pl, w^{j_0}(i_0)}$, and we cannot have $j_0 = 0$.  Let $j_1$ be the largest integer such that $0 \leq j_1 < j_0$ and $\wt_{\pl, w^{j_1}(i_0)} - \sgn(w, w^{j_1}(i_0)) \, \wt_{\pl, w^{j_1 + 1}(i_0)} = \wt_{\pl, w^{j_1}(i_0)} - \wt_{\pl, w^{j_1 + 1}(i_0)} \in \{ 1, 2 \}$ \Pth{\ie, it is nonzero}.  Then $\wt_{\pl, w^j(i_0)} - \sgn(w, w^j(i_0)) \, \wt_{\pl, w^{j + 1}(i_0)} = \wt_{\pl, w^j(i_0)} - \wt_{\pl, w^{j + 1}(i_0)} = 0$, or rather $\wt_{\pl, w^j(i_0)} = \wt_{\pl, w^{j + 1}(i_0)}$, for all $j_1 < j < j_0$.  Therefore, $\wt_{\pl, w^{j_1 + 1}(i_0)} = \cdots = \wt_{\pl, w^{j_0}(i_0)} = - \wt_{\pl, w^{j_0 + 1}(i_0)}$, and $\wt_{\pl, w^{j_1}(i_0)} - \wt_{\pl, w^{j_1 + 1}(i_0)} = \wt_{\pl, w^{j_1}(i_0)} + \wt_{\pl, w^{j_0 + 1}(i_0)} \in \{ 1, 2 \}$.  Thus, we can verify the claim by setting $i_1 = w^{j_1}(i_0)$ and $i_2 = w^{j_0 + 1}(i_0)$.

Secondly, suppose $w(i_0) = i_0$.  Then we must have $\sgn(w, i_0) = -1$ and $2 \wt_{\pl, i_0} \in \{ 1, 2 \}$.  If $2 \wt_{\pl, i_0} = 2$, then we can verify the claim by setting $i_1 = i_2 = i_0$.  Otherwise, $2 \wt_{\pl, i_0} = 1$.  Since the number of occurrences of $1$'s in $\wt_{0, \pl} - \wt_{0, \pl}'$ is even, we know there exists some $i_0'$ such that $\wt_{\pl, i_0'} - \sgn(w, i_0') \, \wt_{\pl, w(i_0')} = 1$.  If $w(i_0') = i_0'$, then we must have $2 \wt_{\pl, i_0'} = 1$, and hence $\wt_{\pl, i_0} + \wt_{\pl, i_0'} = 1$, in which case we can verify the claim by setting $i_1 = i_0$ and $i_2 = i_0'$.  If $w(i_0') \neq i_0'$, then we can replace $i_0$ with $i_0'$, and resort to the previous paragraph.  Now the claim has been verified in all cases.

Conversely, suppose $\rt$ is a root of $\mathfrak{n}_\pl^\std$.
\begin{itemize}
    \item Suppose $\rt = 2 e_i$, with $\dual{\rt} = e_i$, for some $1 \leq i \leq n$.  Let $w \in \WG_{\mathfrak{g}_\pl}$ denote the simple reflection with respect to $\rt$.  Suppose $(\wt_\pl, \dual{\rt}) = 1$.  Then $\wt_\pl - w(\wt_\pl) = (\wt_\pl, \dual{\rt}) \rt = \rt = 2 e_i = \wt_{0, \pl} - \wt_{0, \pl}'$, if we take $\wt_{0, \pl}' = w(\wt_{0, \pl})$.

    \item Suppose $\rt = e_i + e_j$ for some $1 \leq i < j \leq n$, with $\dual{\rt} = e_i + e_j$.  Let $w \in \WG_{\mathfrak{g}_\pl}$ denote the simple reflection with respect to $\rt$.  Suppose $(\wt_\pl, \dual{\rt}) \in \{ 1, 2 \}$.  Then $\wt_\pl - w(\wt_\pl) = (\wt_\pl, \dual{\rt}) \rt$ is either $e_i + e_j$ or $2 e_i + 2 e_j$, both of which can occur as $\wt_{0, \pl} - \wt_{0, \pl}'$ for some $\wt_{0, \pl}'$.  \Pth{See the above.}
\end{itemize}

\subsubsection{Type $\typeD_n^\bR$}\label{sec-key-obs-type-D-n-R}

In this case, by using the same notation system as in \cite[\aSec 3.3.4, case of type $\typeD_n^\bR$]{Lan:2016-vtcac}, we have the following:
\begin{itemize}
    \item $\wt_{0, \pl} = \varpi_1 = (1; 0, \ldots, 0)$;

    \item $\WG_{\mathfrak{g}_\pl} \cong \SG_n \ltimes \{ \pm 1 \}^{n, +}$, where $+$ indicates an even number of $-1$'s, and we shall denote by $\sgn(w, i)$ the sign of $w \in \WG_{\mathfrak{g}_\pl}$ at the $i$-th factor of $\{ \pm 1 \}^{n, +}$; and

    \item $\wt_{0, \pl} - \wt_{0, \pl}'$ is either $(2; 0, \ldots, 0)$ or $(1; 0, \ldots, 0, \pm 1, 0, \ldots, 0)$.  \Pth{See, \eg, the explicit description on the so-called \emph{standard representation} of $\mathfrak{so}_{2n}(\bC)$ in \cite[\S 18.1 and \S 19.2]{Fulton/Harris:1991-RT}.}
\end{itemize}
Hence, $\wt_\pl - w(\wt_\pl) = \wt_{0, \pl} - \wt_{0, \pl}'$ exactly when one of the following three cases holds:
\begin{enumerate}
    \item\label{case-D-n-R-1}
        \begin{itemize}
            \item $\wt_{\pl, 1} - \sgn(w, 1) \, \wt_{\pl, w(1)} = 2$; and

            \item $\wt_{\pl, i} - \sgn(w, i) \, \wt_{\pl, w(i)} = 0$, for all $1 < i \leq n$.
        \end{itemize}

    \item\label{case-D-n-R-2}
        \begin{itemize}
            \item $\wt_{\pl, 1} - \sgn(w, 1) \, \wt_{\pl, w(1)} = 1$;

            \item $\wt_{\pl, i_0} - \sgn(w, i_0) \, \wt_{\pl, w(i_0)} = \pm 1$, for exactly one $1 < i_0 \leq n$; and

            \item $\wt_{\pl, i} - \sgn(w, i) \, \wt_{\pl, w(i)} = 0$, for all $i \neq 1, i_0$.
        \end{itemize}
\end{enumerate}
Assuming that one of the above cases holds, we would like to show that $(\wt_\pl, \dual{\rt}) = 1$ for some root $\rt$ of $\mathfrak{n}_\pl^\std$.  Note that the roots of $\mathfrak{n}_\pl^\std$ are $e_1 \pm e_i$, for all $1 < i \leq n$ and both signs $\pm$, and there are no shorter roots.

In case \Refenum{\ref{case-D-n-R-1}}, consider the orbit of $1$ under the action of powers of $w$ \Pth{as elements of $\SG_n$}.  Let $N_1$ denote the cardinality of this orbit.  For all $1 \leq j < N_1$, since $w^j(1) \neq 1$, we have $\wt_{\pl, w^j(1)} - \sgn(w, w^j(1)) \, \wt_{\pl, w^{j + 1}(1)} = 0$, or rather $\wt_{\pl, w^j(1)} = \sgn(w, w^j(1)) \, \wt_{\pl, w^{j + 1}(1)}$.  Since we cannot have $\wt_{\pl, 1} - \wt_{\pl, 1} = 2$, we must have $\wt_{\pl, 1} - \sgn(w, 1) \, \wt_{\pl, w(1)} = \wt_{\pl, 1} - \prod_{0 \leq j < N_1} \sgn(w, w^j(1)) \, \wt_{\pl, 1} = \wt_{\pl, 1} + \wt_{\pl, 1} = 2$, in which case $\sgn(w, w^j(1)) = -1$ for an odd number of exponents $0 \leq j < N_1$.  Consider all the orbits of the action of powers of $w$ on $\{ 1, 2, \ldots, n \}$.  Since the total number of $i$'s in the orbit of $1$ such that $\sgn(w, i) = -1$ is odd, and since the total number of $i$'s in $\{ 1, 2, \ldots, n \}$ such that $\sgn(w, i) = -1$ is even, there must exist some orbit not containing $1$ such that the total number of $i$'s in that orbit such that $\sgn(w, i) = -1$ is also odd.  Let $i_0$ be any element of this latter orbit, and let $N_0$ be the cardinality of this orbit.  Then $\wt_{\pl, w^j(i_0)} - \sgn(w, w^j(i_0)) \, \wt_{\pl, w^{j + 1}(i_0)} = 0$, or rather $\wt_{\pl, w^j(i_0)} = \sgn(w, w^j(i_0)) \, \wt_{\pl, w^j(i_0)}$, for all $0 \leq j < N_0$.  In this case, $\wt_{\pl, i_0} = \prod_{0 \leq j < N_0} \sgn(w, w^j(i_0)) \, \wt_{\pl, i_0} = - \wt_{\pl, i_0}$, which implies that $\wt_{\pl, i_0} = 0$.  If we consider the root $\rt = e_1 - e_{i_0}$ of $\mathfrak{n}_\pl^\std$, with coroot $e_1 - e_{i_0}$, then $(\wt_\pl, \dual{\rt}) = \wt_{\pl, 1} - \wt_{\pl, i_0} = 1$.

In case \Refenum{\ref{case-D-n-R-2}}, firstly, suppose $w(1) = 1$, in which case $\sgn(w, 1) = -1$ and $2 \wt_{\pl, 1} = 1$.  By assumption, we have $\wt_{\pl, i_0} - \sgn(w, i_0) \, \wt_{\pl, w(i_0)} = \pm 1$, for exactly one $1 < i_0 \leq n$, for some sign $\pm$.  Consider the orbit of $i_0$ under the action of powers of $w$.  Let $N_0$ denote the cardinality of this orbit.  For all $1 \leq j < N_0$, since $w^j(i_0) \neq i_0$, we have $\wt_{\pl, w^j(i_0)} - \sgn(w, w^j(i_0)) \, \wt_{\pl, w^{j + 1}(i_0)} = 0$, or rather $\wt_{\pl, w^j(i_0)} = \sgn(w, w^j(i_0)) \, \wt_{\pl, w^{j + 1}(i_0)}$.  Since we cannot have $\wt_{\pl, i_0} - \wt_{\pl, i_0} = \pm 1$, we must have $\wt_{\pl, i_0} - \sgn(w, i_0) \, \wt_{\pl, w(i_0)} = \wt_{\pl, i_0} - \prod_{0 \leq j < N_0} \sgn(w, w^j(i_0)) \, \wt_{\pl, i_0} = \wt_{\pl, i_0} + \wt_{\pl, i_0} = \pm 1$, in which case $\sgn(w, w^j(i_0)) = -1$ for an odd number of exponents $0 \leq j < N_0$.  Thus, $\wt_{\pl, 1} \pm \wt_{\pl, i_0} = 1$.  Secondly, suppose $w(1) \neq 1$.  Then $\wt_{\pl, 1} - \sgn(w, 1) \, \wt_{\pl, w(1)} = 1$ is also of the form $\wt_{\pl, 1} \pm \wt_{\pl, i_0}$ for some $1 < i_0 = w(1) \leq n$ and some sign $\pm$.  In both cases, if we consider the root $\rt = e_1 \pm e_{i_0}$, where the sign $\pm$ is the same as above, with coroot $\dual{\rt} = e_1 \pm e_{i_0}$, then $(\wt_\pl, \dual{\rt}) = \wt_{\pl, 1} \pm \wt_{\pl, i_0} = 1$.

Conversely, suppose $\rt = e_1 \pm e_i$ is a root of $\mathfrak{n}_\pl^\std$, for some $1 < i \leq n$ and some sign $\pm$.  Suppose $(\wt_\pl, \dual{\rt}) = 1$.  Let $w \in \WG_{\mathfrak{g}_\pl}$ denote the simple reflection with respect to $\rt$.  Then $\wt_\pl - w(\wt_\pl) = (\wt_\pl, \dual{\rt}) \rt = \rt = e_1 \pm e_i = \wt_{0, \pl} - \wt_{0, \pl}'$, if we take $\wt_{0, \pl}' = w(\wt_{0, \pl})$ in this case; and we are in case \Refenum{\ref{case-D-n-R-2}} above.  \Pth{By combining this with the above arguments, we see that case \Refenum{\ref{case-D-n-R-1}} is actually part of case \Refenum{\ref{case-D-n-R-2}}.}

\subsubsection{Type $\typeD_n^\bH$}\label{sec-key-obs-type-D-n-H}

In this case, by using the same notation system as in \cite[\aSec 3.3.4, case of type $\typeD_n^\bH$]{Lan:2016-vtcac}, up to a change of coordinates, we have the following:
\begin{itemize}
    \item $\wt_{0, \pl} = \varpi_n = (\frac{1}{2}, \frac{1}{2}, \ldots, \frac{1}{2})$;

    \item $\WG_{\mathfrak{g}_\pl} \cong \SG_n \ltimes \{ \pm 1 \}^{n, +}$ \Pth{with the same meaning as before}; and

    \item $\wt_{0, \pl} - \wt_{0, \pl}'$ is any vector with entries in $\{ 0, 1 \}$, with a positive and even number of occurrences of $1$'s.  \Pth{See, \eg, the explicit description on the so-called \emph{half-spin representations} of $\mathfrak{so}_{2n}(\bC)$ in \cite[\S 20.1]{Fulton/Harris:1991-RT}.}
\end{itemize}
Then $\wt_\pl - w(\wt_\pl) = \wt_{0, \pl} - \wt_{0, \pl}'$ exactly when both of the following holds:
\begin{itemize}
    \item $\wt_{\pl, i} - \sgn(w, i) \, \wt_{\pl, w(i)} \in \{ 0, 1 \}$, with an even number of occurrences of $1$'s as $i$ runs over all integers from $1$ to $n$; and

    \item there exists some $i_0$ such that $\wt_{\pl, i_0} - \sgn(w, i_0) \, \wt_{\pl, w(i_0)} \neq 0$.
\end{itemize}
Note that the roots of $\mathfrak{n}_\pl^\std$ are of the form $e_{i_1} + e_{i_2}$, with associated coroots of the form $e_{i_1} + e_{i_2}$, for $1 \leq i_1 < i_2 \leq n$.  There are no shorter roots.

Assuming that both of the above hold, we claim that $(\wt_\pl, e_{i_1} + e_{i_2}) = 1$, for some $1 \leq i_1, i_2 \leq n$ such that $i_1 \neq i_2$.  Let $i_0$ be any integer such that $1 \leq i_0 \leq n$ and $\wt_{\pl, i_0} - \sgn(w, i_0) \, \wt_{\pl, w(i_0)} = 1$.

Firstly, suppose $w(i_0) \neq i_0$.  Consider its orbit under the action of powers of $w$ \Pth{as elements of $\SG_n$}.  Let $N_0$ denote the cardinality of this orbit.  Since $w(i_0) \neq i_0$, we have $N_0 > 1$.  Moreover, $w^{j + 1}(i_0) \neq w^j(i_0)$, for all $j \in \bZ$.  If $\sgn(w, w^j(i_0)) = 1$, for all $0 \leq j < N_0$, then $\wt_{\pl, w^j(i_0)} - \sgn(w, w^j(i_0)) \, \wt_{\pl, w^{j + 1}(i_0)} = \wt_{\pl, w^j(i_0)} - \wt_{\pl, w^{j + 1}(i_0)} \in \{ 0, 1 \}$, for all $0 \leq j < N_0$; and \Pth{as integers} $0 = \sum_{j = 0}^{N_0 - 1} (\wt_{\pl, w^j(i_0)} - \wt_{\pl, w^{j + 1}(i_0)}) \geq \wt_{\pl, i_0} - \wt_{\pl, w(i_0)} > 0$, which is a contradiction.  Hence, there exists some $0 \leq j_0 < N_0$ such that $\sgn(w, w^j(i_0)) = 1$, for all $0 \leq j < j_0$; but $\sgn(w, w^{j_0}(i_0)) = -1$.  In this case, $\wt_{\pl, w^{j_0}(i_0)} - \sgn(w, w^{j_0}(i_0)) \, \wt_{\pl, w^{j_0 + 1}(i_0)} = \wt_{\pl, w^{j_0}(i_0)} + \wt_{\pl, w^{j_0 + 1}(i_0)} \in \{ 0, 1 \}$.  If $\wt_{\pl, w^{j_0}(i_0)} + \wt_{\pl, w^{j_0 + 1}(i_0)} = 1$, then we can verify the claim by setting $i_1 = w^{j_0}(i_0)$ and $i_2 = w^{j_0 + 1}(i_0)$.  Otherwise, $\wt_{\pl, w^{j_0}(i_0)} + \wt_{\pl, w^{j_0 + 1}(i_0)} = 0$, or rather $\wt_{\pl, w^{j_0 + 1}(i_0)} = - \wt_{\pl, w^{j_0}(i_0)}$, and we cannot have $j_0 = 0$.  Let $j_1$ be the largest integer such that $0 \leq j_1 < j_0$ and $\wt_{\pl, w^{j_1}(i_0)} - \sgn(w, w^{j_1}(i_0)) \, \wt_{\pl, w^{j_1 + 1}(i_0)} = \wt_{\pl, w^{j_1}(i_0)} - \wt_{\pl, w^{j_1 + 1}(i_0)} = 1$ \Pth{\ie, it is nonzero}.  Then $\wt_{\pl, w^j(i_0)} - \sgn(w, w^j(i_0)) \, \wt_{\pl, w^{j + 1}(i_0)} = \wt_{\pl, w^j(i_0)} - \wt_{\pl, w^{j + 1}(i_0)} = 0$, or rather $\wt_{\pl, w^j(i_0)} = \wt_{\pl, w^{j + 1}(i_0)}$, for all $j_1 < j < j_0$.  Therefore, $\wt_{\pl, w^{j_1 + 1}(i_0)} = \cdots = \wt_{\pl, w^{j_0}(i_0)} = - \wt_{\pl, w^{j_0 + 1}(i_0)}$, and $\wt_{\pl, w^{j_1}(i_0)} - \wt_{\pl, w^{j_1 + 1}(i_0)} = \wt_{\pl, w^{j_1}(i_0)} + \wt_{\pl, w^{j_0 + 1}(i_0)} = 1$.  Thus, we can verify the claim by setting $i_1 = w^{j_1}(i_0)$ and $i_2 = w^{j_0 + 1}(i_0)$.

Secondly, suppose $w(i_0) = i_0$.  Then we must have $\sgn(w, i_0) = -1$ and $2 \wt_{\pl, i_0} = 1$.  Since the number of occurrences of $1$'s in $\wt_{0, \pl} - \wt_{0, \pl}'$ is even, we know there exists some $i_0'$ such that $\wt_{\pl, i_0'} - \sgn(w, i_0') \, \wt_{\pl, w(i_0')} = 1$.  If $w(i_0') = i_0'$, then we must have $2 \wt_{\pl, i_0'} = 1$, and hence $\wt_{\pl, i_0} + \wt_{\pl, i_0'} = 1$, in which case we can verify the claim by setting $i_1 = i_0$ and $i_2 = i_0'$.  If $w(i_0') \neq i_0'$, then we can replace $i_0$ with $i_0'$, and resort to the previous paragraph.  Now the claim has been verified in all cases.

Conversely, suppose $\rt = e_i + e_j$, for some $1 \leq i < j \leq n$, is a root of $\mathfrak{n}_\pl^\std$, with $\dual{\rt} = e_i + e_j$.  Let $w \in \WG_{\mathfrak{g}_\pl}$ denote the simple reflection with respect to $\rt$.  Suppose $(\wt_\pl, \dual{\rt}) = 1$.  Then $\wt_\pl - w(\wt_\pl) = (\wt_\pl, \dual{\rt}) \rt = \rt = e_i + e_j$, which can occur as $\wt_{0, \pl} - \wt_{0, \pl}'$ for some weight $\wt_{0, \pl}'$ of $\rV_{0, \pl}$ other than $\wt_{0, \pl}$.  \Pth{See the above.}

\subsection{Computations in exceptional cases}\label{sec-key-obs-ex}

In this subsection, we shall finish the proof of Proposition \ref{prop-key-obs}, or rather Proposition \ref{prop-key-obs-factor}, by explaining our computations in the remaining cases of types $\typeE_6$ and $\typeE_7$.  Since there are no shorter roots in these two cases, we just need to verify that a weight $\wt_\pl: \mathfrak{h}_\pl \to C$ satisfies $\wt_\pl - w(\wt_\pl) = \wt_{0, \pl} - \wt_{0, \pl}'$, for some $1 \neq w \in \WG_{\mathfrak{g}_\pl}$ and some weight $\wt_{0, \pl}'$ of $\rV_{0, \pl}$ other than $\wt_{0, \pl}$, if and only if there exists some root $\rt$ of $\mathfrak{n}_\pl^\std$ such that $(\wt_\pl, \dual{\rt}) = 1$.  Unlike the classical cases in Section \ref{sec-key-obs}, where the Weyl groups can be explicitly described as signed or unsigned permutations, the Weyl groups of Lie algebras of types $\typeE_6$ and $\typeE_7$ do not have such handy descriptions.

\subsubsection{General description of the methods}\label{sec-key-obs-ex-method}

Let us start with the \Qtn{if} part, which is easier and can be done without using computers.  We shall resort to the following lemma, and verify the two conditions in it:
\begin{lemma}\label{lem-key-obs-factor-if}
    For each $\pl \in \Pl$, the condition $(\wt_\pl, \dual{\rt}) = 1$ for some root $\rt$ of $\mathfrak{n}_\pl^\std$ implies the condition that $\wt_\pl - w(\wt_\pl) = \wt_{0, \pl} - \wt_{0, \pl}'$ for some weight $\wt_{0, \pl}'$ of $\rV_{0, \pl}$ other than $\wt_{0, \pl}$, provided that the following two conditions hold:
    \begin{enumerate}
        \item\label{lem-key-obs-factor-if-trans} $\WG_{\mathfrak{m}_\pl}$ acts transitively on the roots of $\mathfrak{n}_\pl^\std$.

        \item\label{lem-key-obs-factor-if-base} There exists at least one root $\rt_0$ of $\mathfrak{n}_\pl^\std$ such that $\rt_0 = \wt_{0, \pl} - \wt_{0, \pl}''$ for some weight $\wt_{0, \pl}''$ of $\rV_{0, \pl}$ other than $\wt_{0, \pl}$.
    \end{enumerate}
\end{lemma}
\begin{proof}
    Suppose $(\wt_\pl, \dual{\rt}) = 1$ for some root $\rt$ of $\mathfrak{n}_\pl^\std$.  By \Refenum{\ref{lem-key-obs-factor-if-trans}}, there exists some $w_0 \in \WG_{\mathfrak{m}_\pl}$ such that $\rt = w_0(\rt_0)$.  Note that $w_0(\wt_{0, \pl}) = \wt_{0, \pl}$, because $\wt_{0, \pl}$ is $\WG_{\mathfrak{m}_\pl}$-invariant; and $(w_0^{-1}(\wt_\pl), \dual{\rt}_0) = (\wt_\pl, w_0(\dual{\rt}_0)) = (\wt, \dual{\rt}) = 1$, because the pairing is $\WG$-invariant.  Let $\rt_0 = \wt_{0, \pl} - \wt_{0, \pl}''$ be as in \Refenum{\ref{lem-key-obs-factor-if-base}}, and let $w_1 \in \WG_{\mathfrak{g}_\pl}$ be the simple reflection with respect to $\rt_0$.  Let $w := w_0 w_1 w_0^{-1}$ and $\wt_{0, \pl}' := w_0(\wt_{0, \pl}'')$.  Then $\wt_\pl - w(\wt_\pl) = w_0(w_0^{-1}(\wt_\pl) - w_1(w_0^{-1}(\wt_\pl))) = w_0((w_0^{-1}(\wt_\pl), \dual{\rt}_0) \rt_0) = w_0(\rt_0) = w_0(\wt_{0, \pl} - \wt_{0, \pl}'') = \wt_{0, \pl} - \wt_{0, \pl}'$, verifying the condition, as desired.
\end{proof}

As for the \Qtn{only if} part, we shall use computers, and our method can be described as follows, which works on any computer algebra system that performs Gaussian elimination and allows the creation of multiple loops:
\begin{method}\label{method-key-obs-ex-only-if}
    Fix any $\pl \in \Pl$.
    \begin{enumerate}
        \item\label{method-key-obs-ex-only-if-1} Filter the Weyl group $\WG_{\mathfrak{g}_\pl}$ by an increasing sequence
            \[
                \WG_{\mathfrak{g}_\pl, 0} = \{1\} \subset \WG_{\mathfrak{g}_\pl, 1} \subset \cdots \subset \WG_{\mathfrak{g}_\pl, m} = \WG_{\mathfrak{g}_\pl},
            \]
            with explicitly known coset representatives $\WG_{\mathfrak{g}_\pl}^j$ for $\WG_{\mathfrak{g}_\pl, j} / \WG_{\mathfrak{g}_\pl, j - 1}$, for $j = 1, \ldots, m$.  Then we can create loops running through all elements of $\WG_{\mathfrak{g}_\pl}$ by forming products of the form
            \[
                w = w_m \cdot w_{m - 1} \cdot \cdots \cdot w_1,
            \]
            with $w_j$ running through all elements of $\WG_{\mathfrak{g}_\pl}^j$, for $j = 1, \ldots, m$.

        \item\label{method-key-obs-ex-only-if-2} For each $w \in \WG_{\mathfrak{g}_\pl}$ in the above loops in \Refenum{\ref{method-key-obs-ex-only-if-1}}, compute the rank $R_w$ of the matrix $A_w$ of $\Id - w$.  Create a loop running over all weights $\wt_{0, \pl}'$ of $\rV_{0, \pl}$ other than $\wt_{0, \pl}$.

        \item\label{method-key-obs-ex-only-if-3} For each $\wt_{0, \pl}'$ in the above loop in \Refenum{\ref{method-key-obs-ex-only-if-2}}, construct the augmented matrix $A_{w, \wt_{0, \pl}'}$ for the system of linear equations $(\Id - w)(\wt_\pl) = \wt_{0, \pl} - \wt_{0, \pl}'$ \Pth{with variables given by the entries of $\wt_\pl$} by adding a column given by $\wt_{0, \pl} - \wt_{0, \pl}'$.  Compute the rank $R_{w, \wt_{0, \pl}'}$ of $A_{w, \wt_{0, \pl}'}$ by performing Gaussian elimination.  If $R_w \neq R_{w, \wt_{0, \pl}'}$, move on to the next step of the loop for $\wt_{0, \pl}'$.  Otherwise, create a loop running over all roots $\rt$ of $\mathfrak{n}_\pl^\std$.

        \item\label{method-key-obs-ex-only-if-4} For each $\rt$ in the above loop in \Refenum{\ref{method-key-obs-ex-only-if-3}}, construct an augmented matrix $A_{w, \wt_{0, \pl}', \rt}$ by adding to $A_{w, \wt_{0, \pl}'}$ \Pth{or its row echelon form obtained after performing Gaussian elimination in \Refenum{\ref{method-key-obs-ex-only-if-3}}} one more row for the additional equation $(\wt_\pl, \dual{\rt}) = 1$, and compute the rank $R_{w, \wt_{0, \pl}', \rt}$ of $A_{w, \wt_{0, \pl}', \rt}$ by performing Gaussian elimination.  If $R_{w, \wt_{0, \pl}', \rt} \neq R_{w, \wt_{0, \pl}'}$, move on to the next step of the loop for $\rt$.  Otherwise, $R_{w, \wt_{0, \pl}', \rt} = R_{w, \wt_{0, \pl}'}$, and we stop the loop for $\rt$ here.  In the latter case, the condition $(\Id - w)(\wt_\pl) = \wt_{0, \pl} - \wt_{0, \pl}'$ implies the condition $(\wt_\pl, \dual{\rt}) = 1$, for these particular $w$, $\wt_{0, \pl}'$, and $\rt$.
    \end{enumerate}
    The verification succeeds for a particular pair of $w \in \WG_{\mathfrak{g}_\pl}$ and weight $\wt_{0, \pl}'$ of $\rV_{0, \pl}$ other than $\wt_{0, \pl}$ if either $R_{w, \wt_{0, \pl}'} \neq R_w$, or if $R_{w, \wt_{0, \pl}'} = R_w$ and we can find at least one root $\rt$ of $\mathfrak{n}_\pl^\std$ such that $R_{w, \wt_{0, \pl}', \rt} = R_{w, \wt_{0, \pl}'}$.  The whole verification succeeds if the verification succeeds for all pairs of $w$ and $\wt_{0, \pl}'$.
\end{method}

\begin{remark}\label{rem-method-key-obs-ex-only-if}
    The loops in \Refenum{\ref{method-key-obs-ex-only-if-1}} can run through at most $\# \WG_{\mathfrak{g}_\pl}$ possibilities of $w$.  For each $w$, there is one Gaussian elimination for the matrix $A_w$.  The loop in \Refenum{\ref{method-key-obs-ex-only-if-2}} can run through at most $\dim_C(\rV_{0, \pl}) - 1$ possibilities of $\wt_{0, \pl}'$.  For each pair of $w$ and $\wt_{0, \pl}'$, there is one Gaussian elimination for the matrix $A_{w, \wt_{0, \pl}'}$.  The loop in \Refenum{\ref{method-key-obs-ex-only-if-3}} can run through at most $d_\pl = \dim_C(\mathfrak{n}_\pl) = \dim_C(\mathfrak{n}_\pl^\std) = \dim_C(\mathfrak{g}_\pl) - \dim_C(\mathfrak{p}_\pl)$ possibilities of $\rt$, but this last loop is only needed when $R_{w, \wt_{0, \pl}'} = R_w$, and even when it is needed, it does not need to run through all the $d_\pl$ possibilities.  For each triple of $w$, $\wt_{0, \pl}'$, and $\rt$, there is one Gaussian elimination for the matrix $A_{w, \wt_{0, \pl}', \rt}$.  Hence, the total number of Gaussian eliminations is bounded by $(\# \WG_{\mathfrak{g}_\pl}) + (\# \WG_{\mathfrak{g}_\pl}) \cdot (\dim_C(\rV_{0, \pl}) - 1) + (\# \WG_{\mathfrak{g}_\pl}) \cdot (\dim_C(\rV_{0, \pl}) - 1) \cdot d_\pl = (\# \WG_{\mathfrak{g}_\pl}) \cdot [1 + (\dim_C(\rV_0) - 1) \cdot (d_\pl + 1)]$.  Since the computations for distinct pairs $(w, \wt_{0, \pl}')$ are unrelated to each other, they can be performed in a parallel manner.
\end{remark}

\subsubsection{Type $\typeE_6$}\label{sec-key-obs-type-E-6}

In this case, we shall follow the same notation system as in \cite[\aSec 3.3.5]{Lan:2016-vtcac}.  We shall assume that $\rt_6 \not\in \RT_{\Grp{M}_\pl}$, so that $\wt_{0, \pl} = \varpi_6$.  \Pth{By using an automorphism that swaps simple roots as in \cite[\aSec 5.1.5]{Lan:2016-vtcac}, the case where $\rt_1 \not\in \RT_{\Grp{M}_\pl}$ can be deduced from the case where $\rt_6 \not\in \RT_{\Grp{M}_\pl}$.}

Let $s_1, s_2, \ldots, s_6 \in \WG_{\mathfrak{g}_\pl}$ denote the simple reflections associated with the six positive roots $\rt_1, \rt_2, \ldots, \rt_6$, respectively.  We filter the Weyl group $\WG_{\mathfrak{g}_\pl}$ by subgroups
\[
\begin{split}
    \{ 1 \} \subset \, & \WG_{\mathfrak{g}_\pl, \typeA_1} = \langle s_1 \rangle \subset \, \WG_{\mathfrak{g}_\pl, \typeA_2} = \langle s_1, s_2 \rangle \subset \, \WG_{\mathfrak{g}_\pl, \typeA_3} = \langle s_1, s_2, s_3 \rangle \\
    \subset \, & \WG_{\mathfrak{g}_\pl, \typeA_4} = \langle s_1, s_2, s_3, s_4 \rangle \subset \, \WG_{\mathfrak{g}_\pl, \typeD_5} = \langle s_1, s_2, s_3, s_4, s_5 \rangle \\
    \subset \, & \WG_{\mathfrak{g}_\pl, \typeE_6} = \langle s_1, s_2, s_3, s_4, s_5, s_6 \rangle = \WG_{\mathfrak{g}_\pl},
\end{split}
\]
where the notation $\langle \cdots \rangle$ means the subgroup generated by the elements listed within the angular brackets.  Between successive filtered pieces, we choose coset representatives as follows:
\[
\begin{split}
    \WG_{\mathfrak{g}_\pl}^{\typeA_1} & = \{ 1, s_1 \} \Utext{~for~} \WG_{\mathfrak{g}_\pl, \typeA_1} / \{ 1 \}; \\
    \WG_{\mathfrak{g}_\pl}^{\typeA_2 / \typeA_1} & = \{ 1, s_2, s_1 s_2 s_1 \} \Utext{~for~} \WG_{\mathfrak{g}_\pl, \typeA_2} / \WG_{\mathfrak{g}_\pl, \typeA_1}; \\
    \WG_{\mathfrak{g}_\pl}^{\typeA_3 / \typeA_2} & = \{ 1, s_3, s_2 s_3 s_2, s_1 s_2 s_3 s_2 s_1 \} \Utext{~for~} \WG_{\mathfrak{g}_\pl, \typeA_3} / \WG_{\mathfrak{g}_\pl, \typeA_2}; \\
    \WG_{\mathfrak{g}_\pl}^{\typeA_4 / \typeA_3} & = \{ 1, s_4, s_3 s_4 s_3, s_2 s_3 s_4 s_3 s_2, s_1 s_2 s_3 s_4 s_3 s_2 s_1 \} \Utext{~for~} \WG_{\mathfrak{g}_\pl, \typeA_4} / \WG_{\mathfrak{g}_\pl, \typeA_3}; \\
    \WG_{\mathfrak{g}_\pl}^{\typeD_5 / \typeA_4} & = \{ \pm 1 \}^{5, +} \Utext{~for~} \WG_{\mathfrak{g}_\pl, \typeD_5} / \WG_{\mathfrak{g}_\pl, \typeA_4},
\end{split}
\]
where $\{ \pm 1 \}^{5, +}$ means reflections in the first five coordinates of the form $(x_i) \mapsto (\pm x_i)$ with an even number of $-$'s, as in Section \ref{sec-key-obs-type-D-n-H}.  \Pth{There are $2^{5 - 1} = 16$ such reflections in total.}  As for the representatives $\WG_{\mathfrak{g}_\pl}^{\typeE_6 / \typeD_5}$ for $\WG_{\mathfrak{g}_\pl, \typeE_6} / \WG_{\mathfrak{g}_\pl, \typeD_5}$, we choose them to be the 27 Weyl elements in the third column of Table \ref{tab-ex-e-6-alpha-6}, which define distinct cosets modulo $\WG_{\mathfrak{g}_\pl, \typeD_5}$, because $\WG_{\mathfrak{g}_\pl, \typeD_5}$ fixes $\varpi_6$ when it only acts on the first five coordinates by evenly signed permutations, and the first column of Table \ref{tab-ex-e-6-alpha-6} shows that distinct $w \in \WG_{\mathfrak{g}_\pl}^{\typeE_6 / \typeD_5}$ define distinct $w \varpi_6$.  Using these representatives, we can express all elements of $\WG_{\mathfrak{g}_\pl}$ as
\[
    w = w^{\typeE_6 / \typeD_5} \cdot w^{\typeD_5 / \typeA_4} \cdot w^{\typeA_4 / \typeA_3} \cdot w^{\typeA_3 / \typeA_2} \cdot w^{\typeA_2 / \typeA_1} \cdot w^{\typeA_1},
\]
where $w^? \in \WG_{\mathfrak{g}_\pl}^?$, for $? = \typeE_6 / \typeD_5$, $\typeD_5 / \typeA_4$, $\typeA_4 / \typeA_3$, $\typeA_3 / \typeA_2$, $\typeA_2 / \typeA_1$, $\typeA_1$.  Since $\wt_{0, \pl} = \varpi_6$ is a minuscule weight \Pth{\ie, $|(\wt_{0, \pl}, \dual{\rt})| \leq 1$ for all roots $\rt$ of $\mathfrak{g}_\pl$}, by the theory of highest weights, all weights of the irreducible $\mathfrak{g}_\pl$-representation $\rV_{0, \pl}$ of highest weight $\wt_{0, \pl}$ are extremal and hence lie in the same $\WG_{\mathfrak{g}_\pl}$-orbit.  Since $\wt_{0, \pl}$ is fixed by $\WG_{\mathfrak{g}_\pl, \typeD_5}$, we can express each weight of $\rV_{0, \pl}$ as
\[
    w^{\typeE_6 / \typeD_5} \varpi_6,
\]
for some \Pth{uniquely determined} $w^{\typeE_6 / \typeD_5} \in \WG_{\mathfrak{g}_\pl}^{\typeE_6 / \typeD_5}$; and we have recorded all such weights in the first column of Table \ref{tab-ex-e-6-alpha-6}.  These allow us to efficiently carry out Method \ref{method-key-obs-ex-only-if} and verify the \Qtn{only if} part.  Note that
\[
    \# \WG_{\mathfrak{g}_\pl} = (\# \WG_{\mathfrak{g}_\pl}^{\typeE_6 / \typeD_5}) \cdot (\# \WG_{\mathfrak{g}_\pl}^{\typeD_5 / \typeA_4}) \cdot (\# \WG_{\mathfrak{g}_\pl, \typeA_4}) = 27 \cdot 16 \cdot 5! = 51840,
\]
$\dim_C(\rV_{0, \pl}) - 1 = 27 - 1 = 26$, and $d_\pl = 16$.  As explained in Remark \ref{rem-method-key-obs-ex-only-if}, the total number of Gaussian eliminations needed is bounded by
\[
    (\# \WG_{\mathfrak{g}_\pl}) \cdot [1 + (\dim_C(\rV_{0, \pl}) - 1) \cdot (d_\pl + 1)] = 51840 \cdot [1 + 26 \cdot 17] = 22965120.
\]
Assume that we can compute at least $10000$ Gaussian eliminations per second.  This is not unreasonable, given that computations for distinct $(w, \wt_{0, \pl}')$ can be performed in a parallel manner, and that the matrices all have small sizes and entries.  Then the total amount of time needed is bounded by $38.2752$ minutes.  Our actual computations took \emph{less than one minute} on a personal computer with an $8$-core CPU.

\begin{table}[ht!]
    \tiny
    \caption{$\{ w \varpi_6 \}_{w \in \WG_{\mathfrak{g}_\pl}^{\typeE_6 / \typeD_5}}$ in the case of type $\typeE_6$}\label{tab-ex-e-6-alpha-6}
    \begin{tabular}{|l|l|l|}
        \hline
        $(w \varpi_6)^T \text{(transposed)}$ & $l(w)$ & $w \in \WG_{\mathfrak{g}_\pl}^{\typeE_6 / \typeD_5}$ \\
        \hline
        $(0, 0, 0, 0, 0, \frac{2}{\sqrt{3}})$ & $0$ & $1$ \\
        $(\frac{1}{2}, \frac{1}{2}, \frac{1}{2}, \frac{1}{2}, \frac{1}{2}, \frac{1}{2\sqrt{3}})$ & $1$ & $w_1 = s_6$ \\
        $(\frac{1}{2}, \frac{1}{2}, \frac{1}{2}, -\frac{1}{2}, -\frac{1}{2}, \frac{1}{2\sqrt{3}})$ & $2$ & $w_2 = s_5 w_1$ \\
        $(\frac{1}{2}, \frac{1}{2}, -\frac{1}{2}, \frac{1}{2}, -\frac{1}{2}, \frac{1}{2\sqrt{3}})$ & $3$ & $w_3 = s_3 w_2$ \\
        $(\frac{1}{2}, -\frac{1}{2}, \frac{1}{2}, \frac{1}{2}, -\frac{1}{2}, \frac{1}{2\sqrt{3}})$ & $4$ & $w_{4_{\Utext{I}}} = s_2 w_3$ \\
        $(\frac{1}{2}, \frac{1}{2}, -\frac{1}{2}, -\frac{1}{2}, \frac{1}{2}, \frac{1}{2\sqrt{3}})$ & $4$ & $w_{4_{\Utext{II}}} = s_4 w_3$ \\
        $(-\frac{1}{2}, \frac{1}{2}, \frac{1}{2}, \frac{1}{2}, -\frac{1}{2}, \frac{1}{2\sqrt{3}})$ & $5$ & $w_{5_{\Utext{I}}} = s_1 w_{4_{\Utext{I}}}$ \\
        $(\frac{1}{2}, -\frac{1}{2}, \frac{1}{2}, -\frac{1}{2}, \frac{1}{2}, \frac{1}{2\sqrt{3}})$ & $5$ & $w_{5_{\Utext{II}}} = s_2 w_{4_{\Utext{II}}}$ \\
        $(-\frac{1}{2}, \frac{1}{2}, \frac{1}{2}, -\frac{1}{2}, \frac{1}{2}, \frac{1}{2\sqrt{3}})$ & $6$ & $w_{6_{\Utext{I}}} = s_4 w_{5_{\Utext{I}}} = s_1 w_{5_{\Utext{II}}}$ \\
        $(\frac{1}{2}, -\frac{1}{2}, -\frac{1}{2}, \frac{1}{2}, \frac{1}{2}, \frac{1}{2\sqrt{3}})$ & $6$ & $w_{6_{\Utext{II}}} = s_3 w_{5_{\Utext{II}}}$ \\
        $(-\frac{1}{2}, \frac{1}{2}, -\frac{1}{2}, \frac{1}{2}, \frac{1}{2}, \frac{1}{2\sqrt{3}})$ & $7$ & $w_{7_{\Utext{I}}} = s_3 w_{6_{\Utext{I}}} = s_1 w_{6_{\Utext{II}}}$ \\
        $(\frac{1}{2}, -\frac{1}{2}, -\frac{1}{2}, -\frac{1}{2}, -\frac{1}{2}, \frac{1}{2\sqrt{3}})$ & $7$ & $w_{7_{\Utext{II}}} = s_5 w_{6_{\Utext{II}}}$ \\
        $(-\frac{1}{2}, -\frac{1}{2}, \frac{1}{2}, \frac{1}{2}, \frac{1}{2}, \frac{1}{2\sqrt{3}})$ & $8$ & $w_{8_{\Utext{I}}} = s_2 w_{7_{\Utext{I}}}$ \\
        $(-\frac{1}{2}, \frac{1}{2}, -\frac{1}{2}, -\frac{1}{2}, -\frac{1}{2}, \frac{1}{2\sqrt{3}})$ & $8$ & $w_{8_{\Utext{II}}} = s_5 w_{7_{\Utext{I}}} = s_1 w_{7_{\Utext{II}}}$ \\
        $(1, 0, 0, 0, 0, -\frac{1}{\sqrt{3}})$ & $8$ & $w_{8_{\Utext{III}}} = s_6 w_{7_{\Utext{II}}}$ \\
        $(-\frac{1}{2}, -\frac{1}{2}, \frac{1}{2}, -\frac{1}{2}, -\frac{1}{2}, \frac{1}{2\sqrt{3}})$ & $9$ & $w_{9_{\Utext{I}}} = s_5 w_{8_{\Utext{I}}} = s_2 w_{8_{\Utext{II}}}$ \\
        $(0, 1, 0, 0, 0, -\frac{1}{\sqrt{3}})$ & $9$ & $w_{9_{\Utext{II}}} = s_6 w_{8_{\Utext{II}}} = s_1 w_{8_{\Utext{III}}}$ \\
        $(-\frac{1}{2}, -\frac{1}{2}, -\frac{1}{2}, \frac{1}{2}, -\frac{1}{2}, \frac{1}{2\sqrt{3}})$ & $10$ & $w_{10_{\Utext{I}}} = s_3 w_{9_{\Utext{I}}}$ \\
        $(0, 0, 1, 0, 0, -\frac{1}{\sqrt{3}})$ & $10$ & $w_{10_{\Utext{II}}} = s_6 w_{9_{\Utext{I}}} = s_2 w_{9_{\Utext{II}}}$ \\
        $(-\frac{1}{2}, -\frac{1}{2}, -\frac{1}{2}, -\frac{1}{2}, \frac{1}{2}, \frac{1}{2\sqrt{3}})$ & $11$ & $w_{11_{\Utext{I}}} = s_4 w_{10_{\Utext{I}}}$ \\
        $(0, 0, 0, 1, 0, -\frac{1}{\sqrt{3}})$ & $11$ & $w_{11_{\Utext{II}}} = s_6 w_{10_{\Utext{I}}} = s_3 w_{10_{\Utext{II}}}$ \\
        $(0, 0, 0, 0, 1, -\frac{1}{\sqrt{3}})$ & $12$ & $w_{12_{\Utext{I}}} = s_6 w_{11_{\Utext{I}}}$ \\
        $(0, 0, 0, 0, -1, -\frac{1}{\sqrt{3}})$ & $12$ & $w_{12_{\Utext{II}}} = s_5 w_{11_{\Utext{II}}}$ \\
        $(0, 0, 0, -1, 0, -\frac{1}{\sqrt{3}})$ & $13$ & $w_{13} = s_5 w_{12_{\Utext{I}}} = s_4 w_{12_{\Utext{II}}}$ \\
        $(0, 0, -1, 0, 0, -\frac{1}{\sqrt{3}})$ & $14$ & $w_{14} = s_3 w_{13}$ \\
        $(0, -1, 0, 0, 0, -\frac{1}{\sqrt{3}})$ & $15$ & $w_{15} = s_2 w_{14}$ \\
        $(-1, 0, 0, 0, 0, -\frac{1}{\sqrt{3}})$ & $16$ & $w_{16} = s_1 w_{15}$ \\
        \hline
    \end{tabular}
\end{table}

As for the \Qtn{if} part, we shall apply Lemma \ref{lem-key-obs-factor-if} by verifying the two conditions \Refenum{\ref{lem-key-obs-factor-if-trans}} and \Refenum{\ref{lem-key-obs-factor-if-base}} there.  As for \Refenum{\ref{lem-key-obs-factor-if-trans}}, the roots of $\mathfrak{n}_\pl^\std$ \Pth{\Refcf{} \cite[(3.82)]{Lan:2016-vtcac}} is
\[
    \{ \text{$(\pm \tfrac{1}{2}, \pm \tfrac{1}{2}, \pm \tfrac{1}{2}, \pm \tfrac{1}{2}, \pm \tfrac{1}{2}, +\tfrac{\sqrt{3}}{2})^T$ \Pth{transposed} with an odd number of $+$'s} \},
\]
which admits a transitive action of $\WG_{\mathfrak{m}_\pl} = \WG_{\mathfrak{g}_\pl, \typeD_5} \cong \SG_5 \ltimes \{ \pm 1 \}^{5, +}$ \Pth{acting on the first five coordinates}.  As for \Refenum{\ref{lem-key-obs-factor-if-base}}, by taking $\wt_{0, \pl}'' = (\frac{1}{2}, \frac{1}{2}, \frac{1}{2}, \frac{1}{2}, \frac{1}{2}, \frac{1}{2\sqrt{3}})^T$ \Pth{see the second row of Table \ref{tab-ex-e-6-alpha-6}}, we have $\wt_{0, \pl} - \wt_{0, \pl}'' = (-\frac{1}{2}, -\frac{1}{2}, -\frac{1}{2}, -\frac{1}{2}, -\frac{1}{2}, \frac{\sqrt{3}}{2})^T$, which is a root of $\mathfrak{n}_\pl^\std$ \Pth{\Refcf{} \cite[(3.82)]{Lan:2016-vtcac}}.  Thus, the \Qtn{if} part has also been verified.

\subsubsection{Type $\typeE_7$}\label{sec-key-obs-type-E-7}

In this case, we shall follow the same notation system as in \cite[\aSec 3.3.6]{Lan:2016-vtcac}, so that $\rt_1 \not\in \RT_{\Grp{M}_\pl}$ and $\wt_{0, \pl} = \varpi_1$.

Let $s_1, s_2, \ldots, s_7 \in \WG_{\mathfrak{g}_\pl}$ denote the simple reflections associated with the seven positive roots $\rt_1, \rt_2, \ldots, \rt_7$, respectively.  We filter the Weyl group $\WG_{\mathfrak{g}_\pl}$ by subgroups
\[
\begin{split}
    \{ 1 \} \subset \, & \WG_{\mathfrak{g}_\pl, \typeA_1} = \langle s_2 \rangle \subset \, \WG_{\mathfrak{g}_\pl, \typeA_2} = \langle s_2, s_3 \rangle \subset \, \WG_{\mathfrak{g}_\pl, \typeA_3} = \langle s_2, s_3, s_4 \rangle \\
    \subset \, & \WG_{\mathfrak{g}_\pl, \typeA_4} = \langle s_2, s_3, s_4, s_5 \rangle \subset \, \WG_{\mathfrak{g}_\pl, \typeD_5} = \langle s_2, s_3, s_4, s_5, s_6 \rangle \\
    \subset \, & \WG_{\mathfrak{g}_\pl, \typeE_6} = \langle s_2, s_3, s_4, s_5, s_6, s_7 \rangle \subset \, \WG_{\mathfrak{g}_\pl, \typeE_7} = \langle s_1, s_2, s_3, s_4, s_5, s_6, s_7 \rangle = \WG_{\mathfrak{g}_\pl}.
\end{split}
\]
Note that, for $? = \typeA_1, \typeA_2, \typeA_3, \typeA_4, \typeD_5, \typeE_6$, the subgroup $\WG_{\mathfrak{g}_\pl, ?}$ is as in Section \ref{sec-key-obs-type-E-6}, except that the indices of the generators are shifted by one.  Between successive filtered pieces, we choose coset representatives as follows:
\[
\begin{split}
    \WG_{\mathfrak{g}_\pl}^{\typeA_1} & = \{ 1, s_2 \} \Utext{~for~} \WG_{\mathfrak{g}_\pl, \typeA_1} / \{ 1 \}; \\
    \WG_{\mathfrak{g}_\pl}^{\typeA_2 / \typeA_1} & = \{ 1, s_3, s_2 s_3 s_2 \} \Utext{~for~} \WG_{\mathfrak{g}_\pl, \typeA_2} / \WG_{\mathfrak{g}_\pl, \typeA_1}; \\
    \WG_{\mathfrak{g}_\pl}^{\typeA_3 / \typeA_2} & = \{ 1, s_4, s_3 s_4 s_3, s_2 s_3 s_4 s_3 s_2 \} \Utext{~for~} \WG_{\mathfrak{g}_\pl, \typeA_3} / \WG_{\mathfrak{g}_\pl, \typeA_2}; \\
    \WG_{\mathfrak{g}_\pl}^{\typeA_4 / \typeA_3} & = \{ 1, s_5, s_4 s_5 s_4, s_3 s_4 s_5 s_4 s_2, s_2 s_3 s_4 s_5 s_4 s_3 s_2 \} \Utext{~for~} \WG_{\mathfrak{g}_\pl, \typeA_4} / \WG_{\mathfrak{g}_\pl, \typeA_3}; \\
    \WG_{\mathfrak{g}_\pl}^{\typeD_5 / \typeA_4} & = \{ \pm 1 \}^{5, +} \Utext{~for~} \WG_{\mathfrak{g}_\pl, \typeD_5} / \WG_{\mathfrak{g}_\pl, \typeA_4},
\end{split}
\]
where $\{ \pm 1 \}^{5, +}$ means reflections in the middle five coordinates of the form $(x_i) \mapsto (\pm x_i)$ with an even number of $-$'s, as in Section \ref{sec-key-obs-type-D-n-H}.  \Pth{There are $2^{5 - 1} = 16$ such reflections in total.}  As for the representatives $\WG_{\mathfrak{g}_\pl}^{\typeE_6 / \typeD_5}$ for $\WG_{\mathfrak{g}_\pl, \typeE_6} / \WG_{\mathfrak{g}_\pl, \typeD_5}$, we choose them to be the 27 Weyl elements in the third column of Table \ref{tab-ex-e-7-alpha-1-alpha-7}, which define distinct cosets modulo $\WG_{\mathfrak{g}_\pl, \typeD_5}$, because $\WG_{\mathfrak{g}_\pl, \typeD_5}$ fixes $\varpi_7$ when it only acts on the middle five coordinates by evenly signed permutations, and the first column of Table \ref{tab-ex-e-7-alpha-1-alpha-7} shows that distinct $w \in \WG_{\mathfrak{g}_\pl}^{\typeE_6 / \typeD_5}$ define distinct $w \varpi_6$.  As for the representatives $\WG_{\mathfrak{g}_\pl}^{\typeE_7 / \typeE_6}$ for $\WG_{\mathfrak{g}_\pl, \typeE_7} / \WG_{\mathfrak{g}_\pl, \typeE_6}$, we choose them to be the 56 Weyl elements in the third column of Table \ref{tab-ex-e-7-alpha-1}, which define distinct cosets modulo $\WG_{\mathfrak{g}_\pl, \typeE_6}$, because $\WG_{\mathfrak{g}_\pl, \typeE_6} = \langle s_2, s_3, s_4, s_5, s_6, s_7 \rangle$ fixes $\varpi_1$ when $(\varpi_1, \dual{\rt}_i) = 0$ for all $2 \leq i \leq 7$, and because the first column of Table \ref{tab-ex-e-7-alpha-1} showes that distinct $w \in \WG_{\mathfrak{g}_\pl}^{\typeE_6 / \typeD_5}$ define distinct $w \varpi_6$.  Using these representatives, we can express all elements of $\WG_{\mathfrak{g}_\pl}$ as
\[
    w = w^{\typeE_7 / \typeE_6} \cdot w^{\typeE_6 / \typeD_5} \cdot w^{\typeD_5 / \typeA_4} \cdot w^{\typeA_4 / \typeA_3} \cdot w^{\typeA_3 / \typeA_2} \cdot w^{\typeA_2 / \typeA_1} \cdot w^{\typeA_1},
\]
where $w^? \in \WG_{\mathfrak{g}_\pl}^?$, for $? = \typeE_7 / \typeE_6$, $\typeE_6 / \typeD_5$, $\typeD_5 / \typeA_4$, $\typeA_4 / \typeA_3$, $\typeA_3 / \typeA_2$, $\typeA_2 / \typeA_1$, $\typeA_1$.  Since $\wt_{0, \pl} = \varpi_1$ is a minuscule weight \Pth{\Refcf{} Section \ref{sec-key-obs-type-E-6}} fixed by $\WG_{\mathfrak{g}_\pl, \typeE_6}$, we can express each weight of the irreducible $\mathfrak{g}_\pl$-representation $\rV_{0, \pl}$ of highest weight $\wt_{0, \pl}$ as
\[
    w^{\typeE_7 / \typeE_6} \varpi_1,
\]
for some \Pth{uniquely determined} $w^{\typeE_7 / \typeE_6} \in \WG_{\mathfrak{g}_\pl}^{\typeE_6 / \typeD_5}$; and we have recorded all such weights in the first column of Table \ref{tab-ex-e-7-alpha-1}.  These allow us to efficiently carry out Method \ref{method-key-obs-ex-only-if} and verify the \Qtn{only if} part.  Note that
\[
    \# \WG_{\mathfrak{g}_\pl} = (w_\pl^{\typeE_7 / \typeE_6}) \cdot (\# \WG_{\mathfrak{g}_\pl, \typeE_6}) = 56 \cdot 51840 = 2903040,
\]
$\dim_C(\rV_{0, \pl}) - 1 = 56 - 1 = 55$, and $d_\pl = 27$.  As explained in Remark \ref{rem-method-key-obs-ex-only-if}, the total number of Gaussian eliminations needed is bounded by
\[
    (\# \WG_{\mathfrak{g}_\pl}) \cdot [1 + (\dim_C(\rV_{0, \pl}) - 1) \cdot (d_\pl + 1)] = 2903040 \cdot [1 + 55 \cdot 28] = 4473584640.
\]
Assuming as in Section \ref{sec-key-obs-type-E-6} that we can compute at least $10000$ Gaussian eliminations per second, the total amount of time needed is bounded by $5.17776$ days.  Our actual computations took \emph{less than 3 hours} on a personal computer with an $8$-core CPU.

\begin{table}[ht!]
    \tiny
    \caption{$\{ w \varpi_7 \}_{w \in \WG_{\mathfrak{g}_\pl}^{\typeE_6 / \typeD_5}}$ in the case of type $\typeE_7$}\label{tab-ex-e-7-alpha-1-alpha-7}
    \begin{tabular}{|l|l|l|}
        \hline
        $(w \varpi_7)^T \text{(transposed)}$ & $l(w)$ & $w \in \WG_{\mathfrak{g}_\pl}^{\typeE_6 / \typeD_5}$ \\
        \hline
        $(0, 0, 0, 0, 0, 0, \sqrt{2})$ & $0$ & $1$ \\
        $(\frac{1}{2}, \frac{1}{2}, \frac{1}{2}, \frac{1}{2}, \frac{1}{2}, \frac{1}{2}, \frac{1}{\sqrt{2}})$ & $1$ & $w_1 = s_7$ \\
        $(\frac{1}{2}, \frac{1}{2}, \frac{1}{2}, \frac{1}{2}, -\frac{1}{2}, -\frac{1}{2}, \frac{1}{\sqrt{2}})$ & $2$ & $w_2 = s_6 w_1$ \\
        $(\frac{1}{2}, \frac{1}{2}, \frac{1}{2}, -\frac{1}{2}, \frac{1}{2}, -\frac{1}{2}, \frac{1}{\sqrt{2}})$ & $3$ & $w_3 = s_4 w_2$ \\
        $(\frac{1}{2}, \frac{1}{2}, -\frac{1}{2}, \frac{1}{2}, \frac{1}{2}, -\frac{1}{2}, \frac{1}{\sqrt{2}})$ & $4$ & $w_{4_{\Utext{I}}} = s_3 w_3$ \\
        $(\frac{1}{2}, \frac{1}{2}, \frac{1}{2}, -\frac{1}{2}, -\frac{1}{2}, \frac{1}{2}, \frac{1}{\sqrt{2}})$ & $4$ & $w_{4_{\Utext{II}}} = s_5 w_3$ \\
        $(\frac{1}{2}, -\frac{1}{2}, \frac{1}{2}, \frac{1}{2}, \frac{1}{2}, -\frac{1}{2}, \frac{1}{\sqrt{2}})$ & $5$ & $w_{5_{\Utext{I}}} = s_2 w_{4_{\Utext{I}}}$ \\
        $(\frac{1}{2}, \frac{1}{2}, -\frac{1}{2}, \frac{1}{2}, -\frac{1}{2}, \frac{1}{2}, \frac{1}{\sqrt{2}})$ & $5$ & $w_{5_{\Utext{II}}} = s_3 w_{4_{\Utext{II}}}$ \\
        $(\frac{1}{2}, -\frac{1}{2}, \frac{1}{2}, \frac{1}{2}, -\frac{1}{2}, \frac{1}{2}, \frac{1}{\sqrt{2}})$ & $6$ & $w_{6_{\Utext{I}}} = s_5 w_{5_{\Utext{I}}} = s_2 w_{5_{\Utext{II}}}$ \\
        $(\frac{1}{2}, \frac{1}{2}, -\frac{1}{2}, -\frac{1}{2}, \frac{1}{2}, \frac{1}{2}, \frac{1}{\sqrt{2}})$ & $6$ & $w_{6_{\Utext{II}}} = s_4 w_{5_{\Utext{II}}}$ \\
        $(\frac{1}{2}, -\frac{1}{2}, \frac{1}{2}, -\frac{1}{2}, \frac{1}{2}, \frac{1}{2}, \frac{1}{\sqrt{2}})$ & $7$ & $w_{7_{\Utext{I}}} = s_4 w_{6_{\Utext{I}}} = s_2 w_{6_{\Utext{II}}}$ \\
        $(\frac{1}{2}, \frac{1}{2}, -\frac{1}{2}, -\frac{1}{2}, -\frac{1}{2}, -\frac{1}{2}, \frac{1}{\sqrt{2}})$ & $7$ & $w_{7_{\Utext{II}}} = s_6 w_{6_{\Utext{II}}}$ \\
        $(\frac{1}{2}, -\frac{1}{2}, -\frac{1}{2}, \frac{1}{2}, \frac{1}{2}, \frac{1}{2}, \frac{1}{\sqrt{2}})$ & $8$ & $w_{8_{\Utext{I}}} = s_3 w_{7_{\Utext{I}}}$ \\
        $(\frac{1}{2}, -\frac{1}{2}, \frac{1}{2}, -\frac{1}{2}, -\frac{1}{2}, -\frac{1}{2}, \frac{1}{\sqrt{2}})$ & $8$ & $w_{8_{\Utext{II}}} = s_6 w_{7_{\Utext{I}}} = s_2 w_{7_{\Utext{II}}}$ \\
        $(1, 1, 0, 0, 0, 0, 0)$ & $8$ & $w_{8_{\Utext{III}}} = s_7 w_{7_{\Utext{II}}}$ \\
        $(\frac{1}{2}, -\frac{1}{2}, -\frac{1}{2}, \frac{1}{2}, -\frac{1}{2}, -\frac{1}{2}, \frac{1}{\sqrt{2}})$ & $9$ & $w_{9_{\Utext{I}}} = s_6 w_{8_{\Utext{I}}} = s_3 w_{8_{\Utext{II}}}$ \\
        $(1, 0, 1, 0, 0, 0, 0)$ & $9$ & $w_{9_{\Utext{II}}} = s_7 w_{8_{\Utext{II}}} = s_2 w_{8_{\Utext{III}}}$ \\
        $(\frac{1}{2}, -\frac{1}{2}, -\frac{1}{2}, -\frac{1}{2}, \frac{1}{2}, -\frac{1}{2}, \frac{1}{\sqrt{2}})$ & $10$ & $w_{10_{\Utext{I}}} = s_4 w_{9_{\Utext{I}}}$ \\
        $(1, 0, 0, 1, 0, 0, 0)$ & $10$ & $w_{10_{\Utext{II}}} = s_7 w_{9_{\Utext{I}}} = s_3 w_{9_{\Utext{II}}}$ \\
        $(\frac{1}{2}, -\frac{1}{2}, -\frac{1}{2}, -\frac{1}{2}, -\frac{1}{2}, \frac{1}{2}, \frac{1}{\sqrt{2}})$ & $11$ & $w_{11_{\Utext{I}}} = s_5 w_{10_{\Utext{I}}}$ \\
        $(1, 0, 0, 0, 1, 0, 0)$ & $11$ & $w_{11_{\Utext{II}}} = s_7 w_{10_{\Utext{I}}} = s_4 w_{10_{\Utext{II}}}$ \\
        $(1, 0, 0, 0, 0, 1, 0)$ & $12$ & $w_{12_{\Utext{I}}} = s_7 w_{11_{\Utext{I}}}$ \\
        $(1, 0, 0, 0, 0, -1, 0)$ & $12$ & $w_{12_{\Utext{II}}} = s_6 w_{11_{\Utext{II}}}$ \\
        $(1, 0, 0, 0, -1, 0, 0)$ & $13$ & $w_{13} = s_6 w_{12_{\Utext{I}}} = s_5 w_{12_{\Utext{II}}}$ \\
        $(1, 0, 0, -1, 0, 0, 0)$ & $14$ & $w_{14} = s_4 w_{13}$ \\
        $(1, 0, -1, 0, 0, 0, 0)$ & $15$ & $w_{15} = s_3 w_{14}$ \\
        $(1, -1, 0, 0, 0, 0, 0)$ & $16$ & $w_{16} = s_2 w_{15}$ \\
        \hline
    \end{tabular}
\end{table}

\begin{table}[ht!]
    \tiny
    \caption{$\{ w \varpi_1 \}_{w \in \WG_{\mathfrak{g}_\pl}^{\typeE_7 / \typeE_6}}$ in the case of type $\typeE_7$}\label{tab-ex-e-7-alpha-1}
    \begin{tabular}{|l|l|l|}
        \hline
        $(w \varpi_1)^T \text{(transposed)}$ & $l(w)$ & $w \in \WG_{\mathfrak{g}_\pl}^{\typeE_7 / \typeE_6}$ \\
        \hline
        $(1, 0, 0, 0, 0, 0, \frac{1}{\sqrt{2}})$ & $0$ & $1$ \\
        $(0, 1, 0, 0, 0, 0, \frac{1}{\sqrt{2}})$ & $1$ & $w_1 = s_1$ \\
        $(0, 0, 1, 0, 0, 0, \frac{1}{\sqrt{2}})$ & $2$ & $w_2 = s_2 w_1$ \\
        $(0, 0, 0, 1, 0, 0, \frac{1}{\sqrt{2}})$ & $3$ & $w_3 = s_3 w_2$ \\
        $(0, 0, 0, 0, 1, 0, \frac{1}{\sqrt{2}})$ & $4$ & $w_4 = s_4 w_3$ \\
        $(0, 0, 0, 0, 0, 1, \frac{1}{\sqrt{2}})$ & $5$ & $w_{5_{\Utext{I}}} = s_5 w_4$ \\
        $(0, 0, 0, 0, 0, -1, \frac{1}{\sqrt{2}})$ & $5$ & $w_{5_{\Utext{II}}} = s_6 w_4$ \\
        $(0, 0, 0, 0, -1, 0, \frac{1}{\sqrt{2}})$ & $6$ & $w_{6_{\Utext{I}}} = s_6 w_{5_{\Utext{I}}}$ \\
        $(\frac{1}{2}, \frac{1}{2}, \frac{1}{2}, \frac{1}{2}, \frac{1}{2}, -\frac{1}{2}, 0)$ & $6$ & $w_{6_{\Utext{II}}} = s_7 w_{5_{\Utext{II}}}$ \\
        $(0, 0, 0, -1, 0, 0, \frac{1}{\sqrt{2}})$ & $7$ & $w_{7_{\Utext{I}}} = s_4 w_{6_{\Utext{I}}}$ \\
        $(\frac{1}{2}, \frac{1}{2}, \frac{1}{2}, \frac{1}{2}, -\frac{1}{2}, \frac{1}{2}, 0)$ & $7$ & $w_{7_{\Utext{II}}} = s_5 w_{6_{\Utext{II}}}$ \\
        $(0, 0, -1, 0, 0, 0, \frac{1}{\sqrt{2}})$ & $8$ & $w_{8_{\Utext{I}}} = s_3 w_{7_{\Utext{I}}}$ \\
        $(\frac{1}{2}, \frac{1}{2}, \frac{1}{2}, -\frac{1}{2}, \frac{1}{2}, \frac{1}{2}, 0)$ & $8$ & $w_{8_{\Utext{II}}} = s_4 w_{7_{\Utext{II}}}$ \\
        $(0, -1, 0, 0, 0, 0, \frac{1}{\sqrt{2}})$ & $9$ & $w_{9_{\Utext{I}}} = s_2 w_{8_{\Utext{I}}}$ \\
        $(\frac{1}{2}, \frac{1}{2}, -\frac{1}{2}, \frac{1}{2}, \frac{1}{2}, \frac{1}{2}, 0)$ & $9$ & $w_{9_{\Utext{II}}} = s_3 w_{8_{\Utext{II}}}$ \\
        $(\frac{1}{2}, \frac{1}{2}, \frac{1}{2}, -\frac{1}{2}, -\frac{1}{2}, -\frac{1}{2}, 0)$ & $9$ & $w_{9_{\Utext{III}}} = s_6 w_{8_{\Utext{II}}}$ \\
        $(-1, 0, 0, 0, 0, 0, \frac{1}{\sqrt{2}})$ & $10$ & $w_{10_{\Utext{I}}} = s_1 w_{9_{\Utext{I}}}$ \\
        $(\frac{1}{2}, -\frac{1}{2}, \frac{1}{2}, \frac{1}{2}, \frac{1}{2}, \frac{1}{2}, 0)$ & $10$ & $w_{10_{\Utext{II}}} = s_7 w_{9_{\Utext{I}}} = s_2 w_{9_{\Utext{II}}}$ \\
        $(\frac{1}{2}, \frac{1}{2}, -\frac{1}{2}, \frac{1}{2}, -\frac{1}{2}, -\frac{1}{2}, 0)$ & $10$ & $w_{10_{\Utext{III}}} = s_3 w_{9_{\Utext{III}}}$ \\
        $(-\frac{1}{2}, \frac{1}{2}, \frac{1}{2}, \frac{1}{2}, \frac{1}{2}, \frac{1}{2}, 0)$ & $11$ & $w_{11_{\Utext{I}}} = s_7 w_{10_{\Utext{I}}} = s_1 w_{10_{\Utext{II}}}$ \\
        $(\frac{1}{2}, -\frac{1}{2}, \frac{1}{2}, \frac{1}{2}, -\frac{1}{2}, -\frac{1}{2}, 0)$ & $11$ & $w_{11_{\Utext{II}}} = s_6 w_{10_{\Utext{II}}} = s_2 w_{10_{\Utext{III}}}$ \\
        $(\frac{1}{2}, \frac{1}{2}, -\frac{1}{2}, -\frac{1}{2}, \frac{1}{2}, -\frac{1}{2}, 0)$ & $11$ & $w_{11_{\Utext{III}}} = s_4 w_{10_{\Utext{III}}}$ \\
        $(-\frac{1}{2}, \frac{1}{2}, \frac{1}{2}, \frac{1}{2}, -\frac{1}{2}, -\frac{1}{2}, 0)$ & $12$ & $w_{12_{\Utext{I}}} = s_6 w_{11_{\Utext{I}}} = s_1 w_{11_{\Utext{II}}}$ \\
        $(\frac{1}{2}, -\frac{1}{2}, \frac{1}{2}, -\frac{1}{2}, \frac{1}{2}, -\frac{1}{2}, 0)$ & $12$ & $w_{12_{\Utext{II}}} = s_4 w_{11_{\Utext{II}}} = s_2 w_{11_{\Utext{III}}}$ \\
        $(\frac{1}{2}, \frac{1}{2}, -\frac{1}{2}, -\frac{1}{2}, -\frac{1}{2}, \frac{1}{2}, 0)$ & $12$ & $w_{12_{\Utext{III}}} = s_5 w_{11_{\Utext{III}}}$ \\
        $(-\frac{1}{2}, \frac{1}{2}, \frac{1}{2}, -\frac{1}{2}, \frac{1}{2}, -\frac{1}{2}, 0)$ & $13$ & $w_{13_{\Utext{I}}} = s_4 w_{12_{\Utext{I}}} = s_1 w_{12_{\Utext{II}}}$ \\
        $(\frac{1}{2}, -\frac{1}{2}, -\frac{1}{2}, \frac{1}{2}, \frac{1}{2}, -\frac{1}{2}, 0)$ & $13$ & $w_{13_{\Utext{II}}} = s_3 w_{12_{\Utext{II}}}$ \\
        $(\frac{1}{2}, -\frac{1}{2}, \frac{1}{2}, -\frac{1}{2}, -\frac{1}{2}, \frac{1}{2}, 0)$ & $13$ & $w_{13_{\Utext{III}}} = s_5 w_{12_{\Utext{II}}} = s_2 w_{12_{\Utext{III}}}$ \\
        $(-\frac{1}{2}, \frac{1}{2}, -\frac{1}{2}, \frac{1}{2}, \frac{1}{2}, -\frac{1}{2}, 0)$ & $14$ & $w_{14_{\Utext{I}}} = s_3 w_{13_{\Utext{I}}} = s_1 w_{13_{\Utext{II}}}$ \\
        $(\frac{1}{2}, -\frac{1}{2}, -\frac{1}{2}, \frac{1}{2}, -\frac{1}{2}, \frac{1}{2}, 0)$ & $14$ & $w_{14_{\Utext{II}}} = s_5 w_{13_{\Utext{II}}} = s_3 w_{13_{\Utext{III}}}$ \\
        $(-\frac{1}{2}, \frac{1}{2}, \frac{1}{2}, -\frac{1}{2}, -\frac{1}{2}, \frac{1}{2}, 0)$ & $14$ & $w_{14_{\Utext{III}}} = s_1 w_{13_{\Utext{III}}}$ \\
        $(-\frac{1}{2}, -\frac{1}{2}, \frac{1}{2}, \frac{1}{2}, \frac{1}{2}, -\frac{1}{2}, 0)$ & $15$ & $w_{15_{\Utext{I}}} = s_2 w_{14_{\Utext{I}}}$ \\
        $(\frac{1}{2}, -\frac{1}{2}, -\frac{1}{2}, -\frac{1}{2}, \frac{1}{2}, \frac{1}{2}, 0)$ & $15$ & $w_{15_{\Utext{II}}} = s_4 w_{14_{\Utext{II}}}$ \\
        $(-\frac{1}{2}, \frac{1}{2}, -\frac{1}{2}, \frac{1}{2}, -\frac{1}{2}, \frac{1}{2}, 0)$ & $15$ & $w_{15_{\Utext{III}}} = s_1 w_{14_{\Utext{II}}} = s_3 w_{14_{\Utext{III}}}$ \\
        $(-\frac{1}{2}, -\frac{1}{2}, \frac{1}{2}, \frac{1}{2}, -\frac{1}{2}, \frac{1}{2}, 0)$ & $16$ & $w_{16_{\Utext{I}}} = s_5 w_{15_{\Utext{I}}} = s_2 w_{15_{\Utext{III}}}$ \\
        $(-\frac{1}{2}, \frac{1}{2}, -\frac{1}{2}, -\frac{1}{2}, \frac{1}{2}, \frac{1}{2}, 0)$ & $16$ & $w_{16_{\Utext{II}}} = s_1 w_{15_{\Utext{II}}}$ \\
        $(\frac{1}{2}, -\frac{1}{2}, -\frac{1}{2}, -\frac{1}{2}, -\frac{1}{2}, -\frac{1}{2}, 0)$ & $16$ & $w_{16_{\Utext{III}}} = s_6 w_{15_{\Utext{II}}}$ \\
        $(-\frac{1}{2}, -\frac{1}{2}, \frac{1}{2}, -\frac{1}{2}, \frac{1}{2}, \frac{1}{2}, 0)$ & $17$ & $w_{17_{\Utext{I}}} = s_4 w_{16_{\Utext{I}}}$ \\
        $(-\frac{1}{2}, \frac{1}{2}, -\frac{1}{2}, -\frac{1}{2}, -\frac{1}{2}, -\frac{1}{2}, 0)$ & $17$ & $w_{17_{\Utext{II}}} = s_6 w_{16_{\Utext{II}}}$ \\
        $(1, 0, 0, 0, 0, 0, -\frac{1}{\sqrt{2}})$ & $17$ & $w_{17_{\Utext{III}}} = s_7 w_{16_{\Utext{III}}}$ \\
        $(-\frac{1}{2}, -\frac{1}{2}, -\frac{1}{2}, \frac{1}{2}, \frac{1}{2}, \frac{1}{2}, 0)$ & $18$ & $w_{18_{\Utext{I}}} = s_3 w_{17_{\Utext{I}}}$ \\
        $(-\frac{1}{2}, -\frac{1}{2}, \frac{1}{2}, -\frac{1}{2}, -\frac{1}{2}, -\frac{1}{2}, 0)$ & $18$ & $w_{18_{\Utext{II}}} = s_2 w_{17_{\Utext{II}}}$ \\
        $(0, 1, 0, 0, 0, 0, -\frac{1}{\sqrt{2}})$ & $18$ & $w_{18_{\Utext{III}}} = s_1 w_{17_{\Utext{III}}}$ \\
        $(-\frac{1}{2}, -\frac{1}{2}, -\frac{1}{2}, \frac{1}{2}, -\frac{1}{2}, -\frac{1}{2}, 0)$ & $19$ & $w_{19_{\Utext{I}}} = s_6 w_{18_{\Utext{I}}} = s_3 w_{18_{\Utext{II}}}$ \\
        $(0, 0, 1, 0, 0, 0, -\frac{1}{\sqrt{2}})$ & $19$ & $w_{19_{\Utext{II}}} = s_7 w_{18_{\Utext{II}}} = s_2 w_{18_{\Utext{III}}}$ \\
        $(0, 0, 0, 1, 0, 0, -\frac{1}{\sqrt{2}})$ & $20$ & $w_{20_{\Utext{I}}} = s_7 w_{19_{\Utext{I}}} = s_3 w_{19_{\Utext{II}}}$ \\
        $(-\frac{1}{2}, -\frac{1}{2}, -\frac{1}{2}, -\frac{1}{2}, \frac{1}{2}, -\frac{1}{2}, 0)$ & $20$ & $w_{20_{\Utext{II}}} = s_4 w_{19_{\Utext{I}}}$ \\
        $(0, 0, 0, 0, 1, 0, -\frac{1}{\sqrt{2}})$ & $21$ & $w_{21_{\Utext{I}}} = s_4 w_{20_{\Utext{I}}}$ \\
        $(-\frac{1}{2}, -\frac{1}{2}, -\frac{1}{2}, -\frac{1}{2}, -\frac{1}{2}, \frac{1}{2}, 0)$ & $21$ & $w_{21_{\Utext{II}}} = s_5 w_{20_{\Utext{II}}}$ \\
        $(0, 0, 0, 0, 0, -1, -\frac{1}{\sqrt{2}})$ & $22$ & $w_{22_{\Utext{I}}} = s_6 w_{21_{\Utext{I}}}$ \\
        $(0, 0, 0, 0, 0, 1, -\frac{1}{\sqrt{2}})$ & $22$ & $w_{22_{\Utext{II}}} = s_7 w_{21_{\Utext{II}}}$ \\
        $(0, 0, 0, 0, -1, 0, -\frac{1}{\sqrt{2}})$ & $23$ & $w_{23_{\Utext{I}}} = s_5 w_{22_{\Utext{I}}} = s_6 w_{22_{\Utext{II}}}$ \\
        $(0, 0, 0, -1, 0, 0, -\frac{1}{\sqrt{2}})$ & $24$ & $w_{24_{\Utext{I}}} = s_4 w_{23_{\Utext{I}}}$ \\
        $(0, 0, -1, 0, 0, 0, -\frac{1}{\sqrt{2}})$ & $25$ & $w_{25_{\Utext{I}}} = s_3 w_{24_{\Utext{I}}}$ \\
        $(0, -1, 0, 0, 0, 0, -\frac{1}{\sqrt{2}})$ & $26$ & $w_{26_{\Utext{I}}} = s_2 w_{25_{\Utext{I}}}$ \\
        $(-1, 0, 0, 0, 0, 0, -\frac{1}{\sqrt{2}})$ & $27$ & $w_{27_{\Utext{I}}} = s_1 w_{26_{\Utext{I}}}$ \\
        \hline
    \end{tabular}
\end{table}

As for the \Qtn{if} part, we shall apply Lemma \ref{lem-key-obs-factor-if} by verifying the two conditions \Refenum{\ref{lem-key-obs-factor-if-trans}} and \Refenum{\ref{lem-key-obs-factor-if-base}} there.  The roots of $\mathfrak{n}_\pl^\std$ \Pth{\Refcf{} \cite[(3.92)]{Lan:2016-vtcac}} is a union of three subsets
\[
\begin{split}
    & \{ e_1 \pm e_j : j = 2, 3, 4, 5, 6 \} \\
    & \cup \{ \text{$(\tfrac{1}{2}, \pm \tfrac{1}{2}, \pm \tfrac{1}{2}, \pm \tfrac{1}{2}, \pm \tfrac{1}{2}, \pm \tfrac{1}{2}, \tfrac{\sqrt{2}}{2})^T$ \Pth{transposed} with an even number of $\tfrac{1}{2}$'s} \} \\
    & \cup \{ (0, 0, 0, 0, 0, 0, \sqrt{2})^T \},
\end{split}
\]
each of which admits a transitive action of $\WG_{\mathfrak{g}_\pl, \typeD_5} \cong \SG_5 \ltimes \{ \pm 1 \}^{5, +}$ \Pth{acting on the middle five coordinates}.  In order to verify \Refenum{\ref{lem-key-obs-factor-if-trans}}, it suffices to show that some elements of $\WG_{\mathfrak{m}_\pl}$ bring, for example, $e_1 + e_2$ \Pth{in the first subset} and $(0, 0, 0, 0, 0, 0, \sqrt{2})^T$ \Pth{in the third subset} to $(\frac{1}{2}, \frac{1}{2}, \frac{1}{2}, \frac{1}{2}, \frac{1}{2}, \frac{1}{2}, \frac{\sqrt{2}}{2})^T$ \Pth{in the second subset}, respectively.  For the former, the reflection with respect to the root $(-\frac{1}{2}, -\frac{1}{2}, \frac{1}{2}, \frac{1}{2}, \frac{1}{2}, \frac{1}{2}, \frac{\sqrt{2}}{2})^T$ of $\mathfrak{m}_\pl$ \Pth{\Refcf{} \cite[(3.90)]{Lan:2016-vtcac}} brings $e_1 + e_2$ to $(\frac{1}{2}, \frac{1}{2}, \frac{1}{2}, \frac{1}{2}, \frac{1}{2}, \frac{1}{2}, \frac{\sqrt{2}}{2})^T$.  For the latter, the reflection with respect to the root $(-\frac{1}{2}, -\frac{1}{2}, -\frac{1}{2}, -\frac{1}{2}, -\frac{1}{2}, -\frac{1}{2}, \frac{\sqrt{2}}{2})^T$ of $\mathfrak{m}_\pl$ \Pth{\Refcf{} \cite[(3.90)]{Lan:2016-vtcac} again} brings $(0, 0, 0, 0, 0, 0, \sqrt{2})^T$ to $(\frac{1}{2}, \frac{1}{2}, \frac{1}{2}, \frac{1}{2}, \frac{1}{2}, \frac{1}{2}, \frac{\sqrt{2}}{2})^T$.  As for \Refenum{\ref{lem-key-obs-factor-if-base}}, by taking $\wt_{0, \pl}'' = (0, 1, 0, 0, 0, 0, \frac{1}{\sqrt{2}})^T$ \Pth{see the second row of Table \ref{tab-ex-e-7-alpha-1}}, we have $\wt_{0, \pl} - \wt_{0, \pl}'' = e_1 - e_2$, which is a root of $\mathfrak{n}_\pl^\std$.  Thus, the \Qtn{if} part has also been verified.

\section{Vanishing results for completed cohomology}\label{sec-van-res}

\subsection{Completed cohomology}\label{sec-cmpl-coh}

The $p$-adically completed cohomology for locally symmetric spaces was introduced by Emerton to serve as a $p$-adic analogue of spaces of automorphic forms.  We refer to \cite{Emerton:2014-ccplg, Calegari/Emerton:2012-ccs} for more details.  In this paper, we will only consider the case of Shimura varieties.

\begin{definition}\label{def-cmpl-coh}
    Let $(\Grp{G}, \Shdom)$ be a Shimura datum.  Fix a tame level, \ie, an open compact subgroup $\levcp^p$ of $\Grp{G}(\bAip)$.  The \Pth{$i$-th} completed cohomology at tame level $\levcp^p$ is defined as
    \[
        \tilde{H}^i := \varprojlim_n \varinjlim_{\levcp_p} H^i(\Sh_{\levcp^p \levcp_p, \bC}^\an, \bZ / p^n),
    \]
    where $\levcp_p$ runs through all open compact subgroups of $\Grp{G}(\bQ_p)$ such that $\levcp^p \levcp_p$ is neat.  \Pth{See Section \ref{sec-Sh-var} for the notation here.}
\end{definition}

The completed cohomology $\tilde{H}^i$ is $p$-adically complete and admits a natural admissible representation of $\Grp{G}(\bQ_p)$.  In particular, $\tilde{H}^i \otimes_\bZ \bQ$ is an admissible $p$-adic Banach space representation of $\Grp{G}(\bQ_p)$.  We denote by $\tilde{H}^{i, \la}$ the subspace of $\Grp{G}(\bQ_p)$-locally analytic vectors, which is a $\Grp{G}(\bQ_p)$-invariant dense subspace, by \cite[\aThm 7.1]{Schneider/Teitelbaum:2003-apdar}.  The Lie algebra $\Lie \Grp{G}_{\bQ_p}$ of $\Grp{G}(\bQ_p)$ acts naturally on it by derivation.  We put $\tilde{H}_C^i := \tilde{H}^i \ho_{\bZ_p} C$, the $p$-adically completed tensor product, and denote by $\tilde{H}_C^{i, \la}$ its subspace of $\Grp{G}(\bQ_p)$-locally analytic vectors.  Equivalently,
\[
    \tilde{H}_C^{i, \la} = \tilde{H}^{i, \la} \ho_{\bQ_p} C,
\]
the completed tensor product for LB-spaces, as in \cite[\aProp 1.1.32]{Emerton:2017-LAR}.  Note that, after fixing an isomorphism $\iota: \bC \Mi C$, by \cite[\aThm 6.2.6]{RodriguezCamargo:2022-lacc} and the very construction of $\cO^\la = R \pi_{\HT, *}^\la(\cO_\aSh^\la)$ in \Refeq{\ref{eq-thm-geom-summary}}, we have $\mathfrak{g}$-equivariant isomorphisms
\begin{equation}\label{eq-RC-prim-comp}
    \tilde{H}_C^{i, \la} \cong H^i(\aSh_{\levcp^p}^\Tor, \cO_\aSh^\la) \cong H^i(\Fl, \cO^\la).
\end{equation}
\Pth{Here we freely use the notation introduced in Section \ref{sec-geom-Sen-Sh}.}  As explained in Section \ref{sec-hor-act}, there is a natural action of $Z(U(\mathfrak{m}))$ on $\cO^\la$ and hence \Pth{by \Refeq{\ref{eq-RC-prim-comp}}} also on $\tilde{H}_C^{i, \la}$, extending the natural action of $Z(U(\mathfrak{g}))$.  Recall that, in Definitions \ref{def-suff-reg-wt} and \ref{def-suff-reg-wt-id}, we introduced the notion of $\hc_\hd^\iota$-sufficiently regular infinitesimal characters and ideals $\cI^\iota \subset Z(U(\mathfrak{g}))$ and $\cI_{\mathfrak{m}}^\iota \subset Z(U(\mathfrak{m}))$.

\begin{theorem}\label{thm-main}
    Let $d = \dim_\bC(\Shdom)$ and $\iota: \bC \Mi C$ an isomorphism.  Then $(\cI^\iota)^n$ and $(\cI_{\mathfrak{m}}^\iota)^n$ annihilate $\tilde{H}_C^{< d, \la}$, for some $n > 0$.  In particular, if $[\wt]: Z(U(\mathfrak{g})) \to C$ \Pth{\resp $[\wt]_{\mathfrak{m}}: Z(U(\mathfrak{m})) \to C$} is $\hc_\hd^\iota$-sufficiently regular, then the $[\wt]$- \Pth{\resp $[\wt]_{\mathfrak{m}}$-} isotypic part $\tilde{H}_{C, [\wt]}^{< d, \la}$ \Pth{\resp $\tilde{H}_{C, [\wt]_{\mathfrak{m}}}^{< d, \la}$} of $\tilde{H}_C^{< d, \la}$ is zero.
\end{theorem}

Recall that, in Definitions \ref{def-suff-reg-wt} and \ref{def-suff-reg-wt-id}, we also introduce the notion of $\hc_\hd$-sufficiently regular infinitesimal characters and ideal $\cI$ of $Z(U(\Lie\Grp{G}))$.  Note that $Z(U(\Lie \Grp{G})) \subset Z(U(\Lie \Grp{G}_{\bQ_p}))$ acts naturally on $\tilde{H}^{i, \la}$.

\begin{corollary}\label{cor-thm-main}
    The ideal $\cI^n$ annihilates $\tilde{H}^{< d, \la}$, for some $n > 0$.  In particular, if $[\wt]: Z(U(\Lie \Grp{G})) \to \bQ_p$ is $\hc_\hd$-sufficiently regular, then $\tilde{H}_{[\wt]}^{< d, \la} = 0$.
\end{corollary}
\begin{proof}
    By Proposition \ref{prop-suff-reg-m-vs-g}\Refenum{\ref{prop-suff-reg-m-vs-g-3}}, we can find isomorphisms $\iota_1, \cdots, \iota_l: \bC \Mi C$ such that the radical ideal of $\sum_{i = 1}^l \cI^{\iota_i}$ contains $\cI$.  By Theorem \ref{thm-main}, there exists $N > 0$ such that $(\cI^{\iota_i})^N$ annihilates $\tilde{H}_C^{< d, \la}$, for all $i = 1, \cdots, l$.  Hence, a power of $\cI$ annihilates it as well.
\end{proof}

\begin{remark}\label{rem-thm-main-spec}
    When Condition \ref{cond-der-sc-no-need-c} holds, Theorem \ref{thm-main} follows from Theorem \ref{thm-sc-van-m} and Corollary \ref{cor-sc-van-g}, which we have already proved in Section \ref{sec-sc-van-res}.  In the next subsection, we will explain how to reduce the general case to this special case.
\end{remark}

\begin{remark}\label{rem-LS-van-reprove}
    Let $\levcp_p$ be an open compact subgroups of $\Grp{G}(\bQ_p)$ such that $\levcp^p \levcp_p$ is neat.  Suppose that $\rV$ is a finite-dimensional representation of $\Grp{G}^c$ over $\bC$.  Denote by $\GrSh{\rV_\bC}$ the corresponding $\bC$-local system on $\Sh_{\levcp^p \levcp_p, \bC}^\an$.  Our main result implies that, if $\rV_C$ has $\hc_\hd$-sufficiently regular infinitesimal weights for an isomorphism $\bC\cong C$, then
    \[
        H^{< d}(\Sh_{\levcp^p \levcp_p, \bC}^\an, \GrSh{\rV_\bC}) = 0.
    \]
    This vanishing result was previously obtained by the first-named author \cite[\aThm 4.10]{Lan:2016-vtcac} by a different method.  To see the implication, by using an isomorphism $\iota: \bC \Mi C$ over $\ReFl$, we may assume that $\rV$ is defined over a finite extension $L$ of $\bQ_p$ and $\rV_C$ has $\hc_\hd^\iota$-sufficiently regular infinitesimal character.  By the comparison theorems for various \'etale cohomology, it suffices to prove the analogue for the \'etale local system $\etSh{\rV_L}$ on $\Sh_{\levcp^p \levcp_p, C}$.  By \cite[\aThm 0.5]{Emerton:2014-ccplg}, there is a spectral sequence
    \[
        E_2^{i, j} =  \Ext_{\Lie \widetilde{\levcp_p}}^i(\dual{\rV}_L, \tilde{H}^{j, \la}) \Rightarrow H_\et^{i + j}(\etSh{\rV_L}),
    \]
    where $\dual{\rV}_L$ denotes the dual of $\rV_L$, and $H_\et^{i + j}(\etSh{\rV_L}) := \varinjlim_{\levcp_p} H_\et^{i + j}(\Sh_{\levcp^p \levcp_p, C}, \etSh{\rV_L})$.  Note that there is a natural smooth action of $\levcp_p$ on $H_\et^{i + j}(\etSh{\rV_L})$, and $H_\et^{i + j}(\etSh{\rV_L})^{\levcp_p} \cong H_\et^{i + j}(\Sh_{\levcp^p \levcp_p, C}, \etSh{\rV_L})$. By considering the action of the center $Z(U(\Lie\Grp{G}_{\bQ_p}))$, we see that our vanishing result gives $E_2^{i, < d} = 0$ \Pth{see, \eg, the proof of \cite[\aThm 1.1]{Milicic:2014-vocot}}, and hence $H_\et^{< d}(\Sh_{\levcp^p \levcp_p, C}, \etSh{\rV_L}) = 0$.
\end{remark}

A similar argument as in Remark \ref{rem-LS-van-reprove} works for more general pro-\'etale $\bZ_p$-local systems arising from a locally analytic representation of $\widetilde{\levcp_p}$ \Pth{not necessarily of finite rank}.  Let $\Sh_{\levcp^p \levcp_p}^\Tor$ be a toroidal compactification of $\Sh_{\levcp^p \levcp_p}$, as before, and consider the tower $\varprojlim_{\levcp_p' \subset \levcp_p} \aSh_{\levcp^p \levcp_p'}^\Tor$ defining a pro-Kummer-\'etale $\widetilde{\levcp_p}$-torsor over $\aSh_{\levcp^p \levcp_p}^\Tor$, as in the beginning of Section \ref{sec-HT-mor}.  \Pth{So, implicitly, we fix an isomorphism $\iota: \bC \Mi C$.}  Recall that, in Construction \ref{constr-proket-ls}, for each $p$-adic unitary Banach space representation $\rV$ of $\widetilde{\levcp_p}$, we defined a pro-Kummer-\'etale sheaf $\proketSh{\rV} = \proketSh{\rV}_{\aSh_{\levcp^p \levcp_p}^\Tor}$ functorial in $\rV$.  \Pth{For simplicity, we shall similarly omit the subscript \Qtn{$\aSh_{\levcp^p \levcp_p}^\Tor$} in the notation for the similarly constructed sheaves below.}  When $\rV$ is a general $p$-adic Banach space representation, one can choose a $\widetilde{\levcp_p}$-invariant open bounded lattice of $\rV$ and apply the same construction.  Then it is standard that the resulting pro-Kummer-\'etale sheaf is independent of the choice of the lattice.

\begin{corollary}\label{cor-van-la-sys}
    Suppose that $\rV$ is a locally analytic \Pth{$p$-adic} Banach space representation of $\widetilde{\levcp_p}$.  The center $Z(U(\Lie \Grp{G}_{\bQ_p}))$ acts naturally on $\rV$, $\proketSh{\rV}$, and their cohomology.  The ideal $\cI^n$ annihilates $H_\proket^{< d}(\aSh_{\levcp^p \levcp_p}^\Tor, \proketSh{\rV})$, for some $n>0$.  In particular, $H_\proket^{< d}(\aSh_{\levcp^p \levcp_p}^\Tor, \proketSh{\rV}) = 0$ if $\rV$ has a $\hc_\hd$-sufficiently regular infinitesimal character.
\end{corollary}
\begin{proof}
    One may shrink $\levcp_p$ and assume that $\widetilde{\levcp_p}$ is uniform pro-$p$.  \Pth{See the discussion after Lemma \ref{lem-fin-cov-coh-comp} below.}  For $r \in (1 / p, 1) \cap p^\bQ$, we can consider the space $C^{(r)} \subset \sC^\la(\widetilde{\levcp_p}, \bQ_p)$ of $r$-analytic functions on $\widetilde{\levcp_p}$, which is dual to the distribution algebra $D_{< r}(\widetilde{\levcp_p}, \bQ_p)$ considered in \cite{Schneider/Teitelbaum:2003-apdar}.  See also \cite[\aDef IV.1]{Colmez/Dospinescu:2014-curgl}, where it was denoted by $\mathrm{LA}^{(h)}$.  Via the left translation action, $C^{(r)}$ is a locally analytic unitary Banach space representation of $\widetilde{\levcp_p}$.  We have the following:
    \begin{lemma}\label{lem-cmpl-coh-la-sys}
        There is a natural isomorphism
        \[
            H_\proket^i(\aSh_{\levcp^p \levcp_p}^\Tor, \proketSh{C^{(r)}}) \cong (\tilde{H}^i \ho_{\bZ_p} C^{(r)})^{\widetilde{\levcp_p}} \subset (\tilde{H}^i \ho_{\bZ_p} \sC^\la(\widetilde{\levcp_p}, \bQ_p))^{\widetilde{\levcp_p}} \cong \tilde{H}^{i, \la}.
        \]
        Therefore, $\cI^n$ annihilates $H_\proket^{< d}(\aSh_{\levcp^p \levcp_p}^\Tor, \proketSh{C^{(r)}})$, for some $n > 0$.
    \end{lemma}
    \begin{proof}[Proof of Lemma \ref{lem-cmpl-coh-la-sys}]
        For each $n > 0$, the $\widetilde{\levcp_p}$-representation $\rV_n := C^{(r), \circ} / p^n$ defines an \'etale sheaf $\etSh{\rV_n}$ on $\Sh_{\levcp^p \levcp_p, C}$.  The same argument as in \cite[(2.1.10)]{Emerton:2006-iseah} gives a spectral sequence
        \[
            E_2^{i, j} = H^i(\widetilde{\levcp_p}, \tilde{H}^j \otimes_{\bZ_p} \rV_n) \Rightarrow H^{i + j}(\Sh_{\levcp^p \levcp_p, C}, \etSh{\rV_n}).
        \]
        A result of Schneider--Teitelbaum's implies that $H^{> 0}(\widetilde{\levcp_p}, \tilde{H}^j \ho_{\bZ_p} C^{(r)}) = 0$, as $\tilde{H}^j$ is admissible \Pth{see \cite[\aThm 2.2.3]{Pan:2022-laccm}}.  A standard argument using the open mapping theorem \Pth{see \cite[\aProp 2.3.3]{Pan:2022-laccm}} shows that $H^{> 0}(\widetilde{\levcp_p}, \tilde{H}^j \ho_{\bZ_p} C^{(r), \circ})$ is annihilated by $p^N$, for some $N > 0$, and hence
        \[
            p^N H^{> 0}(\widetilde{\levcp_p}, \tilde{H}^j \otimes_{\bZ_p} \rV_n) = 0,
        \]
        for all $n \geq 0$.  Up to enlarging $N$ if necessary, we see that $p^N$ kills the kernel and cokernel of the natural map
        \[
            (\tilde{H}^j \otimes_{\bZ_p} \rV_n)^{\widetilde{\levcp_p}} \to H_\et^j(\Sh_{\levcp^p \levcp_p, C}, \etSh{\rV_n}).
        \]
        Moreover, $p^N$ kills the cokernel of $(\tilde{H}^j \otimes_{\bZ_p} \rV_{n + 1})^{\widetilde{\levcp_p}} \to (\tilde{H}^j \otimes_{\bZ_p} \rV_n)^{\widetilde{\levcp_p}}$ \Pth{by considering the exact sequence $0 \to \rV_1 \Mapn{\times p^n} \rV_{n + 1} \to \rV_n \to 0$}, and so it also kills $R^1\lim_n (\tilde{H}^j \otimes_{\bZ_p} \rV_n)^{\widetilde{\levcp_p}}$.  Therefore, we have
        \[
            (\tilde{H}^j \ho_{\bZ_p} C^{(r)})^{\widetilde{\levcp_p}} \cong \bigl(\lim_n (\tilde{H}^j \otimes_{\bZ_p} \rV_n)^{\widetilde{\levcp_p}}\bigr) \otimes_{\bZ_p} \bQ_p \cong \bigl(R\lim_n (\tilde{H}^j \otimes_{\bZ_p} \rV_n)^{\widetilde{\levcp_p}}\bigr) \otimes_{\bZ_p} \bQ_p.
        \]
        Thus, we see that $(\tilde{H}^j \ho_{\bZ_p} C^{(r)})^{\widetilde{\levcp_p}} \cong \bigl(R\lim_n H_\et^j(\Sh_{\levcp^p \levcp_p, C}, \etSh{\rV_n})\bigr) \otimes_{\bZ_p} \bQ_p$.  Finally, by Huber's comparison theorem \cite[\aThm 3.8.1]{Huber:1996-ERA}, by Proposition \ref{prop-ket}\Refenum{\ref{prop-ket-1}}, and by the construction of $\proketSh{C^{(r)}}$, we see that
        \[
            \bigl(R\lim_n H_\et^j(\Sh_{\levcp^p \levcp_p, C}, \etSh{\rV_n})\bigr) \otimes_{\bZ_p} \bQ_p \cong H_\proket^i(\aSh_{\levcp^p \levcp_p}^\Tor, \proketSh{C^{(r)}}).
        \]
        The lemma now follows.
    \end{proof}

    Back to the proof of Corollary \ref{cor-van-la-sys}, for a general Banach locally analytic representation $\rV$, we shall proceed by induction on $i$.  For any such $\rV$, there exists some $r$ such that, for any $v \in \rV$, the orbit map $o_v: g \in \widetilde{\levcp_p} \mapsto g(v)$ is a function in $C^{(r)} \ho_{\bQ_p} \rV$.  The map $\rV \to C^{(r)} \ho_{\bQ_p} \rV$ sending $v$ to $o_v$ is a closed embedding, and is $\widetilde{\levcp_p}$-equivariant with respect to the diagonal action of $\widetilde{\levcp_p}$ on $C^{(r)} \ho_{\bQ_p} \rV$.  We shall denote the cokernel by $N$, and consider the long exact sequence \Pth{for simplicity, we shall drop all the $\aSh_{\levcp^p \levcp_p}^\Tor$ and the subscripts \Qtn{$\proket$}}
    \[
        \cdots \to H^{i - 1}(\proketSh{N}) \to H^i(\proketSh{C^{(r)}}) \to H^i(\proketSh{C^{(r)} \ho_{\bQ_p} \rV}) \to \cdots.
    \]
    By our choice of $r$, there is an isomorphism $C^{(r)} \ho_{\bQ_p} \rV \cong C^{(r)} \ho_{\bQ_p} \rV$, where $\widetilde{\levcp_p}$ acts on every term except the last $\rV$ \Pth{\Refcf{} \cite[\aSec 2.1.1]{Pan:2026-laccm-2}}.  Hence,
    \[
        H^i(\proketSh{C^{(r)} \ho_{\bQ_p} \rV}) \cong H^i(\proketSh{C^{(r)}}) \ho_{\bQ_p} \rV
    \]
    is annihilated by $\cI^n$, for $n \gg 0$.  Note that $\cI^n$ also annihilates $H^{i - 1}(\proketSh{N})$, by the induction hypothesis.  The proof of Corollary \ref{cor-van-la-sys} is now complete.
\end{proof}

\begin{remark}
    Fix an isomorphism $\iota: \bC \Mi C$ and an infinitesimal character $[\wt]: Z(U(\mathfrak{g})) \to C$, not necessarily $\mu_\hd^\iota$-sufficiently regular.  By Proposition \ref{prop-suff-reg-m-vs-g}, there always exists a $\hc_\hd^\iota$-sufficiently regular character $[\wt']_{\mathfrak{m}}: Z(U(\mathfrak{m})) \to C$ extending $[\wt]$ via $\gamma^{\mathfrak{m}}$ such that
    \[
        \tilde{H}_{C, [\wt']_{\mathfrak{m}}}^{< d, \la}=0.
    \]
    Intuitively, in view of the relationship between the $Z(U(\mathfrak{m}))$-action and the $\iota|_E$-arithmetic Sen operator on $\tilde{H}_C^{< d, \la}$, fixing $[\wt]$ determines all possible Hodge--Tate--Sen weights appearing in the $[\wt]$-isotypic part, and this result says that at least one possible Hodge--Tate--Sen weight is missing in $\tilde{H}_{C, [\wt]}^{< d, \la}$.
\end{remark}

\subsection{Proof of Theorem \ref{thm-main}}

The basic idea is quite simple.  We will relate the completed cohomology for $(\Grp{G}, \Shdom)$ to the completed cohomology of another Shimura datum $(\Grp{G}_1, \Shdom_1')$ with $\Grp{G}_1^\der$ being the simply-connected cover of $\Grp{G}^\der$, in which case we can apply the result in Section \ref{sec-sc-van-res}.  Firstly, we need to understand the relationship between the usual completed cohomology and a connected version of it.  As before, we shall fix an isomorphism $\iota: \bC \Mi C$.

Same setup as in Section \ref{sec-cmpl-coh} and, in particular, Definition \ref{def-cmpl-coh}.  Fix $\levcp_p$ such that $\levcp = \levcp^p \levcp_p$ is neat, and fix a connected component $\Shdom^+$ of $\Shdom$.  Consider the \emph{neutral component} $\Sh_{\levcp, \bC}^{\circ, \an} := \Gamma_\levcp \Lquot \Shdom^+$ of $\Sh_{\levcp, \bC}^\an$, as in \Refeq{\ref{eq-dsc-circ-an}}, which forms a tower as $\levcp$ varies.

\begin{definition}\label{def-cmpl-coh-conn}
    Let $\levcp = \levcp^p \levcp_p$ be a neat level.  We define the \Pth{$i$-th} completed cohomology of the neutral component $\Sh_{\levcp, \bC}^{\circ, \an}$ as
    \[
        \tilde{H}^i(\Sh_{\levcp, \bC}^{\circ, \an}) := \varprojlim_n \varinjlim_{\levcp_p} H^i(\Sh_{\levcp^p \levcp_p, \bC}^{\circ, \an}, \bZ / p^n),
    \]
    equipped with the projective limit topology.
\end{definition}

In general, $\Grp{G}(\bQ_p)$ no longer acts on $\tilde{H}^i(\Sh_{\levcp, \bC}^{\circ, \an})$, but a subgroup of $\Grp{G}^\der(\bQ_p)$ still does.  More precisely, let $\Grp{G}^\scc$ be the simply-connected cover of $\Grp{G}^\der$, and let $\homom: \Grp{G}^\scc \to \Grp{G}^\der$ denote the central isogeny, as in Section \ref{sec-lsv-funct}.  Denote by
\[
    \homom_p := \homom(\bQ_p): \Grp{G}^\scc(\bQ_p) \to \Grp{G}^\der(\bQ_p)
\]
the induced map between $\bQ_p$-points.  By \cite[2.1.3]{Deligne:1979-vsimc}, $\im(\homom_p)$ acts trivially on the set of connected components of Shimura varieties.  It follows that there is a natural action of $\im(\homom_p)$ on $\tilde{H}^i(\Sh_\levcp^{\circ, \an})$.  Thus, an open subgroup of $\levcp_p \cap \Grp{G}^\der(\bQ_p)$ acts on $\tilde{H}^i(\Sh_\levcp^{\circ, \an})$, by the following well-known fact.
\begin{proposition}
    Let $\phi: G_1 \to G_2$ be a homomorphism of $p$-adic Lie groups such that $d\phi: \Lie(G_1) \to \Lie(G_2)$ is an isomorphism.  Then there exist open subgroups $H_1 \subset G_1$ and $H_2 \subset G_2$ such that $\phi(H_1) = H_2$ and $\phi|_{H_1}: H_1 \to H_2$ is an isomorphism of $p$-adic Lie groups.
\end{proposition}
\begin{proof}
    See \cite[\aProp 18.17]{Schneider:2011-PLG}.
\end{proof}

As explained in Section \ref{sec-Sh-var}, by Corollary \ref{cor-KZ-mod-KZ'-tower}, the action of $\levcp_p$ on the tower $\varprojlim_{\levcp_p' \subset \levcp_p} \Sh_{\levcp^p \levcp_p, \bC}^\an$ factors through the quotient $\widetilde{\levcp_p}$ given by \Refeq{\ref{eq-ex-KZ-mod-KZ'-tower-p-infty}} \Pth{with $\Grp{H}$ there replaced with $\Grp{G}$ here}.  Let $\levcp_p^\circ \subset \widetilde{\levcp_p}$ denote the stabilizer of the neutral component $\varprojlim_{\levcp_p' \subset \levcp_p} \Sh_{\levcp^p \levcp_p', \bC}^{\circ, \an}$ of $\varprojlim_{\levcp_p' \subset \levcp_p} \Sh_{\levcp^p \levcp_p', \bC}^\an$.  Our discussion above shows that there is a natural homomorphism of $p$-adic Lie groups
\[
    \levcp_p \cap \im(\homom_p) \to \levcp_p^\circ.
\]

\begin{lemma}\label{lem-K-p-circ-vs-K-p-der}
    This homomorphism is locally an isomorphism.  Equivalently, it induces $\Lie(\levcp_p^\circ) = \Lie(\Grp{G}_{\bQ_p}^\der)$.
\end{lemma}
\begin{proof}
    It suffices to prove this for sufficiently small $\levcp^p$ and $\levcp_p$.  In particular, by \cite[\aCor 2.0.13]{Deligne:1979-vsimc}, we may shrink $\levcp = \levcp^p \levcp_p$ and assume that
    \begin{equation}\label{eq-lem-K-p-circ-vs-K-p-der-prod}
        \Gamma_\levcp = \Gamma' \times \Xi
    \end{equation}
    for some neat congruence subgroup $\Gamma'$ of $\Grp{G}^\der(\bQ)$ and some congruence subgroup $\Xi$ of $\Grp{Z}^\circ(\bQ)$, where $\Grp{Z}^\circ$ denotes the neutral component of the center $\Grp{Z}$ of $\Grp{G}$.  Clearly, $\Gamma_\levcp \to \Gamma_\levcp^\ad$ induces $\Gamma' \Mi \Gamma_\levcp^\ad$.  Note that subgroups of the form $\levcp_1 M_1$, where $\levcp_1$ is an open compact subgroup of $\Grp{G}^\der(\bQ_p)$ and $M_1$ is an open compact subgroup of $\Grp{Z}(\bQ_p)$, form a system of open neighborhoods of $1 \in \Grp{G}(\bQ_p)$.  Moreover, for sufficiently small $\levcp_1$ and $M_1$, the above \Refeq{\ref{eq-lem-K-p-circ-vs-K-p-der-prod}} induces
    \[
        \Gamma_{\levcp_1 M_1} = \Gamma_\levcp \cap (\levcp_1 M_1) = (\Gamma' \cap \levcp_1) \times (\Xi \cap M_1).
    \]
    Hence, $\Gamma_{\levcp_1 M_1} \to \Gamma_{\levcp_1 M_1}^\ad$ induces $\Gamma' \cap \levcp_1 \Mi \Gamma_{\levcp_1 M_1}^\ad$.  Consequently, $\levcp_p^\circ$ is nothing but the closure of $\Gamma'$ in $\Grp{G}^\der(\bQ_p)$, which is, in particular, a subgroup of $\Grp{G}^\der(\bQ_p)$.

    On the other hand, since the kernel of the natural map $\levcp_p \to \widetilde{\levcp_p}$ is contained in $\Grp{Z}(\bQ_p)$, the induced homomorphism $\levcp_p \cap \im(\homom_p) \to \levcp_p^\circ \subset \Grp{G}^{\der}(\bQ_p)$ is injective, and hence is an isomorphism in an open neighborhood of $1 \in \Grp{G}^\der(\bQ_p)$.
\end{proof}

\begin{remark}\label{rem-cmpl-coh-conn-csg}
    It follows from the proof of Lemma \ref{lem-K-p-circ-vs-K-p-der} that, if $\Gamma_\levcp = \Gamma' \times \Xi$ for some neat congruence subgroup $\Gamma'$ of $\Grp{G}^\der(\bQ)$ and some congruence subgroup $\Xi$ of $\Grp{Z}^\circ(\bQ)$, as in \Refeq{\ref{eq-lem-K-p-circ-vs-K-p-der-prod}}, then we have
    \[
        \tilde{H}^i(\Sh_{\levcp, \bC}^{\circ, \an}) = \varprojlim_n \varinjlim_{\levcp_p^\der} H^i((\Gamma' \cap \levcp_p^\der) \Lquot \Shdom^+, \bZ / p^n),
    \]
    where $\levcp_p^\der$ runs through open compact subgroups of $\Grp{G}^\der(\bQ_p)$.
\end{remark}

\begin{proposition}\label{prop-ind-cmpl-coh}
    Let $\{ g_i \}_{i \in I}$ be any set of representatives of $\Grp{G}(\bQ)_+ \Lquot \Grp{G}(\bAi) / \levcp$.  There is a natural $\levcp_p$-equivariant isomorphism
    \[
        \tilde{H}^i \cong \bigoplus_{i \in I} \Ind_{g_i \levcp_p^\circ g_i^{-1}}^{g_i \widetilde{\levcp_p} g_i^{-1}} \tilde{H}^i(\Sh_{g_i \levcp g_i^{-1}}^{\circ, \an}),
    \]
    where $\Ind_{g_i \levcp_p^\circ g_i^{-1}}^{g_i \widetilde{\levcp_p} g_i^{-1}} \tilde{H}^i(\Sh_{g_i \levcp g_i^{-1}}^\circ)$ denote the continuous induction, \ie, continuous functions $f: g_i \widetilde{\levcp_p} g_i^{-1} \to \tilde{H}^i(\Sh_{g_i \levcp g_i^{-1}}^\circ)$ such that $f(g h) = g \cdot f(h)$, for all $g \in g_i \levcp_p^\circ g_i^{-1}$ and $h \in g_i \widetilde{\levcp_p} g_i^{-1}$, and such that $\levcp_p$ acts on it via $(l \cdot f)(h) = f(h g_i \bar{l} g_i^{-1})$, for all $l \in \levcp_p$ with image $\bar{l}$ in $\widetilde{\levcp_p}$.
\end{proposition}
\begin{proof}
    For an open subgroup $\levcp_p' \subset \levcp_p$, we denote by $\pi_{\levcp_p'}: \Sh_{\levcp^p \levcp_p', \bC}^\an \to \Sh_{\levcp, \bC}^\an$ the natural projection map.  Then $\{ \pi_{\levcp_p'}^{-1}(\Sh_{\levcp, \bC}^{\circ, \an}) \}_{\levcp_p' \subset \levcp_p}$ forms a projective system.  Since $\widetilde{\levcp_p}$ acts transitively on the set of connected components of $\varprojlim \pi_{\levcp_p'}^{-1}(\Sh_{\levcp, \bC}^{\circ, \an})$ and $\levcp_p^\circ$ is its stabilizer at $\varprojlim \Sh_{\levcp^p \levcp_p', \bC}^{\circ, \an}$, we see that there is a natural $\widetilde{\levcp_p}$-equivariant isomorphism
    \[
        \varinjlim_{\levcp_p'} H^i(\pi_{\levcp_p'}^{-1}(\Sh_{\levcp, \bC}^{\circ, \an}), \bZ / p^n) \cong \Ind_{\levcp_p^\circ}^{\widetilde{\levcp_p}}\Bigl(\varinjlim_{\levcp_p'} H^i(\Sh_{\levcp^p \levcp_p', \bC}^{\circ, \an}, \bZ / p^n)\Bigr)
    \]
    sending $s \in \varinjlim_{\levcp_p'} H^i(\pi_{\levcp_p'}^{-1}(\Sh_{\levcp, \bC}^{\circ, \an}, \bZ / p^n)$ to the function $f: h \in \widetilde{\levcp_p} \mapsto r(h \cdot s)$, where $r: \varinjlim H^i(\pi_{\levcp_p'}^{-1}(\Sh_{\levcp, \bC}^{\circ, \an}), \bZ / p^n) \to \varinjlim H^i(\Sh_{\levcp^p \levcp_p', \bC}^{\circ, \an}, \bZ / p^n)$ is the restriction map induced by the natural inclusion $\Sh_{\levcp^p \levcp_p', \bC}^{\circ, \an} \subset \pi_{\levcp_p'}^{-1}(\Sh_{\levcp, \bC}^{\circ, \an})$.

    In general, the Hecke action of $g_i$ induces an isomorphism between $\Sh_{g_i \levcp g_i^{-1}, \bC}^{\circ, \an}$ and the connected component $\Gamma_{\levcp, g_i} \Lquot \Shdom^+$ of $\Sh_{\levcp, \bC}^\an$ and identifies the completed cohomology of the tower $\{ \pi_{\levcp_p'}^{-1}(\Gamma_{\levcp, g_i} \Lquot \Shdom^+) \}_{\levcp_p' \subset \levcp_p}$ with the induction of the completed cohomology of $\Sh_{g_i \levcp g_i^{-1}, \bC}^{\circ, \an}$ from $g_i \levcp_p^\circ g_i^{-1}$ to $g_i \widetilde{\levcp_p} g_i^{-1}$.  Then the proposition follows by writing the completed cohomology $\tilde{H}^i$ as the direct sum over $i \in I$ of the completed cohomology of the towers $\{ \pi_{\levcp_p'}^{-1}(\Gamma_{\levcp, g_i} \Lquot \Shdom^+) \}_{\levcp_p' \subset \levcp_p}$.
\end{proof}

Denote by $\tilde{H}^i(\Sh_{\levcp, \bC}^{\circ, \an})_C := \tilde{H}^i(\Sh_{\levcp, \bC}^{\circ, \an}) \ho_{\bZ_p} C$, the $p$-adically completed tensor product, and by $\tilde{H}^i(\Sh_{\levcp, \bC}^{\circ, \an})_C^\la$ its subspace of $\im(\homom_p)$-locally analytic vectors.  The derived subalgebra $\mathfrak{g}^\der \cong \Lie(\Grp{G}_C^\der)$ naturally acts on it, and induces an action of $Z(U(\mathfrak{g}^\der))$ on $\tilde{H}^i(\Sh_{\levcp, \bC}^{\circ, \an})_C^\la$.  The natural restriction map $\tilde{H}_C^{i, \la} \to \tilde{H}^i(\Sh_{\levcp, \bC}^{\circ, \an})_C^\la$ is $\levcp_p^\circ$-equivariant and hence $Z(U(\mathfrak{g}^\der))$-equivariant.  Recall that there is a natural action of $Z(U(\mathfrak{m}))$ on $\tilde{H}_C^{i, \la}$.  This algebra $Z(U(\mathfrak{m}))$ has a subalgebra $Z(U(\mathfrak{m} \cap \mathfrak{g}^\der))$.

\begin{proposition}
    The action of $Z(U(\mathfrak{g}^\der))$ on $\tilde{H}^i(\Sh_{\levcp, \bC}^{\circ, \an})_C^\la$ naturally extends to an action of $Z(U(\mathfrak{m} \cap \mathfrak{g}^\der))$ such that the restriction map $\tilde{H}_C^{i, \la} \to \tilde{H}^i(\Sh_{\levcp, \bC}^{\circ, \an})_C^\la$ is $Z(U(\mathfrak{m} \cap \mathfrak{g}^\der))$-equivariant.
\end{proposition}
\begin{proof}
    The construction of the $Z(U(\mathfrak{m}))$-action works verbatim here.  One can define a $\mathfrak{g}^\der$-equivariant object $\cO^{\la, \circ} \in D^b(\cO_\Fl)$ similar to $\cO^\la$, which is annihilated by $\mathfrak{n}^0$ and whose cohomology computes $\tilde{H}^i(\Sh_{\levcp, \bC}^{\circ, \an})_C^\la$.  In fact, $\cO^{\la, \circ}$ is nothing but the pushforward along $\pi_\HT$ of the restriction of $\cO_\aSh^\la$ to the neutral component \Pth{see Section \ref{sec-geom-Sen-Sh} for the notation}.  The equivariance of $Z(U(\mathfrak{m} \cap \mathfrak{g}^\der))$ is tautological.
\end{proof}

In view of the definition of $\mu_\hd^\iota$-sufficiently regular infinitesimal characters in Definition \ref{def-suff-reg-wt-id}, we see that the ideals $\cI^\iota \subset Z(U(\mathfrak{g}))$ and $\cI_{\mathfrak{m}}^\iota \subset Z(U(\mathfrak{m}))$ are generated by $\cI^{\iota, \der} := \cI^\iota \cap Z(U(\mathfrak{g}^\der))$ and $\cI_{\mathfrak{m}}^{\iota, \der} := \cI_{\mathfrak{m}}^\iota \cap Z(U(\mathfrak{m} \cap \mathfrak{g}^\der))$, respectively.

We will prove the following connected version of Theorem \ref{thm-main}:
\begin{theorem}\label{thm-main-conn}
    $(\cI^{\iota, \der})^n$ and $(\cI_{\mathfrak{m}}^{\iota, \der})^n$ annihilate $\tilde{H}^{< d}(\Sh_{\levcp, \bC}^{\circ, \an})_C^\la$, for some $n > 0$.
\end{theorem}

\begin{remark}
    By Proposition \ref{prop-suff-reg-m-vs-g}, it suffices to show that $(\cI_{\mathfrak{m}}^{\iota, \der})^n$ annihilates $\tilde{H}^{< d}(\Sh_{\levcp, \bC}^{\circ, \an})_C^\la$, for some $n > 0$.
\end{remark}

\begin{lemma}
    Theorems \ref{thm-main-conn} and \ref{thm-main} imply each other.
\end{lemma}
\begin{proof}
    It follows from the proof of Proposition \ref{prop-ind-cmpl-coh} that the isomorphism there is $Z(U(\mathfrak{m} \cap \mathfrak{g}^\der))$-equivariant.
\end{proof}

By Remark \ref{rem-thm-main-spec}, Theorem \ref{thm-sc-van-m} and Corollary \ref{cor-sc-van-g} imply the following:
\begin{corollary}\label{cor-main-thm-cond-sc}
    Theorem \ref{thm-main-conn} holds when Condition \ref{cond-der-sc-no-need-c} holds.
\end{corollary}

In order to prove Theorem \ref{thm-main-conn} in general, we may shrink the level if necessary:
\begin{lemma}\label{lem-fin-cov-coh-comp}
    Let $\Gamma_1$ be a subgroup of $\Gamma_\levcp$ of finite index $N$, and $\levcp'$ an open subgroup of $\levcp$.  Then $(\Gamma_{\levcp'} \cap \Gamma_1) \Lquot \Shdom^+$ is a finite cover of $\Gamma_{\levcp'} \Lquot \Shdom^+$, and the kernel of the natural map $H^i(\Gamma_{\levcp'} \Lquot \Shdom^+, \bZ / p^n) \to H^i((\Gamma_{\levcp'} \cap \Gamma_1) \Lquot \Shdom^+, \bZ / p^n)$ is $N!$-torsion.
\end{lemma}
\begin{proof}
    There is a natural injection of sets of cosets $\Gamma_{\levcp'} / (\Gamma_{\levcp'} \cap \Gamma_1) \to \Gamma_\levcp / \Gamma_1$.  Thus, the degree of the covering map $(\Gamma_{\levcp'} \cap \Gamma_1) \Lquot \Shdom^+ \to \Gamma_{\levcp'} \Lquot \Shdom^+$ is bounded by $N$ and hence divides $N!$.  Our claim follows as the kernel of the induced map of cohomology groups is annihilated by the degree by the standard argument using the trace map \Pth{or restriction-corestriction sequence}.
\end{proof}

\begin{corollary}\label{cor-fin-cov-coh-comp}
    The natural map $\tilde{H}^i(\Sh_{\levcp, \bC}^{\circ, \an})_C \to \tilde{H}^i(\Sh_{\levcp_1, \bC}^{\circ, \an})_C$ is injective, for any open subgroup $\levcp_1 = \levcp_1^p \levcp_{1, p}$ of $\levcp$, by taking $\Gamma_1 = \Gamma_{\levcp_1}$ and $\levcp' = \levcp^p \levcp_p'$, where $\levcp_p'$ varies over open subgroups of $\levcp_p$.
\end{corollary}

Let us now begin the proof of Theorem \ref{thm-main-conn}.  By Corollary \ref{cor-fin-cov-coh-comp}, we may shrink $\levcp$ and assume as in Remark \ref{rem-cmpl-coh-conn-csg} that
\[
    \Gamma_\levcp = \Gamma' \times \Xi,
\]
for some neat congruence subgroup $\Gamma'$ of $\Grp{G}^\der(\bQ)$ and some congruence subgroup $\Xi$ of $\Grp{Z}^\circ(\bQ)$ \Pth{by \cite[\aCor 2.0.13]{Deligne:1979-vsimc}}.  As in the proof of Lemma \ref{lem-K-p-circ-vs-K-p-der}, for sufficiently small open compact subgroups $\levcp_1$ of $\Grp{G}^\der(\bQ_p)$ and $M_1$ of $\Grp{Z}(\bQ_p)$, the canonical homomorphism $\Gamma_\levcp \to \Gamma_\levcp^\ad$ induces
\[
    \Gamma' \cap \levcp_1 \Mi \Gamma_{\levcp^p \levcp_1 M_1}^\ad.
\]
Recall that $\homom: \Grp{G}^\scc \to \Grp{G}^\der$ denotes the simply-connected cover of $\Grp{G}^\der$.  By Lemma \ref{lem-comp-ad}, $\Grp{G}^\scc$ and $\Grp{G}^\der$ have the same adjoint quotient $\Grp{G}^\ad$.  By \cite[\aCor 5.8]{Milne:2005-isv}, $(\Grp{G}^\der, \Shdom^+)$ is a connected Shimura datum, as in \cite[\aProp 4.8]{Milne:2005-isv}.  It follows that $(\Grp{G}^\scc, \Shdom^+)$ is a connected Shimura datum as well.

\begin{lemma}\label{lem-ext-G-sc}
    We can extend $(\Grp{G}^\scc, \Shdom^+)$ to a Shimura datum $(\tilde{\Grp{G}}, \tilde{\Shdom})$ such that $\tilde{\Grp{G}}^\der = \Grp{G}^\scc$ and $\tilde{\Grp{G}} \Mi \tilde{\Grp{G}}^c$.  In this case, $(\tilde{\Grp{G}}, \tilde{\Shdom})$ satisfies Condition \ref{cond-der-sc-no-need-c}.
\end{lemma}
\begin{proof}
    By \cite[\aProp 8.5]{Milne:2013-svm}, we can extend $(\Grp{G}^\scc, \Shdom^+)$ to a Shimura datum $(\tilde{\Grp{G}}_0, \tilde{\Shdom}_0)$ with $\tilde{\Grp{G}}_0^\der = \Grp{G}^\scc$ such that the image of any $\tilde{\hd}_0: \Res_{\bC / \bR} \bG_{m, \bC} \to \tilde{\Grp{G}}_{0, \bR}$ \Pth{parameterized by $\tilde{\Shdom}_0$ as in Section \ref{sec-Sh-var}} in $(\tilde{\Grp{G}}_0 / \tilde{\Grp{G}}_0^\der)_\bR$ is anisotropic modulo the image of its weight cocharacter defined over $\bQ$.  Let $\Grp{T}$ be the smallest subtorus of $\tilde{\Grp{G}}_0 / \tilde{\Grp{G}}_0^\der$ such that $\Grp{T}_\bR$ contains this image of $\tilde{\hd}_0$ in $(\tilde{\Grp{G}}_0 / \tilde{\Grp{G}}_0^\der)_\bR$.  Let $\tilde{\Grp{G}}$ be the preimage of $\Grp{T}$ in $\tilde{\Grp{G}}_0$, and replace $\tilde{\Shdom}_0$ with the $\tilde{\Grp{G}}(\bR)$-orbit of $\tilde{\hd}_0$.  Then $Z_s(\tilde{\Grp{G}})$ is trivial, and $\tilde{\Grp{G}} \Mi \tilde{\Grp{G}}^c$, as desired.  \Pth{This is essentially the same argument as in the proof of \cite[\aCor 8.6]{Milne:2013-svm}, but there is a minor mistake there which missed the quotient by the image of a weight cocharacter defined over $\bQ$.}
\end{proof}

By \cite[\aCor 2.0.13]{Deligne:1979-vsimc}, we can find a neat open compact subgroup $\levcp^\scc = \levcp^{\scc, p} \levcp_p^\scc$ of $\tilde{\Grp{G}}(\bAi)$ such that $\homom(\bAi)(\levcp^\scc \cap \Grp{G}^\scc(\bAi)) \subset \levcp$, and
\[
    \Gamma_{\levcp^\scc} = \Gamma^{\scc, \prime} \times \Xi',
\]
for some neat congruence subgroup $\Gamma^{\scc, \prime}$ of $\Grp{G}^\scc(\bQ)$ and congruence subgroup $\Xi'$ of $\Grp{Z}(\tilde{\Grp{G}})^\circ(\bQ)$.  Note that $\varrho(\bQ)|_{\Gamma^{\scc, \prime}}$ is injective by the neatness, and
\[
    \homom(\bQ)(\Gamma^{\scc, \prime}) \subset \Gamma',
\]
which has finite index, say $N$, by a result of Borel's \cite[\aThm 8.9]{Borel:2019-IAG}.  This gives a natural finite covering map $\Sh_{\levcp^\scc, \bC}^{\circ, \an} \to \Sh_{\levcp, \bC}^{\circ, \an}$, which canonically algebraizes to a finite \'etale map $\Sh_{\levcp^\scc}^\circ \to \Sh_\levcp^\circ$, as in Construction \ref{constr-dcs-comp-funct}.  Suppose that $\levcp_p^\der$ is a torsionfree open compact subgroup of $\Grp{G}^\scc(\bQ_p)$.  Then $\homom_p|_{\levcp_p^\der}$ is injective, and its image is an open compact subgroup of $\Grp{G}^\der(\bQ_p)$.  Hence, the natural map of cosets
\[
    (\Gamma' \cap \homom_p(\levcp_p^\der)) / \homom(\bQ)(\Gamma^{\scc, \prime} \cap \levcp_p^\der) \to \Gamma' / \homom(\bQ)(\Gamma^{\scc, \prime})
\]
is injective.  The same argument as in the proof of Lemma \ref{lem-fin-cov-coh-comp} shows that the natural map
\[
    H^i\bigl((\Gamma' \cap \homom_p(\levcp_p^\der)) \Lquot \Shdom^+, \bZ / p^n\bigr) \to H^i\bigl((\Gamma^{\scc, \prime} \cap \levcp_p^\der) \Lquot \Shdom^+, \bZ / p^n\bigr)
\]
is injective modulo $N!$-torsion.  Therefore, by Remark \ref{rem-cmpl-coh-conn-csg}, the induced map
\begin{equation}\label{eq-cmpl-coh-conn-scc}
    \tilde{H}^i(\Sh_{\levcp, \bC}^{\circ, \an})_C \to \tilde{H}^i(\Sh_{\levcp^\scc, \bC}^{\circ, \an})_C
\end{equation}
is injective.  Since $\Grp{G}_\bC^\scc$ is simply-connected as the notion of simply-connectedness does not depend on the base field \Pth{see \cite[\aCor A.4.11]{Conrad/Gabber/Prasad:2010-PRG}}, the desired vanishing result for $\tilde{H}^i(\Sh_{\levcp^\scc, \bC}^{\circ, \an})_C^\la$ is known by Corollary \ref{cor-main-thm-cond-sc}.  Thus, in order to prove Theorem \ref{thm-main-conn}, it suffices to prove the following.

\begin{lemma}\label{lem-Z-U-m-der-equivar}
    The map $\tilde{H}^i(\Sh_{\levcp, \bC}^{\circ, \an})_C^\la \to \tilde{H}^i(\Sh_{\levcp^\scc, \bC}^{\circ, \an})_C^\la$ canonically induced by \Refeq{\ref{eq-cmpl-coh-conn-scc}} is $Z(U(\mathfrak{m} \cap \mathfrak{g}^\der))$-equivariant.
\end{lemma}
\begin{proof}
    We will see that, in view of the construction of the $Z(U(\mathfrak{m} \cap \mathfrak{g}^\der))$-action, this is a question about the compatibility of Hodge--Tate period morphisms for $\Grp{G}$ and $\tilde{\Grp{G}}$.  More precisely, using the notation in Section \ref{sec-HT-mor}, with compatible choices of toroidal compactifications made there, consider the topological space
    \[
        \aSh_{\levcp^p}^{\circ, \Tor} := \varprojlim_{\levcp_p' \subset \levcp_p} |\aSh_{\levcp^p \levcp_p'}^{\circ, \Tor}| \subset \aSh_{\levcp^p}^\Tor
    \]
    where $\aSh_{\levcp^p \levcp_p'}^{\circ, \Tor}$ denotes the neutral component of $\aSh_{\levcp^p \levcp_p'}^\Tor$.  One can define a $\levcp_p^\circ$-equivariant sheaf of $C$-algebras $\cO_{\aSh^\circ}^\la$ as in Definition \ref{def-HT-mor-la} from the pro-Kummer-\'etale $\levcp_p^\circ$-torsor $\varprojlim_{\levcp_p' \subset \levcp_p} \aSh_{\levcp^p \levcp_p'}^{\circ, \Tor}$.  It inherits a Hodge--Tate period morphism
    \[
        \pi_\HT^{\circ, \Tor}: (\aSh_{\levcp^p}^{\circ, \Tor}, \cO_{\aSh^\circ}^\la) \to \Fl
    \]
    from $\aSh_{\levcp^p}^\Tor$.  With compatible choices of toroidal compactifications, by Proposition \ref{prop-lsv-funct}, we have a natural map
    \[
        \varprojlim_{\levcp_p^{\scc, \prime} \subset \levcp_p^\scc} \Sh_{\levcp^{\scc, p} \levcp_p^{\scc, \prime}}^{\circ, \Tor} \to \varprojlim_{\levcp_p' \subset \levcp_p} \Sh_{\levcp^p \levcp_p'}^{\circ, \Tor}
    \]
    equivariant for the open compact subgroup $\levcp_p^{\scc, \circ}$ of $\Grp{G}^\scc(\bQ_p)$.  We can similarly define the ringed site $(\aSh_{\levcp^{\scc, p}}^{\circ, \Tor}, \cO_{\aSh^{\scc, \circ}}^\la)$, with a $\mathfrak{g}^\der$-equivariant map of ringed sites
    \[
        \phi: (\aSh_{\levcp^{\scc, p}}^{\circ, \Tor}, \cO_{\aSh^{\scc, \circ}}^\la) \to (\aSh_{\levcp^p}^{\circ, \Tor}, \cO_{\aSh^\circ}^\la)
    \]
    induced by the above natural map, and with its Hodge--Tate period morphism
    \[
        \pi_\HT^{\scc, \circ, \Tor}: (\aSh_{\levcp^{\scc, p}}^{\circ, \Tor}, \cO_{\aSh^{\scc, \circ}}^\la) \to \Fl.
    \]
    Note that both $\pi_\HT^{\circ, \Tor}$ and $\pi_\HT^{\scc, \circ, \Tor}$ have the same target, because it only depends on the shared adjoint Shimura datum $(\Grp{G}^\ad, \Shdom^\ad) = (\tilde{\Grp{G}}^\ad, \tilde{\Shdom}^\ad)$, by Lemma \ref{lem-comp-ad-isog} and Remark \ref{rem-HT-mor-comp}.  Since the $Z(U(\mathfrak{m} \cap \mathfrak{g}^\der))$-action is extracted from the action of $\cO_\Fl \otimes_C \mathfrak{g}^\der$, where $\cO_\Fl$ acts via the Hodge--Tate period morphism \Pth{see Section \ref{sec-hor-act}}, it remains to prove the following:

    \begin{lemma}\label{lem-HT-mor-tor-conn-scc-compat}
        For sufficiently small $\levcp_p$, one can compatibly choose toroidal compactifications such that $\pi_\HT^{\scc, \circ, \Tor} = \pi_\HT^{\circ, \Tor} \circ \phi$.
    \end{lemma}
    \begin{proof}[Proof of Lemma \ref{lem-HT-mor-tor-conn-scc-compat}]
        Consider $(\Grp{G}^\ad, \Shdom^\ad) = (\tilde{\Grp{G}}^\ad, \tilde{\Shdom}^\ad)$, the adjoint Shimura datum associated with both $(\Grp{G}, \Shdom)$ and $(\tilde{G}, \tilde{\Shdom})$.  Up to shrinking $\levcp_p$ and hence $\levcp_p^\ad$ if necessary, choose a neat level $\levcp^\ad = \levcp^{\ad, p} \levcp_p^\ad$ such that the natural map $(\Grp{G}, \Shdom) \to (\Grp{G}^\ad, \Shdom^\ad)$ induces a morphism of Shimura varieties $\Sh_\levcp \to \Sh_{\levcp^\ad}$, and such that this morphism can be extended to some morphism of smooth toroidal compactifications with normal crossing boundary divisors $\aSh_\levcp^\Tor \to \aSh_{\levcp^\ad}^\Tor$, up to refinement of cone decompositions for $\Sh_\levcp$, as in Proposition \ref{prop-lsv-funct}.  One can define $(\aSh_{\levcp^{\ad, p}}^{\ad, \circ, \Tor}, \cO_{\aSh^{\ad, \circ}}^\la)$ and $\pi_\HT^{\ad, \circ, \Tor}: (\aSh_{\levcp^{\ad, p}}^{\circ, \Tor}, \cO_{\aSh^{\ad, \circ}}^\la) \to \Fl$, and similarly for $\aSh_{\levcp^\ad}^{\circ, \Tor}$.  There are natural morphisms
        \[
            (\aSh_{\levcp^{\scc, p}}^{\circ, \Tor}, \cO_{\aSh^{\scc, \circ}}^\la) \Mapn{\phi} (\aSh_{\levcp^p}^{\circ, \Tor}, \cO_{\aSh^\circ}^\la) \to (\aSh_{\levcp^{\ad, p}}^{\ad, \circ, \Tor}, \cO_{\aSh^{\ad, \circ}}^\la),
        \]
        whose composition is the canonical morphism induced by the canonical map of Shimura data $(\tilde{\Grp{G}}, \tilde{\Shdom}) \to (\tilde{\Grp{G}}^\ad, \tilde{\Shdom}^\ad) = (\Grp{G}^\ad, \Shdom^\ad)$.  By Remark \ref{rem-HT-mor-la-comp}, both $\pi_\HT^{\scc, \circ, \Tor}$ and $\pi_\HT^{\circ, \Tor}$ factor through $(\aSh_{\levcp^{\ad, p}}^{\ad, \circ, \Tor}, \cO_{\aSh^{\ad, \circ}}^\la)$, and Lemma \ref{lem-HT-mor-tor-conn-scc-compat} follows.
    \end{proof}

    The proofs of Lemma \ref{lem-Z-U-m-der-equivar} and hence of Theorem \ref{thm-main-conn} are now complete.
\end{proof}

\section*{Acknowledgements}

Firstly, we would like to thank Bhargav Bhatt for providing us with his powerful Kodaira-type almost vanishing result in mixed characteristics and for many useful discussions at various stages of this work.  It should be clear that this paper could not have existed without his help!

Secondly, we would like to thank George Boxer and Vincent Pilloni for helpful comments and for pointing out a mistake in an earlier version of our work.

K.-W.L.\@\xspace would like to thank the National Center for Theoretical Sciences (NCTS) in Taipei, Kyoto University, Princeton University, Institute for Advanced Study in Mathematics (IASM) at Zhejiang University, and Academia Sinica for their hospitality during the preparation of this work.  L.P.\@\xspace would like to thank Chenyang Xu and Ziquan Zhuang for helpful discussions on Kodaira-type vanishing results.  Both of us would like to thank the Simons Foundation for its support during various stages of this collaboration.  L.P.\@\xspace was partially supported by a Simons Junior Faculty Fellows award from the Simons Foundation and a Sloan Research Fellowship.

%\bibliographystyle{amsalpha}
%\bibliography{van-cmpl-coh}

\providecommand{\bysame}{\leavevmode\hbox to3em{\hrulefill}\thinspace}
\providecommand{\MR}{\relax\ifhmode\unskip\space\fi MR }
\providecommand{\MRhref}[2]{%
  \href{http://www.ams.org/mathscinet-getitem?mr=#1}{#2}
}
\providecommand{\href}[2]{#2}

\end{document}